\pdfoutput=1
\newif\ifpersonal
\documentclass[11pt,a4paper,dvipsnames,x11names]{amsart} 

\usepackage{amsmath,amsthm,amssymb,mathrsfs,mathtools,tensor,eucal} 
\usepackage[all,cmtip]{xy}

\usepackage[LGR,T1]{fontenc} 
\usepackage[utf8]{inputenc} 
\usepackage{CJKutf8}

\usepackage{xr}
\usepackage{appendix}

\usepackage{microtype,inconsolata} 
\usepackage{enumerate,comment,braket,xspace,tikz,tikz-cd,csquotes} 
\usetikzlibrary{shapes.geometric}
\usepackage[centering,vscale=0.7,hscale=0.7]{geometry}
\usepackage[hidelinks,colorlinks=true,linkcolor=Red4,citecolor=RoyalBlue]{hyperref}
\usepackage[capitalize]{cleveref}
\usepackage{xcolor}

\usepackage{mathpazo}
\usepackage{euler}
\usepackage{tikzarrows}

\tikzcdset{
	cells={font=\everymath\expandafter{\the\everymath\displaystyle}},
}

\numberwithin{equation}{subsection}
\newcounter{equation-intro}[section]
\numberwithin{equation-intro}{section}

\theoremstyle{plain}
\newtheorem{thm}[equation]{Theorem}
\newtheorem{lem}[equation]{Lemma}
\newtheorem{prop}[equation]{Proposition}

\newtheorem{cor}[equation]{Corollary}
\newtheorem{assumption}[equation]{Assumption}

\newtheorem{thm-intro}[equation-intro]{Theorem}
\newtheorem{slogan-intro}[equation-intro]{Slogan}

\theoremstyle{definition}

\newtheorem{defin}[equation]{Definition}
\newtheorem{notation}[equation]{Notation}
\newtheorem{eg}[equation]{Example}
\newtheorem{rem}[equation]{Remark}
\newtheorem{variant}[equation]{Variant}
\newtheorem{observation}[equation]{Observation}
\newtheorem{recollection}[equation]{Recollection}
\newtheorem{warning}[equation]{Warning}
\newtheorem{construction}[equation]{Construction}

\newtheorem{defin-intro}[equation-intro]{Definition}
\newtheorem{eg-intro}[equation-intro]{Example}
\newtheorem{rem-intro}[equation-intro]{Remark}

\newcommand{\todo}[1]{\textcolor{red}{(Todo: #1)}}

\ifpersonal
\newcommand{\personal}[1]{\textcolor[rgb]{0,0,1}{(Personal: #1)}}
\newcommand{\discussion}[1]{\textcolor{violet}{(Discussion: #1)}}
\else
\newcommand{\personal}[1]{\ignorespaces}
\newcommand{\discussion}[1]{\ignorespaces}
\fi

\newcommand{\Z}{\mathbb Z}

\newcommand{\cA}{\mathcal A}
\newcommand{\cB}{\mathcal B}
\newcommand{\cC}{\mathcal C}
\newcommand{\cD}{\mathcal D}
\newcommand{\cE}{\mathcal E}

\newcommand{\cI}{\mathcal I}
\newcommand{\cJ}{\mathcal J}
\newcommand{\cK}{\mathcal K}
\newcommand{\cL}{\mathcal L}
\newcommand{\cM}{\mathcal M}

\newcommand{\cO}{\mathcal O}

\newcommand{\cT}{\mathcal T}

\newcommand{\cX}{\mathcal X}
\newcommand{\cY}{\mathcal Y}
\newcommand{\cZ}{\mathcal Z}
\DeclareFontFamily{U}{BOONDOX-calo}{\skewchar\font=45 }
\DeclareFontShape{U}{BOONDOX-calo}{m}{n}{<-> s*[1.05] BOONDOX-r-calo}{}
\DeclareFontShape{U}{BOONDOX-calo}{b}{n}{<-> s*[1.05] BOONDOX-b-calo}{}
\DeclareMathAlphabet{\mathcalboondox}{U}{BOONDOX-calo}{m}{n}

\newcommand{\bbI}{\mathbb I}

\makeatletter
\let\save@mathaccent\mathaccent
\newcommand*\if@single[3]{%
	\setbox0\hbox{${\mathaccent"0362{#1}}^H$}%
	\setbox2\hbox{${\mathaccent"0362{\kern0pt#1}}^H$}%
	\ifdim\ht0=\ht2 #3\else #2\fi
}
\newcommand*\rel@kern[1]{\kern#1\dimexpr\macc@kerna}
\newcommand*\widebar[1]{\@ifnextchar^{{\wide@bar{#1}{0}}}{\wide@bar{#1}{1}}}
\newcommand*\wide@bar[2]{\if@single{#1}{\wide@bar@{#1}{#2}{1}}{\wide@bar@{#1}{#2}{2}}}
\newcommand*\wide@bar@[3]{%
	\begingroup
	\def\mathaccent##1##2{%
		\let\mathaccent\save@mathaccent
		\if#32 \let\macc@nucleus\first@char \fi
		\setbox\z@\hbox{$\macc@style{\macc@nucleus}_{}$}%
		\setbox\tw@\hbox{$\macc@style{\macc@nucleus}{}_{}$}%
		\dimen@\wd\tw@
		\advance\dimen@-\wd\z@
		\divide\dimen@ 3
		\@tempdima\wd\tw@
		\advance\@tempdima-\scriptspace
		\divide\@tempdima 10
		\advance\dimen@-\@tempdima
		\ifdim\dimen@>\z@ \dimen@0pt\fi
		\rel@kern{0.6}\kern-\dimen@
		\if#31
		\overline{\rel@kern{-0.6}\kern\dimen@\macc@nucleus\rel@kern{0.4}\kern\dimen@}%
		\advance\dimen@0.4\dimexpr\macc@kerna
		\let\final@kern#2%
		\ifdim\dimen@<\z@ \let\final@kern1\fi
		\if\final@kern1 \kern-\dimen@\fi
		\else
		\overline{\rel@kern{-0.6}\kern\dimen@#1}%
		\fi
	}%
	\macc@depth\@ne
	\let\math@bgroup\@empty \let\math@egroup\macc@set@skewchar
	\mathsurround\z@ \frozen@everymath{\mathgroup\macc@group\relax}%
	\macc@set@skewchar\relax
	\let\mathaccentV\macc@nested@a
	\if#31
	\macc@nested@a\relax111{#1}%
	\else
	\def\gobble@till@marker##1\endmarker{}%
	\futurelet\first@char\gobble@till@marker#1\endmarker
	\ifcat\noexpand\first@char A\else
	\def\first@char{}%
	\fi
	\macc@nested@a\relax111{\first@char}%
	\fi
	\endgroup
}
\makeatother

\newcommand{\sSet}{\mathrm{sSet}}

\newcommand{\St}{\mathrm{St}}

\newcommand{\CAlg}{\mathrm{CAlg}}

\newcommand{\fib}{\mathrm{fib}}

\DeclareMathOperator{\LSt}{LSt}
\DeclareMathOperator{\RSt}{RSt}

\DeclareMathOperator{\Gr}{Gr}
\DeclareMathOperator{\expGr}{expGr}
\DeclareMathOperator{\ens}{set}

\DeclareMathOperator{\cart}{cart} 
\DeclareMathOperator{\cocart}{cocart} 
\DeclareMathOperator{\Fil}{Fil} 
\DeclareMathOperator{\Sing}{Sing}
\DeclareMathOperator{\spe}{sp}
\DeclareMathOperator{\SPE}{SP}

\DeclareMathOperator{\Loc}{Loc}

\DeclareMathOperator{\PS}{PS}

\DeclareMathOperator{\Env}{Env}
\DeclareMathOperator{\order}{ord}

\DeclareMathOperator{\Triv}{Triv}
\newcommand{\PrLR}{\categ{Pr}^{\kern0.05em\operatorname{L}, \operatorname{R}}}
\newcommand{\PrFibLR}{\categ{PrFib}^{\mathrm{L},\mathrm{R}}}

\newcommand{\Funcocart}{\Fun^{\cocart}}
\newcommand{\fcartlowershriek}{f^{\mathrm c}_!}
\newcommand{\fcartlowerstar}{f^{\mathrm c}_\ast}
\newcommand{\fcocartlowershriek}{f^{\mathrm{cc}}_!}

\newcommand{\fcocartlowerstar}{f^{\mathrm{cc}}_\ast}

\DeclareMathOperator{\Tw}{Tw}

\newcommand{\ev}{\mathrm{ev}}

\newcommand{\id}{\mathrm{id}}

\newcommand{\op}{^\mathrm{op}}

\usetikzlibrary{decorations.markings} 
\tikzset{
  closed/.style = {decoration = {markings, mark = at position 0.5 with { \node[transform shape, xscale = .8, yscale=.4] {/}; } }, postaction = {decorate} },
  open/.style = {decoration = {markings, mark = at position 0.5 with { \node[transform shape, scale = .7] {$\circ$}; } }, postaction = {decorate} }
}

\DeclareMathOperator{\cofib}{cofib}

\DeclareMathOperator{\Fun}{Fun}

\DeclareMathOperator{\Image}{Im}

\DeclareMathOperator{\Map}{Map}
\DeclareMathOperator{\Mor}{Mor}

\DeclareMathOperator*{\colim}{colim}

\newcommand{\categ}[1]{\textbf{\textup{#1}}}

\newcommand{\Spc}{\categ{Spc}}

\newcommand{\PrR}{\categ{Pr}^{\kern0.05em\operatorname{R}}}
\newcommand{\PrL}{\categ{Pr}^{\kern0.05em\operatorname{L}}}
\newcommand{\PrLomega}{\categ{Pr}^{\kern0.05em\operatorname{L},\omega}}
\newcommand{\PrLkappa}{\categ{Pr}^{\kern0.05em\operatorname{L},\kappa}}
\newcommand{\PrLRomega}{\categ{Pr}^{\kern0.05em\operatorname{L}, \operatorname{R},\omega}}
\newcommand{\PrLotimes}{\categ{Pr}^{\kern0.05em\operatorname{L},\otimes}}
\newcommand{\PrLat}{\categ{Pr}^{\kern0.05em\operatorname{L}, \operatorname{at}}}
\newcommand{\Cat}{\categ{Cat}}
\newcommand{\CAT}{\textbf{\textsc{Cat}}}

\newcommand{\PosFib}{\categ{PosFib}}

\newcommand{\hCoCart}{\textbf{\textsc{CoCart}}}

\newcommand{\Poset}{\categ{Poset}}
\newcommand{\Mod}{\mathrm{Mod}}

\newcommand{\Cart}{\categ{Cart}}
\newcommand{\CoCart}{\categ{CoCart}}

\newcommand{\PrFibL}{\categ{PrFib}^{\mathrm{L}}}

\newcommand{\PrFibLotimes}{\categ{PrFib}^{\mathrm{L},\otimes}}

\newcommand{\PPSh}{\mathbb P\mathrm{Sh}}

\begin{document}

\title{Homotopy theory of Stokes data}

\author{Mauro PORTA}
\address{Mauro PORTA, Institut de Recherche Mathématique Avancée, 7 Rue René Descartes, 67000 Strasbourg, France}
\email{porta@math.unistra.fr}

\author{Jean-Baptiste Teyssier}
\address{Jean-Baptiste Teyssier, Institut de Mathématiques de Jussieu, 4 place Jussieu, 75005 Paris, France}
\email{jean-baptiste.teyssier@imj-prg.fr}

\subjclass[2020]{}

\begin{abstract}
In this paper we lay the foundations of an $\infty$-categorical  theory of Stokes data.
\end{abstract}

\maketitle


\tableofcontents

\section{Introduction}
Given an algebraic flat bundle $(E,\nabla)$ on $\mathbb{C}\setminus \{P_1,\dots, P_n\}$,  several algebraic structures arising from its flat sections can be defined.
The most basic one is the \textit{monodromy representation}, obtained by analytic continuation of the flat sections along loops in $\mathbb{C}\setminus \{P_1,\dots, P_n\}$.
A more sophisticated algebraic structure governing the asymptotic behaviours of the flat sections near the $P_i$ was discovered by Stokes \cite{Stokes_2009} and  formalized by Deligne and Malgrange \cite{DeligneLettreMalgrange,BV} by means of \textit{Stokes filtered local systems}.
See also \cite{SVDP,Boalch_Topology_Stokes} for alternative viewpoints in dimension 1.
In a punctured neighbourhood of $P_i$,  a Stokes filtered local system consists in a local system $L$ whose germs are filtered by an ordered set varying in a constructible way.
In a nutshell,  $L$ encodes the flat sections of $(E,\nabla)$ and the filtrations pertain to their growth order  when approaching $P_i$.
Stokes filtered local systems were discovered by Mochizuki in any dimension \cite{Mochizuki2}.\medskip

During the quest for derived moduli for higher dimensional Stokes filtered local systems \cite{Geometric_Stokes},  we realized that they could be viewed as particular functors that we called \textit{Stokes functors} and that we describe now in more detail. 
Let $X$ be a complex manifold admitting a smooth compactification.
	Let $D$ be a normal crossing divisor in $X$ and put $U\coloneqq X\setminus D$.
	Let $\pi \colon \widetilde{X}\to X$  be the real-blow up along $D$   and $j \colon U \to  \widetilde{X}$ the inclusion.
	Let $\mathscr{I} \subset \cO_{X}(\ast D)/ \cO_{X}$ be a good sheaf of irregular values in the sense of \cite{Mochizuki2}.
	A point $x\in \widetilde{X}$ with $\pi(x)\in D$ can be thought of as a line passing through $\pi(x)$ and a section of $\pi^{\ast}\mathscr{I}$ near $x$ as a meromorphic function defined on some small multi-sector emanating from $\pi(x)$.
	For two such sections $a$ and $b$, the relation
	\[
	\text{$a\leq_x b$ if and only if $e^{a-b}$ has moderate growth at $x$}
	\]
	defines an order on the germs of $\pi^{\ast}\mathscr{I}$ at $x$.
	This collection of orders upgrades $\pi^{\ast}\mathscr{I}$ into a sheaf of posets that turns out to be constructible for a suitable choice of finite subanalytic stratification $P$ of $\widetilde{X}$.
Then,  the \textit{topological exodromy equivalence} from \cite{Lurie_Higher_algebra,Exodromy_coefficient,jansen2023stratified,Beyond_conicality} converts $\pi^{\ast}\mathscr{I}$ into a functor $\Pi_{\infty}(\widetilde{X},P)\to \Poset$ from the $\infty$-category of Exit Paths $\Pi_{\infty}(\widetilde{X},P)$ attached to $(\widetilde{X},P)$.
By design, the objects of $\Pi_{\infty}(\widetilde{X},P)$ are the points of $\widetilde{X}$ and the morphisms between two points $x$ and $y$ can be thought of as continuous paths $\gamma \colon  [0,1]\to \widetilde{X}$ such that $\gamma((0,1])$ lies in the same stratum as $y$.
Via the Grothendieck construction,  the functor $\Pi_{\infty}(\widetilde{X},P)\to \Poset$   corresponds  to a cocartesian fibration in posets $\cI \to \Pi_{\infty}(\widetilde{X},P)$.
In this language, Stokes filtered local systems are special functors $F \colon  \cI \to \cE$ where $\cE$ is the category of $\mathbb{C}$-vector spaces.
A substantial part of the present work is devoted to the $\infty$-categorical analysis of the two conditions that make these functors special.
\subsection*{Splitting condition}
This condition is punctual.
For $x\in \widetilde{X}$,  let $\cI_x \in \Poset$ be the fibre of $\cI \to \Pi_{\infty}(\widetilde{X},P)$ above $x$ and consider the restricted functor $F_x \colon  \cI_x \to \cE$. 
Let $i_{\cI_x} \colon  \cI_x^{\ens}\to \cI_x$ be the underlying set of $\cI_x$.
Let $i_{\cI_x, ! } \colon   \Fun(\cI_x^{\ens},\cE)\to \Fun(\cI_x,\cE)$ be the left Kan extension of  $i_{\cI_x }^* \colon   \Fun(\cI_x,\cE)\to \Fun(\cI_x^{\ens},\cE)$.
Then $F_x$ is requested to lie in the essential image of $i_{\cI_x, ! }$.
Unravelling the definition, this means that there is   $V \colon   \cI_x^{\ens}\to \cE$ such that for  every $a\in \cI_x$, we have
\[
F_x(a)\simeq \bigoplus_{b\leq a \text{ in }\cI_x} V(b)   \ .
\]

\subsection*{Induction condition}
If $\gamma \colon  x\to y$ is an exit path for $(\widetilde{X},P)$, it pertains to a prescription of $F_y$  by $F_x$ via $\gamma$ referred  as \textit{induction} in \cite{Mochizuki1}.
If $\gamma \colon  \cI_x \to \cI_y$ is the  morphism of posets 
induced by  $\gamma \colon  x\to y$ and if $
\gamma_{ ! } \colon   \Fun(\cI_x,\cE)\to \Fun(\cI_y,\cE)$ is the 
left Kan extension of the pull-back $\gamma^* \colon   
\Fun(\cI_y,\cE)\to \Fun(\cI_x,\cE)$, Mochizuki's 
condition translates purely categorically into the requirement that the natural map $\gamma_!(F_x)\to F_y$ is an equivalence. \medskip

The splitting and induction conditions have purely categorical formulations.
This motivates the following 
\begin{defin-intro}\label{Stokes_functor_intro}
Let $\cX$ be an $\infty$-category.
Let $\cI \to \cX$ be a cocartesian fibration in posets.
Let $\cE$ be a presentable $\infty$-category.
A \textit{Stokes functor} is a functor $F \colon \cI \to \cE$ satisfying the splitting and induction conditions.
We denote by $\St_{\cI,\cE}\subset \Fun(\cI,\cE)$ the full subcategory spanned by Stokes functors.
\end{defin-intro}
Working in a purely abstract $\infty$-categorical context rather than just with the cocartesian fibration $\cI \to \Pi_{\infty}(\widetilde{X},P)$ is a necessity rather than a luxury.
Indeed,  most arguments needed to construct the derived moduli of Stokes filtered local systems in \cite{Geometric_Stokes} take place outside the initial framework of 
$(X,D,\mathcal{\cI})$.
One thus needs a robust theory encompassing all geometric situations encountered in \cite{Geometric_Stokes}, which leads naturally to \cref{Stokes_functor_intro}.
From this respect, the present work is the minimal requirement for the proofs of \cite{Geometric_Stokes} to make sense. \medskip

One major obstacle to work $\infty$-categorically is to make sense of the induction condition in a sufficiently synthetic way to minimize the amount of $\infty$-categorical data required for its check.
This is achieved through the \textit{specialization} formalism developed in a first part of the paper.
The gains with this purely $\infty$-categorical approach are streamlined proofs of crucial properties for Stokes functors: preservation under cartesian pull-back and induction over a fixed base (\cref{cor:stokes_functoriality}), invariance under localization of the base (\cref{Stokes_and_localization}), preservation under graduation (\cref{Gr_of_Stokes}), explicit description when $\cX$  has an initial object (\cref{prop:Stokes_functors_in_presence_of_initial_object}), 
spreading out (\cref{induction_poset_cocart}), compatibility with tensor product in $\PrL$ (\cref{prop:Stokes_tensor_comparison_fully_faithful}), categorical actions of local systems (\cref{cor:finite_etale_fibration_Stokes_relative_tensor}), Van Kampen (\cref{prop:Van_Kampen_Stokes}) and existence of $t$-structures (\cref{prop:t_structure_Stokes}).\medskip

Beside these structural results, let us highlight two theorems playing a crucial role in \cite{Geometric_Stokes}.
Assume that $X\subset \mathbb{C}^n$ is a polydisc with coordinates $(z,y)=(z_1,\dots, z_l,y_1,\dots, y_{n-l})$.
Let $D$ be the divisor defined by $z_1 \cdots z_l = 0$.
Let $\mathscr{I}\subset \cO_{X}(\ast D)/\cO_{X}$ be a good sheaf of irregular values.
A classical dévissage in the theory of Stokes data is by means of the \textit{levels} of $\mathscr{I} \subset \cO_{X}(\ast D)/ \cO_{X}$, that is the pole orders of the differences between the sections of $\mathscr{I}$.
By design of a good sheaf of irregular values,  there is a sequence
\begin{equation}\label{auxiliary_sequence}
m(0)<m(1)<\dots <m(d)=0
\end{equation} 
in $\mathbb{Z}^l$ such that for every $k=0,\dots, d-1$, the vectors $m(k)$ and $m(k+1)$ differ only by $1$ at exactly one coordinate and every $\order(a-b)$ for $a,b\in \mathscr{I}(X)$ distinct appears in this sequence
(such a sequence is referred to as an auxiliary sequence in \cite[§2.1.2]{Mochizuki1}).
Fix $k=0,\dots, d$ and put
\[
\mathscr{I}^k \coloneqq  \Image(\mathscr{I} \to \cO_{X}(\ast D)/z^{m(k)}\cO_{X} )  \ .
\] 
Then,  the chain of constructible sheaves on $(X,D)$
\[
\mathscr{I}=\mathscr{I}^d \to \mathscr{I}^{d-1} \to \dots \to \mathscr{I}^0=\ast   
\]
induces a chain of $P$-constructible sheaves in finite posets over $\widetilde{X}$
$$
\pi^{\ast}\mathscr{I}=\pi^{\ast}\mathscr{I}^d \to \pi^{\ast}\mathscr{I}^{d-1} \to \dots \to \pi^{\ast}\mathscr{I}^0=\ast    
$$
which in turn induces a chain of cocartesian fibrations in finite posets on $\Pi_{\infty}(\widetilde{X},P)$
\begin{equation}\label{level_intro}
\cI=\cI^d \to \cI^{d-1} \to \dots \to \cI^0=\ast  \ .
\end{equation}

The following definition is an axiomatization of the features of this chain :

\begin{defin-intro}\label{level_intro_def}
Let $\cX$ be an $\infty$-category and let $p :\cI \to \cJ$  be a morphism of cocartesian fibrations in posets over $\cX$. 
We say that $p\colon \cI \to \cJ$ is a \textit{level morphism} if it is essentially surjective and for every  $x \in \cX$ and every  $a, b \in \cI_x$, we have
	\[ p(a) < p(b) \text{ in } \cJ_x \Rightarrow a < b \text{ in } \cI_x \ . \]
\end{defin-intro}

If we consider the fibre product $\pi \colon \cI_p \coloneqq  \cJ^{\ens} \times_{\cJ} \cI\to \cJ^{\ens}$, the classical level dévissage is traditionally used to reduce the study of $\cI\to \cX$ to that of $\cJ\to \cX$ and $\cI_p\to \cX$.
This is effective since the level morphisms from \eqref{level_intro}  are so that $\cJ$ has less objects than $\cI$ while $\cI_p$ comes with extra favorable properties.
In this work,  we provide a purely categorical explanation for the level dévissage  :
\begin{thm-intro}[\cref{prop:Level_induction}]\label{level_dévissage_intro}
Let $\cX$ be an $\infty$-category and let $p \colon \cI \to \cJ$ be a level graduation morphism of cocartesian fibrations in posets over $\cX$.
Let $\cE$ be a presentable stable $\infty$-category.
Then, there is a pullback square
	\[ \begin{tikzcd}
		\St_{\cI, \cE} \arrow{r} \arrow{d}& \St_{\cJ,\cE} \arrow{d}\\
		\St_{\cI_p, \cE} \arrow{r} & \St_{\cJ^{\ens}, \cE}
	\end{tikzcd} \]
	in $ \CAT_{\infty} $.
\end{thm-intro}
\cref{level_dévissage_intro} is our main tool in \cite{Geometric_Stokes} to reduce statements on Stokes filtered local systems to plain local systems.\medskip

One of the key result of \cite{Geometric_Stokes} is the presentability of $\St_{\cI, \cE}$ for $\cE$ presentable stable in the situation coming from flat bundles.
By Ragimov-Schlank's $\infty$-categorical reflection theorem \cite{Ragimov_Schlank_Reflection}, presentability is a consequence of the fact that $\St_{\cI, \cE}$ is stable under limits and colimits in $\Fun(\cI,\cE)$.
This may seem miraculous since both the splitting and induction conditions are not stable under limits and colimits already over a point.
To leverage Ragimov-Schlank's theorem using the level induction dévissage from \cref{level_dévissage_intro}, a prerequisite is to show that the graduation and induction functors from  \cref{exp_construction} and \cref{sec:graduation} commute with limits.
This is not granted since both functors are left adjoints and have in general no reason to be right adjoints as well.
By \cref{induction_limit_stable}, this is nonetheless true when the total space $\cI$ of the cocartesian fibration in posets $\cI \to \Pi_{\infty}(\widetilde{X},P)$ is compact.
To check compactness, we prove the following

\begin{thm-intro}[\cref{compact_cocartesian_fibration}]
\label{compact_cocartesian_fibration_intro}
	Let $\cX$ be an $\infty$-category and let $\cA \to \cX$ be a cocartesian fibration.
	Assume that $\cX$ is compact and that for every $x \in \cX$, the fiber $\cA_x$ is compact in $\Cat_\infty$.
	Then $\cA$ is compact in $\Cat_\infty$ as well.
\end{thm-intro}

In the situation coming from flat bundles, the fibres of $\cI \to \Pi_{\infty}(\widetilde{X},P)$ are finite posets, which are automatically compact.
On the other hand, the compactness of $\Pi_{\infty}(\widetilde{X},P)$ follows from a result obtained by the authors in collaboration with P. Haine in \cite[Theorems 0.4.2 \& 0.4.3]{Beyond_conicality}. 
It can be seen as a stratified generalization of theorems of Lefschetz–Whitehead, Łojasiewicz and Hironaka on the finiteness of the underlying homotopy types of compact subanalytic spaces and real algebraic varieties.
 \medskip

\paragraph*{\textbf{Linear overview}}
In \cref{exp_construction}, we start by recalling the exponential construction as an $\infty$-functor.
In \cref{sec:specialization_equivalence}, we refine our analysis of the exponential construction via the \emph{specialization equivalence}.
In \cref{cocar_functor_section} and \cref{punctually_split_Stokes_section}, we separately study the property of being \emph{cocartesian} and \emph{punctually split}.
This leads to the basic functorialities of the $\infty$-categories of Stokes functors (see \cref{cor:stokes_functoriality}) and to their fundamental properties such the invariance by localization (see \cref{Stokes_and_localization}), Van Kampen (see \cref{prop:Van_Kampen_Stokes}) and the existence of $t$-structures (see \cref{prop:t_structure_Stokes}).
In \cref{sec:graduation} and \cref{sec:level_structures}, we develop the theory of \emph{graduation} and the notion of \emph{level structure}.
We study the compatibility of the graduation procedure with Stokes functors and the interaction with their basic functorialities.
\Cref{prop:Level_induction} is in many ways the crucial result of this section, establishing the categorical basis of the level induction technique used in \cite{Geometric_Stokes}.\medskip

The remaining part is essentially intended as an appendix.
It turns out that the language of the specialization equivalence is a powerful categorical tool that allows to prove structural results on cocartesian fibrations.
In \cref{compact_cocartesian_fibration}, we establish a local-to-global principle for compactness of the total space of a cocartesian fibration.
In \cref{prop:pullback_localization}, we give a new and model-independent proof of Hinich's theorem \cite{Hinich_DK_revisited}.
Finally,  we introduce in  \cref{sec:locally_constant_cocartesian_fibrations} and \cref{sec:categorical_actions} the notion of \emph{finite étale (cocartesian) fibration}.
This notion plays a crucial role in the proof of the retraction lemma (see \cref{cor:retraction_lemma}) that allows to treat \emph{ramified} Stokes structures in \cite{Geometric_Stokes}.

\medskip

\paragraph*{\textbf{Acknowledgments}}

We are grateful to Enrico Lampetti, Guglielmo Nocera, Tony Pantev, Marco Robalo and Marco Volpe for useful conversations about this paper.
We especially thank Peter J.\ Haine for fruitful collaborations on the exodromy theorems.
We thank the Oberwolfach MFO institute that hosted the Research in Pairs ``2027r: The geometry of the Riemann-Hilbert correspondence''.
We also thank the CNRS for delegations and PEPS ``Jeunes Chercheurs Jeunes Chercheuses'' fundings, as well as the ANR CatAG from which both authors benefited during the writing of this paper.

\section{Cocartesian fibrations and the exponential construction}
\label{exp_construction}
We first review some $\infty$-category theory that has been developed in the companion paper \cite{PortaTeyssier_Day}.
We need this technology for two reasons: (i) to provide a streamlined definition of the category of \emph{Stokes stratified spaces}, and (ii) to show that we can functorially attach to every Stokes stratified space a constructible sheaf of $\infty$-categories, whose global sections is exactly the associated $\infty$-category of Stokes structures.

\subsection{Dual fibrations}

Following the companion paper \cite{PortaTeyssier_Day} we introduce the $\infty$-category $\CoCart$.
We start from the cartesian fibration
\[ \mathrm{t} \colon \Cat_\infty^{[1]} \coloneqq \Fun(\Delta^1, \Cat_\infty) \to  \Cat_\infty \  \]
sending a functor $\cA \to \cX$ to its target $\infty$-category.
We then pass to the dual cocartesian fibration, in the following sense:

\begin{defin}
	Let $p \colon \cA \to \cX$ be a cartesian fibration and let $\Upsilon_\cA \colon \cX\op \to \Cat_\infty$ be its  straightening.
	The \emph{dual cocartesian fibration $p^\star \colon \cA^\star \to \cX\op$} is the cocartesian fibration classified by $\Upsilon_\cA$.
\end{defin}

\begin{recollection}\label{recollection:dual_fibration}
	In the setting of the above definition, recall from \cite{Barwick_Dualizing} that objects of $\cA^\star$ coincide with the objects of $\cA$, while $1$-morphisms $a \to b$ in $\cA^\star$ are given by spans
	\[ \begin{tikzcd}[column sep = 20pt]
		a & c \arrow{l}[swap]{u} \arrow{r}{v} & b
	\end{tikzcd} \]
	where $u$ is $p$-cocartesian and $p(v)$ is equivalent to the identity of $p(b)$.
\end{recollection}

We let
\[ \mathbb B \colon \Cat_\infty^{[1]\star} \to  \Cat_\infty\op \]
be the cocartesian fibration dual to $\mathrm{t}$.
Specializing \cref{recollection:dual_fibration} to this setting, we see that objects of $\Cat_\infty^{[1]\star}$ are functors $\cA \to \cX$, and morphisms $\mathbf f = (f,u,v)$ from $\cB \to \cY$ to $\cA \to \cX$ are commutative diagrams in $\Cat_\infty$ of the form
\begin{equation}\label{eq:morphism_in_Cocart}
	\begin{tikzcd}
		\cB \arrow{d} & \cB_\cX \arrow{d} \arrow{r}{v} \arrow{l}[swap]{u} & \cA \arrow{dl} \\
		\cY & \cX \arrow{l}[swap]{f}
	\end{tikzcd}
\end{equation}
where the square is a pullback.
With respect to this description, $\mathbb B$ sends $\cA \to \cX$ to its target (or base) $\cX$, and a diagram as above defines a $\mathbb B$-cocartesian morphism if and only if $v$ is an equivalence.

\medskip

We define $\CoCart$ to be the (non-full) subcategory of $\Cat_\infty^{[1]\star}$ whose objects are cocartesian fibrations, and whose $1$-morphisms are commutative diagrams as above where $v$ is required to preserve cocartesian edges.
In this way, $\CoCart$ becomes a cocartesian fibration over $\Cat_\infty\op$ such that $\CoCart\to \Cat_\infty^{[1]\star}$ preserves cocartesian edges.
Notice that the fiber at $\cX \in \Cat_\infty\op$ coincides with the $\infty$-category $\CoCart_{/\cX}$.
We will also need a couple of variants of this construction:

\begin{variant}
 We let $\PosFib \subset \CoCart$ be the full subcategory spanned by those cocartesian fibrations $\cA \to \cX$ whose fibers are posets.
\end{variant}

\begin{variant}
	Let $\CAT_\infty$ be the $\infty$-category of large $\infty$-categories and consider the following fiber product:
	\[ \cC \coloneqq \Fun(\Delta^1, \CAT_\infty) \times_{\CAT_\infty} \Cat_\infty \ , \]
	where we used the target morphism $\mathrm t \colon \Fun(\Delta^1, \CAT_\infty) \to \CAT_\infty$.
	In other words, objects in $\cC$ are morphisms $p \colon \cA \to \cX$ where $\cX$ is a small $\infty$-category and the fibers of $p$ are not necessarily small $\infty$-categories.
	The induced morphism $\mathrm t \colon \cC \to \Cat_\infty$ is a cartesian fibration.
	Inside the dual cocartesian fibration $\cC^\star$, we define $\hCoCart$ as the subcategory spanned by cocartesian fibrations and whose $1$-morphisms are diagrams \eqref{eq:morphism_in_Cocart} where  $v$ preserves cocartesian edges.
\end{variant}		

\begin{variant}
We let $\PrFibL \subset \hCoCart$ be the subcategory spanned by cocartesian fibrations with presentable fibres and whose $1$-morphisms are diagrams \eqref{eq:morphism_in_Cocart} that are morphisms in $\hCoCart$ such that  for every $x \in \cX$, the induced functor $v_x \colon \cB_{f(x)} \to \cA_x$ is a morphism in $\PrL$, i.e.\  is cocontinuous.
$\PrFibL$ is the $\infty$-category of \textit{presentable cocartesian fibrations} \cite[\S3.4]{PortaTeyssier_Day}.
\end{variant}

\begin{recollection}\label{PrL_straightening}
	Both $\CoCart$ and $\PrFibL$ can be promoted to $\Cat_\infty\op$-families of symmetric monoidal $\infty$-categories $\CoCart^\otimes$ and $\PrFibLotimes$, in the sense of \cite[Definition A.1]{PortaTeyssier_Day}.
	Concretely, this provides for every $\cX \in \Cat_\infty\op$ a symmetric monoidal structure on the fiber $\CoCart_\cX$ and $\PrFibL_\cX$ of $\mathbb B \colon \CoCart \to \Cat_\infty\op$ and $\mathbb B \colon \PrFibL \to \Cat_\infty\op$.
	Invoking the straightening equivalence \cite[Theorem 3.2.0.1]{HTT}, we find canonical identifications
\begin{equation}\label{straightening_PrL}
	\CoCart_\cX \simeq \Fun(\cX, \Cat_\infty) \qquad \text{and} \qquad \PrFibL_\cX \simeq \Fun(\cX, \PrL) \ . 
\end{equation}
	Under these equivalences, the above  symmetric monoidal structures correspond  to those induced respectively by the cartesian product on $\Cat_\infty$ and the tensor product  on $\PrL$ as defined in \cite[\S4.8.1]{Lurie_Higher_algebra}.
\end{recollection}

\subsection{Exponential construction}\label{subsec:exponential_construction}

Fix a presentable $\infty$-category $\cE$.

\begin{construction}\label{construction:exponential}
	Let $p \colon \cA \to \cX$ be a cocartesian fibration.
	Let $\Upsilon_\cA \colon \cX \to \Cat_\infty$ be its straightening and consider the functor
	\[ \Fun_!(\Upsilon_\cA(-), \cE) \colon \cX \to  \PrL \ , \]
	where $\Fun_!$ denotes the functoriality given by left Kan extensions.
	We write
	\[ \exp_\cE(\cA/\cX) \to \cX \]
	for the presentable cocartesian fibration classifying $\Fun_!(\Upsilon_\cA(-),\cE)$.
	We refer to $\exp_\cE(\cA/\cX)$ as the \emph{exponential fibration with coefficients in $\cE$ associated to $p \colon \cA \to \cX$}.
\end{construction}

\begin{eg}\label{eg:exponential_over_trivial_base}
		Assume that $\cX = \ast$ is the category with one object and one (identity) morphism.
		Then $\CoCart_\cX \simeq \Cat_\infty$ and $\PrFibL_\cX \simeq \PrL$.
		In this case, $\exp_\cE(\cA) \simeq \Fun(\cA, \cE)$.
\end{eg}		

\begin{eg} 
       Assume that $\cX = \Delta^1$, so that we can represent $\Phi_\cA$ as a single functor $f \colon \cA_0 \to \cA_1$.
		In this case, the functor $\Fun_!(\Upsilon_\cA(-),\cE) \colon \Delta^1 \to \PrL$ is identified with the functor
		\[ f_! \colon \Fun(\cA_0, \cE) \to  \Fun(\cA_1,\cE) \ , \]
		where $f_!$ denotes the left Kan extension along $f$.
		Therefore we can understand $\exp_\cE(\cA/\Delta^1)$ as the presentable cocartesian fibration over $\Delta^1$ whose objects are pairs $(F,i)$ where $i \in \Delta^1$ and $F \colon \cA_i \to \cE$ is a functor.
		Besides, using \cite[Proposition 2.4.4.2]{HTT}, we deduce that
		\[ \Map_{\exp_\cE(\cA/\Delta^1)}((F,i),(G,j)) = \begin{cases}
			\Map_{\Fun(\cA_0,\cE)}(F,G) & \text{if } i = j = 0 \ , \\
			\Map_{\Fun(\cA_1,\cE)}(f_!(F),G) & \text{if } i = 0 \text{ and } j = 1 \ , \\
			\Map_{\Fun(\cA_1,\cE)}(F,G) & \text{if } i = j = 1 \ , \\
			\emptyset & \text{if } i = 1 \text{ and } j = 0 \ .
		\end{cases} \]
		Finally, a morphism $(F,0) \to (G,1)$ in $\exp_\cE(\cA/\Delta^1)$ is cocartesian if and only if the induced morphism $f_!(F) \to G$ is an equivalence.
	\end{eg}	

\begin{eg}
Combining the previous two points with the general properties of the straightening equivalence, we deduce that  for any morphism $\gamma \colon x \to y$ in $\cX$ the fibers of $\exp_\cE(\cA/\cX)$ at $x$ and $y$ are canonically identified with $\Fun(\cA_x, \cE)$ and $\Fun(\cA_y,\cE)$, and a morphism $\alpha \colon F \to G$ in $\exp_\cE(\cA/\cX)$ lying over $\gamma$ is cocartesian if and only if for any choice of a cocartesian straightening $f_\gamma \colon \cA_x \to \cA_y$ of $\gamma$, $\alpha$ exhibits $G$ as left Kan extension of $F$ along $f_\gamma$.
\end{eg}

It follows from \cite[Variant 3.20 \& Remark 3.21]{PortaTeyssier_Day} that \cref{construction:exponential} can be canonically promoted to an $\infty$-functor
\[ \exp_\cE \colon \CoCart \to  \PrFibL \ . \]
Let us spell out the functoriality of $\exp_\cE$ in more concrete terms.
With respect to morphisms in $\CoCart$ as in \eqref{eq:morphism_in_Cocart}, we will use the following notation:
\begin{equation}\label{eq:exp_on_morphisms}
	\begin{tikzcd}
		\cB \arrow{d} & \cB_\cX \arrow{d} \arrow{r}{v} \arrow{l}[swap]{u} & \cA \arrow{dl} \\
		\cY & \cX \arrow{l}[swap]{f}
	\end{tikzcd} \stackrel{\exp_\cE}{\longmapsto} \begin{tikzcd}
		\exp_\cE(\cB / \cY) \arrow{d} & \exp_\cE(\cB_\cX / \cX) \arrow{l}[swap]{\cE^u} \arrow{r}{\cE^v_!} \arrow{d} & \exp_\cE(\cA / \cX) \arrow{dl} \\
		\cY & \cX \arrow{l}[swap]{f} 
\end{tikzcd}
\end{equation}
We refer to the functor $\cE^v_!$ as the \emph{exponential induction functor}.

\begin{prop}\label{prop:functoriality_exponential}
	With respect to \eqref{eq:exp_on_morphisms}, we have:
	\begin{enumerate}\itemsep=0.2cm
		\item the functor $\cE^u \colon \exp_\cE(\cB_\cX / \cX) \to \exp_\cE(\cB / \cY)$ makes the the right square a pullback;
		
		\item the functor $\cE^v_!$ preserves cocartesian edges.
	\end{enumerate}
	In particular, $\exp_\cE$ takes $\mathbb B$-cocartesian edges in $\CoCart$ to $\mathbb B$-cocartesian edges in $\PrFibL$.
\end{prop}

\begin{proof}
	Statement (1) simply follows unraveling the definitions, as in \cite[Lemma 3.8]{PortaTeyssier_Day}.
	Statement (2) is automatic from the definition of morphisms in $\PrFibL$, but the reader should observe that for fixed $\cX \in \Cat_\infty\op$, the induced functor $\exp_{\cE, \cX} \colon \CoCart_\cX \to \PrFibL_\cX$ is precisely given by \cref{construction:exponential}.
	In other words, $\cE^v_!$ is the unstraightening of the natural transformation
	\[ \Fun_!(\Upsilon_{\cB_\cX}(-),\cE) \to  \Fun_!(\Upsilon_\cA(-),\cE) \]
	induced by left Kan extension along the natural transformation $\Upsilon_v \colon \Upsilon_{\cB_\cX} \to \Upsilon_\cA$.
	Therefore, $\cE^v_!$ preserves cocartesian edges by construction.
\end{proof}

\begin{cor}\label{pullback_induction_exp}
	Consider a commutative diagram in $\CoCart$
	\[ \begin{tikzcd}
		\cB_{\cY}  \arrow{rr}{v_{\cY} } \arrow{dr} & & \cA_{\cY} \arrow{dl} \arrow{drr}{u_\cA} \\
		{} &  \cY \arrow{drr} & \cB \arrow[crossing over,leftarrow,swap]{ull}{u_\cB} \arrow{rr}{v} \arrow{dr} & & \cA \arrow{dl} \\
		{} & & & \cX
	\end{tikzcd} \]
whose diagonal squares are pullback.
Then,  the squares of the commutative diagram
	\[ \begin{tikzcd}
		 \exp_\cE(\cB_{\cY}/\cY)  \arrow{rr}{\cE^{v_{\cY}}_{!}}\arrow{dr} & & \exp_\cE(\cA_{\cY}/\cY)  \arrow{dl} \arrow{drr}{\cE^{u_\cA}}   \\
		{} &  \cY \arrow{drr} & \exp_\cE(\cB/\cX)  \arrow[crossing over,leftarrow,swap]{ull}{\cE^{u_\cB}}  \arrow{rr}{\cE^{v}_{!}} \arrow{dr} & &\exp_\cE(\cA/\cX)\arrow{dl} \\
		{} & & & \cX
	\end{tikzcd} \]
	are pullback.
\end{cor}

Via the identifications of the fibres of the exponential fibration supplied by   \cref{eg:exponential_over_trivial_base},  \cref{pullback_induction_exp}  specializes to 

\begin{cor}\label{cor:fibers_exponential}
	In the situation from \cref{pullback_induction_exp} where $\cY$ is  an object $x \in \cX$,
	the squares of the commutative diagram
	 \[ \begin{tikzcd}
		\Fun(\cB_x, \cE) \arrow{rr}{v_{x,!}}\arrow{dr} & &\Fun(\cA_x, \cE)  \arrow{dl} \arrow{drr} \\
		{} &  \ast \arrow[swap]{drr}{x} & \exp_\cE(\cB/\cX)  \arrow[crossing over,leftarrow]{ull} \arrow{rr}{\cE^{v}_{!}} \arrow{dr} & &\exp_\cE(\cA/\cX)\arrow{dl} \\
		{} & & & \cX
	\end{tikzcd} \]
	are pullback,  where $v_{x,!}$ is the left Kan extension along  $v_x \colon \cA_x \to \cB_x$.
\end{cor}

\begin{recollection}
	Assume that $\cE$ has an additional symmetric monoidal structure $\cE^\otimes$.
	Then \cite[Example 3.22]{PortaTeyssier_Day} shows that $\exp_\cE$ admits a natural extension
	\[ \exp_\cE \colon \CoCart^\otimes \to  \PrFibLotimes \]
	to a $\Cat_\infty\op$-lax symmetric monoidal functors, in the sense of Definition A.3 in \emph{loc.\ cit.}
\end{recollection}

\subsection{Section functors}\label{subsec:section_functors}

Given a cocartesian fibration $\cA \to \cX$ we can associate to it two different $\infty$-categories:
\[ \Sigma_\cX(\cA/\cX) \coloneqq \Fun_{/\cX}(\cX, \cA) \qquad \text{and} \qquad \Sigma_\cX^{\cocart}(\cA / \cX) \coloneqq \Fun_{/\cX}^{\cocart}(\cX, \cA) \ . \]
These are respectively the $\infty$-categories of sections and of cocartesian sections.
It follows from \cite[Corollary 3.23 \& Variant 3.24]{PortaTeyssier_Day} that these constructions promote to global functors
\[ \Sigma,\ \Sigma^{\cocart} \colon \CoCart \to  \Cat_\infty \times \Cat_\infty\op \qquad \text{and} \qquad \Sigma,\ \Sigma^{\cocart} \colon \PrFibL \to  \PrL \times \Cat_\infty\op \ . \]
The same considerations of \emph{loc.\ cit.} shows that the same holds for $\hCoCart$ in place of $\CoCart$.

\begin{rem}\label{construction_tensor_PrL_ex}
The functor $\Sigma_{\cX}^{\cocart} \colon \PrFibL \to \PrL$ admits a monoidal left adjoint $\Triv_\cX \colon \PrL \to \PrFibL $ informally given by $\cE\to (\cE \times \cX)/\cX$.  
In particular, given an object $\cA \to \cX$ of $\PrFibL$ and $\cE, \cE' \in \PrL$, we have
\[
\exp_\cE(\cA /\cX) \otimes_{\cX} \Triv_{\cX}(\cE') \simeq\exp_{\cE\otimes \cE'}(\cA /\cX) \ .
\]
\end{rem}

\begin{notation}
	Often we will also write $\Sigma$ and $\Sigma^{\cocart}$ for the induced functors $\PrFibL \to \PrL$ and its variants obtained composing the above functors with the canonical projection $\PrL \times \Cat_\infty\op \to \PrL$.
\end{notation}

\medskip

The subtlety here is in the great amount of functoriality encoded in $\Sigma$ and $\Sigma^{\cocart}$.
To fix ideas, let us discuss the case of $\PrFibL$ and the functor $\Sigma^{\cocart}$, although similar considerations will apply to both $\CoCart$ and $\hCoCart$ in place of $\PrFibL$ and $\Sigma$ in place of $\Sigma^{\cocart}$.
Morphisms in $\PrL$ are commutative diagrams of the form
\begin{equation*}
	\begin{tikzcd}
		\cB \arrow{d} & \cB_\cX \arrow{d} \arrow{r}{v} \arrow{l}[swap]{u} & \cA \arrow{dl} \\
		\cY & \cX \arrow{l}[swap]{f}
	\end{tikzcd}
\end{equation*}
where the square is a pullback and $v$ preserves cocartesian edges.
Applying $\Sigma^{\cocart}$, this diagram is sent to the composition
\[ \begin{tikzcd}
	\Funcocart_{/\cY}(\cY,\cB) \arrow{r}{u^\ast} & \Funcocart_{/\cX}(\cX, \cB_\cX) \arrow{r}{v \circ -} & \Funcocart_{/\cX}(\cX,\cA) \ .
\end{tikzcd} \]
Concretely, $u^\ast$ takes a cocartesian section $s \colon \cY \to \cB$, considers the composition $s \circ f$ and applies the universal property of pullbacks to produce a section $u^\ast(s) \colon \cX \to \cB_\cX$ of $\cB_\cX \to \cX$.
An immediate check reveals that this is again a cocartesian section, so that $u^\ast$ is in fact well defined.
On the other hand, $v \circ -$ takes a cocartesian section $t \colon \cX \to \cB_\cX$ to the composite cocartesian section $v \circ t \colon \cX \to \cA$.
That these operations can be performed $\infty$-functorially in $\PrFibL$ is precisely the content of \cite[Corollary 3.23]{PortaTeyssier_Day}.

\medskip

We will often be interested in taking sections of exponential constructions.
The following result is essentially a consequence of the theory of lax limits developed in \cite{Gepner_Enriched_2013}:

\begin{prop}[{See \cite[Proposition 4.1]{PortaTeyssier_Day}}]\label{prop:sections_exponential}
	Let $\cE$ be a presentable $\infty$-category and let $\cA \to \cX$ be a cocartesian fibration.
	There are canonical equivalences
	\[ \Fun(\cA, \cE) \simeq \Sigma_\cX(\exp_\cE(\cA/\cX)) \simeq \Fun_{/\cX}(\cX, \exp_\cE(\cA/\cX)) \ . \]
\end{prop}

\begin{warning}\label{warning:global_sections}
	If instead of applying $\Sigma_\cX$ we use $\Sigma_\cX^{\cocart}$, we obtain a full subcategory $\Funcocart(\cA,\cE)$ of $\Fun(\cA, \cE)$.
	We refer to objects in $\Funcocart(\cA, \cE)$ as \emph{cocartesian functors}.
	We will provide a in \cref{prop:cocartesian_functors_reformulation} a  characterization intrinsic to $\Fun(\cA, \cE)$ of what it means for a functor $F \colon \cA \to \cE$ to be cocartesian.
\end{warning}

\section{The specialization equivalence}\label{sec:specialization_equivalence}

\subsection{Global functoriality statements.}

Fix a cocartesian fibration $p \colon \cA \to \cX$ as well as a presentable $\infty$-category $\cE$.
Write
\[ p_\cE \colon \exp_\cE(\cA / \cX) \to \cX \]
for the structural map of the exponential construction of $p$.
Recall from \cref{prop:sections_exponential} that there is a canonical equivalence
\begin{equation}\label{eq:specialization_equivalence}
	\spe_{\cX,p}^\cE \colon \Fun(\cA, \cE) \simeq \Fun_{/\cX}(\cX, \exp_\cE(\cA/\cX)) \ ,
\end{equation}
which we refer to as the \emph{specialization equivalence}.
When $\cX$ and $\cE$ are clear out of the context, we will use the notation $\spe_\cA$ (or even just $\spe$) instead of $\spe_{\cX,p}^\cE$.

\medskip

The right hand side of \eqref{eq:specialization_equivalence} is functorial in $p \colon \cA \to \cX$ with respect to the morphisms in $\CoCart$.
Explicitly, this means that to every morphism \eqref{eq:morphism_in_Cocart}
\[ \begin{tikzcd}
	\cB \arrow{d} & \cB_\cX \arrow{d} \arrow{r}{v} \arrow{l}[swap]{u} & \cA \arrow{dl} \\
	\cY & \cX \arrow{l}[swap]{f}
\end{tikzcd} \]
one can first apply $\exp_\cE \colon \CoCart \to \PrFibL$ to obtain the morphism
\[ \begin{tikzcd}
	\exp_\cE(\cB / \cY) \arrow{d} & \exp_\cE(\cB_\cX / \cX) \arrow{l}[swap]{\cE^u} \arrow{r}{\cE^v_!} \arrow{d} & \exp_\cE(\cA / \cX) \arrow{dl} \\
	\cY & \cX \arrow{l}[swap]{f} 
\end{tikzcd} \]
and then apply the section functor $\Sigma \colon \PrFibL \to \PrL$ to obtain the composition
\[ \begin{tikzcd}
	\Fun_{/\cY}(\cY, \exp_\cE(\cB/\cY)) \arrow{r}{\Sigma(\cE^u)} & \Fun_{/\cX}(\cX, \exp_\cE(\cB_\cX/\cX)) \arrow{r}{\Sigma(\cE^v_!)} & \Fun_{/\cX}(\cX, \exp_\cE(\cA / \cX)) \ .
\end{tikzcd} \]
We defer to \cite{PortaTeyssier_Day} for the justification that these operations can be performed in an $\infty$-functorial way.
The goal of this section is to explain how this functoriality interacts with the specialization equivalence.
More precisely, observe that applying \eqref{eq:specialization_equivalence} to every term in the above composition, we obtain respectively $\Fun(\cB,\cE)$, $\Fun(\cB_\cX,\cE)$ and $\Fun(\cA, \cE)$.
The following is the main result of this section:

\begin{prop}\label{prop:global_functoriality}
	\hfill
	\begin{enumerate}\itemsep=0.2cm
		\item \label{prop:global_functoriality:pullback} There exists a canonically commutative square
		\begin{equation}\label{eq:global_pullback}
			\begin{tikzcd}
				\Fun_{/\cY}(\cY,\exp_\cE(\cB/\cY)) \arrow{r}{\Sigma(\cE^u)} \arrow{d}{\spe_{\cB}} & \Fun_{/\cX}(\cX, \exp_\cE(\cB_\cX / \cX)) \arrow{d}{\spe_{\cB_\cX}} \\
				\Fun(\cB,\cE) \arrow{r}{u^\ast} & \Fun(\cB_\cX, \cE) \ ,
			\end{tikzcd}
		\end{equation}
		providing a canonical identification $\Sigma(\cE^u) \simeq u^\ast$.
		
		\item \label{prop:global_functoriality:induction} There exists a canonically commutative square
		\begin{equation}\label{eq:global_induction}
			\begin{tikzcd}
				\Fun_{/\cX}(\cX, \exp_{\cE}(\cB_\cX / \cX)) \arrow{r}{\Sigma(\cE^v_!)} \arrow{d}{\spe_{\cB_\cX}} & \Fun_{/\cX}(\cX, \exp_{\cE}(\cA/\cX)) \arrow{d}{\spe_{\cA}} \\
				\Fun(\cB_\cX, \cE) \arrow{r}{v_!} & \Fun(\cA,\cE) \ ,
			\end{tikzcd}
		\end{equation}
		providing a canonical identification $\Sigma(\cE^v_!) \simeq v_!$.
	\end{enumerate}
\end{prop}

Before starting the proof, let us record a couple of handy consequences.
First, recall from \cref{cor:fibers_exponential} that the fiber of $\exp_\cE(\cA/\cX)$ at $x \in \cX$ is canonically identified with $\Fun(\cA_x, \cE)$.
In particular, this means that for a functor $F \colon \cA \to \cE$, the value of its specialization $\spe F$ at an object $x \in \cX$ is a functor
\[ (\spe F)_x \colon \cA_x \to  \cE \ . \]
We have:

\begin{cor}\label{prop:evaluation_specialization}
	Let $j_x \colon \cA_x \to \cA$ be the natural inclusion.
	Then there is a canonical identification
	\[ (\spe F)_x \simeq j_x^\ast(F) \ . \]
	In particular, for every $a \in \cA$ we have a canonical identification
	\[ (\spe F)_{p(a)}(a) \simeq F(a) \in \cE \ .  \]
\end{cor}

\begin{proof}
	The pullback square
	\[ \begin{tikzcd}
		\cA \arrow{d}[swap]{p} & \cA_x \arrow{l} \arrow{d} \\
		\cX & \ast \arrow{l}[swap]{x}
	\end{tikzcd} \]
	defines a morphism from $\cA \to \cX$ to $\cA_x \to \ast$ in $\CoCart$.
	It is then enough to apply \cref{prop:global_functoriality}-\eqref{prop:global_functoriality:pullback} to this morphism.
\end{proof}

\begin{cor}\label{cor:induction_specialization_Beck_Chevalley}
	Consider a commutative diagram in $\CoCart$
	\[ \begin{tikzcd}
		\cB_{\cY}  \arrow{rr}{v_{\cY} } \arrow{dr} & & \cA_{\cY} \arrow{dl} \arrow{drr}{u_\cA} \\
		{} &  \cY \arrow{drr} & \cB \arrow[crossing over,leftarrow,swap]{ull}{u_\cB} \arrow{rr}{v} \arrow{dr} & & \cA \arrow{dl} \\
		{} & & & \cX
	\end{tikzcd} \]
	whose diagonal squares are pullback.
	Let $\cE$ be a presentable $\infty$-category.
	Then, the squares
	\[ \begin{tikzcd}
		\Fun(\cB_\cY,\cE)  & \Fun(\cA_\cY,\cE)\arrow{l}[swap]{v_{\cY}^*}\\
		 \Fun(\cB,\cE) \arrow{u}{u_\cB^*}  & \arrow{l}{v^* }  \Fun(\cA,\cE) \arrow{u}[swap]{u_\cA^*}
	\end{tikzcd} 	
	\qquad \text{ and } \qquad
		\begin{tikzcd}
		\Fun(\cB_\cY,\cE)  & \Fun(\cB,\cE)\arrow{l}[swap]{u_\cB^*}\\
		\Fun(\cA_\cY,\cE) \arrow{u}{v_{\cY}^* }  & \arrow{l}{ u_\cA^*}  \Fun(\cA,\cE) \arrow{u}[swap]{v^*}
	\end{tikzcd} \]
	are respectively horizontally left and right adjointable.
\end{cor}

\begin{proof}
	It is enough to prove the left adjointability statement, which follows by applying the section functor to the commutative diagram 
	\[ \begin{tikzcd}
	 \exp_\cE(\cB_{\cY}/\cY)  \arrow{rr}{\cE^{v_{\cY}}_{!}}\arrow{dr} & & \exp_\cE(\cA_{\cY}/\cY)  \arrow{dl} \arrow{drr}{\cE^{u_\cA}}   \\
		{} &  \cY \arrow{drr} & \exp_\cE(\cB/\cX)  \arrow[crossing over,leftarrow,swap]{ull}{\cE^{u_\cB}}  \arrow{rr}{\cE^{v}_{!}} \arrow{dr} & &\exp_\cE(\cA/\cX)\arrow{dl} \\
		{} & & & \cX
	\end{tikzcd} \]
	supplied by \cref{pullback_induction_exp} and then invoke \cref{prop:global_functoriality}.
\end{proof}

\subsection{Some categorical calculus}

As a preliminary for \cref{prop:global_functoriality}, we revisit and extend part of the content of \cite{Gepner_Lax_colimits}.

\begin{recollection}\label{recollection:pushforward_fibrations}
	Let $f \colon \cX \to \cY$ be a functor of $\infty$-categories.
	The pullback
	\[ f^\ast \coloneqq \cX \times_\cY - \colon \Cat_{\infty / \cY} \to  \Cat_{\infty / \cX} \]
	preserves (co)cartesian fibrations and therefore it gives rise to functors
	\[ f^\ast \colon \Cart_\cY \to  \Cart_\cX \qquad \text{and} \qquad f^\ast \colon \CoCart_\cY \to  \CoCart_\cX \ . \]
	Under the straightening equivalences, we see these functors admit both a left and a right adjoint, denoted respectively
	\[ \fcartlowershriek , \fcartlowerstar \colon \Cart_\cX \to  \Cart_\cY \qquad \text{and} \qquad \fcocartlowershriek, \fcocartlowerstar \colon \CoCart_\cX \to  \CoCart_\cY \ . \]
\end{recollection}

Recall the following standard notation in category theory:

\begin{notation}
	Let $\cX$ be an $\infty$-category.
	We write $\Tw(\cX)$ for the associated $\infty$-category of twisting arrows, see \cite[\S5.2.1]{Lurie_Higher_algebra} and $\lambda \colon \Tw(\cX) \to \cX \times \cX\op$ for the right fibration constructed in \cite[Proposition 5.2.1.11]{Lurie_Higher_algebra}.
	Given a functor
	\[ F \colon \cX \times \cX\op \to  \Cat_\infty \ , \]
	we respectively write
	\[ \int^\cX F \qquad \text{and} \qquad \int_\cX F \]
	for the coend and the end of $F$, i.e.\ for the colimit and the limit of the composite
	\[ \begin{tikzcd}[column sep=small]
		\Tw(\cX) \arrow{r}{\lambda} & \cX \times \cX\op \arrow{r}{F} & \Cat_\infty \ .
	\end{tikzcd} \]
\end{notation}

\begin{notation}
	Write $\pi_\cX \colon \cX \times \cX\op \to \cX$ and $\pi_{\cX\op} \colon \cX \times \cX\op \to \cX\op$ for the canonical projections.
	Given two functors
	\[ F \colon \cX \to  \Cat_\infty \qquad \text{and} \qquad G \colon \cX\op \to  \Cat_\infty \ , \]
	we write $F \boxtimes G$ for the functor
	\[ F \boxtimes G \coloneqq \pi_\cX^\ast(F) \times \pi_{\cX\op}^\ast(G) \ . \]
	When $\cA \to \cX$ is a cocartesian fibration and $\cB \to \cX$ is a cartesian fibration, we write
	\[ \int^\cX \cA \boxtimes \cB \coloneqq \int^\cX \Upsilon_\cA \boxtimes \Upsilon_\cB \qquad \text{and} \qquad \int_\cX \cA \boxtimes \cB \coloneqq \int_\cX \Upsilon_\cA \boxtimes \Upsilon_\cB \ . \]
\end{notation}

To state the first fundamental result, we need to introduce one final notation:

\begin{notation}\label{notation:fake_exponential}
	Let $\cX$ and $\cE$ be two $\infty$-categories.
	For $\cA \to \cX$ a cartesian fibration, write $\Upsilon_\cA \colon \cX\op \to \Cat_\infty$ for its straightening and $\cE^\cA_{\mathrm{cc}}$ for the \emph{cocartesian} fibration classifying the functor
	\[ \Fun(\Upsilon_\cA, \cE) \colon \cX \to  \Cat_\infty \ . \]
	Similarly, for a cocartesian fibration $\cB \to \cA$, write $\Upsilon_\cB \colon \cX \to \Cat_\infty$ for its straightening and $\cE^\cB_{\mathrm{c}}$ for the \emph{cartesian} fibration classifying the functor
	\[ \Fun(\Upsilon_\cB, \cE) \colon \cX\op \to  \Cat_\infty \ . \]
	Notice that given a functor $f \colon \cY \to \cX$, there are canonical equivalences
	\begin{equation}\label{eq:fake_exponential}
		f^\ast \cE^\cA_{\mathrm{cc}} \simeq \cE^{f^\ast(\cA)}_{\mathrm{cc}} \qquad \text{and} \qquad f^\ast \cE^\cB_{\mathrm{c}} \simeq \cE^{f^\ast(\cB)}_{\mathrm{c}} \ .
	\end{equation}
\end{notation}

\begin{lem}\label{lem:change_of_variable}
	Let $f \colon \cX \to \cY$ be a functor of $\infty$-categories.
	\begin{enumerate}\itemsep=0.2cm
		\item For $\cB \to \cX$ a cartesian fibration and $\cA \to \cY$ a cocartesian fibration, there is a canonical equivalence
		\[ \int^\cY \cA \boxtimes \fcartlowershriek(\cB) \simeq \int^\cX f^\ast(\cA) \boxtimes \cB \ . \]
				
		\item For $\cB \to \cX$ a cocartesian fibration and $\cA \to \cY$ a cartesian fibration, there is a canonical equivalence
		\[ \int^\cY \cA \boxtimes \fcocartlowershriek(\cB) \simeq \int^\cX f^\ast(\cA) \boxtimes \cB \ . \]
	\end{enumerate}
\end{lem}

\begin{proof}
	To prove (1), it suffices to fix $\cE \in \Cat_\infty$ and observe that there is the following chain of natural equivalences:
	\begin{align*}
		\Map_{\Cat_\infty}\Big( \int^\cY \cA \boxtimes \fcartlowershriek(\cB) , \cE \Big) & \simeq \int_\cY \Map_{\Cat_\infty}( \Upsilon_\cA \boxtimes \Upsilon_{\fcartlowershriek(\cB)}, \cE ) \\
		& \simeq \int_\cY \Map_{\Cat_\infty}(\Upsilon_{\fcartlowershriek(\cB)}, \Fun(\Upsilon_\cA, \cE)) \\
		& \simeq  \Map_{\Cart_\cY}( \fcartlowershriek(\cB), \cE^\cA_{\mathrm c}) & \text{By \cite[Prop.\ 6.9]{Gepner_Lax_colimits}} \\
		& \simeq \Map_{\Cart_\cX}( \cB, f^\ast \cE^\cA_{\mathrm c} ) \\
		& \simeq \Map_{\Cart_\cX}( \cB, \cE^{f^\ast(\cA)}_{\mathrm c} ) & \text{By eq.\ \eqref{eq:fake_exponential}} \\
		& \simeq \int_\cX \Map_{\Cat_\infty}( \Upsilon_\cB, \Fun(\Upsilon_{f^\ast(\cA)}, \cE) ) \\
		& \simeq \Map_{\Cat_\infty}\Big( \int^\cX f^\ast (\cA) \boxtimes \cB, \cE \Big) \ ,
	\end{align*}
	so the conclusion follows from the Yoneda lemma.
	As for (2), it follows by the same argument, using $\cE^{\cA}_{\mathrm{cc}}$ instead of $\cE^{\cA}_{\mathrm c}$ and working in $\CoCart_\cY$ instead of in $\Cart_\cY$.
	\personal{Here's the full proof. Fix $\cE \in \Cat_\infty$. We have:
	\begin{align*}
		\Map_{\Cat_\infty}\Big( \int^\cY \cA \boxtimes \fcocartlowershriek(\cB) , \cE \Big) & \simeq \int_\cY \Map_{\Cat_\infty}( \Upsilon_\cA \boxtimes \Upsilon_{\fcocartlowershriek(\cB)}, \cE ) \\
		& \simeq \int_\cY \Map_{\Cat_\infty}(\Upsilon_{\fcocartlowershriek(\cB)}, \Fun(\Upsilon_\cA, \cE)) \\
		& \simeq  \Map_{\CoCart_\cY}( \fcocartlowershriek(\cB), \cE^\cA_{\mathrm{cc}}) & \text{By \cite[Prop.\ 6.9]{Gepner_Lax_colimits}} \\
		& \simeq \Map_{\CoCart_\cX}( \cB, f^\ast \cE^\cA_{\mathrm{cc}} ) \\
		& \simeq \Map_{\CoCart_\cX}( \cB, \cE^{f^\ast (\cA)}_{\mathrm{cc}} ) & \text{By eq.\ \ref{eq:fake_exponential}} \\
		& \simeq \int_\cX \Map_{\Cat_\infty}( \Upsilon_\cB, \Fun(\Upsilon_{f^\ast(\cA)}, \cE) ) \\
		& \simeq \Map_{\Cat_\infty}\Big( \int^\cX f^\ast(\cA) \boxtimes \cB, \cE \Big) \ ,
	\end{align*}
	so the conclusion follows from the Yoneda lemma.}
\end{proof}

Next, recall the following:

\begin{thm}[{\cite[Theorem 4.5]{Gepner_Lax_colimits}}]\label{thm:free_cartesian_fibration}
	Let $\cX$ be an $\infty$-category.
	The forgetful functor
	\[ \mathrm U_\cX \colon \Cart_\cX \to  \Cat_{\infty / \cX} \]
	admits a left adjoint $\mathrm F_\cX$.
\end{thm}

\begin{rem}\label{rem:free_cartesian_fibration}
	Given a functor $f \colon \cY \to \cX$, we refer to $\mathrm F_\cX(f)$ as the \emph{free cartesian fibration over $\cX$ generated by $f$}.
	It follows from the explicit description provided in \cite[Definition 4.1 \& Remark 4.4]{Gepner_Lax_colimits}, that $\mathrm F_\cX$ satisfies the following two conditions:
	\begin{enumerate}\itemsep=0.2cm
		\item when $f = \id_\cX$, $\mathrm F_\cX(\id_\cX) = \cX^{[1]}$, and the structural map is $\ev_1 \colon \cX^{[1]} \to \cX$.
		In other words, $\mathrm F_\cX(\id_\cX)$ classifies the functor
		\[ \cX_{-/} \colon \cX\op \to  \Cat_\infty \ . \]
		
		\item For a general $f \colon \cY \to \cX$, one has the following commutative diagram
		\[ \begin{tikzcd}
			\mathrm F_\cX(f) \arrow{r} \arrow{d} & \cX^{[1]} \arrow{r}{\ev_1} \arrow{d}{\ev_0} & \cX \\
			\cY \arrow{r}{f} & \cX & \phantom{\cX} \ ,
		\end{tikzcd} \]
		where the left square is a pullback and where the top horizontal composition is the structural map of the cartesian fibration $\mathrm F_\cX(f)$.
	\end{enumerate}
\end{rem}

\begin{lem}\label{lem:pushforward_free_cartesian_fibration}
	Let $f \colon \cX \to \cY$ be a functor of $\infty$-categories.
	Then there is a canonical equivalence
	\[ \fcartlowershriek( \mathrm F_\cX(\id_\cX) ) \simeq \mathrm F_\cY(f) \]
	in $\Cart_\cY$.
\end{lem}

\begin{proof}
	Indeed, for every cartesian fibration $\cB \to \cY$, we have:
	\begin{align*}
		\Map_{\Cart_\cY}( f_!(\mathrm F_\cX(\id_\cX)), \cB) & \simeq \Map_{\Cart_\cX}( \mathrm F_\cX(\id_\cX), f^\ast(\cB) ) \\
		& \simeq \Map_{/\cX}(\cX, f^\ast(\cB)) \\
		& \simeq \Map_{/\cY}(\cX, \cB) \\
		& \simeq \Map_{\Cart_\cY}(\mathrm F_\cY(f), \cB) \ ,
	\end{align*}
	so the conclusion follows from the Yoneda lemma.
\end{proof}

Finally, observe that \cite[Proposition 7.1]{Gepner_Lax_colimits} can be rewritten as follows:

\begin{cor}\label{cor:ninja_Yoneda}
	Let $\cX$ be an $\infty$-category and let $\cA \to \cX$ be a cocartesian fibration.
	Then there is a canonical equivalence
	\[ \cA \simeq \int^\cX \cA \boxtimes \mathrm F_\cX(\id_\cX) \]
	in $\Cat_\infty$.
\end{cor}

\subsection{Exponential pullback vs.\ global pullback}

Before proving \cref{prop:global_functoriality}-(1), let us revisit the proof of the equivalence \eqref{eq:specialization_equivalence} in terms of the categorical calculus we just introduced.

\begin{recollection}[{\cite[Proposition 4.1]{PortaTeyssier_Day}}]\label{recollection:specalization_equivalence_proof}
	Fix a cocartesian fibration $p \colon \cA \to \cX$ and a presentable $\infty$-category $\cE$.
	Using the equivalence $\PrL \simeq (\PrR)\op$, we see that the presentable fibration $\exp_{\cE}(\cA / \cX) \to \cX$ is at the same time a cocartesian and a cartesian fibration.
	Seen as a cartesian fibration, it classifies the functor
	\[ \Fun(\Upsilon_\cA, \cE) \colon \cX\op \to  \PrR \ . \]
	We use this second description to compute the sections of $\exp_{\cE,\cX}(\cA)$.
	Then the specialization equivalence $\spe_\cA$ is identified with the following composition of equivalences:
	\begin{align*}
		\Fun_{/\cX}(\cX, \exp_\cE(\cA / \cX)) & \simeq \Fun_{/\cX}^{\mathrm{cart}}( \mathrm F_\cX(\id_\cX), \exp_\cE(\cA / \cX) ) \\
		& \simeq \int_\cX \Fun( \cX_{-/}, \Fun(\Upsilon_\cA, \cE) ) & \text{By \cite[Prop.\ 6.9]{Gepner_Lax_colimits} \& Rem.\ \ref{rem:free_cartesian_fibration}} \\
		& \simeq \Fun\Big( \int^\cX \cA \boxtimes \mathrm F_\cX(\id_\cX), \cE \Big) \\
		& \simeq \Fun(\cA, \cE) & \text{By Cor.\ \ref{cor:ninja_Yoneda}.}
	\end{align*}
\end{recollection}

We are now ready for:

\begin{proof}[Proof of \cref{prop:global_functoriality}-(1)]
	Fix a pullback square
	\begin{equation}\label{eq:global_functoriality_pullback_I}
		\begin{tikzcd}
			\cB \arrow{d} \arrow{r}{u} & \cA \arrow{d} \\
			\cX \arrow{r}{f} & \cY
		\end{tikzcd}
	\end{equation}
	where the vertical functors are cocartesian fibrations.
	Recall from \cref{prop:functoriality_exponential}-(1) the canonical equivalence
	\[ f^\ast( \exp_\cE(\cA / \cY) ) \simeq \exp_\cE(\cB / \cX) \]
	We therefore obtain a canonical equivalence
	\begin{align*}
		\Sigma_\cX( \exp_\cE(\cB / \cX) ) & = \Fun_{/\cX}(\cX, \exp_\cE(\cB / \cX)) \\
		& \simeq \Fun_{/\cY}(\cX, \exp_\cE(\cA / \cY)) \\
		& \simeq \Fun_{/\cY}^{\mathrm{cart}}( \mathrm F_\cY(f), \exp_\cE(\cA / \cY) ) \ .
	\end{align*}
	Similarly,
	\[ \Sigma_\cY( \exp_\cE(\cA / \cY) ) \simeq \Fun_{/\cY}^{\mathrm{cart}}( \mathrm F_\cY(\id_\cY), \exp_\cE(\cA / \cY) ) \ . \]
	Since $\id_\cY$ is the final object in $\Cat_{\infty / \cY}$, we find a canonical map
	\[ \alpha_f \colon \mathrm F_\cY(f) \to  \mathrm F_\cY(\id_\cY) \]
	in $\Cart_\cY$ between free cartesian fibrations, and unwinding the definitions we find that the sections of the exponential pullback $\Sigma(\cE^u)$ are canonically identified with the functor
	\[ \alpha_f^\ast \colon \Fun_{/\cY}^{\mathrm{cart}}( \mathrm F_\cY(\id_\cY), \exp_\cE(\cA / \cY) ) \to  \Fun_{/\cY}^{\mathrm{cart}}( \mathrm F_\cY(f), \exp_\cE(\cA / \cY) ) \ . \]
	Applying the same chain of equivalences of \cref{recollection:specalization_equivalence_proof}, we find a canonical identification of $\alpha_f^\ast$ with the map
	\[ \Fun\Big( \int^\cY \cA \boxtimes \mathrm F_\cY(\id_\cY), \cE \Big) \to  \Fun\Big( \int^\cY \cA \boxtimes \mathrm F_\cY(f) , \cE \Big) \]
	induced by pullback along the canonical map
	\[ \beta_f \colon \cA \simeq \int^\cY \cA \boxtimes \mathrm F_\cY(\id_\cY) \to  \int^\cY \cA \boxtimes \mathrm F_\cY(f) \]
	constructed out of $\alpha_f$.
	Recall now from \cref{lem:pushforward_free_cartesian_fibration} that there is a canonical equivalence
	\[ \mathrm F_\cY(f) \simeq \fcartlowershriek(\mathrm F_\cX(\id_\cX)) \ , \]
	so that \cref{lem:change_of_variable} and \cref{cor:ninja_Yoneda} supply a canonical identification
	\[ \int^\cY \cA \boxtimes \mathrm F_\cY(f) \simeq \int^\cY \cA \boxtimes \fcartlowershriek(\mathrm F_\cX(\id_\cX)) \simeq \int^\cX f^\ast(\cA) \boxtimes \mathrm F_\cX(\id_\cX) \simeq \cB \ . \]
	Unwinding the definitions, we see that $\beta_f$ is identified with $u$, whence the conclusion.
\end{proof}

\subsection{Exponential induction vs.\ global induction}

We now deal with \cref{prop:global_functoriality}-(\ref{prop:global_functoriality:induction}).
Fix an $\infty$-category $\cX$ and consider a morphism
\[ \begin{tikzcd}[column sep=small]
	\cB \arrow{rr}{v} \arrow{dr} & & \cA \arrow{dl} \\
	& \cX
\end{tikzcd} \]
in $\CoCart_\cX$.
Applying $\exp_\cE(-/\cX)$, we find the morphism
\[ \begin{tikzcd}[column sep=small]
	\exp_\cE(\cB / \cX) \arrow{rr}{\cE^v_!} \arrow{dr} & & \exp_\cE(\cA / \cX) \arrow{dl} \\
	& \cX
\end{tikzcd} \]
in $\PrFibL_\cX$.

\begin{lem}\label{lem:exponential_induction_adjoint}
	The functor $\cE^v_!$ admits a right adjoint
	\[ \cE^{v,\ast} \colon \exp_\cE(\cA / \cX) \to  \exp_\cE(\cB / \cX) \]
	relative to $\cX$.
\end{lem}

\begin{proof}
	Since both $\exp_\cE(\cA / \cX)$ and $\exp_\cE(\cB / \cX)$ are cocartesian fibrations and $\cE^v_!$ preserves cocartesian edges, applying \cite[Proposition 7.3.2.6]{Lurie_Higher_algebra} shows that it is enough to prove that for every $x \in \cX$, the induced functor on the fibers at $x$
	\[ \cE^v_{!,x} \colon \exp_\cE(\cA / \cX)_x \to  \exp_\cE(\cB / \cX)_x \]
	admits a right adjoint.
	However, \cref{cor:fibers_exponential} identifies this functor with the left Kan extension
	\[ v_{x,!} \colon \Fun(\cA_x, \cE) \to  \Fun(\cB_x, \cE) \ , \]
	which is tautologically left adjoint to the restriction $v_x^\ast$.
	The conclusion follows.
\end{proof}

At this point, \cref{prop:global_functoriality} immediately follows from the following more precise statement:

\begin{prop} \label{prop:exponential_vs_global_induction}
	Keeping the same notations as above, both diagrams
	\[ \begin{tikzcd}[column sep=large]
		\Fun_{/\cX}(\cX, \exp_\cE(\cA/\cX)) \arrow{r}{\Sigma_\cX( \cE^v_! )} \arrow{r} \arrow{d}{\spe_\cA} & \Fun_{/\cX}(\cX, \exp_\cE(\cB/\cX)) \arrow{d}{\spe_\cB} \\
		\Fun(\cA, \cE) \arrow{r}{v_!} & \Fun(\cB, \cE)
	\end{tikzcd} \]
	and
	\[ \begin{tikzcd}[column sep=large]
		\Fun_{/\cX}(\cX, \exp_\cE(\cB/\cX)) \arrow{r}{\Sigma_\cX( \cE^{v,\ast} )}  \arrow{r} \arrow{d}{\spe_\cA} & \Fun_{/\cX}(\cX, \exp_\cE(\cA/\cX)) \arrow{d}{\spe_\cB} \\
		\Fun(\cB, \cE) \arrow{r}{v^\ast} & \Fun(\cA, \cE)
	\end{tikzcd} \]
	are canonically commutative.
\end{prop}

\begin{proof}
	Since $\cE^v_!$ is left adjoint to $\cE^{v,\ast}$ by \cref{lem:exponential_induction_adjoint}, it follows that $\Sigma_\cX(\cE^v_!)$ is left adjoint to $\Sigma_\cX(\cE^{v,\ast})$.
	Since $\spe_\cA$ and $\spe_\cB$ are equivalences, it is then enough to prove the commutativity of the second diagram.
	Notice that since $v$ preserves cocartesian arrows, it induces a natural transformation
	\[ \alpha_v \colon \Upsilon_\cA \times \cX_{-/} \to  \Upsilon_\cB \times \cX_{-/}  \]
	of functors $\cX \times \cX\op \to \Cat_\infty$.
	Following the construction of the specialization equivalence (see \cref{recollection:specalization_equivalence_proof}), we reduce to check that the map
	\[ \cA \simeq \int^\cX \cA \boxtimes \mathrm F_\cX(\id_\cX) \to  \int^\cX \cB \boxtimes \mathrm F_{\cX}(\id_\cX) \simeq \cB \]
	induced by $\alpha_v$ is canonically identified with $v$.
	This follows from \cref{cor:ninja_Yoneda} and the Yoneda lemma.
\end{proof}

\subsection{Change of coefficients}\label{subsec:change_of_coefficients}

It follows from 
\cite[Variant 3.20 \& Remark 3.21-(1)]{PortaTeyssier_Day}  and the functoriality of the tensor product of presentable $\infty$-categories that the exponential construction $\exp_\cE$ depends functorially on $\cE$.
In other words, we have a bifunctor
\[ \exp \colon \CoCart \times \PrL \to \PrFibL \ ,  \]
that sends a pair $(p \colon \cA \to \cX, \cE)$ to the presentable fibration $p \colon \exp_\cE(\cA / \cX) \to \cX$.

\medskip

Let $f \colon \cE \to \cE'$ be a morphism in $\PrL$ and fix a cocartesian fibration $p \colon \cA \to \cX$.
The functor $f$ induces morphisms
\[ f^{\cA/\cX} \colon \exp_\cE(\cA/\cX) \to \exp_{\cE'}(\cA / \cX) \qquad \text{and} \qquad f \colon \Fun(\cA, \cE) \to \Fun(\cA, \cE') \ , \]
in $\PrFibL$ and in $\PrL$, respectively.
Here we wrote $f$ in place of the more accurate $f \circ (-)$, to keep the notations light.
These two operations are related by the following relation:

\begin{prop}\label{prop:change_of_coefficients}
	Keeping the above notations, the diagram
	\[ \begin{tikzcd}[column sep=large]
		\Fun_{/\cX}( \cX, \exp_\cE(\cA / \cX) ) \arrow{r}{\Sigma_\cX(f^{\cA / \cX})} \arrow{d}{\spe_{\cA}^\cE} & \Fun_{/\cX}(\cX, \exp_{\cE'}(\cA / \cX)) \arrow{d}{\spe_\cA^{\cE'}} \\
		\Fun(\cA, \cE) \arrow{r}{f} & \Fun(\cA,\cE') 
	\end{tikzcd} \]
	commutes.
\end{prop}

\begin{proof}
	This simply follows unraveling the chain of equivalences in \cref{recollection:specalization_equivalence_proof} and observing that they are natural in $\cE$.
\end{proof}

Finally, let us observe that $f^{\cA / \cX}$ is natural in $\cA \to \cX$:

\begin{prop}\label{prop:change_of_coefficients_naturality}
	Let
	\[ \begin{tikzcd}
		\cB \arrow{d} & \cB_\cX \arrow{d} \arrow{r}{v} \arrow{l}[swap]{u} & \cA \arrow{dl} \\
		\cY & \cX \arrow{l}[swap]{f}
	\end{tikzcd} \]
	be a morphism in $\CoCart$ and let $f \colon \cE \to \cD$ be a morphism in $\PrL$.
	Then the diagram
	\[ \begin{tikzcd}
		\exp_{\cE}(\cB / \cY) \arrow{d}{f^{\cB/\cY}} & \exp_{\cE}(\cB_\cX / \cX) \arrow{l}[swap]{\cE^u} \arrow{r}{\cE^v_!} \arrow{d}{f^{\cB_\cX / \cX}} & \exp_\cE(\cA / \cX) \arrow{d}{f^{\cA /\cX}} \\
		\exp_{\cD}(\cB / \cY) & \exp_{\cD}(\cB_\cX / \cX) \arrow{l}[swap]{\cD^u} \arrow{r}{\cD^v_!} & \exp_{\cD}(\cA / \cX)
	\end{tikzcd} \]
	commutes and the left square is a pullback.
	In particular, the diagram
	\[ \begin{tikzcd}
		\Fun(\cB, \cE) \arrow{d}{f} \arrow{r}{u^\ast} & \Fun(\cB_\cX, \cE) \arrow{d}{f} \arrow{r}{v_!} & \Fun(\cA, \cE) \arrow{d}{f} \\
		\Fun(\cB, \cD) \arrow{r}{u^\ast} & \Fun(\cB_\cX, \cD) \arrow{r}{v_!} & \Fun(\cA, \cE)
	\end{tikzcd} \]
	commutes.
\end{prop}

\begin{proof}
	The first half simply follows from the bifunctoriality of $\exp \colon \CoCart \times \PrL \to \PrFibL$.
	The second half follows applying $\Sigma$ and combining Propositions~\ref{prop:global_functoriality} and \ref{prop:change_of_coefficients}.
	Alternatively, the second half can be proven directly observing that, since $f$ commutes with colimits, it also commutes with the formation of arbitrary left Kan extensions.
\end{proof}

\section{Cocartesian functors}\label{cocar_functor_section}

\subsection{The space of specialization morphisms}

Fix a cocartesian fibration $p \colon \cA \to \cX$ as well as a presentable $\infty$-category $\cE$.
Write
\[ p_\cE \colon \exp_\cE(\cA/\cX) \to \cX \]
for the structural morphism of the exponential construction of $p$.
Recall from \cref{prop:sections_exponential} that there is a canonical equivalence
\[ \spe_{\cX,p}^\cE \colon \Fun(\cA, \cE) \simeq \Fun_{/\cX}(\cX, \exp_\cE(\cA/\cX)) \ , \]
which we refer to as the \emph{specialization equivalence}.
When $\cX$, $p$ and $\cE$ are clear out of the context, we drop the decorations and write $\spe$ instead of $\spe_{\cX,p}^\cE$.

\begin{rem}\label{rem:specialization_vs_evaluation_I}
	Recall from \cref{eg:exponential_over_trivial_base}-(1) that the fiber of $\exp_\cE(\cA/\cX)$ at $x \in \cX$ is canonically identified with $\Fun(\cA_x,\cE)$.
	In particular, for $F \colon \cA \to \cE$, the value $(\spe F)_x$ of the section $\spe F$ on $x$ is a functor $(\spe F)_x \colon \cA_x \to \cE$.
	Denoting by $j_x \colon \cA_x \to \cA$ the natural inclusion, \cref{prop:evaluation_specialization} supplies a canonical identification $(\spe F)_x \simeq j_x^\ast(F)$.
\end{rem}

\begin{defin}\label{def:space_specialization_morphisms}
	Let $F \in \Fun(\cA,\cE)$ be a functor and let $\gamma \colon x \to y$ be a morphism in $\cX$.
	The \emph{space of specialization morphisms for $F$ relative to $\gamma$} is the space $\SPE_\gamma(F)$
	\[ \begin{tikzcd}[column sep=small]
		(\spe F)_x \arrow{r}{\beta} & G \arrow{r}{\alpha} & (\spe F)_y
	\end{tikzcd} \]
	where $\beta$ is a $p_\cE$-cocartesian lift of $\gamma$ in $\exp_\cE(\cA/\cX)$.
	In this case, we say that $\alpha$ is a \emph{specialization morphism for $F$ relative to $\gamma$}.
\end{defin}

\begin{rem}
	Since $\exp_\cE(\cA / \cX)$ is also a cartesian fibration, there is a dual notion of \emph{cospecialization morphism}, that are obtained choosing $p_\cE$-cartesian lifts of $\gamma$.
\end{rem}

We immediately discuss a fundamental example.

\begin{notation}\label{notation:simplexes_cocartesian_fibration}
	Let $p \colon \cA \to \cX$ be a cocartesian fibration.
	For $\sigma \colon \Delta^n \to \cX$, write $\cA_\sigma \coloneqq \Delta^n \times_\cX \cA$ and $p_\sigma \colon \cA_\sigma \to \Delta^n$ for the induced cocartesian fibration.
	Notice that \cref{prop:functoriality_exponential} provides a canonical and functorial identification
	\[ \exp_\cE(\cA/\cX)_\sigma \simeq \exp_\cE(\cA_\sigma / \Delta^n) \ . \]
\end{notation}

\begin{eg}\label{eg:specialization_over_Delta1}
	Let $p \colon \cA \to \cX$ be a cocartesian fibration and let $\gamma \colon x \to y$ be a morphism in $\cX$.
	\emph{Choose} a straightening
	\[ f \colon \cA_x \to  \cA_y \]
	for $p_\gamma \colon \cA_\gamma \to \Delta^1$.
	The functor $f$ fits in the following triangle
	\[ \begin{tikzcd}[column sep=small]
		\cA_x \arrow{rr}{f_\gamma} \arrow{dr}[swap]{j_x} & & \cA_y \arrow{dl}{j_y} \\
		{} & \cA
	\end{tikzcd} \]
	where $j_x$ and $j_y$ denote the canonical inclusions of the fibers of $p$ inside $\cA$.
	This triangle is \emph{not commutative} but we can choose a natural transformation
	\[ s \colon j_x \to  j_y \circ f_\gamma \]
	in $\Fun(\cA_x, \cE)$ with the property that for every $a \in \cA_x$ the morphism $s_\gamma(a) \colon j_x(a) \to j_y(f_\gamma(a))$ is $p$-cocartesian in $\cA$.
	Applying the contravariant functor $\Fun(-,\cE)$ we obtain a natural transformation
	\[ s^\ast \colon j_x^\ast \to  f_\gamma^\ast \circ j_y^\ast \]
	of functors from $\Fun(\cA, \cE) \to \Fun(\cA_x, \cE)$.
	There is therefore an induced Beck-Chevalley morphism
	\begin{equation}\label{eq:Beck_Chevalley_specialization}
		\alpha_{f,s} \colon f_{\gamma,!} \circ j_x^\ast \to  j_y^\ast \ .
	\end{equation}
	Unraveling the definition of $\exp_{\cE}(\cA / \cX)$ we see that for every $F \colon \cA \to \cE$, the induced morphism
	\[ \alpha_{f,s}(F) \colon f_{\gamma,!}j_x^\ast(F) \to  j_y^\ast(F) \]
	is a specialization morphism for $F$ relative to $\gamma$.
\end{eg}

\begin{rem}\label{rem:specialization_space}
	Let $p \colon \cA \to \cX$ be a cocartesian fibration and let $F \colon \cA \to \cE$ be a fixed functor.
	Since $p_\cE \colon \exp_\cE(\cA/\cX) \to \cX$ is a cocartesian fibration, it immediately follows that the space $\SPE_\gamma(F)$ is contractible.
	Observe that, in the setting of \cref{eg:specialization_over_Delta1}, neither $f$ nor $s$ are uniquely determined in a strict sense (although the spaces of choices for the pair $(f,s)$ is contractible).
	Every such choice gives rise to an element $\SPE_\gamma(F)$, whose underlying specialization morphism is $\alpha_{f,s}(F)$.
	The contractibility of $\SPE_\gamma(F)$ shows that the actual choices for $f$ and $s$ are immaterial as they give rise to equivalent specialization morphisms, and this in a homotopy unambiguous way.
	\personal{(Mauro) Notice however, that following \cref{construction:exponential}, $\exp_\cE(\cA/\cX)$ itself depends on a straightening. The super-functorial construction of \cite{PortaTeyssier_Day} of the functor $\exp_\cE$ does not solve this issue: it ultimately discharge the responsibility on the functor $\PPSh$, which arises as the left adjoint to the forgetful from cocomplete cocartesian fibrations to plain cocartesian fibrations. Now, adjoint functors are only well defined up to a contractible space of choices, which is the same as saying that we are not making any specific choice, but rather we \emph{can} make a choice and that everything else we will say will be (automatically) independent of this choice.
	This is the usual picture I have in mind for whatever $\infty$-categorical construction I judge acceptable.}
\end{rem}

\begin{eg}\label{eg:specialization_over_Delta1_continuation}
	We maintain the notation introduced in \cref{eg:specialization_over_Delta1}.
	It is worth unpacking the specialization equivalence when $\cX = \Delta^1$.
	Write $\Delta^1 = \{ \gamma \colon x \to y \}$ and fix a cocartesian fibration $p \colon \cA \to \Delta^1$ together with a straightening $f_\gamma \colon \cA_x \to \cA_y$ and a natural transformation $s \colon j_x \to j_y \circ f_\gamma$ as in \cref{eg:specialization_over_Delta1}.
	Notice that $\Tw(\Delta^1)$ can be represented as
	\[ \begin{tikzcd}
		{} & x \arrow{d}{\gamma} \arrow{dr}{\gamma} \\
		x \arrow[equal]{ur} \arrow[equal]{d} & y & y \arrow[equal]{d} \\
		x \arrow[equal]{u} \arrow{ur}{\gamma} & & y \arrow[equal]{u} \arrow[equal]{ul} ,
	\end{tikzcd} \]
	where the vertical arrows are the objects of $\Tw(\Delta^1)$.
	In other words, $\Tw(\Delta^1)$ is equivalent to $\mathrm{Span} = \{ * \leftarrow * \rightarrow *\}$.
	It follows that the chain of equivalences of \cref{recollection:specalization_equivalence_proof} in this case simply asserts that the square
	\[ \begin{tikzcd}
		\Fun(\cA, \cE) \arrow{d} \arrow{r}{j_y^*} & \Fun(\cA_y, \cE) \arrow{d}{f_\gamma^*} \\
		\Fun(\Delta^1, \Fun(\cA_x, \cE)) \arrow{r}{\ev_1} & \Fun(\cA_x, \cE) .
	\end{tikzcd} \]
	is a pullback.
	Unraveling the definitions, we see that the left vertical map sends $F \colon \cA \to \cE$ to $s^\ast \colon j_x^\ast(F) \to f_\gamma^\ast( j_y^\ast(F) )$.
	Vice-versa, given
	\[ F_x \colon \cA_x \to  \cE \ , \qquad F_y \colon \cA_y \to  \cE \]
	and a natural transformation
	\[ \alpha \colon F_x \to  f_\gamma^\ast(F_y) \ , \]
	we can produce a functor $F \colon \cA \to \cE$ together with the following data:
	\begin{enumerate}\itemsep=0.2cm
		\item equivalences $\beta_x \colon F_x \simeq j_x^\ast(F)$ and $\beta_y \colon F_y \simeq j_y^\ast(F)$;
		
		\item whenever $\phi \colon a \to f_\gamma(a)$ is a $p$-cocartesian morphism in $\cA$, an equivalence
		\[ \beta_a \colon F(\phi) \simeq \alpha(a) \]
		in $\Map_{\cE}( F_x(a), F_y(f_\gamma(a)) )$.
	\end{enumerate}
\end{eg}

The above analysis allows to obtain an improvement on \cref{prop:evaluation_specialization}.
To state it, we need to first introduce the following:

\begin{notation}\label{notation:specialization_on_morphisms}
	Let $\gamma \colon x \to y$ be a morphism in $\cX$.
	Let $F \in \Fun(\cA_x, \cE)$ and $G \in \Fun(\cA_y, \cE)$ and let $\alpha \colon F \to G$ be a morphism in $\exp_\cE(\cA / \cX)$ lying over $\gamma$.
	We can factor $\alpha$ as
	\[ \begin{tikzcd}
		F \arrow{r}{\alpha_0} & G' \arrow{r}{\alpha_1} & G \ ,
	\end{tikzcd} \]
	where $\alpha_1$ is $p_\cE$-cartesian.
	Unraveling the definitions, we see that for every $p$-cocartesian lift $\phi \colon a \to b$ of $\gamma$, $\alpha_1$ induces a canonical equivalence $\alpha_1(\phi) \colon G'(a) \simeq G(b)$, and in particular we obtain a well defined morphism
	\[ \alpha(\phi) \coloneqq \alpha_1(\phi) \circ \alpha_0(a) \colon F(a) \to  G(b) \]
	in $\cE$.
\end{notation}

\begin{cor}\label{cor:specialization_on_morphisms}
	Let $F \colon \cA \to \cE$ be a functor and let $\phi \colon a \to b$ be a $p$-cocartesian morphism in $\cA$.
	Then there is a canonical identification
	\[ F(\phi) \simeq (\spe F)_{p(\phi)}(\phi) \]
	of morphisms in $\cE$.
\end{cor}

\begin{proof}
	Using \cref{prop:global_functoriality}-(\ref{prop:global_functoriality:pullback}) we can assume without loss of generality that $\cX = \Delta^1$.
	Choose a straightening $f_\gamma \colon \cA_x \to \cA_y$ together with a morphism $s \colon j_x \to j_y \circ f_\gamma$ as in \cref{eg:specialization_over_Delta1}.
	Using \cref{prop:evaluation_specialization} we see that $(\spe F)_\gamma$ can be factored as
	\[ \begin{tikzcd}
		j_x^\ast(F) \arrow{r}{s^\ast} & f_\gamma^\ast(j_y^\ast(F)) \arrow{r} & j_y^\ast(F) \ ,
	\end{tikzcd} \]
	where the second morphism is $p_\cE$-cartesian.
	With these choices, the notation introduced in \cref{notation:specialization_on_morphisms} collapses to $(\spe F)_{p(\phi)}(a) \simeq s^\ast(a)$, so the conclusion follows from the analysis of the specialization equivalence over $\Delta^1$ carried out in \cref{eg:specialization_over_Delta1_continuation}.
\end{proof}

\subsection{Cocartesian functors}

We fix as usual a cocartesian fibration $p \colon \cA \to \cX$ and a presentable $\infty$-category $\cE$.
We let $p_\cE \colon \exp_\cE(\cA / \cX) \to \cX$ be the canonical projection.

\begin{defin}\label{def:cocartesian_functor}
	Let $F \colon \cA \to \cE$ be a functor and let $\gamma \colon x \to y$ be a morphism in $\cX$.
	We say that $F$ is \emph{cocartesian at $\gamma$} if every specialization morphism for $F$ relative to $\gamma$ is an equivalence in $\Fun(\cA_y,\cE)$.
	
	\medskip
	
	\noindent We say that $F$ is \emph{cocartesian} if it is cocartesian at every morphism $\gamma$ of $\cX$.
	We write $\Funcocart(\cA, \cE)$ for the full subcategory of $\Fun(\cA, \cE)$ spanned by cocartesian functors.
\end{defin}

\begin{rem}\label{rem:contractibility_specialization_space}
	Recall from \cref{rem:specialization_space} that $\SPE_\gamma(F)$ is a contractible space.
	In particular, in order to check that $F$ is cocartesian at $\gamma$, it is enough to check that there exists \emph{one} specialization morphism $\alpha$ that is an equivalence.
\end{rem}

We now collect a couple of elementary facts concerning these objects.
We keep the cocartesian fibration $p \colon \cA \to \cX$ and the presentable $\infty$-category $\cE$ fixed in all the following statements:

\begin{prop}\label{prop:cocartesian_functors_reformulation}
	Let $F \colon \cA \to \cE$ be a functor and let $\gamma \colon x \to y$ be a morphism in $\cX$.
	The following statements are equivalent:
	\begin{enumerate}\itemsep=0.2cm
		\item $F$ is cocartesian at $\gamma$;
		
		\item the specialization $\spe F \colon \cX \to \exp_\cE(\cA / \cX)$ takes $\gamma$ to a $p_\cE$-cocartesian edge;
		
		\item let $f_\gamma \colon \cA_x \to \cA_y$ be \emph{any} straightening for $p_\gamma \colon \cA_\gamma \to \Delta^1$.
		Then the canonical Beck-Chevalley transformation \eqref{eq:Beck_Chevalley_specialization}
		\[ f_{\gamma,!} j_x^\ast(F) \to  j_y^\ast(F) \]
		is an equivalence.
	\end{enumerate}
\end{prop}

\begin{proof}
	Any element of $\SPE_\gamma(F)$ corresponds to a factorization
	\[ \begin{tikzcd}[column sep=small]
		(\spe F)_x \arrow{rr}{\beta} \arrow{dr}[swap]{(\spe F)_\gamma} & & G \arrow{dl}{\alpha} \\
		{} & (\spe F)_y
	\end{tikzcd} \]
	inside $\exp_\cE(\cA / \cX)$, where $\beta$ is $p_\cE$-cocartesian and $\alpha$ is the associated specialization morphism.
	It follows that $(\spe F)_\gamma$ is $p_\cE$-cocartesian if and only if $\alpha$ is an equivalence.
	This shows that (1) $\Leftrightarrow$ (2).
	The equivalence (2) $\Leftrightarrow$ (3) follows now from \cref{eg:specialization_over_Delta1}.
\end{proof}

\begin{cor}\label{cor:cocart_presentable_stable}
	Denoting $\Upsilon_\cA \colon \cX \to \Cat_\infty$ the straightening of the cocartesian fibration $\cA \to \cX$, there are canonical equivalences
	\[ \Funcocart(\cA, \cE) \simeq \Sigma^{\cocart}_\cX(\exp_\cE(\cA/\cX)) \simeq \lim_\cX \Fun_!(\Upsilon_\cA, \cE) \]
	In particular:
	\begin{enumerate}\itemsep=0.2cm
		\item $\Funcocart(\cA, \cE)$ is presentable;
		
		\item if $\cE$ is stable, $\Funcocart(\cA, \cE)$ is stable.
	\end{enumerate}
\end{cor}

\begin{proof}
	Combining the specialization equivalence \eqref{eq:specialization_equivalence} and the equivalence (1) $\Leftrightarrow$ (2) of \cref{prop:cocartesian_functors_reformulation}, we see that $\Funcocart(\cA, \cE)$ coincides with the full subcategory of $\Fun_{/\cX}(\cX, \exp_\cE(\cA / \cX))$ spanned by cocartesian sections
	This proves the first equivalence, and the second follows directly from \cite[Proposition 3.3.3.1]{HTT}.
	For point (1) it is now sufficient to observe that the functor $\Fun_!(\Upsilon_\cA, \cE) \colon \cX \to \Cat_\infty$ takes values in $\PrL$, so the conclusion follows from \cite[Propositon 5.5.3.13]{HTT}.
	Similarly, point (2) follows from \cite[Theorem 1.1.4.4]{Lurie_Higher_algebra}.
\end{proof}

\begin{warning}
	There is another natural condition that we can impose on a functor $F \colon \cA \to \cE$: namely, we can ask that $F$ takes $p$-cocartesian arrows in $\cA$ to equivalences in $\cE$.
	This condition cuts a full subcategory $\Fun'(\cA, \cE)$ of $\Fun(\cA,\cE)$, that however does not coincide with $\Funcocart(\cA,\cE)$.
	Indeed, \cite[Corollary 3.3.4.3]{HTT} yields an identification
	\[ \Fun'(\cA, \cE) \simeq \Fun\big( \colim_{\cX} \Upsilon_{\cA}, \cE \big) \simeq \lim_{\cX\op} \Fun^\ast( \Upsilon_{\cA}, \cE ) \simeq \Sigma_\cX^{\cart}(\exp_\cE(\cA / \cX)) \ , \]
	where $\Sigma_\cX^{\cart}$ denotes the functor of \emph{cartesian} sections.
\end{warning}

\begin{cor}\label{cocart_at_equivalence}
	A functor $F \colon \cA \to \cE$ is cocartesian at every equivalence of $\cX$.
\end{cor}

\begin{proof}
	Immediate from the equivalence (1) $\Leftrightarrow$ (2) of \cref{prop:cocartesian_functors_reformulation} and \cite[2.4.1.5]{HTT}.
\end{proof}

\begin{cor}\label{cocart_almost_2_to_3}
	Let
	\[ \begin{tikzcd}[column sep=small]
		{} & y \arrow{dr}{\gamma_1} \\
		x \arrow{ur}{\gamma_0} \arrow{rr}{\gamma_2} & & z. 
	\end{tikzcd} \]
	be a commutative triangle in $\cX$.
	Let $F \colon \cA \to \cE$ be a functor, and assume that it is cocartesian at $\gamma_0$.
	Then $F$ is cocartesian at $\gamma_1$ if and only if it is cocartesian at $\gamma_2$.
\end{cor}

\begin{proof}
	Immediate from the equivalence (1) $\Leftrightarrow$ (2) of \cref{prop:cocartesian_functors_reformulation} and from \cite[2.4.1.7]{HTT}.
	\personal{Old version: write
	\[ F_x \coloneqq (\spe F)_x, \quad F_y \coloneqq (\spe F)_y, \quad F_z \coloneqq (\spe F)_z . \]
	Applying $\spe F$ to the above $2$-simplex in $\cX$ yields the following $2$-simplex in $\cE^{\cA}$
	\[ \begin{tikzcd}[ampersand replacement=\&]
		{} \& F_y \arrow{dr}{ (\spe F)_{\gamma_1} } \\
		F_x \arrow{ur}{ (\spe F)_{\gamma_0}  } \arrow{rr}{ (\spe F)_{\gamma_2} } \& \& F_z. 
	\end{tikzcd} \]
	Then, \cref{cocart_almost_2_to_3} follows from  \cite[2.4.1.7]{HTT} applied to the cartesian fibration $p_{\cE} \colon  \cE^{\cA}\to  \cX$.}
\end{proof}

\begin{cor}\label{local_cocart_stability_colimits}
Let $p:\cA\to \cX$ be a cocartesian fibration.
Let $\gamma : x\to y$ be a morphism in $\cX$.
Let $\cE$ be a presentable $\infty$-category.
Then, the full subcategory of $\Fun(\cA,\cE)$ spanned by functors cocartesian at $\gamma$  is stable under colimits.
\end{cor}

\begin{proof}
This follows from the equivalence (1) $\Leftrightarrow$ (3) in \cref{prop:cocartesian_functors_reformulation} and the fact that the functors $f_{\gamma,!}$, $j_x^\ast$ and $j_y^\ast$ commute with colimits.
\end{proof}

\begin{prop}\label{cor:cocartesianization}
	Let $p \colon \cA \to \cX$ be a cocartesian fibration and let $\cE$ be a presentable $\infty$-category.
	Then $\Funcocart(\cA, \cE)$ is stable under colimits $\Fun(\cA, \cE)$. 
	In particular, $\Funcocart(\cA, \cE)$ is a coreflective subcategory of $\Fun(\cA, \cE)$, that is the inclusion
	\begin{equation}\label{eq:inclusion_cocartesian}
		\Funcocart(\cA, \cE) \hookrightarrow \Fun(\cA, \cE)
	\end{equation}
	admits a right adjoint.
\end{prop}

\begin{proof}
	We know from \cref{cor:cocart_presentable_stable} that $\Funcocart(\cA, \cE)$ is presentable.
	It is thus enough to check that $\Funcocart(\cA, \cE)$ is stable under colimits in $\Fun(\cA, \cE)$, which follows from \cref{local_cocart_stability_colimits}.
\end{proof}

\begin{defin}\label{defin_push_forward_cocart}
	Let $p \colon \cA \to \cX$ be a cocartesian fibration and let $\cE$ be a presentable $\infty$-category.
	We denote by 
	\[ (-)^{\cocart} \colon  \Fun(\cA, \cE)  \to  \Funcocart(\cA, \cE) \]
	the  right adjoint of the inclusion \eqref{eq:inclusion_cocartesian}, and refer to $(-)^{\cocart}$ as the \textit{cocartesianization functor}.
\end{defin}

\begin{rem}
	The functor $(-)^{\cocart}$ can be explicitly computed in some specific situations.
	See \cref{computation_cocart_initial_object}.
\end{rem}

Under extra stability and fiberwise compactness conditions, \cref{local_cocart_stability_colimits} and \cref{cor:cocartesianization}
have the following counterparts for limits :

\begin{lem}\label{local_cocart_stability_limits}
Let $p:\cA\to \cX$ be a cocartesian fibration  and let $\gamma : x\to y$ be a morphism in $\cX$ such that $\cA_x$ is compact and $\cA_y$ is proper (see \cref{proper_category}).
Let $\cE$ be a presentable stable $\infty$-category.
Then, the full subcategory of $\Fun(\cA,\cE)$ spanned by functors cocartesian at $\gamma$  is closed under limits.
\end{lem}
\begin{proof}
This follows from the equivalence (1) $\Leftrightarrow$ (3) in \cref{prop:cocartesian_functors_reformulation} and the fact that the functors $f_{\gamma,!}$, $j_x^\ast$ and $j_y^\ast$ commute with limits in virtue of \cref{induction_limit_stable}.
\end{proof}

\begin{prop}\label{cocart_stability_limits}
Let $p:\cA\to \cX$ be a cocartesian fibration with compact and proper fibers.
Let $\cE$ be a presentable stable $\infty$-category.
Then $\Funcocart(\cA, \cE)$ is stable under limits $\Fun(\cA, \cE)$. 
In particular $\Funcocart(\cA, \cE)$ is a reflective subcategory of $\Fun(\cA, \cE)$, that is the inclusion
	\begin{equation}\label{eq:inclusion_cocartesian}
		\Funcocart(\cA, \cE) \hookrightarrow \Fun(\cA, \cE)
	\end{equation}
	admits a left adjoint.
\end{prop}

\begin{proof}
We know from \cref{cor:cocart_presentable_stable} that $\Funcocart(\cA, \cE)$ is presentable.
	It is thus enough to check that $\Funcocart(\cA, \cE)$ is stable under limits in $\Fun(\cA, \cE)$, which follows from \cref{local_cocart_stability_limits}.
\end{proof}

\subsection{Functoriality of cocartesian functors}

We fix as usual a cocartesian fibration $p \colon \cA \to \cX$ and a presentable $\infty$-category $\cE$.
We saw in \cref{cor:cocart_presentable_stable} that there is a canonical equivalence
\[ \Funcocart(\cA, \cE) \simeq \Sigma^{\cocart}(\exp_{\cE}(\cA / \cX)) \ . \]
Therefore, it follows from \cite[Corollary 3.23]{PortaTeyssier_Day} that this construction depends functorially on the cocartesian fibration $\cA \to \cX$ seen as an element of $\CoCart$.
We now make this explicit in terms of the \emph{lax} functoriality of $\Fun(\cA, \cE) \simeq \Sigma(\exp_\cE(\cA / \cX))$ in $\cA \to \cX$.

\begin{prop}\label{prop:cocartesian_functoriality}
	Let
	\[ \begin{tikzcd}
		\cB \arrow{d} & \cB_\cX \arrow{d} \arrow{r}{v} \arrow{l}[swap]{u} & \cA \arrow{dl} \\
		\cY & \cX \arrow{l}[swap]{f}
	\end{tikzcd} \]
	be a morphism in $\CoCart$.
	Then:
	\begin{enumerate}\itemsep=0.2cm
		\item \label{prop:cocartesian_functoriality:pullback} if $F \colon \cB \to \cE$ is a cocartesian functor, the same goes for $u^\ast(F) \colon \cB_\cX \to \cE$;
		
		\item \label{prop:cocartesian_functoriality:induction} if $G \colon \cB_X \to \cE$ is a cocartesian functor, then the same goes for $v_!(G) \colon \cA \to \cE$.
	\end{enumerate}
	In particular the functors
	\[ u^\ast \colon \Fun(\cB, \cE) \to  \Fun(\cB_\cX, \cE) \qquad \text{and} \qquad v_! \colon \Fun(\cB_\cX, \cE) \to  \Fun(\cA, \cE) \]
	restrict to well-defined functors
	\[ u^\ast \colon \Funcocart(\cB, \cE) \to  \Funcocart(\cB_\cX, \cE) \qquad \text{and} \qquad v_! \colon \Funcocart(\cB_\cX, \cE) \to  \Funcocart(\cA, \cE) \ . \]
\end{prop}

This proposition results of the following two more precise lemmas:

\begin{lem} \label{cor:change_of_base_cocartesian}
	Let
	\[ \begin{tikzcd}
		\cA \arrow{r}{u} \arrow{d} & \cB \arrow{d} \\
		\cX \arrow{r}{f} & \cY
	\end{tikzcd} \]
	be a pullback square in $\Cat_\infty$, where the vertical morphisms are cocartesian fibrations.
	Fix a morphism $\gamma$ in $\cX$ and a functor $F \colon \cB \to \cE$.
	Then $u^\ast(F)$ is cocartesian at $\gamma$ if and only if $F$ is cocartesian at $f(\gamma)$.
\end{lem}

\begin{proof}
	Under the specialization equivalence \eqref{eq:specialization_equivalence} and \cref{prop:global_functoriality}-(\ref{prop:global_functoriality:pullback}), the statement follows from \cref{prop:cocartesian_functors_reformulation} and from \cite[Proposition 2.4.1.3-(2)]{HTT} applied to the square
	\[ \begin{tikzcd}
		\exp_\cE(\cA / \cX) \arrow{r}{\cE^u} \arrow{d} & \exp_\cE(\cB / \cY) \arrow{d} \\
		\cX \arrow{r}{f} & \cY \ ,
	\end{tikzcd} \]
	which is a pullback thanks to \cref{prop:functoriality_exponential}-(1).
	\personal{Old argument:
	Let
	\[ \begin{tikzcd}[ampersand replacement = \&]
		(\spe_{\cX, p}^\cE F)(f(x)) \arrow{r}{\beta} \& G \arrow{r}{\alpha} \& (\spe_{\cX, p}^\cE F)(f(y))
	\end{tikzcd} \]
	be a factorization of $(\spe_{\cX, p}^\cE F)(f(\gamma))$ in $\cE^\cA$, where $\beta$ is locally $p_\cE$-cocartesian.
	Using \cref{prop:specialization_equivalence_change_of_base}, we obtain a canonical equivalence
	\[ (\spe_{\cY,q}^\cE g^* F)(x) \simeq f^*( \spe_{\cX,p}^\cE  F)(x)  \]
	in $\Fun(\cB_x,\cE)$.	
	Because the square \eqref{eq:specialization_change_of_base} is a pullback, we have a canonical equivalence
	\[ (\cE^\cB)_x \simeq   (\cE^\cA)_{f(x)}\]
	through which $f^*( \spe_{\cX,p}^\cE  F)(x)$ corresponds to 
	$(\spe_{\cX, p}^\cE F)(f(x))$.
	Hence, we can interpret the above factorization as a factorization
	\[ \begin{tikzcd}[ampersand replacement = \&]
		( \spe_{\cY, q}^\cE g^* F )(x) \arrow{r}{\beta} \& G \arrow{r}{\alpha} \& (\spe_{\cY, q}^\cE g^* F)(y) ,
	\end{tikzcd} \]
	in $\cE^{\cB}$ with $\beta$ locally $q_\cE$-cocartesian.
	The conclusion follows.}
\end{proof}

\begin{lem}\label{induction_preserves_cocart_functors}
	Let $\cX$ be an $\infty$-category and consider a morphism
	\[ \begin{tikzcd}[column sep=small]
		\cA \arrow{rr}{v} \arrow{dr} & & \cB \arrow{dl} \\
		{} & \cX
	\end{tikzcd} \]
	in $\CoCart_\cX$.
	Let $\gamma$ be a morphism in $\cX$ and let $F \colon \cA \to \cE$ be a functor.
	If $F$ is cocartesian at $\gamma$, then the same goes for $v_!(F)$.
\end{lem}

\begin{proof}
	In virtue of \cref{prop:cocartesian_functors_reformulation}, we have to prove that the section
	\[ \spe( v_!(F) ) \colon \cX \to  \exp_\cE(\cB / \cX) \]
	takes $\gamma$ to a cocartesian edge in $\exp_\cE(\cB / \cX)$.
	Using \cref{prop:global_functoriality}-(\ref{prop:global_functoriality:induction}), we find a canonical identification
	\[ \spe( v_!(F) ) \simeq \cE^v_! \circ \spe(F) \ , \]
	where $\cE^v_! \colon \exp_\cE(\cA / \cX) \to \exp_\cE(\cB / \cX)$ is the exponential induction functor.
	The conclusion now follows from \cref{prop:functoriality_exponential}-(2), that guarantees that $\cE^v_!$ preserves cocartesian edges.
\end{proof}

We conclude with a handy consequence:

\begin{cor}\label{right_adjoint_pull_back_cocart}
	In the setting of \cref{cor:change_of_base_cocartesian}, the composition 
	\[ u_*^{\cocart} \coloneqq (-)^{\cocart}\circ u_* \colon \Funcocart(\cB, \cE)\to  \Funcocart(\cA, \cE) \]
	is right adjoint to the pull-back functor $u^* \colon \Funcocart(\cA, \cE) \to \Funcocart(\cB, \cE)$.
\end{cor}

\begin{proof}
	The functors at play are well-defined from \cref{prop:cocartesian_functoriality}-(\ref{prop:cocartesian_functoriality:pullback}) and \cref{defin_push_forward_cocart}.
	\cref{right_adjoint_pull_back_cocart} is then a routine computation.
\end{proof}

\begin{lem}\label{pullback_induction_cocart_commutes_limits_colimits}
	Let
	\[ \begin{tikzcd}
		\cB \arrow{d} & \cB_\cX \arrow{d} \arrow{r}{v} \arrow{l}[swap]{u} & \cA \arrow{dl} \\
		\cY & \cX \arrow{l}[swap]{f}
	\end{tikzcd} \]
	be a morphism in $\CoCart$.
    Assume  that $\cA\to \cX$ and $\cB \to \cX$ have compact and proper fibers, and that $\cE$ is presentable stable.
    Then, the functors 
	\[ u^\ast \colon \Funcocart(\cB, \cE) \to \Funcocart(\cB_\cX, \cE) \qquad \textrm{and} \qquad v_! \colon \Funcocart(\cB_X, \cE) \to \Funcocart(\cA, \cE) \]
    commute with limits and colimits.
\end{lem}

\begin{proof}
	From \cref{cor:cocartesianization} and \cref{cocart_stability_limits}, $\Funcocart(\cB, \cE)$ and $\Funcocart(\cB_\cX, \cE)$ are stable under limits and colimits in $\Fun(\cB, \cE)$ and $\Fun(\cB_\cX, \cE)$ respectively.
	Hence, it is enough to show that the functors 
	\[ u^\ast \colon \Fun(\cB, \cE) \to \Fun(\cB_\cX, \cE) \qquad \textrm{and} \qquad  v_! \colon \Fun(\cB_{\cX}, \cE) \to \Fun(\cA, \cE) \]
	commute with limits and colimits.
	For the former, this is obvious.
	For the latter, this follows from \cref{induction_limit_stable}.
\end{proof}

\subsection{Van Kampen for cocartesian functors}

Consider the following general fact:

\begin{lem}[Van Kampen for filtered functors] \label{lem:Van_Kampen_filtered}
	Let $\cX_\bullet \colon I \to \Cat_\infty$ be a diagram with colimit $\cX$.
	Let $p \colon \cA \to \cX$ be a cocartesian fibration and set
	\[ \cA_\bullet \coloneqq \cX_\bullet \times_{\cX} \cA \colon I \to \Cat_\infty \ . \]
	Then the canonical morphism
	\[ \colim_{i \in I} \cA_i \to \cA \]
	is an equivalence.
	In particular, for every presentable $\infty$-category $\cE$ the canonical morphism
	\begin{equation}\label{eq:Van_Kampen_filtered}
		\Fun(\cA, \cE) \to \lim_{i \in I\op} \Fun^\ast(\cA_i, \cE)
	\end{equation}
	is an equivalence.
\end{lem}

\begin{proof}
	Since $p \colon \cA \to \cX$ is a cocartesian fibration, it is in particular an exponentiable fibration thanks to \cite[Lemma 2.15]{Ayala_Francis_Fibrations}.
	In particular, the functor
	\[ p^\ast \colon (\Cat_\infty)_{/\cX} \to (\Cat_\infty)_{/\cA} \]
	is a left adjoint.
	It follows in particular that it preserves all colimits.
	Now the conclusion follows from the fact that for every $\infty$-category $\cC$, the forgetful functor
	\[ (\Cat_\infty)_{/\cC} \to \Cat_\infty \]
	is conservative and preserves all colimits.
\end{proof}

To prove a Van Kampen result for cocartesian functors, we need a couple of categorical preliminaries.
Recall the following definitions:

\begin{defin}
	The \emph{maximal spine} of the standard $n$-simplex $\Delta^n$ is the sub-simplicial set formed by the consecutive $1$-simplexes $\Delta^n_{\{0,1\}}, \Delta^n_{\{1,2\}}, \ldots, \Delta^n_{\{n-1,n\}}$.
\end{defin}

\begin{rem}\label{rem:spine_and_horns}
	Notice that the maximal spine of $\Delta^2$ coincides with $\Lambda^2_1$. On the other hand, for $n \geqslant 3$ every horn $\Lambda^n_i$ cointains the maximal spine of $\Delta^n$.
\end{rem}

\begin{defin}
	Let $\cC$ be a quasi-category and let $S \subset \cC$ be a collection of $1$-simplexes.
	We say that $S$ is \emph{closed under identities} if whenever $f \colon x \to y$ belongs to $S$, then $\id_x$ and $\id_y$ belong to $S$ as well.
\end{defin}

\begin{construction}
	Let $\cC$ be a quasicategory and let $S \subset \cC$ be a collection of $1$-simplexes.
	Define $\cC_S$ as the full sub-simplicial set of $\cC$ defined by the following condition: an $n$-simplex $\sigma \colon \Delta^n \to \cC$ belongs to $\cC_S$ if and only if the restriction of $\sigma$ to the maximal spine of $\Delta^n$ factors through $S$.
\end{construction}

\begin{lem}\label{lem:quasicategory_generated}
	Let $\cC$ be a quasi-category and let $S \subset \cC$ be a collection of $1$-simplexes.
	If $S$ is closed under identities, then $\cC_S$ is the smallest full sub-quasicategory of $\cC$ containing $\cC$.
\end{lem}

\begin{proof}
	Let $\cC'$ be the smallest full sub-quasicategory of $\cC$ containing $\cC$.
	It immediately follows from \cref{rem:spine_and_horns} that $\cC_S$ is a quasi-category, and therefore that $\cC' \subseteq \cC_S$.
	Vice-versa, iteratively applying the lifting condition against inner horns we deduce that any sub-quasicategory containing $S$ must contain $\cC_S$.
	Thus, $\cC' = \cC_S$ as full sub-quasicategories of $\cC$.
\end{proof}

\begin{notation}
	Let $f \colon \cY \to \cX$ be a morphism of quasicategories.
	We denote by $S_f$ the collection of $1$-simplexes of $\cX$ that lie in the essential image of $f$.
\end{notation}

\begin{lem}\label{lem:computing_a_random_colimit}
	Let $\cX_\bullet \colon I \to \Cat_\infty$ be a diagram with colimit $\cX$.
	Let $f_i \colon \cX_i \to \cX$ be the structural morphisms and define
	\[ S \coloneqq \bigcup_{i \in I} S_{f_i} \ . \]
	Then $S$ is closed under identities and the inclusion $\cX_S \subseteq \cX$ is an equivalence in $\Cat_\infty$.
\end{lem}

\begin{proof}
	That $S$ is closed under identities simply follows from the definitions.
	Notice that $\cX_S$ is itself an $\infty$-category and that the inclusion $i \colon \cX_S \hookrightarrow \cX$ is fully faithful.
	By definition, every $f_i$ factors as
	\[ \overline{f}_i \colon \cX_i \to \cX_S \ . \]
	Therefore, the universal property of the colimit provides a canonical map $p \colon \cX \to \cX_S$ together with an equivalence $i \circ p \simeq \id_\cX$.
	This implies that $i$ is essentially surjective. Being already fully faithful, it follows that it is an equivalence.
\end{proof}

We are now ready for:

\begin{prop}[Van Kampen for cocartesian functors]\label{prop:Van_Kampen_cocartesian}
	Let $\cX_\bullet \colon I \to \Cat_\infty$ be a diagram with colimit $\cX$.
	Let $p \colon \cA \to \cX$ be a cocartesian fibration and set
	\[ \cA_\bullet \coloneqq \cX_\bullet \times_{\cX} \cA \colon I \to \Cat_\infty \ . \]
	Let $\cE$ be a presentable $\infty$-category.
	Then the equivalence of \cref{lem:Van_Kampen_filtered} restricts to an equivalence
	\[ \Funcocart(\cA, \cE) \simeq \lim_{i \in I} \Funcocart(\cA_i, \cE) \ . \]
\end{prop}

\begin{proof}
	Using \cref{prop:cocartesian_functoriality}-(\ref{prop:cocartesian_functoriality:pullback}), we see that the canonical map \eqref{eq:Van_Kampen_filtered} induces a well defined map between cocartesian functors making the diagram
	\[ \begin{tikzcd}
		\Fun(\cA, \cE) \arrow{r} & \lim_{i \in I} \Fun(\cA_i, \cE) \\
		\Funcocart(\cA, \cE) \arrow{r} \arrow[hook]{u} & \lim_{i \in I} \Funcocart(\cA_i,\cE) \arrow[hook]{u} \ .
	\end{tikzcd} \]
	Since the top horizontal arrow is an equivalence and the vertical ones are fully faithful, it follows that the bottom horizontal functor is fully faithful as well.
	To conclude the proof, it is enough to show that a functor $F \colon \cA \to \cE$ is cocartesian if and only if for every $i \in I$ its image in $\Fun(\cA_i,\cE)$ is cocartesian.
	The ``only if'' follows from \cref{prop:cocartesian_functoriality}-(\ref{prop:cocartesian_functoriality:pullback}).
	For the converse, observe first that combining \cref{cocart_almost_2_to_3} and \cref{lem:computing_a_random_colimit} we deduce that $F$ is cocartesian if and only if it is cocartesian at every morphism in the essential image of the structural map $f_i \colon \cX_i \to \cX$.
	At this point, the conclusion follows from \cref{cor:change_of_base_cocartesian}.
\end{proof}

\subsection{Change of coefficients for cocartesian functors}

Fix a cocartesian fibration $p \colon \cA \to \cX$ and let $f \colon \cE \to \cE'$ be a morphism in $\PrL$.
Recall from \cref{subsec:change_of_coefficients} that this induces a transformation
\[ f^{\cA / \cX} \colon \exp_\cE(\cA / \cX) \to \exp_{\cE'}(\cA / \cX) \]
in $\PrFibL$.
In particular:

\begin{prop}\label{prop:change_of_coefficients_cocartesian}
	The transformation $f^{\cA/\cX}$ preserves cocartesian edges.
	Therefore, the induced functor
	\[ f \colon \Fun(\cA, \cE) \to \Fun(\cA, \cE') \]
	preserves cocartesian functors and induces a well defined morphism
	\[ f \colon \Funcocart( \cA, \cE ) \to \Funcocart(\cA, \cE') \ . \]
\end{prop}

\begin{proof}
	Since $f^{\cA / \cX}$ is a morphism in $\PrFibL$, it automatically preserves cocartesian edges.
	The second half follows then from the identification $f \simeq \Sigma_\cX(f^{\cA/\cX})$ supplied by \cref{prop:change_of_coefficients}.
\end{proof}

We now study the change of coefficients via the tensor product in $\PrL$.
Recall that for every pair of presentable $\infty$-categories $\cE$ and $\cE'$ and for every $\infty$-category $\cA$, there is a canonical equivalence
\[ \Fun(\cA, \cE) \otimes \cE' \simeq \Fun(\cA, \cE \otimes \cE') \ . \]
Under suitable finiteness assumptions, we are going to see that this equivalence preserves cocartesian functors.

\begin{defin}\label{defin_PrLR}
	Define $\PrLR$ as the (non full) subcategory of $\PrL$ whose objects are presentable 
	$\infty$-categories and morphisms are functors that are both left and right adjoints.
\end{defin}

\begin{defin}
	Let $\PrFibLR$ be the full subcategory of $\PrFibL$ corresponding to $\Fun(\cX,\PrLR)$ under the straightening equivalence (\ref{straightening_PrL}).
\end{defin}

\begin{eg}\label{eg:exponential_fibration_PrLR}
	Let $p \colon \cA \to \cX$ be a cocartesian fibration with compact and proper fibers (see \cref{proper_category}).
	Let $\cE$ be a stable presentable $\infty$-category.
	Then the exponential fibration $p_\cE \colon \exp_\cE(\cA / \cX) \to \cX$ defines an object in $\PrFibLR$: indeed, we have to check that for every morphism $\gamma \colon x \to y$ in $\cX$ and any choice of a straightening $f_\gamma \colon \cA_x \to \cA_y$, the induced functor
	\[ f_{\gamma,!} \colon \Fun(\cA_x, \cE) \to \Fun(\cA_y, \cE) \]
	commutes with limits and colimits, and this follows from \cref{induction_limit_stable}.
\end{eg}

Our main use of $\PrLR$ will be through  the following lemma from \cite[2.7.9]{Beyond_conicality}.

\begin{lem}\label{Peter_lemma}
	Let $ A $ be a small $\infty$-category and let $ \cC_{\bullet} \colon A\to \PrLR$ be a diagram of $\infty$-categories.
	Then, 
	\begin{enumerate}\itemsep=0.2cm
		\item The limits of $ \cC_{\bullet} $ when computed in $ \PrR $, $ \PrL $, or $ \CAT_{\infty} $ all agree.
		
		\item For any presentable $\infty$-category $ \cE $, the natural morphism
		\begin{equation*}
			\lim_{\alpha \in A} \cE \otimes \cC_{\alpha}\to  \cE \otimes\lim_{\alpha \in A} \cC_{\alpha}
		\end{equation*}
		in $ \PrL $ is an equivalence.
		(Here, both limits are computed in $ \PrL $).
	\end{enumerate}
\end{lem}

\begin{lem}\label{cocart_tensor}
	Let $\cX$ be an $\infty$-category and let $\cE$ be a presentable $\infty$-category.
	Then commutative diagram
	\[
	\begin{tikzcd}
		\PrL \arrow{d}{(-) \otimes \cE}\arrow{r}{\Triv_{\cX}} &  \PrFibL \arrow{d}{(-)\otimes_{\cX} \Triv_{\cX}(\cE)}    \\
		\PrL \arrow{r}{\Triv_{\cX}}  &\PrFibL
	\end{tikzcd}
	\]
	is horizontally right adjointable on objects of $\PrFibLR$.
	That is, for every object $p \colon \cA\to \cX$  of $\PrFibLR$, the Beck-Chevalley transformation 
	\[
	\Sigma^{\mathrm{cocart}}_{\cX}(\cA/\cX) \otimes\cE \to \Sigma^{\mathrm{cocart}}_{\cX}(\cA  \otimes_{\cX} \Triv_{\cX}(\cE) ) 
	\]
	is an equivalence.
\end{lem}

\begin{proof}
	Let $\Upsilon_\cA \colon \cX \to \PrL$ be  the straightening of $\cA \to \cX$.
	From \cite[3.3.3.2]{HTT}, the Beck-Chevalley transformation  reads as the following natural morphism in $\PrL$
	\[
	\big(\lim_{x\in \cX } \cA_x\big)  \otimes \cE  \to \lim_{x\in \cX } (\cA_x  \otimes \cE)   \ .
	\]
	Then, \cref{cocart_tensor} follows from \cref{Peter_lemma}.
\end{proof}

\begin{cor}\label{cocart_tensor_for_exp}
	Let $p \colon \cA\to \cX$ be an object of $\CoCart$ with proper and compact fibers.
	Let $\cE,\cE'$ be  presentable $\infty$-categories such that $\cE$ is stable.
	Then, the canonical transformation 
	\[
	\Funcocart(\cA,\cE)\otimes\cE'   \to \Funcocart(\cA,\cE\otimes \cE')
	\]
	is an equivalence.
\end{cor}

\begin{proof}
	Recall from \cref{construction_tensor_PrL_ex} that there is a canonical equivalence
	\[ \exp_\cE(\cA / \cX) \otimes_\cX \Triv_\cX(\cE) \simeq \exp_{\cE \otimes \cE'}(\cA / \cX) \ . \]
	By \cref{eg:exponential_fibration_PrLR} the exponential fibration $\exp_\cE( \cA / \cX )$ belongs to $\PrFibLR$.
	Thus, the conclusion follows applying $\Sigma_\cX^{\cocart}$ to the above equivalence and using \cref{cocart_tensor}.
\end{proof}

\subsection{Cocartesian functors in presence of an initial object}

We saw in \cref{cor:cocartesianization} that the inclusion of cocartesian functors inside all functors always admits a right adjoint $(-)^{\cocart}$.
The goal of this section is to provide an explicit description of this functor in the special case where the base $\cX$ admits an initial object.
We start with the following construction:

\begin{construction}\label{construction:strict_induction}
	Fix a cocartesian fibration $p \colon \cA \to \cX$ and let $\gamma \colon x \to y$ be a morphism in $\cX$.
	Define
	\[ \overline{\gamma}_! \coloneqq j_y^\ast \circ j_{x,!} \colon \Fun(\cA_x, \cE) \to  \Fun(\cA_y, \cE) \ . \]
	Write $\varepsilon_x$ for the counit of the adjunction $j_{x,!} \dashv j_x^\ast$.
	It induces a natural transformation
	\[ \alpha_\gamma \coloneqq \varepsilon_x j_x^\ast \colon \overline{\gamma}_! \circ j_x^\ast \to  j_y^\ast \ . \]
	
	Fix now a straightening
	\[ f \colon \cA_x \to  \cA_y \]
	for $p_\gamma \colon \cA_\gamma \to \Delta^1$ together with a natural transformation $s \colon j_x \to j_y \circ f$ as in \cref{eg:specialization_over_Delta1}.
	Write $\eta_x$ for the unit of $j_{x,!} \dashv j_x^\ast$ and consider the transformation
	\[ \begin{tikzcd}[column sep=35pt]
		\beta_{f,s} \colon f_{!} \arrow{r}{f_{!}(\eta_x)} & f_{!} \circ j_{x}^* \circ j_{x!} \arrow{r}{\alpha_{f,s}(j_{x!})} & j_y^* \circ j_{x!} = \overline{\gamma}_! .
	\end{tikzcd} \]
\end{construction}

\begin{prop} \label{prop:strict_induction}
	In the setting of \cref{construction:strict_induction}, the diagram
	\[ \begin{tikzcd}[column sep = 30pt]
		f_! \circ j_x^* \arrow{r}{\beta_{f,s}(j_x^*)} \arrow{dr}[swap]{\alpha_{f,s}} & j_y^* \circ j_{x,!} \circ j_x^* \arrow{d}{j_y^*(\varepsilon_x)} \\
		{} & j_y^*
	\end{tikzcd} \]
	is canonically commutative.
	If in addition $x$ is an initial object of $\cX$, then the natural transformation $\beta_{f,s} \colon f_! \to \overline{\gamma}_!$ is an equivalence.
	In this case, for every $F \in \Fun(\cA, \cE)$, the morphism
	\[ j_y^*(\varepsilon_x) \colon j_y^* j_{x,!} j_x^*(F) \to  j_y^*(F) \]
	is a specialization morphism for $F$ relative to $\gamma$.
\end{prop}

\begin{proof}
	For what concerns the commutativity, a standard diagram chase reduces it to the triangular identities for $j_{0,!} \dashv j_0^\ast$.
	We leave the details to the reader.
	We check that $\beta_{f,s}$ is an equivalence under the assumption that $x$ is an initial object of $\cX$.
	Unraveling the definitions, we reduce ourselves to check that for every $a \in \cA_y$, the canonical functor induced by the pair $(f,s)$
	\begin{equation}\label{eq:prop:strict_induction}
		\cA_x \times_{\cA_y} (\cA_y)_{/a} \to  \cA_x \times_{\cA} \cA_{/j_y(a)} 
	\end{equation}
	is cofinal. 
	We are going to show that it is an equivalence. 
	To do this, \cite[2.2.3.3]{HTT} ensures that it is enough to show that \eqref{eq:prop:strict_induction} is a pointwise equivalence over $\cA_x$.
	The restriction of \eqref{eq:prop:strict_induction}  above $b \in \cA_x$ reads as
	\begin{equation}\label{eq:prop:strict_induction_II}
		\Map_{\cA_y}( f(b), a )  \to   \Map_{\cA}( j_x(b) , j_y(a) ) \ .
	\end{equation}
	Since $x$ is an initial object in $\cX$, $\Map_\cX(x,y)$ is contractible.
	Thus, every morphism $j_x(b) \to j_y(a)$ lies over $\gamma \colon x \to y$.
	Since $s \colon j_x(b) \to j_y(f(b))$ is a $p$-cocartesian lift of $\gamma$, \cite[2.4.4.2]{HTT} implies that \eqref{eq:prop:strict_induction_II} is an equivalence.
	
	\personal{Here's the diagram chase.
	Denote by $\varepsilon_\gamma$ the counit of the adjunction $f_{\gamma !} \dashv f_\gamma^*$.
	Consider the diagram
	\begin{equation} \label{eq:lem:computing_specialization_morphism}
		\begin{tikzcd}[column sep = large,ampersand replacement = \&]
			f_{\gamma !}\circ  j_x^* \arrow{r}{f_{\gamma !} \eta_x j_x^*} \arrow[equal]{dr} \& f_{\gamma !} \circ j_x^* \circ j_{x!} \circ j_x^* \arrow{r}{f_{\gamma !} \sigma j_{x!} j_x^*} \arrow{d}{f_{\gamma !} j_x^* \varepsilon_x} \& f_{\gamma !} \circ f_\gamma^* \circ j_y^* \circ j_{x!} \circ j_x^* \arrow{r}{\varepsilon_\gamma j_y^* j_{x!} j_x^*} \arrow{d}{f_{\gamma !} f_\gamma^* j_y^* \varepsilon_x} \& j_y^* \circ j_{x!} \circ j_x^* \arrow{d}{j_y^* \varepsilon_x} \\
			{} \& f_{\gamma !}\circ  j_x^* \arrow{r}{f_{\gamma !}\sigma} \& f_{\gamma !} \circ f_\gamma^* \circ j_y^* \arrow{r}{\varepsilon_\gamma j_y^*} \& j_y^* .
		\end{tikzcd}
	\end{equation}
	The triangular identities for $j_{x!} \dashv j_x^*$ provide a canonical homotopy making the triangle on the left commutative.
	Applying the natural transformation $j_x^* (\varepsilon_x) \colon j_x^* \circ j_{x!} \circ j_x^* \to j_x^*$ to the morphism $\sigma \colon j_x^* \to f_\gamma^* \circ j_y^*$, we obtain a canonical homotopy making
	\[ \begin{tikzcd}[column sep = large,ampersand replacement=\&]
		j_x^* \circ j_{x!} \circ j_x^* \arrow{r}{\sigma j_{x!} j_x^*} \arrow{d}{j_x^* \varepsilon_x} \& f_\gamma^* \circ j_y^* \circ j_{x!} \circ j_x^* \arrow{d}{f_\gamma^* j_y^* \varepsilon_x} \\
		j_x^* \arrow{r}{\sigma} \& f_\gamma^* \circ j_y^* .
	\end{tikzcd} \]
	Further applying $f_{\gamma !}$ we obtain a canonical homotopy making the middle square in the diagram \eqref{eq:computing_specialization_morphism} commutative.
	As for the square on the right, we simply evaluate the natural transformation $\varepsilon_\gamma \colon f_{\gamma !} \circ f_\gamma^* \to \mathrm{id}$ to the morphism $j_y^* (\varepsilon_x) \colon j_y^* \circ j_{x!} \circ j_x^* \to j_y^*$.
	We therefore deduce that the diagram \eqref{eq:computing_specialization_morphism} is canonically commutative.
	Unraveling the definitions, we see that the top composition followed by the rightmost vertical arrow coincides with $j_y^*(\varepsilon_x) \circ \beta_s(j_x^*)$, while the bottom composition coincides with $\alpha_s$.
	The conclusion follows.}
\end{proof}

For a general morphism $\gamma \colon x \to y$ in $\cX$, we can always replace $p \colon \cA \to \cX$ by $p_\gamma \colon \cA_\gamma \to \Delta^1$ in order to ensure that the hypothesis of \cref{prop:strict_induction} is satisfied.
This yields:

\begin{defin} \label{def:canonical_specialization_morphism}
	Let $p \colon \cA \to \cX$ be a cocartesian fibration.
	Let $F \in \Fun(\cA, \cE)$ be a functor and let $\gamma \colon \Delta^1 \to \cX$ be a morphism in $\cX$.
	Let $j_\gamma \colon \cA_\gamma \to \cA$, $j_{\gamma,x} \colon \cA_x \to \cA_\gamma$ and $j_{\gamma,y} \colon \cA_y \to \cA_\gamma$ be the natural functors.
	Observe that $j_\gamma \circ j_{\gamma,y} \simeq j_y$, and similarly for $x$.
	The \emph{strict induction functor relative to $\gamma$} is the functor
	\[ \gamma_{\cA,!} \coloneqq j_{\gamma,y}^* \circ j_{\gamma,x, !} \colon \Fun(\cA_x, \cE) \to  \Fun(\cA_y, \cE) . \]
	The \emph{strict specialization morphism for $F$ relative $\gamma$} is the natural transformation
	\[ \spe_{\cA, \gamma}(F) \coloneqq j_{\gamma,y}^*( \varepsilon_x j_\gamma^* ) \colon \gamma_{\cA, !} ( j_x^* F ) \to  j_y^*(F) , \]
	where $\varepsilon_x$ denotes the counit of the adjunction $j_{\gamma,x, !} \dashv j_{\gamma,x}^*$.
	
	When $\cA$ is clear from the context, we write $\gamma_!$ and $\spe_\gamma(F)$ instead of $\gamma_{\cA !}$ and $\spe_{\cA, \gamma}$.
	\personal{(Mauro) Beware: change of notation with respect to the previous version of Stokes: instead of $j_{x,\gamma}$, I'm writing $j_{\gamma,x}$. This looks to me more compatible with \cref{notation:simplexes_cocartesian_fibration},as it suggests $j_{\gamma,x} = (j_\gamma)_x$.}
\end{defin}

\begin{rem}
	The terminology is due to the fact that neither $\gamma_{\cA,!}$ nor $\spe_{\cA, \gamma}(F)$ depend on the choice of a straightening of $p \colon \cA \to \cX$.
\end{rem}

\Cref{prop:strict_induction} leads to a complete understanding of $\Funcocart(\cA,\cE)$ when $\cX$ has an initial object.
Before stating the main result of this section, let us collect a couple of general facts:

\begin{prop}\label{invariance_final_pull_back}
	Let $\cX$ and $\cY$ be small $\infty$-categories and let
	\[ \begin{tikzcd}
		\cC \arrow{r}{g} \arrow{d}{q} & \cD \arrow{d}{p}\\
		\cX \arrow{r}{f} & \cY
	\end{tikzcd} \]
	be a pullback square in $\CAT_\infty$, with $p$ being a cocartesian fibration.
	If $f$ is a final functor, then the induced pull-back functor
	\[ f^{\ast} \colon \Funcocart_{/\cY}(\cY, \cD) \to  \Funcocart_{/\cX}(\cX, \cC) \]
	is an equivalence of $\infty$-categories.
\end{prop}

\begin{proof}
	Let $\Upsilon_\cD$ and $\Upsilon_\cC$ be the straightenings of $p \colon \cD \to \cY$ and of $q \colon \cC \to \cX$, respectively.
	Since the given square is a pullback, there is a natural equivalence $\Upsilon_\cC \simeq \Upsilon_\cD \circ f$.
	We find:
	\begin{align*}
		\Funcocart_{/\cY}(\cY, \cD) & \simeq \lim_\cX \Upsilon_\cD & \text{By \cite[Prop.\ 3.3.3.1]{HTT}} \\
		& \simeq \lim_\cY \Upsilon_\cD \circ f & \text{$f$ is cofinal} \\
		& \simeq \lim_\cY \Upsilon_\cC \\
		& \simeq \Funcocart_{/\cX}(\cX, \cC) & \text{By \cite[Prop.\ 3.3.3.1]{HTT},}
	\end{align*}
	and the conclusion follows.
\end{proof}

In the particular case where $\cC = \exp_\cE(\cB / \cY)$, we find:

\begin{cor}\label{cor:invariance_final_pullback}
	Let
	\[ \begin{tikzcd}
		\cA \arrow{r}{u} \arrow{d}{q} & \cB \arrow{d}{p} \\
		\cX \arrow{r}{f} & \cY
	\end{tikzcd} \]
	be a pullback in $\Cat_\infty$, with $p$ being a cocartesian fibration.
	Let $\cE$ be a presentable $\infty$-category.
	If $f$ is a final functor, then
	\begin{equation} \label{eq_invariance_final_pull_back}
		u^\ast \colon \Funcocart(\cB, \cE) \to  \Funcocart(\cA, \cE)
	\end{equation}
	is an equivalence.
\end{cor}

\begin{proof}
	Apply \cref{invariance_final_pull_back} to $\exp_\cE(\cB / \cY)$ and use \cref{prop:global_functoriality}-(\ref{prop:global_functoriality:pullback}).
\end{proof}

\begin{observation}
	Recall from \cref{right_adjoint_pull_back_cocart} that $u^\ast$ admits a right adjoint $u_\ast^{\cocart}$.
	It follows formally that in the situation of \cref{cor:invariance_final_pullback}, the $u_\ast^{\cocart}$ realizes the inverse of $u^\ast$.
\end{observation}

\begin{cor}\label{induction_from_initial_preserves_cocart_functors}
	Let $\cX$ be an $\infty$-category with an initial object $x$.
	Let $p \colon \cA \to \cX$ be a cocartesian fibration and let $\cE$ be a presentable $\infty$-category.
	Then:
	\begin{enumerate}\itemsep=0.2cm
		\item the functor $j_{x,!} \colon \Fun(\cA_x, \cE) \to \Fun(\cA, \cE)$ factors through $\Funcocart(\cA, \cE)$.
		
		\item The adjunction
		\[ j_{x,!} \colon \Fun(\cA_x, \cE) \leftrightarrows \Fun(\cA, \cE) \colon j_x^\ast \]
		restricts to an equivalence of $\infty$-categories between $\Fun(\cA_x, \cE)$ and $\Funcocart(\cA, \cE)$.
	\end{enumerate}
\end{cor}

\begin{proof}
	We prove (1).
	Let $F \in \Fun(\cA_x, \cE)$ and let $\gamma \colon y\to z$ be a morphism in $\cX$.
	We need to show that $j_{x,!}(F) $ is cocartesian at $\gamma$.
	Since $x$ is initial in $\cX$, we can find a commutative triangle
	\[ \begin{tikzcd}[column sep=small]
		{} & y \arrow{dr}{\gamma} \\
		x \arrow{ur}{\gamma_0} \arrow{rr}{\gamma_2} & & z .
	\end{tikzcd} \]
	in $\cX$. 
	From \cref{cocart_almost_2_to_3}, it is enough to prove that $j_{x,!}(F)$ is cocartesian at $\gamma_0$ and $\gamma_2$.
	Equivalently, we can suppose that $y = x$.
	Now we apply \cref{prop:strict_induction} to $j_{x,!}(F)$: notice that since $x$ is initial, the inclusion $\{x\} \hookrightarrow \cX$ is fully faithful and therefore that $j_x \colon \cA_x \to \cA$ is fully faithful as well.
	Thus, the unit transformation $F \to j_x^\ast j_{x,!}(F)$ is an equivalence, and therefore the strict specialization morphism provided by \cref{prop:strict_induction} is an equivalence as well.
	
	\medskip
	
	We now prove (2).
	Since $\Funcocart(\cA, \cE)$ is fully faithful inside $\Fun(\cA,\cE)$ and since $j_{x,!}$ factors through $\Funcocart(\cA, \cE)$, we see that the adjunction $j_{x,!} \dashv j_x^\ast$ descends to an adjunction
	\[ j_{x,!} \colon \Fun(\cA_x, \cE) \leftrightarrows \Fun(\cA, \cE) \colon j_x^\ast \ . \]
	It is therefore enough to prove that $j_x^\ast$ is an equivalence.
	Since the inclusion $\{x\} \hookrightarrow \cX$ is final, this follows from the limit-description of cocartesian functors provided in \cref{cor:cocart_presentable_stable}.
	See also \cref{cor:invariance_final_pullback} below.
\end{proof}

\begin{cor}\label{computation_cocart_initial_object}
	Let $\cX$ be an $\infty$-category with an initial object $x$.
	Let $p \colon \cA \to \cX$ be a cocartesian fibration and let $\cE$ be a presentable $\infty$-category.
	Then there is a natural equivalence
	\[ (-)^{\cocart} \simeq j_{x,!} \circ j_x^\ast \]
	of functors from $\Fun(\cA, \cE)$ to $\Funcocart(\cA, \cE)$.
\end{cor}

\begin{proof}
	Fix $F \in \Funcocart(\cA, \cE)$ and let $G \in \Fun(\cA, \cE)$.
	We have
	\begin{align*}
		\Map_{\Fun(\cA,\cE)}( F, G ) & \simeq \Map_{\Fun(\cA, \cE)}( j_{x,!} j_x^\ast(F), G ) & \text{By Cor.\ \ref{induction_from_initial_preserves_cocart_functors}} \\
		& \simeq \Map_{\Fun(\cA_x,\cE)}( j_x^\ast(F), j_x^\ast(G) ) \\
		& \simeq \Map_{\Funcocart(\cA_x, \cE)}( j_{x,!} j_x^\ast(F), j_{x,!} j_x^\ast(G) ) \\
		& \simeq \Map_{\Funcocart(\cA_x, \cE)}( F, j_{x,!} j_x^\ast(G))
	\end{align*}
	where the last two equivalences are again due to \cref{induction_from_initial_preserves_cocart_functors}.
	Therefore, $j_{x,!} \circ j_x^\ast$ is right adjoint to the inclusion of $\Funcocart(\cA, \cE)$ into $\Fun(\cA, \cE)$, whence the conclusion.
\end{proof}

For later use, let us extract the formal argument used to prove \cref{induction_from_initial_preserves_cocart_functors}-(2):

\begin{lem}\label{pull_back_equivalence_induction}
	Let
	\[ \begin{tikzcd}
		\cA \arrow{r}{u} \arrow{d}{q} & \cB \arrow{d}{p}\\
		\cX \arrow{r}{f} & \cY
	\end{tikzcd} \]
	be a pullback in $\Cat_\infty$, with $p$ being a cocartesian fibration.
	Let $\cE$ be a presentable $\infty$-category.
	Assume that $u^* \colon \Funcocart(\cA,\cE) \to \Funcocart(\cB,\cE)$ is an equivalence of $\infty$-categories.
	Then, the following conditions are equivalent:
	\begin{enumerate}\itemsep=0.2cm
		\item The functor $u_! \colon \Fun(\cB, \cE) \to \Fun(\cA, \cE)$ preserves cocartesian functors;
		
		\item The adjunction $u_! \dashv u^*$ restricts to an equivalence of $\infty$-categories between $\Funcocart(\cA,\cE)$ and $\Funcocart(\cB,\cE)$;
		
		\item For every $F \in \Funcocart(\cB, \cE)$, there is a natural equivalence $u_!(F) \simeq u_\ast^{\cocart}(F)$.
	\end{enumerate}
\end{lem}

\begin{proof}
	Notice that both (2) and (3) imply tautologically (1).
	Since $\Funcocart(\cA,\cE)$ is a full subcategory of $\Fun(\cA, \cE)$, and similarly for $\cB$ in place of $\cA$, we see that as soon as (1) is satisfied the induced functor
	\[ u_! \colon \Funcocart(\cB, \cE) \to  \Funcocart(\cA, \cE) \]
	provides a left adjoint to $u^\ast$.
	So (2) holds, and since $u^\ast$ is an equivalence, (3) follows from the uniqueness of the inverse.
\end{proof}

\subsection{Invariance of cocartesian functors under localization}

We saw in \cref{cor:invariance_final_pullback} that when $f$ is a final functor,
\[ u^\ast \colon \Funcocart(\cB, \cE) \to  \Funcocart(\cA, \cE) \]
is an equivalence, with inverse given by $u_\ast^{\cocart}$.
Furthermore, in \cref{induction_from_initial_preserves_cocart_functors}, we saw that when $f$ is the inclusion of an initial object, then the inverse can be identified with the much simpler left Kan extension $u_!$.
In this section, we analyze a similar situation, where $f$ is assumed to be a localization (recall from \cite[Proposition 7.1.10]{Cisinski_Higher_Category} that all localizations are final), building on the results of the previous section.
Our starting point is the following finer analysis of cocartesian functors in this special situation:

\begin{prop}\label{pull_back_localization_by_cocartesian_is_localization}
	Let
	\begin{equation} \label{eq:pull_back_localization_by_cocartesian_is_localization}
		\begin{tikzcd}
			\cA \arrow{r}{u} \arrow{d}{q} & \cB \arrow{d}{p} \\
			\cX \arrow{r}{f} & \cY \ .
		\end{tikzcd}
	\end{equation}
	be a pullback square in $\Cat_\infty$, with $p$ being a cocartesian fibration.
	Assume that $f$ exhibits $\cY$ as a localization of $\cX$ at a collection of morphisms $W$.
	Then for every presentable $\infty$-category $\cE$ and every functor $G \colon \cA \to \cE$, the following conditions are equivalent:
	\begin{enumerate}\itemsep=0.2cm
		\item $G$ lies in the essential image of $u^{\ast} \colon \Fun(\cB,\cE) \to \Fun(\cA,\cE)$;
		
		\item $G$ is cartesian at every mophism in $W$;
		
		
		\item For every $\gamma \in W$, the morphism $\cE^u((\spe G)(\gamma))$ is an equivalence in $\exp_\cE(\cB / \cY)$;
		
		\item $G$ is cocartesian at every mophism in $W$.
		
	\end{enumerate}
\end{prop}

\begin{proof}
	Let $W_\cA$ be the set of cocartesian lifts of morphisms in $W$.
	We saw in \cref{prop:pullback_localization} that $u \colon \cA \to \cB$ exhibits $\cB$ as a localization of $\cA$ at $W_\cA$.
	Thus, (1) is equivalent to ask that $G$ inverts every arrow in $W_\cA$ and \cref{lem:g_local_via_specialization} shows that this is equivalent to condition (2).
	Combining the specialization equivalence \cref{prop:sections_exponential} and the global functoriality established in \cref{prop:global_functoriality}-(\ref{prop:global_functoriality:pullback}) and the fact that the front square of \eqref{eq:cocartesian_pullback_localization} is a pullback, we deduce that (1) is equivalent to ask that $\cE^u \circ (\spe G) \colon \cX \to \exp_\cE(\cB / \cY)$ inverts all arrows in $W$, i.e.\ to condition (3).
	Finally, we prove the equivalence between (3) and (4): let $\gamma$ be a morphism in $W$.
	Combining \cref{prop:functoriality_exponential}-(1) and \cite[2.4.1.12]{HTT}, we see that $G$ is cocartesian at $\gamma$ if and only if $\cE^u \circ (\spe G)$ takes $\gamma$ into a $p_\cE$-cocartesian morphism in $\exp_\cE(\cB / \cY)$.
	Since $\cE^u( (\spe G)(\gamma) )$ lies over $f(\gamma)$, which is an equivalence in $\cY$, we see that this happens if and only if $\cE^u((\spe G)(\gamma))$ is an equivalence in $\exp_\cE(\cB / \cY)$, whence the conclusion.
\end{proof}

\begin{prop}\label{cocart_and_localization}
		Let
	\begin{equation} \label{eq:pull_back_localization_by_cocartesian_is_localization}
		\begin{tikzcd}
			\cA \arrow{r}{u} \arrow{d}{q} & \cB \arrow{d}{p} \\
			\cX \arrow{r}{f} & \cY \ .
		\end{tikzcd}
	\end{equation}
	be a pullback square in $\Cat_\infty$, with $p$ being a cocartesian fibration.
	Assume that $f$ is a localization functor and let $\cE$ be a presentable $\infty$-category.
	Then:
	\begin{enumerate}\itemsep=0.2cm
		\item A functor $F \in \Fun(\cB, \cE)$ is cocartesian if and only if $u^\ast(F)$ cocartesian.

		\item The functor
		\[ u_! \colon \Fun(\cA, \cE) \to  \Fun(\cB, \cE) \]
		preserves cocartesian functors.

		\item The adjunction
		\[ u_! \colon \Fun(\cA, \cE) \leftrightarrows \Fun(\cB, \cE) \colon u^\ast \]
		restricts to an equivalence of $\infty$-categories between $\Funcocart(\cA,\cE)$ and $\Funcocart(\cB,\cE)$.
	\end{enumerate}
\end{prop}

\begin{proof}
	Let $W$ be the collection of morphisms that in $\cX$ that are inverted by $f$.
	We start by proving (1).
	The ``only if'' direction is a consequence of \cref{prop:cocartesian_functoriality}-(\ref{prop:cocartesian_functoriality:pullback}).
	Suppose on the other hand that $u^\ast(F)$ is cocartesian.
	Notice that the homotopy category $\mathrm h(\cY)$ is the $1$-categorical localization of $\mathrm h(\cX)$ at the image of $W$ in $\mathrm h(\cX)$.
	In particular, every $1$-morphism (in $\mathrm h(\cY)$ and hence) in $\cY$ can be represented as a zig-zag (see \cite{Gabriel_Zisman}):
	\[ x_0 \to  x_1 \leftarrow x_2 \to  \cdots \leftarrow x_{n} \]
	in $\cX$, where the arrows pointing to the left are in $W$.
	Recall from \cref{cocart_at_equivalence}, that $F$ is cocartesian at every equivalence of $\cY$.
	Thus, using \cref{cocart_almost_2_to_3} we are left to show that for every morphism $\gamma$ of $\cX$, the functor $F$ is cocartesian at $f(\gamma)$, and this follows from \cref{cor:change_of_base_cocartesian} and our assumption that $u^\ast(F)$ is cocartesian.
	
	\medskip
	
	We now prove the claim $(2)$.
	Let $G \colon \cA \to \cE $ be a cocartesian functor.
	\Cref{pull_back_localization_by_cocartesian_is_localization} ensures the existence of a functor $F \colon \cB \to \cE$ such that $G \simeq u^*(F)$.
	Point (1) guarantees that $F$ is cocartesian.
	At the same time, we know from \cref{prop:pullback_localization} that $u \colon \cB \to \cA$ is a localization functor.
	Thus, $u^* \colon \Fun(\cB,\cE) \to \Fun(\cA,\cE)$ is fully faithul, and therefore the counit transformation $g_! \circ g^*\to \id$ is an equivalence.
	It follows that
	\[ F \simeq g_!(g^*(F)) \simeq g_!(G) \]
	is cocartesian, and so (2) is proven.
	
	\medskip
	
	Finally, for (3), recall from \cite[Proposition 7.1.10]{Cisinski_Higher_Category} that localization functors are final.
	Thus, (3) follows from (2) combined with \cref{cor:invariance_final_pullback} and \cref{pull_back_equivalence_induction}.
\end{proof}

\begin{cor}\label{pushforward_and_refinement}
	Let
	\[ \begin{tikzcd}
		\cA' \arrow{rr}{u'} \arrow{dr}{s'} \arrow{dd}{} & & \cB' \arrow{dr}{s} \arrow[near start]{dd}{} \\
		{} & \cA \arrow[crossing over]{rr}[near start]{u} & & \cB \arrow{dd}{} \\
		 \cX' \arrow{rr}[near end]{f'}  \arrow{dr}{r'} & & \cY'  \arrow{dr}{r} \\
		{} & \cX \arrow{rr}{f} \arrow[leftarrow, crossing over, near end]{uu}{} & & \cY
	\end{tikzcd} \]
	be a commutative cube in $\Cat_\infty$, with the vertical arrows being cocartesian fibrations.
	Assume that $r$ and $r'$ are localization functors and that the left and right vertical faces are pullbacks.
	Let $\cE$ be a presentable $\infty$-category.
	Then, the following diagrams 
	\[ \begin{tikzcd}
		\Funcocart(\cA', \cE) \arrow{r}{u^{\prime\ast}} \arrow{d}{s_!} & \Funcocart(\cB', \cE) \arrow{d}{s'_!} \\
		\Funcocart(\cA, \cE) \arrow{r}{u^*} & \Funcocart(\cB, \cE) ,
	\end{tikzcd} \qquad 
	\begin{tikzcd}	
		\Funcocart(\cB, \cE) \arrow{r}{g_*^{\cocart}} \arrow{d}{u^{\prime\ast}} & \Funcocart(\cA, \cE) \arrow{d}{s^*} \\
		\Funcocart(\cB', \cE) \arrow{r}{u^{\prime\cocart}_{\ast}} & \Funcocart(\cA', \cE)
	\end{tikzcd} \]
	are canonically commutative.
\end{cor}

\begin{proof}
	Observe that the right square from \cref{pushforward_and_refinement} is obtained from the left square by passing to right adjoints.
	Hence, we are left to prove the commutativity of the left square.
	Since the top face is commutative, we have  
	\[ s^{\prime\ast}\circ u^* =u^{\prime\ast} \circ s^* \]
	From \cref{cocart_and_localization}, the adjunction $s_! \dashv s^*$ induces an equivalence of $\infty$-categories between $\Funcocart(\cA,\cE)$ and $\Funcocart(\cA',\cE)$ and similarly with $s'_! \dashv s^{\prime *}$.
	The commutativity of the left square thus follows.
\end{proof}

\begin{cor}
	Let $f \colon \cX \to \cY$ be a localization functor between $\infty$-categories.
	Let $\cB$ be an $\infty$-category and denote by $u \colon \cB \times \cX \to \cB \times \cY$ the induced functor.
	Let $\cE$ be a presentable $\infty$-category.
	The adjunction $u_! \dashv u^*$ induces an equivalence of $\infty$-categories between $\Funcocart(\cA \times \cY, \cE)$ and $\Funcocart(\cA\times \cX, \cE)$.
\end{cor}

\begin{proof}
	Consider the pullback square
	\[ \begin{tikzcd}
		\cA\times \cY \arrow{r} \arrow{d} & \cA\times \cX \arrow{d} \\
		\cY \arrow{r} & \cX
	\end{tikzcd} \]
	and apply \cref{cocart_and_localization}.
\end{proof}

\begin{rem}\label{localization_equi}
When applied to the localization $\cX \to \Env(\cX)$, the above corollary says that $u_! \dashv u^*$ induces an equivalence of $\infty$-categories between $\Fun(\cA\times \Env(\cX), \cE)$  and  $\Funcocart(\cA\times \cX, \cE)$.
\end{rem}

\subsection{Exceptional functoriality}

Let
\begin{equation}\label{pullback_Cocartesianability criterion}
	\begin{tikzcd}
		\cA \arrow{r}{u}\arrow{d}{q} &\cB \arrow{d}{p} \\
		\cX \arrow{r}{f} & \cY
	\end{tikzcd}
\end{equation}
be a pullback diagram in $\Cat_\infty$, with $p$ being a cocartesian fibration.
We saw in \cref{cor:invariance_final_pullback} that when $f$ is a final functor the pullback
\[ u^\ast \colon \Funcocart(\cB, \cE) \to  \Funcocart(\cA, \cE) \]
is an equivalence for every presentable $\infty$-category $\cE$.
In virtue of \cref{cor:cocartesianization}, the inverse to $u^\ast$ is always given by the functor $u_\ast^{\cocart}$, which is nevertheless very inexplicit in general.
At the same time we saw in two rather different situations (\cref{induction_from_initial_preserves_cocart_functors} and \cref{cocart_and_localization}) that sometimes the inverse can be computed by the left Kan extension $u_!$.
In this section, we analyze this phenomenon more in detail, obtaining a sufficient criterion guaranteeing that $u_!$ preserves cocartesian functors, that will be needed later on.

\medskip

We start with a simple observation:

\begin{prop} \label{cocart_and_ff}
	Let
	\[ \begin{tikzcd}
		\cA \arrow{r}{u} \arrow{d}{q} & \cB \arrow{d}{p} \\
		\cX \arrow{r}{f} & \cY
	\end{tikzcd} \]
	be a pullback square in $\Cat_\infty$, with $p$ being a cocartesian fibration.
	Let $\cE$ be a presentable $\infty$-category.
	Assume that $f$ is fully faithful, and let $F \in \Fun(\cA, \cE)$ be a functor cocartesian at a morphims $\gamma \colon x \to y$ in $\cX$.
	Then $u_!(F)$ is cocartesian at $f(\gamma)$.
\end{prop}

\begin{proof}
	Since $f$ is fully faithful, the same goes for $u$.
	Thus, the unit transformation $F \to u^\ast(u_!(F))$ is an equivalence.
	Using \cref{cor:change_of_base_cocartesian}, we therefore see that $u_!(F)$ is cocartesian at $f(\gamma)$ if and only if $F \simeq u^\ast(u_!(F))$ is cocartesian at $\gamma$.
	The conclusion follows.
\end{proof}

We now carry out a finer analysis.
Fix $x \in \cX$, set $y \coloneqq f(x)$ and fix as well a morphism $\gamma \colon y \to z$ in $\cY$.
Associated to these data, we can form the following commutative cube:
\begin{equation}\label{cube_Cocartesianability criterion}
	\begin{tikzcd}
		\cB_y \arrow{rr}{j_{\gamma,x}} \arrow{dr}{j_{x}} \arrow{dd}{p_x} & & \cB_\gamma \arrow{dr}{j_{\gamma}} \arrow[near end]{dd}{p_\gamma} \\
		{} & \cA \arrow[crossing over]{rr}[near start]{u} & & \cB \arrow{dd}{p} \\
		\ast \arrow{rr}[pos=0.7]{0}  \arrow{dr}{y} & & \Delta^{1}  \arrow{dr}{\gamma} \\
		{} & \cX \arrow{rr}{f} \arrow[leftarrow, crossing over, near end]{uu}{q} & & \cY \ ,
	\end{tikzcd}
\end{equation}
whose vertical faces are pullbacks.
Fix a presentable $\infty$-category $\cE$.
The commutativity of the top face of the above cube induces a Beck-Chevalley transformation
\begin{equation}\label{eq:adjointability_implies_cocartesian}
	j_{\gamma,x,!} \circ j_x^\ast \to  j_\gamma^\ast \circ u_!
\end{equation}
of functors from $\Fun(\cA, \cE)$ to $\Fun(\cB_\gamma, \cE)$.
We have:

\begin{prop}\label{adjointability_implies_cocartesian}
	Assume that the Beck-Chevalley transformation \eqref{eq:adjointability_implies_cocartesian} is an equivalence.
	Then for every $F \in \Fun(\cB, \cE)$, the functor $u_!(F) \colon \cA\to \cE$ is cocartesian at $\gamma$.
\end{prop}

\begin{proof}
	We have to prove that $\spe(u_!(F))$ is cocartesian at $\gamma$.
	By \cref{cor:change_of_base_cocartesian} applied to $\gamma \colon \Delta^1 \to \cY$, this is equivalent to show that $j_{\gamma}^\ast ( u_!(F) )$ is cocartesian at $0 \to 1$.
	Since the Beck-Chevalley transformation \eqref{eq:adjointability_implies_cocartesian} is an equivalence, we are reduced to prove that $j_{\gamma,x,!}(j_x^\ast(F))$ is cocartesian at $0 \to 1$.
	In other words, we are reduced to prove the statement in the special case where \eqref{pullback_Cocartesianability criterion} is the back square of \eqref{cube_Cocartesianability criterion}.
	Since $0$ is initial in $\Delta^1$, this follows directly from \cref{induction_from_initial_preserves_cocart_functors}.
\end{proof}

We now give a sufficient condition on $f$ and $\gamma$ ensuring that the Beck-Chevalley transformation \eqref{eq:adjointability_implies_cocartesian} is an equivalence:

\begin{prop}\label{cocart_criterion}
	In the above setting, assume that:
	\begin{enumerate}\itemsep=0.2cm
		\item for every $(t,\alpha) \in \cX \times_{\cY} \cY_{/y}$, the map
		\[ \Map_\cX(t, x) \to  \Map_{\cY}(f(t), y) \]
		is an equivalence;
		
		\item for every $(s,\beta) \in \cX \times_\cY \cY_{/z}$, the composition
		\[ \Map_{\cX}(s,x) \to  \Map_\cY(f(s), y) \to  \Map_\cY(f(s),z) \]
		is an equivalence.
	\end{enumerate}
	Then for every $F \in \Fun(\cA, \cE)$, $u_!(F)$ is cocartesian at $\gamma$.
\end{prop}

\begin{rem}\label{rem:cocart_criterion}
	Notice that condition (1) above is automatically satisfied when $f$ is fully faithful, or when both $\Map_\cX(t,x)$ and $\Map_\cY(f(t), y)$ are both contractible.
	Similarly, condition (2) is automatically satisfied when both $\Map_\cX(s,x)$ and $\Map_\cY(f(s),z)$ are both contractible.
\end{rem}

\begin{proof}[Proof of \cref{cocart_criterion}]
	In virtue of \cref{adjointability_implies_cocartesian}, it is enough to show that these assumptions guarantee that the Beck-Chevalley transformation \eqref{eq:adjointability_implies_cocartesian} is an equivalence.
	For this, it is enough to check that for every $b \in \cB_\gamma$, the induced functor
	\begin{equation}\label{is_it_cofinal?}
		\begin{tikzcd}
			\cB_y \times_{\cB_\gamma}  (\cB_\gamma)_{/b}   \arrow{r} & \cA \times_{\cB}  \cB_{/j_{\gamma}(b)}
		\end{tikzcd} 
	\end{equation}
	is cofinal.
	Let $v \coloneqq p_\gamma(b) \in \Delta^1$ and set $w \coloneqq \gamma(v)$ (we have $w = y$ if $v = 0$ and $w = z$ if $v = 1$).
	Using \cref{left_adjointability_reduction_trivial_cocart_fibration}, it is sufficient to prove that under our assumptions, the map
	\begin{equation}\label{is_it_cofinal?_bis}
		\begin{tikzcd}
			\{0\} \times_{\Delta^1} \Delta^1_{/v} \arrow{r} & \cX \times_{\cY}  \cY_{/w}  
		\end{tikzcd} 
	\end{equation}
	is cofinal.
	Observe that the left hand side is contractible (and it coincides with the unique morphism $\varepsilon$ from $0$ to $v$ in $\Delta^1$).
	In particular, the map \eqref{is_it_cofinal?_bis} is cofinal if and only if its image coincides with the final object of $\cX \times_{\cY} \cY_{/w}$.
	Now, unraveling the definitions we see that the above map takes $\varepsilon$ to $(x,\id_y)$ if $v = 0$ and to $(x,\gamma)$ if $v = 1$.
	Thus, we have to prove that $(x,\id_y)$ and $(x,\gamma)$ are final objects in $\cX \times_\cY \cY_{/y}$ and in $\cX \times_\cY \cY_{/z}$, respectively.
	Fix $(t,\alpha) \in \cX \times_\cY \cY_{/w}$ and consider the following commutative diagram:
	\[ \begin{tikzcd}
		\Map_{\cX \times_{\cY} \cY_{/w}}((t,\alpha), (x,\gamma(\varepsilon))) \arrow{r} \arrow{d} & \Map_\cX( t, x ) \arrow{d} \\
		\Map_{\cY_{/w}}( \alpha, \gamma(\varepsilon) ) \arrow{r} \arrow{d} & \Map_{\cY}( f(t), y ) \arrow{d} \\
		\ast \arrow{r}{\alpha} & \Map_\cY( f(t), w ) \ .
	\end{tikzcd} \]
	The top square is a pullback by definition and the bottom one is a pullback thanks to the dual of \cite[Lemma 5.5.5.12]{HTT}.
	Our assumptions guarantee that in the two cases under consideration, the right vertical composition is an equivalence.
	Therefore, it follows that the top left corner is contractible, i.e.\ that $(x,\gamma(\varepsilon))$ is a final object in $\cX \times_\cY \cY_{/w}$, thus completing the proof.
\end{proof}

\begin{cor}\label{induction_poset_cocart}
	Let $f \colon \mathrm{J}\to \mathrm{I}$ be a fully faithful functor between posets and consider a pullback square in $\Cat_{\infty}$
	\[ \begin{tikzcd}
		\cA \arrow{r}{u}\arrow{d}{q}&\cB \arrow{d}{p}\\
		\mathrm{J} \arrow{r}{f}&\mathrm{I} \ ,
	\end{tikzcd} \]
	where in addition $p$ is a cocartesian fibration.
	Assume that for every object $i$ in $\mathrm{I}$, the subposet $\mathrm{J}_{/i}$ of $\mathrm{J}$ admits a final object.
	Then, the functor $u_! \colon \Fun(\cA,\cE)\to \Fun(\cB,\cE)$ preserves cocartesian functors.
\end{cor}

\begin{proof}
	Let $F \colon \cB \to \cE$ be a cocartesian functor.
	Let $\gamma \colon i_1 \to i_2$ be a morphism in $\mathrm{I}$. 
	By assumption, there exists  a commutative diagram 
	\[ \begin{tikzcd}[column sep=small]
		{} & i_1 \arrow{dr}{\gamma} \\
		f(j) \arrow{ur} \arrow{rr}  & & i_2 
	\end{tikzcd} \]
	in $\mathrm{I}$ where $j$ belongs to $\mathrm{J}$.
	Using \cref{cocart_almost_2_to_3}, we see it is enough to show that $F$ is cocartesian at a morphism of the form $f(j)\to i$ where $j \in \mathrm{J}$ and $i \in \mathrm{I}$.
	Let $j_{\infty}$ be a final object in $\mathrm{J}_{/i}$.
	Then, there is a commutative diagram
	\[ \begin{tikzcd}[column sep=small]
		{} &f(j_{\infty})\arrow{dr} \\
		f(j) \arrow{ur} \arrow{rr}  & & i 
	\end{tikzcd} \]
	in $\mathrm{I}$.
	Since $f\colon \mathrm{J}\to \mathrm{I}$ is fully faithful and since $F \colon \cB \to \cE$ is cocartesian at $j\to j_{\infty}$, \cref{cocart_and_ff} ensures that $u_!(F)$ is cocartesian at $f(j)\to f(j_{\infty})$.
	Using again \cref{cocart_almost_2_to_3}, we are thus left to show that $u_!(F)$ is cocartesian at $f(j_{\infty})\to i$.
	In that case, the conditions of \cref{cocart_criterion} (in the form of \cref{rem:cocart_criterion}) are satisfied and the proof is achieved.
\end{proof}

\personal{(Mauro) The following is not needed, but it is a very obvious corollary of the above criterion. We can omit it in the final version.}

\begin{cor}
	Let
	\[ \begin{tikzcd}
		\cA \arrow{d}{q} \arrow{r}{u} & \cB \arrow{d}{p} \\
		\cX \arrow{r}{f} & \cY
	\end{tikzcd} \]
	be a pullback square in $\Cat_\infty$, with $p$ being a cocartesian fibration.
	Let $\cE$ be a presentable $\infty$-category.
	Assume that $f$ is fully faithful and admits a right adjoint $g$ and let $\gamma \colon f \circ g \to \id_\cY$ be a counit transformation.
	Then for every $F \in \Fun(\cA,\cE)$, $u_!(F)$ is cocartesian at $\gamma_y$ for every $y \in \cY$.
\end{cor}

\subsection{Induced $t$-structure for cocartesian functors}\label{subsec:t_structure_cocartesian_functors}

Let $p \colon \cA \to \cX$ be a cocartesian fibration and let $\cE$ be a stable presentable $\infty$-category equipped with an accessible $t$-structure $\tau = (\cE_{\geqslant 0}, \cE_{\leqslant 0})$.
Then $\Fun(\cA, \cE)$ has an induced $t$-structure defined by
\[ \Fun(\cA, \cE)_{\geqslant 0} \coloneqq \Fun(\cA, \cE_{\geqslant 0}) \qquad \text{and} \qquad \Fun(\cA, \cE)_{\leqslant 0} \coloneqq \Fun(\cA, \cE_{\leqslant 0}) \ . \]

\begin{defin}
	We say a cocartesian functor $F \in \Funcocart(\cA, \cE)$ is \emph{connective (with respect to $\tau$)} if its image in $\Fun(\cA,\cE)$ belongs to $\Fun(\cA, \cE)_{\geqslant 0}$.
	We let $\Funcocart(\cA, \cE)_{\geqslant 0}$ be the full subcategory of $\Funcocart(\cA,\cE)$ spanned by connective objects.
\end{defin}

\begin{prop}\label{prop:t_structure_cocartesian_functors}
	There exists a unique $t$-structure on $\Funcocart(\cA, \cE)$ whose connective part coincides with $\Funcocart(\cA, \cE)_{\geqslant 0}$.
	In particular, the inclusion $\Funcocart(\cA, \cE) \hookrightarrow \Fun(\cA, \cE)$ is right $t$-exact.
\end{prop}

\begin{proof}
	Since $\Funcocart(\cA, \cE)$ is presentable and stable by \cref{cor:cocart_presentable_stable}, using \cite[Proposition 1.4.4.11]{Lurie_Higher_algebra} we are reduced to check that $\Funcocart(\cA, \cE)_{\geqslant 0}$ is closed under colimits and extensions in $\Funcocart(\cA, \cE)$.
	Closure under colimits follows from \cref{local_cocart_stability_colimits}, and closure under extensions is automatic.
	So the conclusion follows.
\end{proof}

\begin{lem} \label{lem:t_structure_cocartesian_initial_object}
	Assume that $\cX$ has an initial object $x$.
	Then a cocartesian functor $F \in \Funcocart(\cA, \cE)$ is connective if and only if $j_x^\ast(F) \in \Fun(\cA_x,\cE)$ is connective.
\end{lem}

\begin{proof}
	The functor $j_x^\ast \colon \Fun(\cA, \cE) \to \Fun(\cA_x, \cE)$ is $t$-exact, so if $F$ is connective then $j_x^\ast(F)$ is connective as well.
	For the converse, we have to check that $F$ takes values in $\cE_{\geqslant 0}$.
	It suffices to show that for every $y \in \cX$, $j_y^\ast(F) \colon \cA_y \to \cE$ takes values in $\cE_{\geqslant 0}$.
	Since $x$ is an initial object, there exists a morphism $\gamma \colon x \to y$ in $\cX$.
	Choose a straightening $f_\gamma \colon \cA_x \to \cA_y$ for $\cA_\gamma$.
	Then \cref{prop:cocartesian_functors_reformulation} provides a canonical identification
	\[ j_y^\ast(F) \simeq f_{\gamma,!} j_x^\ast(F) \ . \]
	Since $j_x^\ast(F)$ takes values in $\cE_{\geqslant 0}$ by assumption and since $\cE_{\geqslant 0}$ is closed under colimits in $\cE$, the conclusion follows from the formula for left Kan extensions.
\end{proof}

\begin{cor} \label{cor:t_structure_cocartesian_initial_object}
	Assume that $\cX$ has an initial object $x$.
	Then the adjoint equivalence of \cref{induction_from_initial_preserves_cocart_functors} 
	\[ j_{x,!} \colon \Fun(\cA_x, \cE) \leftrightarrows \Funcocart(\cA, \cE) \colon j_x^\ast \]
	is $t$-exact.
\end{cor}

\begin{proof}
	Thanks to \cref{lem:t_structure_cocartesian_initial_object}, we know that $\Funcocart(\cA,\cE)_{\geqslant 0}$ corresponds via the above equivalence to $\Fun(\cA_x,\cE_{\geqslant 0})$.
	The conclusion follows from the uniqueness of the $t$-structure.
\end{proof}

\begin{eg}\label{eg:cocartesian_functor_in_the_heart}
	Consider the posets $I_0$ and $I_1$ having $I = \{a,b,c,d\}$ as the underlying set and order given by the following Hasse diagrams:
	\[ I_0 = \left\{ \begin{tikzcd}[column sep=small, row sep=small]
		b & & c & d \\
		& a \arrow[-]{ur} \arrow[-]{ul} 
	\end{tikzcd} \right\} \ , \qquad I_1 = \left\{ \begin{tikzcd}[column sep=small,row sep=small]
		& d \\
		b \arrow[-]{ur} & & d \arrow[-]{ul} \\
		& a \arrow[-]{ul} \arrow[-]{ur}
	\end{tikzcd} \right\} \]
	The identity of $I$ defines a morphism of posets $f \colon I_0 \to I_1$, which we can reinterpret as a constructible sheaf of posets $\cI$ on $([0,1], \{0\})$.
	Fix a field $k$ and consider the stable derived $\infty$-category $\cE \coloneqq \Mod_k$.
	Let $F \colon \cI_0 \to \Mod_k$ be the functor defined by setting
	\[ F_a = F_d \coloneqq k \ , \qquad F_b = F_c \coloneqq 0 \ . \]
	Then via \cref{cor:t_structure_cocartesian_initial_object}, $F$ determines an object in $\Funcocart( \cI , \Mod_k )^\heartsuit$.
	Notice however that
	\[ f_!(F)_d \simeq F_d \oplus F_a[1] \simeq k \oplus k[1] \]
	does not belong to the abelian category $\Mod_k^\heartsuit$.
\end{eg}

\subsection{Categorical actions on cocartesian functors} \label{subsec:categorical_actions_cocartesian}

We use the terminology on categorical actions reviewed in \cref{sec:categorical_actions}.
Fix a presentably symmetric monoidal $\infty$-category $\cE^\otimes$.
As recalled in \cref{recollection:monoidal_structure_functor_category}, for every (small) $\infty$-category $\cA$, the functor $\infty$-category $\Fun(\cA, \cE)$ inherits a symmetric monoidal structure $\Fun(\cA, \cE)^\otimes$.
When $\cA$ is part of a cocartesian fibration $p \colon \cA \to \cX$, cocartesian functors $\Funcocart(\cA, \cE)$ form a full subcategory of $\Fun(\cA, \cE)$, but they are not closed under tensor product.
Nevertheless, we still see a shadow of the tensor structure of $\Fun(\cA, \cE)$ on cocartesian functors in terms of a categorical action:

\begin{prop}\label{prop:categorical_action_cocartesian}
	Let $p \colon \cA \to \cX$ be a cocartesian fibration.
	Then for every $L \in \Loc(\cX;\cE)$ (see \cref{def:abstract_local_systems}) and every $G \in \Funcocart(\cA,\cE)$, the functor
	\[ p^\ast(L) \otimes G \colon \cA \to \cE \]
	is again cocartesian.
	In particular, the standard action of $\Loc(\cX; \cE)$ on $\Fun(\cA, \cE)$ restricts to a categorical action of $\Loc(\cX; \cE)$ on $\Funcocart(\cA,\cE)$.
\end{prop}

\begin{proof}
	Let $\gamma \colon x \to y$ be a morphism in $\cX$ and let $f_\gamma \colon \cA_x \to \cA_y$ be any straightening for $p_\gamma \colon \cA_\gamma \to \Delta^1$.
	Since $j_x^\ast$ and $j_y^\ast$ are symmetric monoidal, we reduce to check that
	\[ f_{\gamma,!} ( j_x^\ast p^\ast(L) \otimes j_x^\ast(G) ) \to j_y^\ast(p^\ast(L)) \otimes j_y^\ast(G) \]
	is an equivalence.
	Since $j_x^\ast \circ p^\ast(L) \simeq p_x^\ast(L(x))$, \cref{lem:linearity_of_induction} supplies a canonical equivalence
	\[ f_{\gamma,!} ( j_x^\ast p^\ast(L) \otimes j_x^\ast(G) ) \simeq p_y^\ast(L(x)) \otimes f_{\gamma,!}( j_x^\ast(G) ) \ . \]
	\personal{Here the ``hidden'' point is that $\Fun(\cA_x, \cE)$ is seen as a $\cE^\otimes$-module via $p_x^\ast$, while $\Fun(\cA_y,\cE)$ is seen as a $\cE^\otimes$-module via $p_y^\ast$. Since $f_{\gamma,!}$ is $\cE^\otimes$-linear, $p_x^\ast$ goes out as a $p_y^\ast$.}
	Since $G$ is cocartesian, the canonical comparison map
	\[ f_{\gamma,!}( j_x^\ast(G) ) \to j_y^\ast(G) \]
	is an equivalence.
	On the other hand, since $L$ is a local system, the canonical map $L(\gamma) \colon L(x) \to L(y)$ is an equivalence.
	The conclusion follows.
\end{proof}

Consider now a pullback square
\begin{equation}\label{eq:relative_tensor_cocartesian_setting}
	\begin{tikzcd}
		\cB \arrow{r}{u} \arrow{d}{q} & \cA \arrow{d}{p} \\
		\cY \arrow{r}{f} & \cX
	\end{tikzcd}
\end{equation}
in $\Cat_\infty$, where $p$ is a cocartesian fibration.
Then \cref{construction:relative_tensor_of_functor_categories} supplies a canonical comparison map
\[ \mu \colon \Loc(\cY;\cE) \otimes_{\Loc(\cX;\cE)} \Fun(\cA, \cE) \to \Fun(\cB, \cE) \ . \]
Unraveling the definitions, we see that $\mu$ takes $L \otimes G$ to $p^\ast(L) \otimes G$.
In particular, \cref{prop:categorical_action_cocartesian} shows that $\mu$ restricts to a well defined functor
\begin{equation}\label{eq:relative_tensor_cocartesian_comparison}
	\mu^{\cocart} \colon \Loc(\cY;\cE) \otimes_{\Loc(\cX;\cE)} \Funcocart(\cA, \cE) \to \Funcocart(\cB, \cE)
\end{equation}
When $f$ is a finite étale fibration (see \cref{def:finite_etale_fibrations}), \cref{cor:relative_tensor_product} shows that $\mu$ is an equivalence.
The goal of this section is to show that under mild assumptions the same holds in the cocartesian setting.

\begin{observation}
	In the above setting, assume that $f$ is a finite étale fibration.
	Then the composition $q' \coloneqq f \circ q \colon \cB \to \cX$ is a cocartesian fibration.
	This allows to consider the exponential fibrations
	\[ \exp_\cE( \cB / \cY ) \in \PrFibL_\cY \qquad \text{and} \qquad \exp_\cE( \cB / \cX ) \in \PrFibL_\cX \ . \]
	We set
	\[ \Funcocart(\cB / \cY, \cE) \coloneqq \Sigma_\cY^{\cocart}( \exp_\cE( \cB / \cY ) ) \qquad \text{and} \qquad \Funcocart(\cB / \cX, \cE) \coloneqq \Sigma_\cX^{\cocart}( \exp_\cE( \cB / \cX ) ) \ . \]
	Notice that both $\Funcocart(\cB / \cY, \cE)$ and $\Funcocart(\cB / \cX, \cE)$ are full subcategories of $\Fun(\cB, \cE)$.
\end{observation}

\begin{construction}
	Using the notation from \cref{recollection:pushforward_fibrations}, observe the commutativity of
	\[ \begin{tikzcd}
		\cB \arrow[equal]{r} \arrow{d}{q} & \cB \arrow{d}{f \circ q} \\
		\cY \arrow{r}{f} & \cX
	\end{tikzcd} \]
	provides a canonical transformation
	\[ \delta \colon \cB \to f^\ast(\cB) \]
	in $\CoCart_\cY$.
	In turn, \cref{prop:functoriality_exponential}-(1) shows that $\delta$ induces a morphism
	\[ \exp_\cE( \cB / \cY ) \to f^\ast \exp(\cB / \cX) \ , \]
	which by adjunction $f^\ast \dashv f^{\mathrm{cc}}_\ast$ corresponds to a morphism
	\begin{equation}\label{eq:comparison_cocartesian_over_different_basis}
		\alpha \colon \exp_\cE( \cB / \cX ) \to f_\ast^{\mathrm{cc}} \exp(\cB / \cY) \ .
	\end{equation}
\end{construction}

\begin{prop}\label{prop:comparison_cocartesian_over_different_basis}
	In the above setting, assume that $f$ is a finite étale fibration.
	Then the comparison morphism \eqref{eq:comparison_cocartesian_over_different_basis} is an equivalence.
	In particular,
	\[ \Funcocart(\cB / \cY, \cE) = \Funcocart(\cB / \cX, \cE) \]
	as full subcategories of $\Fun(\cB, \cE)$.
\end{prop}

\begin{proof}
	The second half follows applying $\Sigma_\cX^{\cocart}$ to the equivalence \eqref{eq:comparison_cocartesian_over_different_basis}.
	To show that $\alpha$ is an equivalence, it is enough to show that for every $x \in \cX$, $j_x^\ast(\alpha)$ is an equivalence.
	Unraveling the definitions, we see that
	\[ \cB_x \simeq \coprod_{y \in \cY_x} \cA_x \ , \]
	which immediately implies that
	\[ \exp(\cB / \cX)_x \simeq \prod_{y \in \cY_x} \Fun(\cA_x, \cE) \ . \]
	Similarly, the formula for right Kan extensions paired with \cite[\cref*{exodromy-lem:RKE}]{Exodromy_coefficient}, implies that
	\[ \big( f^{\mathrm{cc}}_\ast \exp_\cE( \cB / \cY ) \big)_x \simeq \big( f^{\mathrm{cc}}_\ast f^\ast \exp_\cE( \cA / \cX ) \big)_x \simeq \prod_{y \in \cY_x} \Fun(\cA_x, \cE) \ . \]
	The conclusion follows.
\end{proof}

\begin{cor}\label{cor:finite_etale_fibration_induction_and_cocartesian}
	In the above setting, if $f$ is a finite étale fibration then
	\[ u_! \colon \Fun(\cB, \cE) \to \Fun(\cA, \cE) \]
	preserves cocartesian functors and therefore it induces a well defined functor
	\[ u_! \colon \Funcocart(\cB, \cE) \to \Funcocart(\cA, \cE) \ . \]
\end{cor}

\begin{proof}
	Seeing $\cB$ fibered over $\cX$ via $q' \coloneqq f \circ q$, \cite[Proposition 2.4.1.3-(2)]{HTT} implies that $u$ takes $q'$-cocartesian edges to $p$-cocartesian ones.
	Therefore, \cref{prop:cocartesian_functoriality}-(\ref{prop:cocartesian_functoriality:induction}) shows that $u_!$ restricts to a well defined functor
	\[ u_! \colon \Funcocart(\cB / \cX, \cE) \to \Funcocart(\cA, \cE) \ . \]
	Since $\Funcocart(\cB / \cX, \cE) = \Funcocart(\cB / \cY, \cE)$ by \cref{prop:comparison_cocartesian_over_different_basis}, the conclusion follows.
\end{proof}

\begin{cor} \label{cor:finite_etale_fibration_cocartesian_monadicity}
	In the above setting, assume that $f$ is a finite étale fibration and that $\cE$ is stable.
	Then
	\[ u_! \colon \Funcocart(\cB, \cE) \to \Funcocart(\cA, \cE) \]
	is monadic.
\end{cor}

\begin{proof}
	It follows from \cref{lem:finite_etale_fibration_biadjoint} that the functors
	\[ u_! \colon \Fun(\cB, \cE) \to \Fun(\cA, \cE) \qquad \text{and} \qquad u^\ast \colon \Fun(\cA, \cE) \to \Fun(\cB, \cE) \]
	are biadjoint.
	Combining \cref{prop:cocartesian_functoriality}-(\ref{prop:cocartesian_functoriality:pullback}) and \cref{cor:finite_etale_fibration_induction_and_cocartesian}, we see that both respect cocartesian functors.
	Therefore, $u_! \colon \Funcocart(\cB, \cE) \to \Funcocart(\cA, \cE)$ is biadjoint to $u^\ast$.
	Besides, \cref{lem:finite_etale_fibration_universally_conservative} implies that $u_! \colon \Fun(\cB, \cE) \to \Fun(\cA, \cE)$ is conservative.
	Since $\Funcocart(\cA, \cE)$ is a full subcategory of $\Fun(\cA, \cE)$, it follows that the same goes for the restriction of $u_!$ to cocartesian functors.
	At this point, the conclusion follows from Lurie-Barr-Beck \cite[Theorem 4.7.3.5]{Lurie_Higher_algebra}.
\end{proof}

\begin{cor}\label{cor:finite_etale_fibration_cocartesian_relative_tensor}
	Let
	\[ \begin{tikzcd}
		\cB \arrow{r}{u} \arrow{d}{q} & \cA \arrow{d}{p} \\
		\cY \arrow{r}{f} & \cX
	\end{tikzcd} \]
	be a pullback square in $\Cat_\infty$, where $p$ is a cocartesian fibration.
	Let $\cE^\otimes$ be a presentably symmetric monoidal $\infty$-category.
	If $f$ is a finite étale fibration and $\cE$ is stable, then the comparison functor
	\[ \Loc(\cY) \otimes_{\Loc(\cY)} \Funcocart(\cA,\cE) \to \Funcocart(\cB, \cE) \ . \]
	is an equivalence.
\end{cor}

\begin{proof}
	Using \cref{cor:finite_etale_fibration_cocartesian_monadicity} as input, the same proof of \cref{cor:relative_tensor_product} applies.
\end{proof}

\section{Punctually split and Stokes functors}
\label{punctually_split_Stokes_section} 
In this section, we explore the new features that appear when one specializes the exponential construction to the case of cocartesian fibrations \emph{in posets}.
Every such fibration $\cI$ has an underlying discrete fibration $\cI^{\ens}$ (see \cref{notation_Iset}), and this allows to introduce punctually split functors.
We analyze their role and explore this notion from the point of view of the exponential construction and discuss their basic functorialities.
Finally, we introduce the main object of study of this paper: Stokes functors.

\subsection{Punctually split functors}

\begin{notation}\label{notation_Iset}
	We let
	\[ (-)^{\mathrm{set}} \colon \Poset \to \Poset \]
	be the functor sending a poset $(I,\leqslant)$ to the underlying set $I$, seen as a poset with trivial order.
	By extension, if $\cX \in \Cat_\infty$ and $\cI \to \cX$ is a fibration in posets, we let $\cI^{\mathrm{set}}$ be the cocartesian fibration on $\cX$ obtained by applying $(-)^{\mathrm{set}}$ fiberwise.
	In a more verbose way, if $\mathscr I \colon \cX \to \Poset$ is the unstraightening of $\cI$, then $\cI^{\mathrm{set}}$ is the cocartesian fibration classifying the composition $(-)^{\mathrm{set}} \circ \mathscr I \colon \cX \to \Poset$.
	Notice that $\cI^{\mathrm{set}}$ is in fact a left fibration over $\cX$ and that it comes equipped with a canonical morphism
	\[ i_\cI \colon \cI^{\mathrm{set}} \to  \cI \]
	that preserves cocartesian edges over $\cX$.
	It is immediate that this construction promotes to a global functor
	\[ (-)^{\mathrm{set}} \colon \PosFib \to  \PosFib \ , \]
	equipped with a natural transformation $i \colon (-)^{\mathrm{set}} \to \id_{\PosFib}$.
\end{notation}

We fix a presentable $\infty$-category $\cE$.

\begin{defin}\label{defin_split}
	Let $p \colon \cI \to \cX$ be an object in $\PosFib$.
	Let $F\in \Fun(\cI,\cE)$.
	\begin{enumerate}\itemsep=0.2cm
		\item For $x \in \cX$, we say that $F$ is \emph{split at $x$} if $j_x^*(F)$ lies in the essential image of
		\[ i_{\cI_x,!} \colon \Fun(\cI^{\ens}_x, \cE)\to \Fun(\cI_x, \cE) \]
		
		\item We say that $F$ is \emph{punctually split} if it is split at every object $x \in \cX$.
		
		\item We say that $F$ is \emph{split} if it lies in the essential image of the induction functor
		\[ i_{\cI,!} \colon \Fun(\cI^{\ens}, \cE)\to \Fun(\cI, \cE) \]
	\end{enumerate}
	We denote by $\Fun_{\PS}(\cI,\cE)$ the full subcategory of $\Fun(\cI,\cE)$ formed by punctually split functors.
\end{defin}

\begin{rem}\label{rem:split_implies_punctually_split}
	It follows from \cref{cor:induction_specialization_Beck_Chevalley} that split functors are punctually split.
\end{rem}

\begin{eg}\label{eg:split_functor}
	Let $p \colon \cI \to \cX$ be a cocartesian fibration in posets and let $a \in \cI$ be an element.
	Write $\ev_a^{\cI} \colon \{a\} \hookrightarrow \cI$ for the canonical inclusion.
	Since $\ev_a^{\cI}$ factors through $i_\cI \colon \cI^{\ens} \to \cI$, we see that for every $E \in \cE$ the functor $\ev_{a,!}^{\cI}(E) \in \Fun(\cI, \cE)$ is split, and hence punctually split by \cref{rem:split_implies_punctually_split}.
\end{eg}

\begin{defin}\label{def:splitting}
	In the setting of \cref{defin_split}, a \emph{splitting for $F$} is the given of a functor $F_0 \colon \cI^{\ens} \to \cE$ and an equivalence $\alpha \colon i_{\cI,!}(F_0) \simeq F$.
\end{defin}

\begin{warning}
	In general, splittings do not exist and even when they exist they are typically neither unique nor canonical.
\end{warning}

\begin{lem}\label{EssIm}
	Let $\cX \in \Cat_\infty$ and let
	\[ \begin{tikzcd}[column sep = small]
		\cB \arrow{rr}{f} \arrow{dr}[swap]{q} & & \cA \arrow{dl}{p} \\
		{} & \cX
	\end{tikzcd} \]
	be a commutative diagram where $p$ and $q$ are cocartesian fibrations and $f$ preserves cocartesian edges.
	Letting $\mathrm{EssIm}(f)$ be the essential image of $f$, the composition
	\[ \mathrm{EssIm}(f) \subseteq \cA \stackrel{p}{\to } \cX \]
	is again a cocartesian fibration.
	Furthermore,  the formation of $\mathrm{EssIm}(f)$ commutes with pullback along any morphism $\cY\to \cX$ in $\Cat_{\infty}$.  
	In particular,  the fibers of $\mathrm{EssIm}(f)$ at $x \in \cX$ canonically coincide with the essential image of $f_x \colon \cB_x \to \cA_x$.
\end{lem}

The essential image construction of \cref{EssIm} allows to organize punctually split functors into a subfibration of the exponential fibration $\exp_\cE(\cI/\cX)$:

\begin{defin}
	Let $p \colon \cI \to \cX$ be a cocartesian fibration in posets and let $i_\cI \colon \cI^{\ens} \to \cI$ be the canonical morphism.
	We define the \emph{punctually split exponential fibration with coefficients in $\cE$ associated to $p \colon \cI \to \cX$} as
	\[ \exp_\cE^{\PS}( \cI / \cX ) \coloneqq \mathrm{EssIm}( \cE^{i_\cI}_! ) \ . \]
\end{defin}

\begin{rem}
	Notice that $\exp_\cE^{\PS}( \cI / \cX )$ defines an object in $\hCoCart_\cX$, but typically not in $\PrFibL_\cX$.
	\Cref{EssIm} shows that it is a sub-cocartesian fibration of $\exp_\cE(\cI / \cX)$.
	Under the specialization equivalence, we see that $\Sigma_\cX( \exp_\cE^{\PS}( \cI / \cX ) )$ coincides with the full subcategory of $\Fun(\cI, \cE)$ spanned by punctually split functors.
\end{rem}

Split functors provide a handy set of generators for $\Fun(\cA, \cI)$:

\begin{recollection}\label{recollection:generators_presheaves}
	Let $\cA$ be an $\infty$-category.
	For every $a \in \cA$, write $\ev_a^\cA \colon \{a\} \to \cA$ be the canonical inclusion.
	It follows from the Yoneda lemma that the functor
	\[ \ev_{a,!}^\cA \colon \Spc \to  \Fun(\cA, \Spc) \]
	is the unique colimit-preserving functor sending $*$ to $\Map_\cA(a,-)$.
	The density of the Yoneda embedding implies therefore that $\Fun(\cA, \Spc)$ is generated under colimits by $\{ \ev_{a, !}^\cA(*) \}_{a \in \cA}$.
	More generally, let $\cE$ be a presentable $\infty$-category generated under colimits by a set $\{E_\alpha\}_{\alpha \in I}$.
	Then under the identification
	\[ \Fun(\cA, \cE) \simeq \Fun(\cA, \Spc) \otimes \cE \]
	we see that $\ev_{a,!}(E_\alpha) \simeq \ev_{a,!}^\cA(\ast) \otimes E_\alpha$ and therefore that $\{ \ev_{a, !}^\cA(E_\alpha) \}_{a \in \cA, \alpha \in I}$ generates $\Fun(\cA, \cE)$ under colimits.
\end{recollection}

\begin{prop}\label{prop:generation_by_split_functors}
	Let $p \colon \cI \to \cX$ be a cocartesian fibration in posets.
	Then $\Fun(\cI, \cE)$ is generated under colimits by split functors.
\end{prop}

\begin{proof}
	Combine \cref{eg:split_functor} and \cref{recollection:generators_presheaves}.
\end{proof}

\subsection{Stokes functors}

We introduce here the fundamental object of this paper.
We fix once more a presentable $\infty$-category $\cE$.

\begin{defin}\label{def_Stokes}
	Let $p \colon \cI \to \cX$ be a cocartesian fibration in posets.
	The $\infty$-category of \emph{$\cI$-Stokes functors  with value in $\cE$} is by definition
	\[ \St_{\cI,\cE}  \coloneqq  \Sigma^{\cocart}_\cX(\exp_\cE^{\PS}( \cI / \cX ) ) \ . \]
\end{defin}

\begin{rem}
	Under the specialization equivalence (\ref{eq:specialization_equivalence}), the $\infty$-category $ \St_{\cI,\cE}$ coincides with full subcategory of  $\Fun(\cI,\cE)$ spanned by  functors $F \colon \cI \to \cE$ such that 
	\begin{enumerate}\itemsep=0.2cm
		\item $F$ is cocartesian (\cref{def:cocartesian_functor}).
		\item $F$ is punctually split (\cref{defin_split}). 
	\end{enumerate}
\end{rem}

\begin{eg}\label{eg:Stokes_functors_discrete_fibration}
	Assume that the cocartesian fibration $p \colon \cI \to \cX$ is discrete, i.e.\ that its fibers are sets.
	Then the map $i_\cI \colon \cI^{\ens} \to \cI$ is an equivalence, so in this case every functor $F \colon \cI \to \cE$ is split (and hence punctually split).
	It follows from the above remark that in this case
	\[ \St_{\cI,\cE} \simeq \Funcocart(\cI,\cE) \ . \]
\end{eg}

The cocartesian condition can be used to transport a splitting defined at an object $x \in \cX$ to a point $y$ via a morphism $\gamma \colon x \to y$, as in the following lemma:

\begin{lem}\label{lem:propagation_split_for_cocartesian}
	Let 
	\[ \begin{tikzcd}
		\cJ \arrow{d} & \cJ_\cX \arrow{d} \arrow{l}[swap]{u}\\
		\cY & \cX \arrow{l}[swap]{f}
	\end{tikzcd} \]
	be a pullback square in $\Cat_\infty$, whose vertical morphisms are cocartesian fibrations in posets.
	Let $\gamma \colon f(x) \to y$ be a morphism in $\cY$, with $x \in \cX$.
	Let $F \colon \cJ_\cX \to \cE$ be a functor such that $u_!(F) \colon \cJ \to \cE$ is cocartesian at $\gamma$ and such that the unit $F\to u^* u_!(F)$ is an equivalence above $x$.
	If $F$ is split at $x$, then $u_!(F)$ is split at $y$.
	In particular, when $f = \id_\cX$ and $F$ is cocartesian at $\gamma \colon x \to y$, if $F$ is split at $x$ then it is split at $y$ as well.
\end{lem}

\begin{proof}
	Let $f_\gamma \colon \cJ_{f(x)}\to \cJ_y$ be the morphism of posets induced by $\gamma \colon f(x)\to y$.
	Since $u_!(F) \colon \cJ \to \cE$ is cocartesian at $\gamma $, \cref{prop:cocartesian_functors_reformulation} implies the existence of an equivalence $f_{\gamma,!}((u_!(F))_{f(x)}) \simeq F_y$.
	By assumption $F_x\simeq (u_!(F))_{f(x)}$.
	The conclusion thus follows.
\end{proof}

%

This leads to the following neat description of Stokes functors when $\cX$ admits an initial object:

\begin{prop}\label{prop:Stokes_functors_in_presence_of_initial_object}
	Let $p \colon \cI \to \cX$ be a cocartesian fibration in posets.
	If $\cX$ admits an initial object $x$, then the adjunction
	\[ j_{x,!} \colon \Fun(\cI_x, \cE) \leftrightarrows \Fun(\cI,\cE) \colon j_x^\ast \]
	restricts to an equivalence of $\infty$-categories between $\St_{\cI_x,\cE}$ and $\St_{\cI,\cE}$.
\end{prop}

\begin{proof}
	Using \cref{induction_from_initial_preserves_cocart_functors}, we see that both $j_{x,!}$ and $j_x^\ast$ preserve cocartesian functors and that it restricts to an equivalence between $\Funcocart(\cI_x, \cE)$ and $\Funcocart(\cI,\cE)$.
	That $j_x^\ast$ preserves punctually split functors follows directly from the definition.
	On the other hand, combining together \cref{induction_from_initial_preserves_cocart_functors} and \cref{lem:propagation_split_for_cocartesian} we see that $j_{x,!}$ also preserves the punctually split condition.
	The conclusion follows.
\end{proof}

\begin{cor}\label{cor:induction_from_initial_object_is_essentially_surjective}
	Let $p \colon \cI \to \cX$ be an object of $\PosFib$.
	Assume that $\cX$ admits an initial object.
	Let $\cE$ be a presentable $\infty$-category.
	Then, the induction 
	\[ i_{\cI ! } \colon \St_{\cI^{\ens},\cE}\to  \St_{\cI,\cE} \]
	is essentially surjective.
	That is, every Stokes functor $F \colon \cI \to \cE$ splits.
\end{cor}

\begin{proof}
	Let $x$ be an initial object in $\cX$. 
	From  \cref{prop:Stokes_functors_in_presence_of_initial_object}, the horizontal arrows of the following commutative square
	\[ \begin{tikzcd}
		\St_{\cI_x^{\ens},\cE}   \arrow{r}{j_{x !}^{\ens}}\arrow{d}{i_{\cI_x !}}  &   \St_{\cI^{\ens},\cE}    \arrow{d}{i_{\cI !}}   \\
		\St_{\cI_x,\cE}  \arrow{r}{j_{x !}}&  	\St_{\cI,\cE}
	\end{tikzcd} \]
	are equivalences. 
	On the other hand, the left vertical arrow is essentially surjective by definition.
\end{proof}

\begin{warning}
	The splitting produced by \cref{cor:induction_from_initial_object_is_essentially_surjective} is not unique nor canonical.
\end{warning}

%

\subsection{Functoriality for punctually split and Stokes functors}

Fix a morphism
\[ \begin{tikzcd}
	\cJ \arrow{d} & \cJ_\cX \arrow{d} \arrow{r}{v} \arrow{l}[swap]{u} & \cI \arrow{dl} \\
	\cY & \cX \arrow{l}[swap]{f}
\end{tikzcd} \]
in $\PosFib$.
We now show that the basic functorialities of pullback and induction are well behaved with respect to punctually split and Stokes functors.
We start at the exponential level:

\begin{prop}\label{cor:stokes_functoriality_exponential}
	The functors
	\[ \cE^u \colon \exp_\cE(\cJ / \cY) \to \exp_\cE(\cJ_\cX / \cX) \qquad \text{and} \qquad \cE^v_! \colon \exp_\cE(\cJ_\cX / \cX) \to \exp_\cE(\cI / \cX) \]
	respect the punctually split sub-cocartesian fibrations and thus they induce the following commutative diagram:
	\begin{equation}\label{eq:stokes_functoriality_exponential}
		\begin{tikzcd}
			\exp_\cE^{\PS}(\cJ / \cY) \arrow[hook]{d} & \exp_\cE^{\PS}(\cJ_\cX / \cX) \arrow{l}[swap]{\cE^u} \arrow{r}{\cE^v_!} \arrow[hook]{d} & \exp_\cE^{\PS}(\cI / \cX) \arrow[hook]{d} \\
			\exp_\cE(\cJ / \cY) \arrow{d} & \exp_\cE(\cJ_\cX / \cX) \arrow{r}{\cE^v_!} \arrow{l}[swap]{\cE^u} \arrow{d} & \exp_\cE(\cI / \cX) \arrow{dl} \\
			\cY & \cX \arrow{l}
		\end{tikzcd}
	\end{equation}
	whose left squares are pullbacks.
\end{prop}

\begin{proof}
	An object in $\exp_\cE(\cJ_\cX / \cX)$ is a pair $(x,F)$, where $x \in \cX$ and $F \colon (\cJ_\cX)_x \to \cE$.
	The functor $\cE^u$ takes $(x,F)$ to $(f(x),F)$, where $F$ is now seen as a functor from $\cI_{f(x)} \simeq (\cJ_\cX)_x$ to $\cE$.
	In particular, $\cE^u$ preserves and reflects the punctually spit condition, which shows that the top left square is both commutative and a pullback.
	On the other hand, the commutativity of
	\[ \begin{tikzcd}
		(\cJ_\cX^{\ens})_x \arrow{r}{v_x^{\ens}} \arrow{d}{i_{\cJ_\cX}} & \cI_x^{\ens} \arrow{d}{i_\cI} \\
		(\cJ_\cX)_x \arrow{r}{v_x} & \cI_x
	\end{tikzcd} \]
	immediately implies that $\cE^v_!$ preserves the condition of being split at $x$.
\end{proof}

\begin{cor}\label{cor:ps_functoriality}
	In the above setting:
	\begin{enumerate}\itemsep=0.2cm
		\item Let  $F \colon \cJ \to \cE$ be a functor. 
		Let $x\in \cX$ be an object.
		Then, $F$ is punctually split at $f(x)$ if and only if $u^\ast(F)$ is punctually split at $x$.
		
		\item if $G \colon \cJ_X \to \cE$ is punctually split, then the same goes for $v_!(G) \colon \cI \to \cE$.
	\end{enumerate}
	In particular the functors
	\[ u^\ast \colon \Fun(\cJ, \cE) \to  \Fun(\cJ_\cX, \cE) \qquad \text{and} \qquad v_! \colon \Fun(\cJ_\cX, \cE) \to  \Fun(\cI, \cE) \]
	restrict to well-defined functors
	\[ u^\ast \colon \Fun_{\PS}(\cJ,\cE)\to  \Fun_{\PS}(\cJ_X,\cE) \qquad \text{and} \qquad v_! \colon \Fun_{\PS}(\cJ_X,\cE) \to  \Fun_{\PS}(\cI,\cE) \ . \]
\end{cor}
\begin{proof}
	Apply $\Sigma_\cX$ to the commutative diagram \eqref{eq:stokes_functoriality_exponential} and use \cref{prop:global_functoriality}.
\end{proof}

\begin{cor}\label{cor:stokes_functoriality}
	In the above setting:
	\begin{enumerate}\itemsep=0.2cm
		\item \label{cor:stokes_functoriality:pullback} if $F \colon \cJ \to \cE$ is a Stokes functor, the same goes for $u^\ast(F) \colon \cJ_\cX \to \cE$;
		
		\item \label{cor:stokes_functoriality:induction} if $G \colon \cJ_X \to \cE$ is a Stokes functor, then the same goes for $v_!(G) \colon \cI \to \cE$.
	\end{enumerate}
	Thus, the functors
	\[ u^\ast \colon \Fun(\cJ, \cE) \to  \Fun(\cJ_\cX, \cE) \qquad \text{and} \qquad v_! \colon \Fun(\cJ_\cX, \cE) \to  \Fun(\cI, \cE) \]
	restrict to well-defined functors
	\[ u^\ast \colon \St_{\cJ,\cE} \to  \St_{\cJ_X,\cE}  \qquad \text{and} \qquad v_! \colon \St_{\cJ_X,\cE} \to  \St_{\cI,\cE} \ . \]
\end{cor}

\begin{proof}
	Apply $\Sigma_\cX^{\cocart}$ to the commutative diagram \eqref{eq:stokes_functoriality_exponential} and combine \cref{prop:cocartesian_functoriality} and \cref{cor:ps_functoriality}.
\end{proof}

We conclude this section with the following generalization of \cref{prop:Stokes_functors_in_presence_of_initial_object}:

\begin{prop}\label{Stokes_and_localization}
	Let
	\[ \begin{tikzcd}
		\cJ \arrow{r}{g} \arrow{d} & \cI \arrow{d} \\
		\cY \arrow{r}{f} & \cX
	\end{tikzcd} \]
	be a pullback square in $\Cat_\infty$, where the vertical morphisms are cocartesian fibrations in posets.
	Assume that $f \colon \cY \to \cX$ is a localization functor.
	Let $\cE$ be a presentable $\infty$-category.
	Then, the following statements hold :
	\begin{enumerate}\itemsep=0.2cm
		\item Let $F\in \Fun(\cI, \cE)$. Then, $F$ is a Stokes functor if and only if so is $g^*(F)$.
		
		\item Let $G\in \Fun(\cJ, \cE)$. If $G$ is a Stokes functor, then so is $g_!(G)$.
		
		\item The adjunction $g_! \dashv g^*$ induces an equivalence of $\infty$-categories between $\St_{\cI,\cE}$ and $\St_{\cJ,\cE}$.
	\end{enumerate}
\end{prop}

\begin{proof}
	The claim (1) follows from  \cref{cocart_and_localization}-(1) and \cref{cor:ps_functoriality}-(1).
	Let $G \colon \cJ\to \cE$ be a Stokes functor.
	From \cref{cocart_and_localization}-(2) the functor $g_!(G)$ is cocartesian. 
	To check that $g_!(G)$ is punctually split amounts to show by (1) that $g^*(g_!(G))$ is punctually split.  
	On the other hand, \cref{cocart_and_localization}-(3) gives $g^*(g_!(G))\simeq G$ and (2) is proved.
	The claim (3) then follows from \cref{cocart_and_localization}-(3).
\end{proof}

\subsection{Stokes functors and (co)limits}

Stokes functors are poorly behaved with respect to limits and colimits, as the following next two lemmas are essentially the only stability properties one gets in general:

\begin{prop}\label{prop:stability_lim_colim_discrete_fibrations}
	Let $p \colon \cI \to \cX$ be a cocartesian fibration in \emph{sets}, seen as an object in $\PosFib$.
	Then $\St_{\cI,\cE}$ is presentable and furthermore:
	\begin{enumerate}\itemsep=0.2cm
		\item \label{prop:stability_lim_colim_discrete_fibrations:lim} $\St_{\cI, \cE}$ is stable under colimits in $\Fun(\cI, \cE)$.
		
		\item \label{prop:stability_lim_colim_discrete_fibrations:colim} Assume additionally that the fibers of $p$ are finite and that $\cE$ is presentable stable.
		Then $\St_{\cI,\cE}$ is stable under limits in $\Fun(\cI, \cE)$.
	\end{enumerate}
\end{prop}

\begin{proof}
	Via the equivalence $\St_{\cI,\cE} \simeq \Funcocart(\cI, \cE)$ of \cref{eg:Stokes_functors_discrete_fibration}, presentability follows from \cref{cor:cocart_presentable_stable}, statement (1) follows from \cref{local_cocart_stability_colimits} and statement (2) follows from \cref{cocart_stability_limits}.
\end{proof}

More in general, we have:

\begin{lem}\label{lem:Stokes_functors_coproducts}
	Let $p \colon \cI \to \cX$ be a cocartesian fibration in posets.
	Then $\St_{\cI,\cE}$ is closed under arbitrary coproducts in $\Fun(\cI,\cE)$.
\end{lem}

\begin{proof}
	Thanks to \cref{cor:cocartesianization}, we know that cocartesian functors are closed under arbitrary colimits in $\Fun(\cI,\cE)$.
	Besides, for every $x \in \cX$, the restriction functor $j_x^\ast \colon \Fun(\cI,\cE) \to \Fun(\cI_x, \cE)$ commute with all colimits as well.
	This reduces us to the case where $\cX$ is a single point, and we have to prove that split functors are closed under coproducts.
	Let therefore $\{F_i\}_{i \in I}$ be a family of split functors and fix splittings
	\[ \alpha_i \colon i_{\cI,!}(V_i) \simeq F_i \ . \]
	Since $i_{\cI,!}$ commutes with colimits, it immediately follows that $\coprod_{i \in I} V_i$ provides a splitting for $\coprod_{i \in I} F_i$.
\end{proof}

\begin{defin}
Let $p \colon \cI \to \cX$  be an object of $\PosFib$ and let $\cC\subset \PrL$ be a full subcategory.
We say that $p \colon \cI \to \cX$  is \textit{$\cC$-bireflexive} if the full subcategory $\St_{\cI,\cE}$ of $\Fun(\cI,\cE)$ is closed under limits and colimits for every $\cE\in \cC$.
\end{defin}

\begin{eg}
If $\cC$ only consists in a single category $\cE$, we say that $p \colon \cI \to \cX$  is $\cE$-bireflexive.
If $\cC\subset \PrL$ is the collection of all presentable stable $\infty$-categories, we simply say that $p \colon \cI \to \cX$  is  stably bireflexive.
\end{eg}

\begin{rem}
\cite[\cref*{Geometric_Stokes-stability_lim_colim_ISt}]{Geometric_Stokes} provides many geometrical examples of stably bireflexive cocartesian fibrations in posets.
\end{rem}

\begin{lem}\label{bireflexive_implies_presentable_stable}
Let $p \colon \cI \to \cX$ be an object of $\PosFib$ and let $\cE$ be a presentable (stable) $\infty$-category such that $p \colon \cI \to \cX$ is $\cE$-bireflexive.
	Then $\St_{\cI,\cE}$ is a localization of $\Fun(\cI,\cE)$, and in particular it is presentable (stable).
\end{lem}

\begin{proof}
	Since $\cE$ is presentable (stable),  $\Fun(\cI,\cE)$ is presentable (stable) in virtue of  \cite[Proposition 5.5.3.6]{HTT} and \cite[Proposition 1.1.3.1]{Lurie_Higher_algebra}.
	Then, the conclusion follows from the $\infty$-categorical reflection theorem, see \cite[Theorem 1.1]{Ragimov_Schlank_Reflection}.
\end{proof}

\begin{notation}
	In the setting of \cref{bireflexive_implies_presentable_stable},  the canonical inclusion $\St_{\cI,\cE}\hookrightarrow \Fun(\cI,\cE)$ admits a left  adjoint and a right adjoint, that we denote by $\LSt_{\cI,\cE}$ and $\RSt_{\cI,\cE}$ respectively.
\end{notation}

\begin{lem}\label{bireflexive_compact_generation}
Let $p \colon \cI \to \cX$ be an object of $\PosFib$. 
Let $\cE$ be a presentable stable compactly generated $\infty$-category such that $p \colon \cI \to \cX$ is $\cE$-bireflexive.
	Let $\{E_\alpha\}_{\alpha \in I}$ be a set of compact generators for $\cE$.
	Then $\St_{\cI,\cE}$ is presentable stable compactly generated by the $\{\LSt_{\cI,\cE}(\ev_{a,!}(E_{\alpha}))\}_{\alpha \in I,a\in \cI }$ where the $\ev_a : \{a\}\to \cI$ are the canonical inclusions.
\end{lem}

\begin{proof}
That $\St_{\cI,\cE}$  is presentable stable follows from \cref{bireflexive_implies_presentable_stable}.
	By \cref{recollection:generators_presheaves}, the $\{\ev_{a,!}(E_{\alpha})\}_{\alpha \in I,a\in \cI }$ are compact generators of $\Fun(\cI,\cE)$.
	Then \cref{bireflexive_compact_generation} formally follows from the fact that $\St_{\cI,\cE}\hookrightarrow \Fun(\cI,\cE)$ commutes with colimits.
\end{proof}

The following two lemmas are immediate consequences of \cref{induction_limit_stable}.

\begin{lem}\label{Stokes_p_commutation_limcolim}
	Let $\cX$ be an $\infty$-category.
	Let $p \colon \cI \to \cJ$ be a morphism in $\PosFib_{\cX}^f$.
	Let $\cE$ be a presentable stable $\infty$-category such that $\cI \to \cX$ and $\cJ \to \cX$ are $\cE$-bireflexive.
	Then, $p_! \colon \St_{\cI,\cE} \to\St_{\cJ,\cE}$ commutes with limits and colimits.
\end{lem}

\begin{lem}\label{Stokes_Grp_commutation_limcolim}
	Let $\cX$ be an $\infty$-category.
	Let $p \colon \cI \to \cJ$ be a morphism in $\PosFib_{\cX}^f$.
	Let $\cE$ be a presentable stable $\infty$-category such that $\cI \to \cX$ and $\cI_p \to \cX$ are $\cE$-bireflexive.
	Then, $\Gr_p \colon \St_{\cI,\cE} \to\St_{\cI_p,\cE}$ commutes with limits and colimits.
\end{lem}

\begin{cor}\label{level_pullback_in_PrLR}
	Let $(X,P)$ be an exodromic stratified space.
	Let $p \colon \cI \to \cJ$ be a graduation morphism of cocartesian fibrations in finite posets over $\Pi_{\infty}(X,P)$.
	Let $\cE$ be a presentable stable $\infty$-category and consider the pull-back square
	\[ 
	\begin{tikzcd}
		\St_{\cI, \cE} \arrow{r}{p_!} \arrow{d}{\Gr_p} & \St_{\cJ,\cE} \arrow{d}{\Gr} \\
		\St_{\cI_p, \cE} \arrow{r}{\pi_!} & \St_{\cJ^{\ens}, \cE}
	\end{tikzcd} 
	\]
	supplied by  \cref{prop:Level_induction}.
	If all the above cocartesian fibrations in posets are $\cE$-bireflexive, then the square is a pullback in $\PrLR$.
\end{cor}

\begin{proof}
	The $\infty$-categorical reflection theorem of \cite[Theorem 1.1]{Ragimov_Schlank_Reflection} implies that in this case all the $\infty$-categories of Stokes functors appearing in the above square are presentable.
	Then the conclusion follows combining \cref{Stokes_p_commutation_limcolim} with \cref{Stokes_Grp_commutation_limcolim}.
\end{proof}

\subsection{Van Kampen for Stokes functors}

In \cref{prop:Van_Kampen_cocartesian} we proved a Van Kampen result for cocartesian functors.
We now show that the same holds for Stokes functors:

\begin{prop}[Van Kampen for Stokes functors]\label{prop:Van_Kampen_Stokes}
	Let $\cX_\bullet \colon I \to \Cat_\infty$ be a diagram with colimit $\cX$.
	Let $p \colon \cI \to \cX$ be a cocartesian fibration in posets and set
	\[ \cI_\bullet \coloneqq \cX_\bullet \times_{\cX} \cI \colon I \to \Cat_\infty \ . \]
	Let $\cE$ be a presentable $\infty$-category.
	Then the equivalence of \cref{lem:Van_Kampen_filtered} restricts to an equivalence
	\[ \St_{\cI,\cE} \simeq \lim_{i \in I} \St_{\cI_i,\cE} \ . \]
\end{prop}

\begin{proof}
	Using \cref{cor:stokes_functoriality} in place of \cref{prop:cocartesian_functoriality}, we see that the natural map
	\[ \Fun(\cI, \cE) \to \lim_{i \in I} \Fun(\cI_i,\cE) \]
	gives rise to the following commutative square:
	\begin{equation}\label{eq:Van_Kampen_Stokes}
		\begin{tikzcd}
			\Fun(\cI, \cE) \arrow{r} & \lim_{i \in I} \Fun(\cI_i, \cE) \\
			\St_{\cI,\cE} \arrow{r} \arrow[hook]{u} & \lim_{i \in I} \St_{\cI_i,\cE} \arrow[hook]{u} \ .
		\end{tikzcd}
	\end{equation}
	It follows from \cref{lem:Van_Kampen_filtered} that the top horizontal functor is an equivalence.
	Thus, the bottom horizontal one is fully faithful.
	To conclude the proof, it is enough to prove that a functor $F \colon \cI \to \cE$ is Stokes if and only if for every $i \in I$, its image in $\Fun(\cI_i,\cE)$ is Stokes.
	The ``only if'' follows from \cref{cor:stokes_functoriality}.
	For the converse, we have already shown in \cref{prop:Van_Kampen_cocartesian} that if each restriction of $F$ is cocartesian then $F$ was cocartesian to begin with.
	We are left to check that $F$ is punctually split.
	Combining \cref{cocart_at_equivalence} and \cref{lem:computing_a_random_colimit}, we see that $F$ is punctually split if and only if it is split at every object of $\cX$ lying in the image of some structural map $f_i \colon \cX_i \to \cX$.
	However, if $x \in \cX$ is in the image of $f_i$, then $F$ is split at $x$ thanks to \cref{cor:ps_functoriality}.
\end{proof}

As a consequence of Van Kampen for Stokes functors, we can prove:

\begin{cor}\label{cor:descending_geometricity_of_Stokes_via_colimits}
	In the situation of \cref{prop:Van_Kampen_Stokes}, if furthermore $\cE$ is stable and if $\cI_i\to \cX$ is $\cE$-bireflexive for every $i\in I$, then $\cI\to \cX$ is $\cE$-bireflexive and the limit of \cref{prop:Van_Kampen_Stokes} can be computed inside $\PrLR$.
\end{cor}

\begin{proof}
	Let $f \colon i \to j$ be a morphism in $I$.
	Since $\cI_i$ and $\cI_j$ are $\cE$-bireflexive, it follows that $\St_{\cI_i,\cE}$ and $\St_{\cI_j,\cE}$ are presentable and that the transition functor
	\[ f^\ast \colon \St_{\cI_j,\cE} \to \St_{\cI_i, \cE} \]
	commute with limits and colimits.
	Therefore, it admits both a left and a right adjoint.
	In particular, the diagram $\St_{\cI_\bullet, \cE}$ factors through $\PrLR$.
	Since limits in $\PrL$ can be computed in $\CAT_\infty$, \cref{prop:Van_Kampen_Stokes} implies that $\St_{\cI,\cE}$ is presentable and stable.
	Besides, since all transition maps in $\St_{\cI_\bullet,\cE}$ commute with limits, it automatically follows that the structural functors
	\[ \St_{\cI,\cE} \to \St_{\cI_i,\cE} \]
	commute with limits as well.
	Thus, $\St_{\cI,\cE}$ is closed under limits inside $\Fun(\cI,\cE)$.
	On the other hand, \cref{lem:Stokes_functors_coproducts} shows that $\St_{\cI,\cE}$ is closed under arbitrary coproducts inside $\Fun(\cI,\cE)$.
	Since $\cE$ is stable and we already showed that $\St_{\cI,\cE}$ is stable, closure under finite colimits is automatic.
	The conclusion follows.
\end{proof}

\subsection{Change of coefficients for punctually split and Stokes functors}

Fix a cocartesian fibration in posets $p \colon \cI \to \cX$ and let $f \colon \cE \to \cE'$ be a morphism in $\PrL$.
Recall from \cref{subsec:change_of_coefficients} that this induces a transformation
\[ f^{\cI / \cX} \colon \exp_\cE(\cI / \cX) \to \exp_{\cE'}(\cI / \cX) \]
in $\PrFibL$.
We have:

\begin{prop}\label{prop:change_of_coefficients_stokes}
	The transformation $f^{\cI/\cX}$ respects the punctually split sub-cocartesian fibrations, and thus it induces a functor
	\[ f^{\cI / \cX} \colon \exp_\cE^{\PS}(\cI/ \cX) \to \exp_{\cE'}^{\PS}(\cI / \cX) \]
	in $\PrFibL_\cX$.
	In particular, the induced functor
	\[ f \colon \Fun(\cI, \cE) \to \Fun(\cI, \cE') \]
	induces well defined functors
	\[ f \colon \Fun_{\PS}(\cI, \cE) \to \Fun(\cI, \cE') \qquad \text{and} \qquad f \colon \St_{\cI,\cE} \to \St_{\cI,\cE'} \ . \]
\end{prop}

\begin{proof}
	Since $f$ commutes with colimits, it commutes with the formation of left Kan extensions.
	This immediately implies the first statement.
	Applying $\Sigma_\cX$ and using \cref{prop:change_of_coefficients}, we deduce that $f \colon \Fun(\cI,\cE) \to \Fun(\cI, \cE')$ preserves punctually split functors.
	In turn, this fact and \cref{prop:change_of_coefficients_cocartesian} implies that $f$ also preserves Stokes functors.
\end{proof}

Cocartesian functors exhibit a nice behavior with respect to the tensor product in $\PrL$ (see \cref{cocart_tensor_for_exp}).
On the other hand $\exp_\cE^{\PS}(\cI / \cX)$ is typically not a presentable fibration, and $\St_{\cI,\cE}$ is typically not presentable.
This prevents from formally deducing an analogue of \cref{cocart_tensor_for_exp} for Stokes functors: such a result will be true, but only in a more restrictive geometric setting, see \cite[\cref*{Geometric_Stokes-thm:Stokes_tensor_product}]{Geometric_Stokes}.
For the moment, let us simply collect a couple of elementary observations that will be needed later.

\medskip

When bireflexivity holds, we can construct, for every pair of presentable $\infty$-categories $\cE$ and $\cE'$, a canonical comparison morphism $\St_{\cI,\cE} \otimes \cE' \to \St_{\cI,\cE \otimes \cE'}$.
The key point is the following lemma:

\begin{lem}\label{lem:splitting_tensor}
	Let $p \colon \cI \to \cX$ be a cocartesian fibration in posets.
	Let $\cE$ and $\cE'$ be presentable $\infty$-categories.
	Let $x \in \cX$ be an object and let $F \colon \cI \to \cE$ be a functor that splits at $x$.
	Then for every object $E' \in \cE'$, the functor
	\[ F \otimes E' \in \Fun(\cI,\cE) \otimes \cE' \simeq \Fun(\cI, \cE \otimes \cE') \]
	splits at $x$ as well.
\end{lem}

\begin{proof}
	Let $i_x \colon  \cI_x \to \cI$ be the canonical functor.
	For any presentable $\infty$-category $\cD$, both functors
	\[ i_{x}^\ast \colon \Fun(\cI, \cD) \to \Fun(\cI_x, \cD) \qquad \text{and} \qquad i_{\cI_x,!} \colon \Fun(\cI_x^{\ens}, \cD) \to \Fun(\cI_x, \cD)  \]
	commute with colimits, so we obtain the following canonical identifications:
	\[ 
	\begin{tikzcd}[column sep=large]
		\Fun(\cI,\Spc) \otimes \cD \arrow{r}{i_x^* \otimes \cD } \arrow{d}{\wr}  & \Fun(\cI_x,\Spc) \otimes \cD \arrow{d}{\wr} &    \Fun(\cI_x^{\ens},\Spc ) \otimes \cD  \arrow{l}[swap]{i_{\cI_x, !}\otimes\cD} \arrow{d}{\wr} \\
		\Fun(\cI,\cD) \arrow{r}{i_x^*} & \Fun(\cI_x, \cD)   &     \Fun(\cI_x^{\ens},\cD)        \arrow{l}[swap]{i_{\cI_x, !}} \ .
	\end{tikzcd}
	\]
	This proves the lemma when $\cE = \Spc$, and the general case follows from the associativity of the tensor product in $\PrL$.
\end{proof}

\begin{construction}\label{construction:Stokes_tensor_comparison}
	Let $p \colon \cI \to \cX$ be a cocartesian fibration in posets.
	Let $\cE$ and $\cE'$ be presentable $\infty$-categories such that $p \colon \cI \to \cX$ is $\{\cE,\cE'\}$-bireflexive.
	Consider the following solid commutative diagram:
	\begin{equation}\label{eq:Stokes_tensor_comparison}
		\begin{tikzcd}
			\St_{\cI,\cE} \otimes \cE' \arrow{d} \arrow[dashed]{r} & \St_{\cI,\cE \otimes \cE'} \arrow[hook]{d} \\
			\Funcocart(\cI,\cE) \otimes \cE' \arrow{r}{\sim}  \arrow{d} & \Funcocart(\cI,\cE \otimes \cE' ) \arrow[hook]{d} \\
			\Fun(\cI,\cE) \otimes \cE' \arrow{r}{\sim} & \Fun(\cI,\cE \otimes \cE' ) 
		\end{tikzcd}
	\end{equation}
	in $\PrL$.
	By definition, $\St_{\cI,\cE} \otimes \cE'$ is generated under colimits by objects of the form $F \otimes E'$, where $F \colon \cI \to \cE$ is a Stokes functor and $E' \in \cE'$ is an object.
	\Cref{lem:splitting_tensor} guarantees that such an object is mapped via the bottom horizontal equivalence into an object in $\St_{\cI,\cE \otimes \cE'}$.
	Since the right vertical arrows are fully faithful by definition, it follows that the dashed arrow exist.
\end{construction}

\begin{prop}\label{prop:Stokes_tensor_comparison_fully_faithful}
	Let $p \colon \cI \to \cX$ be a cocartesian fibration in posets.
	Let $\cE$ and $\cE'$ be presentable $\infty$-categories.
	Assume that:
	\begin{enumerate}\itemsep=0.2cm
		\item The fibers of $\cI$ are finite;
		\item $p \colon \cI \to \cX$ is $\cE,\cE'$-bireflexive.
	\end{enumerate}
	Then the canonical comparison functor
	\[ \St_{\cI,\cE} \otimes \cE' \to \St_{\cI,\cE \otimes \cE'} \]
	of \cref{construction:Stokes_tensor_comparison} is fully faithful.
\end{prop}

\begin{proof}
	Since the fibers of $\cI$ are finite, \cref{cocart_stability_limits} implies that $\Funcocart(\cI,\cE)$ is closed under limits and colimits in $\Fun(\cI,\cE)$.
	By (2), it follows that $\St_{\cI,\cE}$ is closed under limits and colimits in $\Funcocart(\cI,\cE)$ as well.
	Since in this situation $\St_{\cI,\cE}$ is presentable by \cref{bireflexive_implies_presentable_stable}, it follows that the inclusion of $\St_{\cI,\cE}$ into $\Funcocart(\cI,\cE)$ has both a left and a right adjoint.
	Therefore, the functoriality of the tensor product in $\PrL$ implies that the top left vertical arrow in the diagram \eqref{eq:Stokes_tensor_comparison} is fully faithful.
	On the other hand, the middle horizontal functor is an equivalence thanks to \cref{cocart_tensor_for_exp}, so the conclusion follows.
\end{proof}

\begin{defin}\label{def:stably_universal}
	Let $p \colon \cI \to \cX$  be an object of $\PosFib$ and let $\cC\subset \PrL$ be a full subcategory stable under tensor product.
	We say that $p \colon \cI \to \cX$  is \emph{$\cC$-universal} if it is $\cC$-bireflexive and the comparison map $\St_{\cI,\cE} \otimes \cE' \to \St_{\cI,\cE\otimes \cE'}$ of \cref{construction:Stokes_tensor_comparison} is an equivalence for every$ \cE,\cE'\in \cC$.
	If $\cC\subset \PrL$ is the collection of all presentable stable $\infty$-categories, we simply say that $p \colon \cI \to \cX$ is stably universal.
\end{defin}

\begin{prop}\label{prop:descent_of_universality}
	Let $\cX_\bullet \colon I \to \Cat_\infty$ be a diagram with colimit $\cX$.
	Let $p \colon \cI \to \cX$ be a cocartesian fibration in posets and set
	\[ \cI_\bullet \coloneqq \cX_\bullet \times_{\cX} \cI \colon I \to \Cat_\infty \ . \]
	Assume that $\cI_i\to \cX$ is stably universal for every $i\in I$.
	Then $\cI\to \cX$ is stably universal.
\end{prop}

\begin{proof}
Note that $\cI\to \cX$ is stably bireflexive by \cref{cor:descending_geometricity_of_Stokes_via_colimits}.
For every presentable stable $\infty$-categories $\cE,\cE'$, we have 
	\begin{align*}
		\St_{\cI,\cE} \otimes \cE' & \simeq (\lim_{i\in I} \St_{\cI_i,\cE} )  \otimes \cE' & \text{\cref{cor:descending_geometricity_of_Stokes_via_colimits}}  \\
		&\simeq \lim_{i\in I}( \St_{\cI_i,\cE} \otimes \cE' )   &\text{\cref{Peter_lemma} }  \\
		& \simeq  \lim_{i\in I} \St_{\cI_i,\cE\otimes \cE' }    & \\
		& \simeq \St_{\cI,\cE\otimes \cE' } & \text{\cref{cor:descending_geometricity_of_Stokes_via_colimits}}
	\end{align*}	
\end{proof}

\subsection{Induced $t$-structures for Stokes functors}\label{subsec:t_structure_Stokes}

Fix a presentable stable $\infty$-category $\cE$ equipped with an induced $t$-structure $\tau = (\cE_{\geqslant 0}, \cE_{\leqslant 0})$.
In \cref{subsec:t_structure_cocartesian_functors} we showed that cocartesian functors inherits a $t$-structure from $\tau$, and we analyzed the basic properties.
We now investigate the behavior with respect to Stokes functor.

\medskip

We start with a couple of general facts concerning $t$-structures.

\begin{construction}\label{construction:comparison_truncation_functors}
	Let $\cC$ and $\cD$ be stable $\infty$-categories equipped with $t$-structures $\tau^{\cC} = (\cC_{\leqslant 0}, \cC_{\geqslant 0})$ and $\tau^{\cD} = (\cD_{\leqslant 0}, \cD_{\geqslant 0})$ and let
	\[ F \colon \cC \to \cD \]
	be a right $t$-exact stable functor.
	For every object $C \in \cC$, one has $F(\tau^\cC_{\geqslant 0}(C)) \in \cD_{\geqslant 0}$, and therefore the mapping space
	\[ \Map_\cD\big( F(\tau^\cC_{\geqslant 0}(C)), \tau^{\cD}_{\leqslant -1}( F(C) ) \big) \]
	is contractible.
	It follows that there exists the dashed morphisms making the diagram
	\begin{equation}\label{eq:comparison_truncation_functors}
		\begin{tikzcd}
			F\big( \tau_{\geqslant 0}^\cC(C) \big) \arrow{r} \arrow[dashed]{d} & F(C) \arrow[equal]{d} \arrow{r} & F\big( \tau_{\leqslant -1}^\cC(C) \big) \arrow[dashed]{d} \\
			\tau^\cD_{\geqslant 0}\big( F(C) \big) \arrow{r} & F(C) \arrow{r} & \tau^\cD_{\leqslant -1}\big( F(C) \big)
		\end{tikzcd}
	\end{equation}
	commutative.
\end{construction}

\begin{lem}\label{lem:t_structure_commuting_with_truncations}
	In the situation of \cref{construction:comparison_truncation_functors}, let $C \in \cC$ be an object.
	If $F( \tau_{\leqslant -1}^\cC(C) ) \in \cD_{\leqslant -1}$ then both canonical comparison maps
	\[ F\big( \tau_{\geqslant 0}^\cC(C) \big) \to \tau^\cD_{\geqslant 0}\big( F(C) \big) \qquad \text{and} \qquad F\big( \tau_{\leqslant -1}^\cC(C) \big) \to \tau^\cD_{\leqslant -1}\big( F(C) \big) \]
	are equivalences.
\end{lem}

\begin{proof}
	Since $F$ is a stable functor, the top row of \eqref{eq:comparison_truncation_functors} is a fiber sequence in $\cD$.
	By definition of $t$-structure, the same holds true for the bottom row.
	Set
	\[ K \coloneqq \fib\big( F\big( \tau_{\geqslant 0}^\cC(C) \big) \to \tau^\cD_{\geqslant 0}\big( F(C) \big) \big) \qquad \text{and} \qquad K' \coloneqq \fib\big( F\big( \tau_{\leqslant -1}^\cC(C) \big) \to \tau^\cD_{\leqslant -1}\big( F(C) \big) \big) \ . \]
	We therefore obtain a fiber sequence
	\[ K \to 0 \to K' \ , \]
	which implies $K' \simeq K[1]$.
	Observe now that $K' \in \cD_{\leqslant -1}$.
	At the same time,
	\[ K[1] \simeq \cofib\big( F\big( \tau_{\geqslant 0}^\cC(C) \big) \to \tau^\cD_{\geqslant 0}\big( F(C) \big) \big) \in \cD_{\geqslant 0} \ . \]
	Thus, it follows that $K' \in \cD_{\geqslant 0} \cap \cD_{\leqslant -1} = \{0\}$.
	Thus, both $K$ and $K'$ are zero, which implies that the comparison morphisms are equivalences.
\end{proof}

We now start analyzing the behavior of the standard $t$-structure on Stokes functors.

\begin{recollection}\label{recollection:standard_t_structure}
	Let $f \colon \cA \to \cB$ be a functor of $\infty$-categories.
	Then
	\[ f^\ast \colon \Fun(\cB, \cE) \to \Fun(\cA, \cE) \]
	is $t$-exact with respect to the standard $t$-structures.
	In particular, $f_!$ is right $t$-exact and $f_\ast$ is left $t$-exact.
\end{recollection}

\begin{lem}\label{lem:t_structure_Stokes_discrete_source}
	Let $f \colon \cI \to \cJ$ be a morphism of posets, where $\cI$ is discrete and finite.
	Then
	\[ f_! \colon \Fun(\cI, \cE) \to \Fun(\cJ,\cE) \]
	is $t$-exact.
\end{lem}

\begin{proof}
	Fix a functor $F \colon \cI \to \cE$ and an object $b \in \cJ$.
	By definition
	\[ f_!(F)_b \simeq \bigoplus_{f(a) \leq b} F_a \ , \]
	so the conclusion follows from the fact that both $\cE_{\geqslant 0}$ and $\cE_{\leqslant 0}$ are closed under finite sums.
\end{proof}

\begin{cor}\label{cor:t_structure_Stokes_truncation_split_is_split}
	Let $\cI$ be a finite poset and let $F \colon \cI \to \cE$ be a functor.
	If $F$ is split, then so are $\tau_{\leqslant n}(F)$ and $\tau_{\geqslant n}(F)$ for every $n \in \Z$.
\end{cor}

\begin{proof}
	It suffices to treat the case $n = 0$.
	Choose a functor $V \colon \cI^{\ens} \to \cE$ with an equivalence $F \simeq i_{\cI,!}(V)$.
	Since $\cI$ is finite, \cref{lem:t_structure_Stokes_discrete_source} implies that
	\[ \tau_{\leqslant 0}( i_{\cI,!}(V) ) \simeq i_{\cI,!}( \tau_{\leqslant 0}(V) ) \qquad \text{and} \qquad \tau_{\geqslant 0}( i_{\cI,!}(V) ) \simeq i_{\cI,!}( \tau_{\geqslant 0}(V) ) \ , \]
	whence the conclusion.
\end{proof}

\begin{lem} \label{lem:t_structure_Stokes_splitting}
	Let $\cI$ be a poset and let $F \colon \cI \to \cE$ be a functor.
	Let $(V,\phi \colon i_{\cI,!}(V) \simeq F)$ be a splitting for $F$.
	Let $n \in \Z$ be an integer.
	If $F$ takes values in $\cE_{\geqslant n}$ (resp.\ $\cE_{\leqslant n}$), then the same goes for $V$.
\end{lem}

\begin{proof}
	It suffices to consider the case $n = 0$.
	For $a \in \cI$, $\phi$ induces
	\[ F_a \simeq \bigoplus_{b \leqslant a} V_b \ . \]
	In particular, $V_a$ is a retract of $F_a$.
	Since $F_a \in \cE_{\geqslant 0}$ (resp.\ $F_a \in \cE_{\leqslant 0}$), it follows that $V_a \in \cE_{\geqslant 0}$ (resp.\ $V_a \in \cE_{\leqslant 0}$) as well.
\end{proof}

\begin{cor} \label{lem:t_structure_Stokes_induction}
	Let $f \colon \cI \to \cJ$ be a morphism of finite posets.
	Let $F \colon \cI \to \cE$ be a split functor and let $n \in \Z$ be an integer.
	If $F$ takes values in $\cE_{\leqslant n}$, then so does $f_!(F)$;
\end{cor}

\begin{proof}
	It suffices to consider the case $n = 0$.
	Since $F$ is split, we can find a functor $V \colon \cI^{\ens} \to \cE$ and an equivalence $\phi \colon F \simeq i_{\cI,!}(V)$.
	\Cref{lem:t_structure_Stokes_splitting} guarantees that $V$ takes values in $\cE_{\leqslant 0}$.
	Thus, we find
	\[ f_!(F) \simeq f_! (i_{\cI,!}(V)) \simeq i_{\cJ,!} (f^{\ens}_!(V)) \ , \]
	and the conclusion now follows from \cref{lem:t_structure_Stokes_discrete_source} applied to $i_\cJ \circ f^{\ens} \colon \cI^{\ens} \to \cJ$.
\end{proof}

\begin{notation}\label{notation:truncations_filtered_functors}
	Given an $\infty$-category $\cA$, we denote again by
	\[ \tau_{\geqslant n} \colon \Fun(\cA, \cE) \to \Fun(\cA, \cE_{\geqslant n}) \qquad \text{and} \qquad \tau_{\leqslant n} \colon \Fun(\cA, \cE) \to \Fun(\cA, \cE_{\leqslant n}) \]
	the induced truncation functors, given respectively by the compositions
	\[ \tau_{\geqslant n}(F) \coloneqq \tau_{\geqslant n} \circ F \qquad \text{and} \qquad \tau_{\leqslant n}(F) \coloneqq \tau_{\leqslant n} \circ F \ . \]
\end{notation}

\begin{lem}\label{lem:t_structures_Stokes_induction_II}
	Let $f \colon \cI \to \cJ$ be a morphism of finite posets and let $F \colon \cI \to \cE$ be a split functor.
	Then for every integer $n$, the canonical maps of \cref{construction:comparison_truncation_functors}
	\[ f_!(\tau_{\geqslant n}(F)) \to \tau_{\geqslant n}( f_!(F) ) \qquad \text{and} \qquad f_!(\tau_{\leqslant n}(F)) \to \tau_{\leqslant n}( f_!(F) ) \]
	are equivalences.
\end{lem}

\begin{proof}
	It suffices to consider the case $n = 0$.
	Since $F$ is split, \cref{cor:t_structure_Stokes_truncation_split_is_split} guarantees that $\tau_{\leqslant -1}(F)$ is again split and takes values in $\cE_{\leqslant -1}$.
	Therefore, \cref{lem:t_structure_Stokes_induction} implies that $f_!( \tau_{\leqslant -1}(F) )$ takes values in $\cE_{\leqslant -1}$.
	At this point, the conclusion follows from \cref{lem:t_structure_commuting_with_truncations}.
\end{proof}

\begin{prop} \label{prop:t_structure_Stokes}
	Let $p \colon \cI \to \cX$ be a cocartesian fibration in finite posets and let $\cE$ be a stable presentable $\infty$-category equipped with an accessible $t$-structure $\tau = (\cE_{\leqslant 0}, \cE_{\geqslant 0})$.
	If $F \colon \cI \to \cE$ is a Stokes functor then for every integer $n \in \Z$, both $\tau_{\leqslant n}(F)$ and $\tau_{\geqslant n}(F)$ are again Stokes functors.
	In particular, $p \colon \cI \to \cX$ is $\cE$-bireflexive, then $\St_{\cI,\cE}$ acquires a unique accessible $t$-structure such that the inclusion
	\[ \St_{\cI,\cE} \hookrightarrow \Fun(\cI,\cE) \]
	is $t$-exact.
	If in addition $\tau$ is compatible with filtered colimits, the same goes for the induced $t$-structure on $\St_{\cI,\cE}$.
\end{prop}

\begin{proof}
	We know from \cref{bireflexive_implies_presentable_stable} that $\St_{\cI,\cE}$ is presentable and stable.
	Since $\St_{\cI,\cE}$ is closed under limits and colimits in $\Fun(\cI,\cE)$, the first half of the statement implies the existence of the desired $t$-structure, its accessibility and its compatibility with filtered colimits.
	
	\medskip
		
	Let us therefore prove the first part.
	It suffices to consider the case $n = 0$.
	Let $F \colon \cI \to \cE$ be a Stokes functor.
	We first prove that $\tau_{\geqslant 0}(F)$ and $\tau_{\leqslant 0}(F)$ are punctually split.
	Fix an object $x \in \cX$.
	Since $j_x^\ast \colon \Fun(\cI, \cE) \to \Fun(\cI_x, \cE)$ is $t$-exact, we find canonical equivalences
	\[ j_x^\ast( \tau_{\geqslant 0}(F) ) \simeq \tau_{\geqslant 0}( j_x^\ast(F) ) \qquad \text{and} \qquad j_x^\ast( \tau_{\leqslant 0}(F) ) \simeq \tau_{\leqslant 0}( j_x^\ast(F) ) \ . \]
	Since $j_x^\ast(F)$ is split by assumption, \cref{cor:t_structure_Stokes_truncation_split_is_split} implies that the same goes for $\tau_{\leqslant 0}(j_x^\ast(F))$ and $\tau_{\geqslant 0}(j_x^\ast(F))$ as well, which proves the first claim.
	
	\medskip
	
	We now prove that $\tau_{\geqslant 0}(F)$ and $\tau_{\leqslant 0}(F)$ are cocartesian.
	Let $\gamma \colon x \to y$ be a morphism in $\cX$ and let $f_\gamma \colon \cI_x \to \cI_y$ be any straightening for $\cI_\gamma \to \Delta^1$.
	By \cref{lem:t_structures_Stokes_induction_II}, the canonical comparison maps
	\[ f_{\gamma,!}( \tau_{\geqslant 0}( j_x^\ast(F) ) ) \to \tau_{\geqslant 0}( f_{\gamma,!}( j_x^\ast(F) ) ) \qquad \text{and} \qquad f_{\gamma,!}( \tau_{\leqslant 0}( j_x^\ast(F) ) ) \to \tau_{\leqslant 0}( f_{\gamma,!}( j_x^\ast(F) ) ) \]
	are equivalences.
	Since $F$ is cocartesian, the canonical map
	\[ f_{\gamma,!}( j_x^\ast(F) ) \to j_y^\ast(F) \]
	is an equivalence.
	The conclusion now follows from the $t$-exactness of both $j_x^\ast$ and $j_y^\ast$.
\end{proof}

\begin{cor}\label{cor:t_structure_Stokes_heart}
	In the setting of \cref{prop:t_structure_Stokes}, one has a canonical equivalence:
	\[ \St_{\cI,\cE}^\heartsuit \simeq \St_{\cI,\cE^\heartsuit} \ . \]
\end{cor}

\begin{proof}
	By definition of the standard $t$-structure on $\Fun(\cI, \cE)$, we have $\Fun(\cI,\cE)^\heartsuit \simeq \Fun(\cI,\cE^\heartsuit)$.
	\Cref{prop:t_structure_Stokes} guarantees that a Stokes functor is connective (resp.\ coconnective) if and only if it is connective (resp.\ coconnective) as an object in $\Fun(\cI,\cE)$, so the conclusion follows.
\end{proof}

\begin{recollection}\label{derived_category_A}
If $\cA$ is a Grothendieck abelian category, we denote by $\mathsf D(\cA)$ the derived $\infty$-category of $\cA$ (see \cite[Definition 1.3.5.8]{Lurie_Higher_algebra}).
By \cite[Propositions 1.3.5.9 \& 1.3.5.21]{Lurie_Higher_algebra} we see that $\mathsf D(\cA)$ is a presentable stable $\infty$-category equipped with an accessible $t$-structure $\tau = (\mathsf D(\cA)_{\geqslant 0}, \mathsf D(\cA)_{\leqslant 0})$ compatible with filtered colimits and such that $\cA \simeq \mathsf D(\cA)^\heartsuit$.
\end{recollection}

\begin{cor}\label{Grothendieck_abelian_abstract}
	Let $p \colon \cI \to \cX$ be an object of $\PosFib$ and let $\cA$ 
	be a Grothendieck abelian category such that $p \colon \cI \to \cX$ is $\mathsf D(\cA)$-bireflexive.
	Then $\St_{\cI,\cA}$ is a Grothendieck abelian category.
\end{cor}

\begin{proof}
\cref{bireflexive_implies_presentable_stable} implies that $\St_{\cI,\mathsf D(\cA)}$ is presentable and stable, while \cref{prop:t_structure_Stokes} guarantees that $\tau$ induces an accessible $t$-structure on $\St_{\cI,\mathsf D(\cA)}$ which is compatible with filtered colimits and such that the inclusion
	\[ \St_{\cI,\mathsf D(\cA)} \hookrightarrow \Fun(\cI, \mathsf D(\cA)) \]
	is $t$-exact. 
	Moreover, \cref{cor:t_structure_Stokes_heart} and \cref{derived_category_A} imply that
	\[ \St_{\cI,\mathsf D(\cA)}^\heartsuit \simeq \St_{\cI,\cA} \ . \]
	Thus, it follows that $\St_{\cI,\cA}$ is a Grothendieck abelian category.
\end{proof}

\begin{cor}\label{cor:t_exactness_induction_on_Stokes}
	Let $\cX$ be an $\infty$-category and let $f \colon \cI \to \cJ$ be a morphism between cocartesian fibrations in finite posets over $\cX$.
	If $\cI\to \cX$ and $\cJ\to \cX$ are $\cE$-bireflexive, then the functor
	\[ f_! \colon \St_{\cI, \cE} \to \St_{\cJ, \cE} \]
	is $t$-exact.
\end{cor}

\begin{proof}
	It follows from \cref{prop:t_structure_Stokes} and \cref{recollection:standard_t_structure} that $f_!$ is right $t$-exact.
	Let $F \in (\St_{\cI,\cE})_{\leqslant 0}$.
	We have to prove that $f_!(F)$ takes values in $\cE_{\leqslant 0}$.
	Combining Corollaries \ref{cor:induction_specialization_Beck_Chevalley} and \ref{cor:stokes_functoriality}, we can reduce ourselves to the case where $\cX$ is reduced to a point, where the result follows from \cref{lem:t_structure_Stokes_induction}.
\end{proof}

\begin{rem}
	The inclusion $\Funcocart(\cI, \cE) \hookrightarrow \Fun(\cI,\cE)$ is typically not left $t$-exact and in general
	\[ \Funcocart(\cI,\cE)^\heartsuit \not\simeq \Funcocart(\cI,\cE^\heartsuit) \ , \]
	as \cref{eg:cocartesian_functor_in_the_heart} shows.
	Notice however that the functor $F$ considered there is not punctually split at $0$.
	Similarly, if $f \colon \cI \to \cJ$ is a morphism between cocartesian fibrations in finite posets, neither
	\[ f_! \colon \Fun(\cI, \cE) \to \Fun(\cJ,\cE) \]
	nor its cocartesian variant are left $t$-exact.
	However, it becomes left $t$-exact once restricted to $\St_{\cI,\cE}$, thanks to \cref{cor:t_exactness_induction_on_Stokes}.
\end{rem}

\begin{cor}\label{cor:t_structures_Stokes_initial_object}
	Let $\cI \to \cX$ be a cocartesian fibration in posets.
	If $\cX$ admits an initial object $x$, then a Stokes functor $F \colon \cI \to \cE$ takes values in $\cE^\heartsuit$ if and only if $j_x^\ast(F) \colon \cI_x \to \cE$ takes values in $\cE^\heartsuit$.
\end{cor}

\begin{proof}
	The ``only if'' direction simply follows from the $t$-exactness of $j_x^\ast \colon \Fun(\cI, \cE) \to \Fun(\cI_x, \cE)$.
	For the ``if'' direction, we equivalently have to show that for every $y \in \cX$, the restriction $j_y^\ast(F)$ takes values in $\cE^\heartsuit$.
	Since $x$ is initial in $\cX$, we can find a morphism $\gamma \colon x \to y$.
	Choose any straightening $f_\gamma \colon \cI_x \to \cI_y$ for $\cI_\gamma \to \Delta^1$.
	Since $F$ is cocartesian, the canonical map
	\[ f_{\gamma,!}( j_x^\ast(F) ) \to j_y^\ast(F) \]
	is an equivalence.
	The conclusion now follows from \cref{lem:t_structure_Stokes_induction}.
\end{proof}

\subsection{Categorical actions on Stokes functors}\label{subsec:categorical_actions_Stokes}

We use the terminology on categorical actions reviewed in \cref{sec:categorical_actions} (see also \cref{subsec:categorical_actions_cocartesian}, of which this section is the continuation).
We fix a presentably symmetric monoidal and stable $\infty$-category $\cE^\otimes$.
In analogy to \cref{prop:categorical_action_cocartesian}, we have:

\begin{prop}\label{prop:categorical_action_Stokes}
	Let $p \colon \cI \to \cX$ be a cocartesian fibration in posets.
	Then for every $L \in \Loc(\cX;\cE)$ and every $G \in \Fun_{\PS}(\cI, \cE)$, the functor
	\[ p^\ast(L) \otimes G \colon \cI \to \cE \]
	is again punctually split.
	In particular, if $G$ is a Stokes functor, the same goes for $p^\ast(L) \otimes G$.
\end{prop}

\begin{proof}
	Fix $x \in \cX$.
	Since the restrictions $j_x^\ast \colon \Fun(\cI,\cE) \to \Fun(\cI_x, \cE)$ are $\cE$-linear, we can assume without loss of generality that $\cX$ is reduced at a single point.
	Choose a splitting $V \colon \cI^{\ens} \to \cE$ for $G$.
	\Cref{lem:linearity_of_induction} implies that $i_{\cI,!} \colon \Fun(\cI^{\ens}, \cE) \to \Fun(\cI,\cE)$ is $\cE$-linear.
	Therefore, for every $L \in \cE$ we obtain
	\[ p^\ast( L ) \otimes G \simeq p^\ast( L ) \otimes i_{\cI,!}(V) \simeq i_{\cI,!}( p^{\ens,\ast}( L ) \otimes V ) \ , \]
	i.e.\ $p^\ast( L ) \otimes G$ is split.
\end{proof}

\begin{cor}\label{cor:categorical_action_Stokes}
	Let $p \colon \cI \to \cX$ be a $\cE$-bireflexive cocartesian fibration in posets.
	Then the categorical action of $\Loc(\cX;\cE)$ on $\Fun(\cI,\cE)$ restricts to a categorical action of $\Loc(\cX;\cE)$ on $\St_{\cI,\cE}$.
\end{cor}

\begin{proof}
	This is obvious from \cref{prop:categorical_action_Stokes}.
\end{proof}

We now derive an analogue of Corollaries \ref{cor:relative_tensor_product} and \ref{cor:finite_etale_fibration_cocartesian_relative_tensor} in the setting of Stokes functors.
We fix a pullback square
\[ \begin{tikzcd}
	\cJ \arrow{r}{u} \arrow{d}{q} & \cI \arrow{d}{p} \\
	\cY \arrow{r}{f} & \cX
\end{tikzcd} \]
in $\Cat_\infty$, where $p$ is a cocartesian fibration in posets.
In addition, we assume that $f$ is a finite étale fibration (see \cref{def:finite_etale_fibrations}) and that both $\cI$ and $\cJ$ are $\cE$-bireflexive.
In this setting, \cref{construction:relative_tensor_of_functor_categories} supplies a canonical transformation
\[ \mu \colon \Loc(\cY;\cE) \otimes_{\Loc(\cX;\cE)} \Fun(\cI,\cE) \to \Fun(\cJ, \cE) \ , \]
and \cref{prop:categorical_action_cocartesian} and \cref{cor:categorical_action_Stokes} imply that this action restricts to a well defined categorical action
\[ \mu \colon \Loc(\cY;\cE) \otimes_{\Loc(\cX;\cE)} \St_{\cI,\cE} \to \St_{\cJ, \cE} \ . \]

\begin{lem}\label{lem:finite_etale_fibration_induction_Stokes}
	In the above setting, the functor
	\[ u_! \colon \Fun(\cJ, \cE) \to \Fun(\cI, \cE) \]
	respects Stokes functors and in particular it induces a well defined functor
	\[ u_! \colon \St_{\cJ, \cE} \to \St_{\cI, \cE} \ . \]
	\personal{More generally, the proof below shows that if $F \colon \cJ \to \cE$ is punctually split at every $y \in \cY_x$, then $u_!(F)$ is punctually split at $x$.}
\end{lem}

\begin{proof}
	We know from \cref{cor:finite_etale_fibration_induction_and_cocartesian} that $u_!$ preserves cocartesian functors.
	It is therefore enough to prove that it preserves punctually split functors as well.
	Let therefore $F \colon \cJ \to \cE$ be a punctually split functor.
	Fix $x \in \cX$.
	For every $y \in \cY_x$, we have a splitting
	\[ V_y \colon \cJ_y^{\ens} \to \cE \]
	for $j_y^\ast(F)$.
	Since $f$ is a finite étale fibration, the same goes for $u$ (see \cref{lem:finite_etale_fibrations_pullback}), and therefore
	\[ j_x^\ast( u_!(F) ) \simeq \bigoplus_{y \in \cY_x} j_y^\ast(F) \ . \]
	It follows from \cref{lem:Stokes_functors_coproducts} that $\bigoplus_{y \in \cY_x} V_y$ provides a splitting for $j_x^\ast( u_!(F) )$, whence the conclusion.
\end{proof}

\begin{prop} \label{prop:finite_etale_fibration_Stokes_monadicity}
	In the above setting, the functor
	\[ u_! \colon \St_{\cJ, \cE} \to \St_{\cI, \cE} \]
	is monadic.
\end{prop}

\begin{proof}
	As in the proof of \cref{cor:finite_etale_fibration_cocartesian_monadicity}, Lemmas \ref{lem:finite_etale_fibration_induction_Stokes} and \ref{lem:finite_etale_fibration_biadjoint} imply that $u_!$ and $u^\ast$ are biadjoint.
	Besides, $u_! \colon \Fun(\cJ, \cE) \to \Fun(\cI, \cE)$ is conservative thanks to \cref{lem:finite_etale_fibration_universally_conservative}, so the same holds true for its restriction to the $\infty$-categories of Stokes functors.
\end{proof}

\begin{cor}\label{cor:finite_etale_fibration_Stokes_relative_tensor}
	Let
	\[ \begin{tikzcd}
		\cJ \arrow{r}{u} \arrow{d}{q} & \cI \arrow{d}{p} \\
		\cY \arrow{r}{f} & \cX
	\end{tikzcd} \]
	be a pullback square in $\Cat_\infty$, where $p$ is a cocartesian fibration in posets.
	Let $\cE^\otimes$ be a presentably symmetric monoidal $\infty$-category.
	Assume that:
	\begin{enumerate}\itemsep=0.2cm
		\item $f$ is a finite étale fibration;
		\item $\cE$ is stable;
		\item both $\cI$ and $\cJ$ are $\cE$-bireflexive.
	\end{enumerate}
	Then, the comparison functor
	\[ \mu \colon \Loc(\cY;\cE) \otimes_{\Loc(\cX;\cE)} \St_{\cI,\cE} \to \St_{\cJ, \cE} \]
	is an equivalence.
\end{cor}

\begin{proof}
	Using \cref{prop:finite_etale_fibration_Stokes_monadicity} as input, the same proof of \cref{cor:relative_tensor_product} applies.
\end{proof}

We conclude this section with the following result, which has been inspired by \cite[Lemma 15.5]{Bachmann_Hoyois_Norms} and that will play an important role later on:

\begin{cor}[Retraction lemma]\label{cor:retraction_lemma}
	Let
	\[ \cX_\bullet \colon \mathbf \Delta_{s}\op \to \Cat_\infty \]
	be a semi-simplicial diagram with colimit $\cX$.
	Let $\cI \to \cX$ be a cocartesian fibration in posets and set
	\[ \cI_\bullet \coloneqq \cX_{\bullet} \times_\cX \cI \ . \]
	Let $\cE^\otimes$ be a presentably symmetric monoidal stable $\infty$-category.
	Assume that:
	\begin{enumerate}\itemsep=0.2cm
		\item $\Env(\cX)$ is compact in $\Spc$;
		\item for every $[n] \in \mathbf \Delta_s$, the structural morphism $\cX_n \to \cX$ is a finite étale fibration;
		\item For every $[n] \in \mathbf \Delta_s$, $\cI_n$  is $\cE$-bireflexive;
	\end{enumerate}
	Then there exists an integer $m \geqslant 0$ depending only on $\Env(\cX)$ such that $\St_{\cI,\cE}$ is a retract of
	\[ \lim_{[n] \in \mathbf \Delta_{s,\leq m}} \St_{\cI_n,\cE} \]
	in $\PrL$.
\end{cor}

\begin{proof}
	To begin with, \cref{cor:descending_geometricity_of_Stokes_via_colimits} implies that $\St_{\cI,\cE}$ is presentable, stable and closed under limits and colimits in $\Fun(\cI,\cE)$ and that besides
	\[ \St_{\cI,\cE} \simeq \lim_{i \in I} \St_{\cI_i,\cE} \ , \]
	the limit being computed in $\PrLR$.
	
	\medskip
	
	For any integer $m \geqslant 0$, set
	\[ \cX^{(m)} \coloneqq \colim_{[n] \in \mathbf \Delta_{s,\leq m}\op} \cX_\bullet \ . \]
	It automatically follows that
	\[ \cX \simeq \colim_{m \in \mathbb N\op} \cX^{(m)} \ , \]
	where the colimit is now filtered.
	Since the enveloping $\infty$-groupoid functor $\Env \colon \Cat_\infty \to \Spc$ is a left adjoint, we see that
	\[ \Env(\cX) \simeq \colim_{m \in \mathbb N} \Env(\cX^{(m)}) \ . \]
	Since $\Env(X)$ is compact, there exists an integer $m \geqslant 0$ such that $\Env(\cX)$ is a retract of $\Env(\cX^{(m)})$.
	As a consequence, we see that $\Loc(\cX;\cE)$ is a retract of $\Loc(\cX^{(m)};\cE)$.
	In particular,
	\[ \St_{\cI,\cE} \simeq \Loc(\cX;\cE) \otimes_{\Loc(\cX;\cE)} \St_{\cI,\cE} \]
	is a retract of
	\[ \Loc(\cX^{(m)};\cE) \otimes_{\Loc(\cX;\cE)} \St_{\cI,\cE} \simeq \Big( \lim_{[n] \in \mathbf \Delta_{s,\leq m}} \Loc(\cX_n;\cE) \Big) \otimes_{\Loc(\cX;\cE)} \St_{\cI;\cE} \ . \]
	Notice that the diagram $\Loc(\cX_\bullet;\cE)$ takes values in $\PrLR$.
	Therefore, \cref{Peter_lemma} supplies a canonical equivalence
	\[ \Big( \lim_{[n] \in \mathbf \Delta_{s,\leq m}} \Loc(\cX_n;\cE) \Big) \otimes_{\Loc(\cX;\cE)} \St_{\cI;\cE} \simeq \lim_{[n] \in \mathbf \Delta_{s,\leq m}} \Loc(\cX_n;\cE) \otimes_{\Loc(\cX;\cE)} \St_{\cI;\cE} \ . \]
	Since each $\cX_n \to \cX$ is a finite étale fibration, \cref{cor:finite_etale_fibration_Stokes_relative_tensor} supplies a canonical equivalence
	\[ \Loc(\cX_n;\cE) \otimes_{\Loc(\cX;\cE)} \St_{\cI;\cE} \simeq \St_{\cI_n;\cE} \ . \]
	Thus, the conclusion follows.
\end{proof}

\section{Graduation}\label{sec:graduation}

In this section we keep working with cocartesian fibrations in posets, but we restrict to stable coefficients.
This allows to introduce a new fundamental operation for Stokes functors: \emph{graduation}.
Intuitively, this allows to break a Stokes functor in more elementary pieces, and it will be the key ingredient needed to develop the theory of level morphisms and level induction.

\subsection{Relative graduation}\label{subsec:graduation}

Let $\cX$ be an $\infty$-category.
Starting with a morphism $p \colon \cI \to \cJ$ in $\PosFib_\cX$, we can perform the following two constructions:

\begin{construction}\label{construction:I_p}
	Consider the fiber product
	\[ \begin{tikzcd}
		\cI_p \arrow{r} \arrow{d}{\pi} & \cI \arrow{d}{p} \\
		\cJ^{\ens} \arrow{r}{i_\cJ} & \cJ
	\end{tikzcd} \]
	Notice that $\cI_p^{\ens} \to \cI^{\ens}$ is an equivalence.
	When $\cX$ is reduced to a point, we can identify $\cI_p$ with the poset $(\cI, \le_p)$, where
	\[ a \leq_p a' \stackrel{\text{def.}}{\iff} p(a) = p(a') \text{ and } a \leq a' \ . \]
	In other words, if $p(a) \ne p(a')$, then $a$ and $a'$ are incomparable with respect to $\le_p$.
\end{construction}

\begin{construction}
	Let
	\[ \Upsilon = \Upsilon_\cJ \coloneqq \mathrm{St}_\cX^{\mathrm{co}}(\cJ) \colon \cX \to  \Cat_\infty \]
	be the straightening of $p_{\cJ} \colon \cJ\to \cX$.
	Consider the composition
	\[ \Upsilon^{\Delta^1} \coloneqq \Fun(\Delta^1, \Upsilon(-)) \colon \cX \to  \Cat_\infty , \]
	and write $\cJ^{\Delta^1} \coloneqq \mathrm{Un}_\cX^{\mathrm{co}}(\Upsilon^{\Delta^1})$ for the associated cocartesian fibration. 
	The source  and identity functors 
	\[ \Fun(\Delta^1, \Upsilon(-))  \to  \Upsilon(-) \qquad \text{and} \qquad \Upsilon(-)  \to\Fun(\Delta^1, \Upsilon(-)) \]
	induce morphisms of cocartesian fibration in posets over $\cX$
	\[ s \colon \cJ^{\Delta^1} \to \cJ \qquad \text{and} \qquad \id \colon  \cJ \to  \cJ^{\Delta^1} \]
	Consider the pullback diagram
	\[ \begin{tikzcd}
		\cI_{\leq} \arrow{d} \arrow{r}{\sigma} &\cI \arrow{d}{p} \\
		\cJ^{\Delta^1}\arrow{r}{s} &  \cJ   \ .
	\end{tikzcd} \]
	Objects of $\cI_{\leq}$ are triples $(x,a,b)$ where $a\in \cI_x$, $b\in \cJ_x$ and where $p(a)\leq b$ in $\cJ_x$.
	We also consider the full subcategory $i_{\cI,<} \colon \cI_< \hookrightarrow \cI_\leq$ spanned by objects $(x,a,b)$ with $p(a) < b$.
	When $\cI$ is clear from the context, we simply write $i_<$ instead of $i_{\cI,<}$.
\end{construction}

\begin{rem}\label{description_I_smaller_over_point}
	The  target functor $\Fun(\Delta^1, \Upsilon(-))  \to \Upsilon(-) $ induces a  morphism  of cocartesian fibration in posets $t \colon \cJ^{\Delta^1} \to \cJ$.
	Let $\tau \colon \cI_{\leq} \to \cJ$ be the composition of $\cI_{\leq} \to \cJ^{\Delta^1}$ with $t \colon \cJ^{\Delta^1} \to \cJ$.
	Then, one checks that if $\cX$ is a point, the induced functor $(\sigma, \tau) \colon \cI_{\leq } \to \cI \times \cJ$ is fully-faithful.
\end{rem}

In general, $\cI_<$ is no longer a cocartesian fibration.
To remedy this, we introduce the following:

\begin{defin}\label{J_loc_constant}
Let $\cX$ be an $\infty$-category. 
Let $p \colon \cI\to \cJ$ be a  morphism  in $\PosFib$ over $\cX$.
We say that $p \colon \cI\to \cJ$ is a \textit{graduation morphism} if the cocartesian fibration $ \cJ^{\ens}\to \cX$ is locally constant in the sense of \cref{def:locally_constant}.
\end{defin}

The following lemma is simply a matter of unraveling the definitions:

\begin{lem}
Let $\cX$ be an $\infty$-category. 
Let $p \colon \cI\to \cJ$ be a graduation morphism  over $\cX$.
Then, $i_< \colon \cI_< \to \cI_\leq$ is a cocartesian subfibration of $\cI_\leq$ over $\cX$.
\end{lem}

Consider the following diagram with pull-back squares:
\begin{equation}\label{eq:graduation_diagram}
	\begin{tikzcd}
		& & \cI_{<} \arrow{d}{i_{<}} & \\
		\cI_p \arrow{d} \arrow{rr}{i_p} & & \cI_{\leq} \arrow{d} \arrow{r}{{\sigma}} & \cI  \arrow{d}{p} \\
		\cJ^{\ens}\arrow{r}{i_{\cJ}} & \cJ \arrow{r}{\id} & \cJ^{\Delta^1} \arrow{r}{s} & \cJ   \ .
	\end{tikzcd} 
\end{equation}
We fix a presentable stable $\infty$-category $\cE$ and write
\[ \varepsilon_< \colon i_{<!} i_<^\ast \to  \id_{\Fun(\cI_{\leq }, \cE)} \]
for the counit of the adjunction $i_{<!} \colon \Fun(\cI_<, \cE) \leftrightarrows \Fun(\cI_\leq, \cE) \colon i_<^\ast$.

\begin{defin}\label{defin_Gr}
	The \emph{graduation functor relative to $p \colon \cI \to \cJ$} (or \emph{$p$-graduation functor})
	\[ \Gr_p \colon \Fun(\cI, \cE) \to  \Fun(\cI_{p}, \cE) \]
	is the cofiber
	\[ \Gr_p \coloneqq \cofib\big( i_p^\ast \varepsilon_< \sigma^\ast \colon i_p^\ast \circ i_{<!} \circ i_<^\ast \circ \sigma^\ast \to i_p^{\ast} \circ \sigma^\ast \big) \ . \]
\end{defin}

\begin{notation}
	When $p = \id$, we note $\Gr$ for $\Gr_{\id}$. 
\end{notation}

In the following basic example, we recall that if  $p \colon \cI \to \cJ$ is  a morphism of posets and if  $b\in \cJ$, we put $\cI_{\leq b}=\cI_{/b} \coloneqq \cI \times_{\cJ} \cJ_{/b}$.
Since $\cJ$ is a poset, the canonical morphism $\cJ_{/b}\to \cJ$ is fully-faithful.
Thus, the canonical morphism $\cI_{/b}\to \cI$ identifies $\cI_{/b}$ with the full subcategory of $\cI$ spanned by objects $a\in \cI$ such that $p(a)\leq b$.
Similarly, $\cI_{< b} \coloneqq \cI \times_{\cJ} \cJ_{<b}$ is the full subcategory of $\cI$ spanned by objects $a \in \cI$ such that $p(a) < b$.

\begin{eg}\label{ex_Gr}
	Let $p \colon \cI \to \cJ$ be a morphism of posets.
	Let $\cE$ be a presentable stable $\infty$-category.
	Let $V \colon \cI^{\ens}\to \cE$ be a functor and put $F \coloneqq i_{\cI !}(V)$.
	Let $a \in \cI_p$.
	Then, there is a canonical equivalence
	\[ (\Gr_p(F))_a \simeq \bigoplus_{\substack{a'\leq a \\ p(a')=p(a)}} V_{a'}  \ . \]
\end{eg}

\begin{proof}
	We have
	\[ (i_p^{\ast }\circ i_{<!} \circ i^*_< \circ \sigma^*(F))_a \simeq \colim_{(a',b)\in (\cI_{<})_{/(a,p(a))}} F_{a'} \simeq  \colim_{(a',a'',b)\in \cC}  V_{a''} \]
	where $\cC$ is the full subcategory of $\cI \times \cJ \times \cI^{\ens}$ spanned by triples $(a',b,a'')$ such that $a''\leq a' \leq a$ and $p(a')<b\leq p(a)$.
	Let $\cD$  be the full subcategory of $\cI \times \cI^{\ens}$ spanned by pairs $(a',a'')$ such that $a''\leq a' \leq a$ and $p(a')< p(a)$.
	Let $\cA$  be the subset of $ \cI^{\ens}$ formed by the  $a''$ such that $a'' \leq a$ and $p(a'')< p(a)$.
	Consider the commutative diagram
	\[ \begin{tikzcd}
		  \cD   \arrow{r}{f}  \arrow{d}{p_2}   & \cC \arrow{d}{p_3}     \\
		   \cA   \arrow{r}         &          \cI^{\ens} 
	\end{tikzcd} \] 
	where $f \colon \cD \to \cC$ is given by $(a',a'')\mapsto (a',p(a),a'')$.
	We claim that $f$ is cofinal.
	Indeed, for every triple $(a',b,a'')\in \cC $, $\cD_{(a',b,a'')/} \coloneqq \cD \times_\cC \cC_{(a',b,a'')/}$ is the subposet of $\cD$ spanned by pairs $(\alpha', a'')$ with  $a'\leq \alpha'$.
	Observe that $(a',a'')$ is a minimal element of $\cD_{(a',b,a'')/}$, which is thus weakly contractible.
	Hence, the claimed cofinality follows from \cite[Theorem 4.1.3.1]{HTT}.
	Thus, 
	\[ (i_p^{\ast }\circ i_{<!} \circ i^*_< \circ \sigma^*(F))_a   \simeq  \colim_{\cD}  V|_{\cD} \]
	We also claim that $p_2 \colon \cD \to \cA$ is cofinal.
	Indeed, if $a''\in \cA$, $\cD_{a''/} \coloneqq \cD \times_\cA \cA_{a''/}$ is the subposet of $\cD$ spanned by couples $(\alpha', a'')$.
	Observe that $(a'',a'')$ is a minimal element of $\cD_{a''/}$, which is thus weakly contractible.
	Hence, the claimed cofinality follows from \cite[Theorem 4.1.3.1]{HTT}.
	Thus,
	\[ (i_p^{\ast }\circ i_{<!} \circ i^*_< \circ \sigma^*(F))_a   \simeq  \colim_{\cA}  V|_{\cA}\simeq \bigoplus_{\substack{a'\leq a \\ p(a')< p(a)}} V_{a'} \]
	On the other hand, 
	\[ (i^*_p \circ \sigma^*(F))_a \simeq F_a \simeq \bigoplus_{\substack{a'\leq a}} V_{a'} \]
 \cref{ex_Gr} thus follows.
\end{proof}

In particular when $p \colon \cI \to \cJ$ is the identity of $\cI$, we obtain:

\begin{eg}\label{Gr_p_identity}
	In the setting of \cref{ex_Gr}, let $V \colon \cI^{\ens}\to \cE$ be a functor and put $F \coloneqq i_{\cI !}(V)$.
	Let $a \in \cI^{\ens}$.
	Then, there is a canonical equivalence 
	\[ (\Gr F)_a \simeq V_{a} \]
\end{eg}

\begin{eg}\label{Gr_morphism_to_point}
Let $\cX$ be an $\infty$-category. 
Let $p \colon \cI\to \cJ$ be a graduation morphism over $\cX$ and assume that $\cJ=\cJ^{\ens}$.
Then, $\cI_p=\cI$  and $\Gr_p  \colon  \Fun(\cI_p,\cE)\to \Fun(\cI,\cE)$  is the identity functor.
\end{eg}

\begin{prop}\label{cor:section_Gr_commutes_with_colimits}
	Let $\cX$ be an $\infty$-category. 
	Let $p \colon \cI\to \cJ$ be a graduation morphism over $\cX$.
	Let $\cE$ be a presentable stable $\infty$-category.
	Then $\Gr_p \colon \Fun(\cI,\cE)\to \Fun(\cI_p,\cE)$ commutes with colimits.
	In particular, $\Gr_p$ admits a right adjoint 
	 \[ \Gr_p^* \colon \Fun(\cI_p,\cE)\to \Fun(\cI,\cE) \]
 	that can be explicitly computed as
	 \[ \Gr_p^\ast \simeq \fib\big( \sigma_\ast \circ i_{p,\ast} \stackrel{\eta_<}{\to } \sigma_\ast \circ i_{<,\ast} \circ i_<^\ast \circ i_{p,\ast} \big) \ , \]
	 where $\eta_<$ is the unit of the adjunction $i_{<}^\ast \dashv i_{<,\ast}$.
\end{prop}
\begin{proof}
	The first half follows immediately from the fact that $\Gr_p$ is a composition of functors commuting with colimits.
	The second half simply follows from the Yoneda lemma.
\end{proof}

\begin{rem}
For an explicit description of the right adjoint $\Gr_p^*$ in the spirit of \cref{ex_Gr}, see \cref{punctual_right_adjoint}.
\end{rem}

Under extra finitness conditions, \cref{cor:section_Gr_commutes_with_colimits} has  a counterpart for limits.
Before stating it, we introduce the following 

\begin{defin}
	We define $\PosFib^f$ as the full subcategory of $\PosFib$ spanned by cocartesian fibrations in posets $p : \cI \to \cX$  such that for every $x\in \cX$, the  poset $\cI_x$ is finite. 
\end{defin}

\begin{prop}\label{cor:section_Gr_commutes_with_limits}
	Let $\cX$ be an $\infty$-category. 
	Let $p \colon \cI\to \cJ$ be a graduation morphism in $\PosFib^f$ over $\cX$.
	Let $\cE$ be a presentable stable $\infty$-category.
	Then $\Gr_p \colon \Fun(\cI,\cE)\to \Fun(\cI_p,\cE)$ commutes with limits.
\end{prop}

\begin{proof}
	Follows immediately from \cref{induction_limit_stable} and the fact that in a stable $\infty$-category, the cofiber functor commutes with limits in virtue of \cref{finite_colimit_and_limit_stable}.
\end{proof}

\begin{prop}\label{cor:section_Gr_conservative}
	Let $\cX$ be an $\infty$-category. 
	Let $p \colon \cI\to \cJ$ be a graduation morphism in $\PosFib^f$ over $\cX$.
	Let $\cE$ be a presentable stable $\infty$-category.
	Then $\Gr_p \colon \Fun(\cI,\cE)\to \Fun(\cI_p,\cE)$ is conservative.
\end{prop}

\begin{proof}
	From \cref{cor:induction_specialization_Beck_Chevalley}, we can suppose that $\cX$ is a point.
	Let $f  \colon F \to G$ be a morphism in $\Fun(\cI,\cE)$ such that $\Gr_p(f) \colon \Gr_p(F)\to \Gr_p(G)$ is an isomorphism.
	Let $A\subset \cI$ be the subset of elements $a$ such that $f$ is not an isomorphism at $a$.
	We argue by contradiction and assume that $A$ is not empty.
	Since $\cI$ is finite, $A$ admits a minimal element $a$.
	If $\cC_a \coloneqq \cI_{<}\times_{\cI_{\leq }} (\cI_{\leq })_{/(a,p(a))}$, there is a morphism of cofibre sequences
	\[ \begin{tikzcd}
			\colim_{\cC_a} F|_{\cC_a} \arrow{r} \arrow{d} & F_a \arrow{r} \arrow{d}{f_a} & \Gr_{p}(F)_a \arrow{d}{\Gr(f)_a} \\
			\colim_{\cC_a} G|_{\cC_a} \arrow{r} & G_a \arrow{r} & \Gr_p(G)_a
	\end{tikzcd} \]
	By definition, an object of $\cC_a$ is a couple $(b,c)\in \cI\times \cJ$ with $p(b)<c$ such that $b\leq a$ and $c\leq p(a)$.
	In particular $p(b)<p(a)$, so that $b<a$.
	That is, the above colimit over $\cC_a$ only features values of $f$ at elements $b\in \cI$ strictly smaller than $a$.
	Thus, the left vertical arrow is an equivalence by the minimality of $a$.
	The right vertical arrow is an equivalence by assumption.
	Hence $f_a  \colon F_a  \to G_a$ is an equivalence. 
	Contradiction.
\end{proof}

\begin{rem}
	Note that \cref{cor:section_Gr_conservative} fail if the finiteness assumption on $\cI$ is dropped. 
	If $\cI=\mathbb{Z}$ and if $F$ is the functor constant to a non zero object in $\cE$, then we have $F\neq 0$ and $\Gr(F)\simeq 0$. 
\end{rem}

\subsection{Exponential graduation}

The graduation functor introduced in \cref{subsec:graduation} should be understood as the global counterpart of the exponential graduation, which we now discuss.
Fix a presentable stable $\infty$-category $\cE$.
For every $\infty$-category $\cX$ and every graduation morphism $p \colon \cI \to \cJ$ in $\PosFib$ over $\cX$, we can apply $\exp_\cE(-/\cX)$ to the diagram \eqref{eq:graduation_diagram}.
This yields the following commutative diagram
\[ \begin{tikzcd}
	& & \exp_\cE( \cI_< / \cX ) \arrow{d}{\cE^{i_<}_!} \\
	\exp_\cE( \cI_p / \cX ) \arrow{rr}{\cE^{i_p}_!} \arrow{d} & & \exp_\cE( \cI_\leq / \cX ) \arrow{r}{\cE^{\sigma}_!} \arrow{d} & \exp_\cE( \cI / \cX ) \arrow{d} \\
	\exp_\cE( \cJ^{\ens} / \cX ) \arrow{r}{\cE^{i_\cJ}_!} & \exp_\cE( \cJ / \cX ) \arrow{r}{\cE^{\id}_!} & \exp_\cE( \cJ^{\Delta^1} / \cX ) \arrow{r}{\cE^s_!} & \exp_\cE( \cJ / \cX )
\end{tikzcd} \]
in $\PrFibL_\cX$.
Recall from \cref{lem:exponential_induction_adjoint} the existence of right adjoints $\cE^{\sigma,\ast}$, $\cE^{i_<,\ast}$ and $\cE^{i_p,\ast}$ for $\cE^\sigma_!$, $\cE^{i_<}_!$ and $\cE^{i_p}_!$, respectively.

\begin{defin}\label{def:exponential_graduation}
	In the above setup, the \emph{exponential graduation relative to $p$} is the functor
	\[ \expGr_p \coloneqq \cofib( \cE^{i_p,\ast} \circ \cE^{i_<}_! \circ \cE^{i_<,\ast} \circ \cE^{\sigma, \ast} \to \cE^{i_p,\ast} \circ \cE^{\sigma, \ast} ) \ , \]
	where the morphism is induced by the counit of the adjunction $\cE^{i_<}_! \dashv \cE^{i_<,\ast}$.
\end{defin}

The following result summarizes the local and the global behavior of the exponential graduation functor:

\begin{prop}\label{prop:exponential_graduation_local_global_behavior}
	Keep the same notations as above.
	Then:
	\begin{enumerate}\itemsep=0.2cm
		\item for every $x \in \cX$, the diagram
		\[ \begin{tikzcd}
			\Fun(\cI_x, \cE) \arrow{d} \arrow{r}{\Gr_{p_x}} & \Fun((\cI_p)_x, \cE) \arrow{d} \\
			\exp_\cE( \cI / \cX ) \arrow{r}{\expGr_p} & \Fun(\cI_p, \cE)
		\end{tikzcd} \]
		commutes.
		
		\item The diagram
		\[ \begin{tikzcd}[column sep=50pt]
			\Fun_{/\cX}(\cX, \exp_\cE( \cI / \cX )) \arrow{r}{\Sigma_\cX( \expGr_p )} \arrow{d}{\spe_\cI} & \Fun_{/\cX}( \cX, \exp_\cE(\cI_p / \cX) ) \arrow{d}{\spe_{\cI_p}} \\
			\Fun(\cI, \cE) \arrow{r}{\Gr_p} & \Fun(\cI_p, \cE)
		\end{tikzcd} \]
		commutes.
	\end{enumerate}
\end{prop}

\begin{proof}
	Statement (1) immediately follows from \cref{cor:induction_specialization_Beck_Chevalley} applied to $\cE^{i_<}_!$ and the fact that the adjunctions $\cE^{\sigma}_! \dashv \cE^{\sigma,\ast}$, $\cE^{i_<}_! \dashv \cE^{i_<,\ast}$ and $\cE^{i_p}_! \dashv \cE^{i_p,\ast}$ are relative to $\cX$, see \cref{lem:exponential_induction_adjoint}.
	On the other hand, statement (2) is a direct consequence of \cref{prop:exponential_vs_global_induction}.
\end{proof}

Our next goal is to understand the behavior of $\expGr$ with the exponential functoriality for morphisms in $\PosFib$.
We start analyzing cartesian morphisms.
Consider therefore a diagram
\begin{equation}
	\begin{tikzcd}
		\cI \arrow{d}{p} & \cI_\cX \arrow{d}{q} \arrow{l}[swap]{u} \\
		\cJ \arrow{d} & \cJ_\cX \arrow{d} \arrow{l}[swap]{u'} \\
		\cY & \cX \arrow{l}[swap]{f}
	\end{tikzcd}
\end{equation}
whose squares are pullbacks and where $\cI \to \cY$ and $\cJ \to \cY$ are cocartesian fibrations in posets.
We also assume that $\cJ^{\ens} \to \cX$ is locally constant.
We have:

\begin{prop}
	In the above setting, the diagram
	\begin{equation}\label{eq:exponential_graduation_functoriality_cartesian}
		\begin{tikzcd}
			\exp_\cE( \cI_p / \cY ) & \exp_\cE( (\cI_\cX)_{p_\cX} / \cX ) \arrow{l}[swap]{\cE^{u_p}} \\
			\exp_\cE( \cI / \cY ) \arrow{u}{\expGr_p} & \exp_\cE( \cI_\cX / \cX ) \arrow{l}[swap]{\cE^u} \arrow{u}[swap]{\expGr_{q}}
		\end{tikzcd}
	\end{equation}
	is canonically commutative, and it is therefore a pullback.
\end{prop}

\begin{proof}
	Unraveling the definitions, we see that an object in $\exp_\cE( \cI_\cX / \cX )$ can be identified with a pair $(F, x)$, where $x \in \cX$ and $F \colon (\cI_\cX)_x \to \cE$ is a functor.
	Under the canonical identification $(\cI_\cX)_x \simeq \cI_{f(x)}$, the functor $\cE^u$ sends $(F,x)$ to $(F,f(x))$.
	At this point, the commutativity follows from \cref{cor:induction_specialization_Beck_Chevalley} and \cref{prop:exponential_graduation_local_global_behavior}-(1), while \cref{prop:functoriality_exponential}-(1) immediately implies that the square in consideration is a pullback.
\end{proof}

\begin{cor}\label{lem:Gr_restriction} 
	In the above setting, the natural transformation
	\[ \Gr_q \circ u^\ast \to  u_p^\ast \circ \Gr_p \]
	between functors from $\Fun(\cI, \cE)$ to $\Fun((\cI_\cX)_q, \cE)$ is an equivalence.
\end{cor}

\begin{proof}
	We can see \eqref{eq:exponential_graduation_functoriality_cartesian} as a commutative square in $\PrFibL$.
	The statement then follows applying $\Sigma$ and invoking \cref{prop:global_functoriality}-(\ref{prop:global_functoriality:pullback}) and \cref{prop:exponential_graduation_local_global_behavior}-(2).
\end{proof}
%
%

\begin{rem} \label{lem:Gr_restriction_points} 
As a particular case of  \cref{lem:Gr_restriction}, we see that relative graduation commutes with restriction over an object of $\cX$.
\end{rem}

\subsection{Graduation and induction}\label{subsection_Gr_ind}

Our next task is to understand how graduation behaves with respect to morphisms in $\PosFib_\cX$ for a fixed $\infty$-category $\cX$.
In other words, we are interested in seeing to which extent (exponential) graduation and (exponential) induction intertwine with each other.
Our starting point is the following.
Let $\cX$ be an $\infty$-category and let
\begin{equation}\label{eq:graduation_and_induction_setup}
	\begin{tikzcd}
		\cI   \arrow{d}{p} \arrow{r}{f}&\cK \arrow{d}{q} \\
		\cJ \arrow{r}{g} &  \cL
	\end{tikzcd}
\end{equation}
be a commutative diagram in  $\PosFib$ over $\cX$.
We make the following running

\begin{assumption}\label{g_ff}
	\hfill
	\begin{enumerate}\itemsep=0.2cm
		\item Both $p : \cI \to \cJ$ and $q  : \cK \to \cL$ are graduation morphisms.
		
		\item For every $x \in \cX$, the map $g^{\ens}_x \colon \cJ_x^{\ens} \to \cL_x^{\ens}$ is injective.
	\end{enumerate}
\end{assumption}

The second half of this assumption guarantees that if $a,b \in \cJ_x$ are such that $a < b$, then $g(a) < g(b)$ as well.
\personal{Indeed, since $g_x \colon \cJ_x \to \cL_x$ is a map of posets, $g(a) \leqslant g(b)$. But one cannot have $g(a) = g(b)$ because $g^{\ens}$ is injective. So $g(a) < g(b)$.}
Thus, the above assumption guarantees the existence and commutativity of the following diagram:
\begin{equation}\label{eq:graduation_and_induction_comparison_diagram}
	\begin{tikzcd}
		& & & \cI_{<} \arrow{dl}[swap]{f_<} \arrow{dd}{i_{\cI,<}} & & \\
		& & \cK_{<} & & & \\
		& \cI_p \arrow{rr}[pos=0.3]{i_p} \arrow{dl}[swap]{f_{p,q}}  \arrow{dd} & & \cI_{\leq} \arrow{rr}{\sigma_\cI} \arrow{dl}[swap]{f_\leq} \arrow{dd} & & \cI \arrow{dd}{p} 	\arrow{dl}[swap]{f} \\
		\cK_q \arrow{dd} \arrow[crossing over]{rr}[pos=0.7]{i_q} & & \cK_{\leq} \arrow[leftarrow, crossing over]{uu}[swap,pos=0.7]{i_{\cK,<}} \arrow[crossing 	over]{rr}[pos=0.7]{\sigma_{\cK}} & & \cK \\
		& \cJ^{\ens} \arrow{dl} \arrow{rr} & & \cJ^{\Delta^1} \arrow{rr} \arrow{dl} & & \cJ \arrow{dl}{g} \\
		\cL^{\ens} \arrow{rr} & & \cL^{\Delta^1}\arrow[leftarrow, crossing over]{uu} \arrow{rr} & & \cL \arrow[leftarrow, crossing over]{uu}[pos=0.6,swap]{q}
	\end{tikzcd}
\end{equation}
Fix a stable presentable $\infty$-category $\cE$ and consider the induced natural transformation
\begin{equation}\label{eq:graduation_and_induction}
	f_{p,q !} \circ \Gr_p \to  \Gr_{q} \circ f_{!}
\end{equation}
of functors from $\Fun(\cI, \cE)$ to $\Fun(\cK_q, \cE)$.
The goal of this section is to establish the following:

\begin{prop}\label{prop:Graduation_induction}
	Let $\cE$ be a stable presentable $\infty$-category and let $F \in \Fun(\cI, \cE)$.
	Under \cref{g_ff}, the natural transformation \eqref{eq:graduation_and_induction} is an equivalence and the diagram
	\[ \begin{tikzcd}
		\exp_\cE( \cI / \cX ) \arrow{r}{\cE^f_!} \arrow{d}{\expGr_p} & \exp_\cE( \cK / \cX ) \arrow{d}{\expGr_q} \\
		\exp_\cE( \cI_p / \cX ) \arrow{r}{\cE^{f_{p,q}}_!} & \exp_\cE( \cK_q / \cX )
	\end{tikzcd} \]
	commutes.
\end{prop}

%
%
%
%
%
%
%

We first deal with the natural transformation \eqref{eq:graduation_and_induction}, and we start by the following particular case:

\begin{lem}\label{cor:compatibility_Gr_i_I!}
	Let $\cX$ be an $\infty$-category. 
	Let $\cE$ be a presentable stable $\infty$-category.
	Let $p \colon \cI \to \cJ$ be a graduation morphism in  $\PosFib$ over $\cX$.
	Consider the commutative diagram 
	\[ \begin{tikzcd}
		\cI^{\ens} \arrow{d}{p \circ i_{\cI}}\arrow{r}{i_{\cI}} & \cI \arrow{d}{p}\\
		\cJ   \arrow{r}{\id} & \cJ
	\end{tikzcd} \]
	Then, the induced natural transformation $i_{\cI_p !} \to\Gr_{p} \circ i_{\cI !}$ is an equivalence.
\end{lem}

\begin{proof}
	Combining \cref{cor:induction_specialization_Beck_Chevalley} and \cref{lem:Gr_restriction}, we can suppose that $\cX$ is a point.
	Let $V \colon \cI^{\ens} \to \cE$ be a functor.
	Then, for every $a\in \cI_p$, we have
	\[ (i_{\cI_p !} (V))_a \simeq  \bigoplus_{\substack{a'\leq a \\ p(a')=p(a)}} V_{a'} \]
	Then, \cref{cor:compatibility_Gr_i_I!} follows from the computation performed in \cref{ex_Gr}.
\end{proof}

\begin{cor}\label{Gr_preserves_PS}
	Under the assumptions of \cref{cor:compatibility_Gr_i_I!}, for every punctually split functor $F \colon \cI \to \cE$, the graduation $\Gr_p(F) \colon \cI_p \to \cE$  is punctually split.
\end{cor}

\begin{proof}
	From \cref{lem:Gr_restriction}, we are left to treat the case where $\cX$ is a point.
	In this case, the statement follows from \cref{cor:compatibility_Gr_i_I!}.
\end{proof}

\begin{cor}\label{section_Gr_gives_splitting}
	Let $\cI\to \cX$ be an object of  $\PosFib^f$ such that $\cI^{\ens}\to \cX$ is locally constant.
	Let $\cE$ be a presentable stable $\infty$-category.
	Let $F \colon  \cI \to \cE$ be a functor.
Then, the following are equivalent: 
	\begin{enumerate}\itemsep=0.2cm
	\item the canonical morphism $i_{\cI}^*(F)\to \Gr(F)$ admits a section $\sigma \colon \Gr(F) \to  i_{\cI}^*(F)$;

	\item the functor $F$ split.
	\end{enumerate}
	If these conditions are satisfied, the morphism 
	\[
	\tau \colon i_{\cI !}\Gr(F)\to F
	\] 
	induced by $\sigma$ is an equivalence.
\end{cor}
\begin{proof}
      Assume that (1) holds.
      To prove (2), it is enough to show that $\tau$ is an equivalence.
	By \cref{cor:section_Gr_conservative}, it is enough to show that 
	\[
	\Gr(\tau) \colon \Gr i_{\cI !}\Gr(F)\to  \Gr(F)
	\]	
is an equivalence.
	Then, \cref{section_Gr_gives_splitting} follows from \cref{cor:compatibility_Gr_i_I!}.
	Assume that $F$ split and let us write $F\simeq i_{\cI,!}(V)$ where $V \colon \cI \to \cE$ is a functor.
	By \cref{cor:compatibility_Gr_i_I!}, the canonical morphism from (1) reads as $i_{\cI}^* i_{\cI,!}(V)\to V$.
	Then, the unit transformation  $V \to i_{\cI}^* i_{\cI,!}(V)$ does the job.
\end{proof}

We are now ready for:

\begin{proof}[Proof of \cref{prop:Graduation_induction}]
	Combining \cref{cor:induction_specialization_Beck_Chevalley} and \cref{lem:Gr_restriction}, we can assume that $\cX$ is a point.
	Recall moreover from \cref{prop:generation_by_split_functors} that $\Fun(\cI,\cE)$ is generated under colimits by punctually split functors.
	Since both source and target of \eqref{eq:graduation_and_induction} commute with colimits, it is enough to check that the canonical morphism
	\[ f_{p,q!}( \Gr_p(F) ) \to  \Gr_q( f_!(F) ) \]
	is an equivalence when $F$ is punctually split.
	We can therefore assume that $F\simeq i_{\cI !}(V)$ for some functor $V \colon \cI^{\ens} \to \cE$.
	Thus, we can compute:
	\begin{align*}
		f_{p,q !} ( \Gr_p(F)) & \simeq f_{p,q !} ( \Gr_p(i_{\cI !}(V)) ) \\
		& \simeq  f_{p,q !}\circ i_{\cI_p !}(V) & \text{By \cref{cor:compatibility_Gr_i_I!}} \\
		& \simeq  i_{\cK_q !}( f^{\ens}_{!}(V) ) \\
		& \simeq \Gr_q ( i_{\cK !} ( f^{\ens}_{!}(V) ) ) & \text{By \cref{cor:compatibility_Gr_i_I!}} \\
		& \simeq \Gr_p\circ f_! \circ i_{\cK ! }(V) \\
		& \simeq \Gr_p\circ f_! (F) \ .
	\end{align*}
	Thus, \eqref{eq:graduation_and_induction} is an equivalence.
	As for the second half of the statement, observe that applying $\exp_\cE(-/\cX)$ to the diagram \eqref{eq:graduation_and_induction} supplies a canonical natural transformation
	\[ \alpha \colon \cE^{f_{p,q}}_! \circ \expGr_p \to  \expGr_q \circ \cE^f_! \ . \]
	To prove that it is an equivalence, it is enough to prove that its restriction $\alpha_x$ is an equivalence for every $x \in \cX$.
	Combining \cref{prop:global_functoriality}-(\ref{prop:global_functoriality:induction}) and \cref{prop:exponential_graduation_local_global_behavior}-(1), we see that $\alpha_x$ coincides with the natural transformation \eqref{eq:graduation_and_induction}, so the conclusion follows from what we have already proven.
\end{proof}

We store the following particular cases of \cref{prop:Graduation_induction} for later use.

\begin{cor}\label{cor:compatibility_Gr_p!}
	Let $\cX$ be an $\infty$-category. 
	Let $p \colon \cI \to \cJ$ be a graduation morphism in $\PosFib$ over $\cX$.
	Consider the commutative square 
	\[ \begin{tikzcd}
		\cI \arrow{d}{p} \arrow{r}{p} & \cJ \arrow{d}{\id} \\
		\cJ \arrow{r}{\id} & \cJ 
	\end{tikzcd} \]
	Let $\pi \colon \cI_p \to \cJ^{\ens}$ be the morphism induced by $p$.
	Let $\cE$ be a presentable stable $\infty$-category.
	Then, for every functor $F \colon \cI\to \cE$, the canonical morphism
	\[ \pi_! ( \Gr_p (F) ) \to  \Gr( p_{!} (F) ) \]
	is an equivalence.
\end{cor}

\begin{cor}\label{cor:compatibility_Gr_id}
	Let $\cX$ be an $\infty$-category. 
		Let $i \colon \cI \hookrightarrow  \cJ$ be a fully faithful functor in  $\PosFib$ over $\cX$.
	Consider the commutative square 
	\[ \begin{tikzcd}
		\cI \arrow{d}{\id} \arrow[hook]{r}{i} & \cJ \arrow{d}{\id} \\
		\cI \arrow[hook]{r}{i} & \cJ  \ .
	\end{tikzcd} \]
	Let $\cE$ be a presentable stable $\infty$-category.
	Then, for every functor $F \colon \cI\to \cE$, the canonical morphism
	\[ i^{\ens}_! ( \Gr (F) ) \to  \Gr( i_{!} (F) ) \]
	is an equivalence.
\end{cor}

\begin{prop}\label{punctual_right_adjoint}
Let $p \colon \cI \to \cJ$ be a morphism of posets.
Let $\cE$ be a presentable stable $\infty$-category.
Then, for every $F\in \Fun(\cI_p,\cE)$ and every $a,b\in \cI$ with $a\leq b$, we have canonical equivalences
\[ (\Gr_p^\ast(F))( a \leq b ) \simeq \begin{cases}
	F( a \leq b ) & \text{if } p(a) = p(b) \\
	0 \colon F_a \to F_b & \text{if } p(a) < p(b)
\end{cases} \]
\end{prop}

\begin{proof}
Let $i \colon p^{-1}(p(a))\to \cI$ and $j \colon  p^{-1}(p(a))\to \cI_p$ be the inclusions.
From \cref{prop:Graduation_induction}  and \cref{Gr_morphism_to_point} applied to the commutative square 
	\[ \begin{tikzcd}
			p^{-1}(p(a))\arrow{d}  \arrow{r}{i}  &\cI \arrow{d}{p} \\
			\ast \arrow{r}{p(a)} & \cJ
	\end{tikzcd} \]
there is a canonical equivalence of functors $j_!\simeq \Gr_p \circ i_!$.
Passing to right adjoints gives a canonical equivalence $i^* \circ \Gr_p^\ast \simeq j^\ast$.
This proves the first claim.
Let $a,b\in \cI$ with $a\leq b$ and $p(a)<p(b)$.
We want to show that 
\[
\alpha \coloneqq (\Gr_p^*(F))(a\leq b) \colon (\Gr_p^*(F))_a \to  (\Gr_p^*(F))_b 
\]
 is the zero morphism.
This amounts to show that for every $V\in \cE$, the morphism
 \[
\Map(V,\alpha) \colon \Map(V, \ev_a^{\cI,*}\Gr_p^*(F)) \to  \Map(V, \ev_b^{\cI,*}\Gr_p^*(F))
\]
is the zero morphism.
By adjunction, this amounts to show that
 \[
\Map(\beta ,F) \colon \Map(\Gr_p \circ \ev_{a,!}^{\cI}(V),F) \to  \Map(\Gr_p \circ \ev_{b,!}^{\cI}(V),F) 
\]
is the zero morphism, where 
\[
\beta : \Gr_p \circ \ev_{b,!}^{\cI}(V)\to  \Gr_p \circ \ev_{a,!}^{\cI}(V)
\] 
is the induced morphism in $\Fun(\cI_p,\cE)$.
We are thus left to show that $\beta$ is the zero morphism.
From \cref{prop:Graduation_induction}, $\beta$ identifies with a morphism of the form $\ev_{b,!}^{\cI_p}(V) \to  \ev_{a,!}^{\cI_p}(V)$.
Let $c\in \cI_p$.
Since $p(a)<p(b)$, then either $p(c)\neq p(a)$ or $p(c)\neq p(b)$.
In the first case,  $a$ and $c$ cannot be compared in $\cI_p$, so that $\ev_{a,!}^{\cI_p}(V)$ sends $c$ to $0$.
In the second case,  $b$ and $c$ cannot be compared in $\cI_p$, so that $\ev_{b,!}^{\cI_p}(V)$ sends $c$ to $0$.
Hence, in both cases $\beta$ is zero  when evaluated at $c$.
Thus, $\beta$ is the zero morphism.
\end{proof}

\begin{prop}\label{right_adjoint_Gr_evaluated}
	Let $\cX$ be an $\infty$-category.
	Let $p \colon \cI \to \cJ$ be a graduation morphism  over $\cX$.
	Let $a\in \cI$.
	Let $\cE$ be a presentable stable $\infty$-category.
	Then the triangle
	\[ \begin{tikzcd}[column sep=small]
		 \Fun(\cI_p,\cE) \arrow{rr}{\Gr_p^\ast } \arrow{dr}[swap]{\ev_a^{\cI_p,*}} & &  \Fun(\cI,\cE)  \arrow{dl}{\ev_a^{\cI,*}} \\
		{} &  \cE
	\end{tikzcd} \]
	is canonically commutative.
\end{prop}

\begin{proof}
	Equivalently, it is enough to check that the canonical map
	\[ \ev^{\cI_p}_{a!} \to  \Gr_p \circ \ev^{\cI}_{a!}  \]
	is an equivalence.
	Since $\ev^{\cI_p}_a$ factors through $\cI^{\ens} \to \cI_p$ and $\ev^{\cI}_a$ factors through $\cI^{\ens} \to \cI$, the statement follows directly from \cref{prop:Graduation_induction}, in the form of the special case treated in \cref{cor:compatibility_Gr_i_I!}.
	\personal{Other argument: Let $x$ be the image of $a$ by $\cI \to \cX$.
	Since $\ev_a^{\cI} \colon \{a\}\to \cI$  factors as $\ev_a^{\cI_x} \colon \{a\} \to \cI_x$ followed by $i_x \colon \cI_x \to \cI$, \todo{add ref} reduces the proof of  \cref{right_adjoint_Gr_evaluated} to the case where $\cX$ is a point.
	In that case, we conclude by applying \cref{punctual_right_adjoint} to the identity of $a$.}
\end{proof}

\begin{cor}\label{right_adjoint_Gr_restriction_commutes_colimits}
Let $\cX$ be an $\infty$-category.
Let $p \colon \cI \to \cJ$ be a graduation morphism over $\cX$.
Then, $\Gr_p^\ast \colon  \Fun(\cI_p,\cE)  \to   \Fun(\cI,\cE)$  commutes with colimits.
In particular, $\Gr_p \colon  \Fun(\cI,\cE)  \to   \Fun(\cI_p,\cE)$  preserves compact objects.
\end{cor}

\begin{proof}
Immediate from \cref{right_adjoint_Gr_evaluated} and the fact that the functors $\ev_a^{\cI,*}, a\in \cI$ are jointly conservative and commute with colimits.
\end{proof}


\subsection{Graduation and cocartesian functors}\label{subsec:exponential_graduation_vs_induction}

\begin{prop}\label{prop:exponential_graduation_cocartesian_edges}
	Let $\cX$ be an $\infty$-category. 
	Let $p \colon \cI \to \cJ$ be a graduation morphism in $\PosFib$ over $\cX$.
	Let $\cE$ be a presentable stable $\infty$-category. 
	Then, the functor
	\[ \expGr_p \colon \exp_\cE(\cI / \cX) \to  \exp_\cE(\cI_p / \cX) \]
	preserves cocartesian edges.
	If in addition both $\cI$ and $\cJ$ belong to $\PosFib^f$, then $\expGr_p$ reflects cocartesian edges as well.
\end{prop}

\begin{warning}
	Since the adjoints $\cE^{\sigma,\ast}$, $\cE^{i_<,\ast}$ and $\cE^{i_p,\ast}$ do \emph{not} preserve cocartesian edges, it is a priori not obvious that $\expGr_p$ defines a morphism of cocartesian fibrations over $\cX$.
\end{warning}

\begin{proof}[Proof of \cref{prop:exponential_graduation_cocartesian_edges}]
	Unraveling the definitions, we have to prove the following statement.
	Let $\gamma \colon x \to y$ be any morphism in $\cX$ and fix compatible straightenings $f_\gamma \colon \cI_x \to \cI_y$, $g_\gamma \colon \cJ_x \to \cJ_y$ making the diagram
	\begin{equation}
		\begin{tikzcd}
			\cI_x \arrow{r}{f_\gamma} \arrow{d}{p_x} & \cI_y \arrow{d}{p_y} \\
			\cJ_x \arrow{r}{g_\gamma} & \cJ_y
		\end{tikzcd}
	\end{equation}
	commutative.
	Then we have to prove that for every pair of functors $F_x \colon \cI_x \to \cE$ and $F_y \colon \cI_y \to \cE$ and every map $\alpha \colon (F_x,x) \to (F_y,y)$ in $\exp_\cE(\cI / \cX)$ lying over $\gamma$, if the canonically induced morphism
	\[ \overline{\alpha} \colon f_{\gamma,!}( F_x ) \to  F_y \]
	is an equivalence then the same goes for the map
	\begin{equation}\label{eq:exponential_graudation_cocartesian_edges}
		\overline{\beta} \colon (f_\gamma)_{p_x,p_y!}( \Gr_{p_x}(F) ) \to  \Gr_{p_y}( F_y )
	\end{equation}
	induced by the morphism $\beta \coloneqq \expGr_p(\alpha) \colon (\Gr_{p_x}(F_x), x) \to (\Gr_{p_y}(F_y),y)$ in $\exp_\cE( \cI_p / \cX )$.
	Notice that, since $\cJ^{\ens} \to \cX$ is locally constant, the underlying map $g_\gamma^{\ens} \colon \cJ_x^{\ens} \to \cJ_y^{\ens}$ is a bijection.
	In particular, \cref{g_ff} is satisfied, and we therefore find a natural transformation $(f_\gamma)_{p_x,p_y!} \circ \Gr_{p_x} \to \Gr_{p_y} \circ f_{\gamma !}$ making the diagram
	\[ \begin{tikzcd}
		(f_\gamma)_{p_x,p_y!}( \Gr_{p_x}( F_x ) ) \arrow{r}{\overline{\beta}} \arrow{d} & \Gr_{p_y}( F_y ) \\
		\Gr_{p_y}( f_{\gamma !}( F_x ) ) \arrow{ur}[swap]{\Gr_{p_y}(\overline{\alpha})} & \phantom{\Gr_{p_y}( F_y )} \ ,
	\end{tikzcd} \]
	commutative.
	Now, \cref{prop:Graduation_induction} guarantees that the vertical arrow is an equivalence, so $\overline{\beta}$ is an equivalence if and only if $\Gr_{p_y}(\overline{\alpha})$ is.
	This immediately proves the first half of the statement, and the second half follows from the conservativity of $\Gr_{p_y}$, that holds when $\cI$ and $\cJ$ are in $\PosFib^f$ thanks to \cref{cor:section_Gr_conservative}.
\end{proof}

\begin{cor} \label{cor:graduation_preserves_cocartesian_functors}
	In the setting of \cref{prop:exponential_graduation_cocartesian_edges}, the functor $\Gr_p \colon \Fun(\cI, \cE) \to \Fun(\cI_p, \cE)$ preserves cocartesian functors.
	If in addition $\cI$ and $\cJ$ belong to $\PosFib^f$, then the resulting commutative square
	\[ \begin{tikzcd}
		\Funcocart(\cI, \cE) \arrow[hook]{r} \arrow{d}{\Gr_p} & \Fun(\cI, \cE) \arrow{d}{\Gr_p} \\
		\Funcocart(\cI_p, \cE) \arrow[hook]{r} & \Fun(\cI_p, \cE)
	\end{tikzcd} \]
	is a pullback.
\end{cor}

\begin{proof}
	Thanks to \cref{prop:exponential_graduation_cocartesian_edges}, we see that $\expGr_p$ is a morphism in $\PrFibL$.
	Applying $\Sigma^{\cocart}_\cX$, we see that $\Gr_p$ preserves cocartesian functors.
	As for the pullback statement, since both horizontal functors are fully faithful, it amounts to check that if $F \colon \cI \to \cE$ is such that $\Gr_p(F)$ is cocartesian, then the same goes for $F$.
	Via the specialization equivalence \eqref{eq:specialization_equivalence}, we can equivalently see $F$ as a section $\spe_\cI(F) \colon \cX \to \exp_\cE(\cI / \cX)$ of the structural map of $\exp_\cE(\cI / \cX)$.
	Using \cref{prop:exponential_graduation_local_global_behavior}-(2), we see that the problem at hand becomes showing that $\spe_\cI(F)$ is a cocartesian section if and only if $\expGr_p \circ \spe_\cI(F)$ is a cocartesian section of $\exp_\cE(\cI_p / \cX)$, and this latter statement follows directly from the second half of \cref{prop:exponential_graduation_cocartesian_edges}.
\end{proof}

In fact, we can extract from the proof of \cref{cor:graduation_preserves_cocartesian_functors} the following more precise statement:

\begin{cor}\label{lem:Gr_and_specialization}
	In the setting of \cref{prop:exponential_graduation_cocartesian_edges}, let $\gamma \colon x \to y$ be a morphism in $\cX$ and let $F \colon \cI \to \cE$ be a functor.
	Then if $F$ is cocartesian at $\gamma$, the same goes for $\Gr_p(F) \colon \cI_p \to \cE$.
	The converse holds provided that both $\cI$ and $\cJ$ belong to $\PosFib^f$.
\end{cor}

\begin{proof}
	Passing to the other side of the specialization equivalence \eqref{eq:specialization_equivalence} and invoking \cref{prop:exponential_graduation_local_global_behavior}-(2), we have to prove that a section $s \colon \cX \to \exp_\cE(\cI / \cX)$ takes $\gamma$ to a cocartesian morphism if and only if $\expGr_p \circ s$ takes $\gamma$ to a cocartesian morphism in $\exp_\cE(\cI_p / \cX)$.
	As this statement is obviously implied by \cref{prop:exponential_graduation_cocartesian_edges}, the conclusion follows.
\end{proof}

Combining together \cref{Gr_preserves_PS} and \cref{cor:graduation_preserves_cocartesian_functors} we obtain:

\begin{cor}\label{cor:graded_of_cocart_is_cocart}
	Let $\cX$ be an $\infty$-category. 
	Let $p \colon \cI \to \cJ$ be a graduation morphism in $\PosFib_\cX$.
	Let $\cE$ be a presentable stable $\infty$-category.
	Then, for every cocartesian punctually split functor $F \colon \cI \to \cE$, its $p$-graduation $\Gr_p(F) \colon \cI_{p} \to \cE$ is cocartesian and punctually split.
\end{cor}

We conclude this section with the following handy consequence:

\begin{cor}
	Let $\cX$ be an $\infty$-category. 
	Let $p \colon \cI \to \cJ$ be a graduation morphism in $\PosFib_\cX$.
	Let $\cE$ be a presentable stable $\infty$-category.
	Then the functor
	\[ \expGr_p \colon \exp_\cE(\cI / \cX) \to  \exp_\cE( \cI_p / \cX ) \]
	admits a right adjoint $\expGr_p^\ast$ relative to $\cX$.
	In particular, for every $x \in \cX$, the diagram
	\[ \begin{tikzcd}[column sep=large]
		\Fun( (\cI_x)_{p_x}, \cE) \arrow{r}{\Gr_{p_x}^\ast} \arrow{d} & \Fun( \cI_x, \cE ) \arrow{d} \\
		\exp_\cE( \cI_p / \cX ) \arrow{r}{\expGr_p^\ast} & \exp_\cE( \cI / \cX )
	\end{tikzcd} \]
	commutes.
	In addition, the diagram
	\[ \begin{tikzcd}[column sep=50pt]
		\Fun_{/\cX}( \cX, \exp_\cE( \cI_p / \cX ) ) \arrow{r}{\Sigma_\cX(\expGr_p^\ast)} \arrow{d}{\spe_{\cI_p}} & \Fun_{/\cX}(\cX, \exp_\cE( \cI / \cX ) ) \arrow{d}{\spe_\cI} \\
		\Fun( \cI_p, \cE ) \arrow{r}{\Gr_p^\ast} & \Fun( \cI, \cE )
	\end{tikzcd} \]
	commutes as well.
\end{cor}

\begin{proof}
	Since $\expGr_p$ preserves cocartesian edges by \cref{prop:exponential_graduation_cocartesian_edges}, \cite[Proposition 7.3.2.6]{Lurie_Higher_algebra} shows that it is enough to prove that for every $x$, the induced functor on the fibers at $x$
	\[ (\expGr_p)_x \colon \Fun( (\cI_x)_{p_x}, \cE ) \to  \Fun( \cI_x, \cE ) \]
	admits a right adjoint.
	By \cref{prop:exponential_graduation_local_global_behavior}-(1), we see that $(\expGr_p)_x$ canonically coincides with $\Gr_{p_x}$, so the existence of the right adjoint is guaranteed by \cref{cor:section_Gr_commutes_with_colimits}.
	This proves at the same time the commutativity of the first diagram.
	As for the second, it simply follows from the uniqueness of the adjoints, the fact that the specialization functors are equivalences and \cref{prop:exponential_graduation_local_global_behavior}-(2).
\end{proof}

\begin{prop}\label{Gr_of_Stokes}
	Let $\cX$ be an $\infty$-category. 
	Let $p \colon \cI\to \cJ$ be a graduation morphism over $\cX$.
	Let $\cE$ be a presentable stable $\infty$-category.
	Then, the graduation functor relative to $p$ (\cref{defin_Gr})
	\[ \Gr_p \colon \Fun(\cI, \cE) \to  \Fun(\cI_p, \cE) \]
	preserves the category of Stokes functors.
	In other words, it restricts to a functor
	\[ \Gr_p \colon\St_{\cI,\cE}  \to  \St_{\cI_p,\cE}   . \]
\end{prop}
\begin{proof}
	This follows from \cref{cor:graded_of_cocart_is_cocart}, \cref{lem:Gr_restriction}  and \cref{cor:compatibility_Gr_i_I!}.
\end{proof}

\begin{cor}\label{cocartesian_splitting}
Let $\cI\to \cX$ be an object of  $\PosFib$ such that $\cI^{\ens}\to \cX$ is locally constant.
	Let $\cE$ be a presentable stable $\infty$-category.
	Then, the following	square
	\[ \begin{tikzcd}
		\St_{\cI^{\ens},\cE}  \arrow[hook]{d} \arrow{r}{i_{\cI,!}} & \St_{\cI,\cE}  \arrow[hook]{d} \\
		\Fun(\cI^{\ens},\cE) \arrow{r}{i_{\cI,!}}  & \Fun(\cI,\cE)
	\end{tikzcd} \]
is a pullback
\end{cor}

\begin{proof}
Let $F \colon \cI \to \cE$ be a split Stokes functor.
Let $V \colon  \cI^{\ens}\to \cE$ such that $F\simeq i_{\cI , !}(V)$.
By \cref{cor:compatibility_Gr_i_I!}, we have
\[
\Gr(F)\simeq \Gr(i_{\cI , !}(V)) \simeq V \ .
\]
By \cref{cor:graduation_preserves_cocartesian_functors}, we deduce that $V \colon  \cI^{\ens}\to \cE$  is cocartesian. 
Since $V$ is automatically punctually split, \cref{cocartesian_splitting} thus follows.

\end{proof}

\subsection{Essential image of a fully-faithul induction}

The following propositions describe the essential image of a fully-faithul induction in terms of graduation.

\begin{lem}\label{image_fully_faithful_induction_PS}
	Let $\cX$ be an $\infty$-category. 
	Let $i \colon \cI \to \cJ$ be a fully faithful functor in  $\PosFib$ over $\cX$.
	Let $\cE$ be a presentable stable $\infty$-category.
	Then, the functor 
	\[
	i_! \colon \Fun_{\PS}(\cI,\cE)\to  \Fun_{\PS}(\cJ,\cE)
	\]
	is fully faithful.
	Let $F\in \Fun_{\PS}(\cJ,\cE)$. 
	Then, the following statements are equivalent :
	\begin{enumerate}\itemsep=0.2cm
		\item $F$ lies in the essential image of $i_! \colon \Fun_{\PS}(\cI,\cE) \to \Fun_{\PS}(\cJ,\cE)$.
		\item $i^*(F)$ lies in $\Fun_{\PS}(\cI,\cE)$ and the counit map $i_{!}(i^*(F))\to F$ is an equivalence.
		\item  $(\Gr F)_a\simeq 0$  for every  $a\in \cJ^{\ens}$ not  in the essential image of $i^{\ens} \colon \cI^{\ens} \to \cJ^{\ens}$.
	\end{enumerate}
\end{lem}
\begin{proof}
	Since $i \colon \cI \to \cJ$  is fully faithful, so is 
	$i_! \colon \Fun(\cI,\cE)\to \Fun(\cJ,\cE)$.
	In particular, the unit of $i_! \dashv i^*$ is an equivalence.
	The fact that $(1)$ implies $(2)$ is then obvious.
	The statement $(2)$ trivially implies $(1)$.
	To show the equivalence with $(3)$, we can suppose from  \cref{cor:induction_specialization_Beck_Chevalley} and \cref{lem:Gr_restriction} that $\cX$ is a point.
	If $(2)$ holds, the sought-after vanishing follows from \cref{Gr_p_identity}.
	Suppose that $(3)$ holds.
	Let us write $F=i_{\cJ !}(V)$ where $V \colon \cJ^{\ens} \to \cE$.
	From \cref{Gr_p_identity}, $V_a\simeq 0$ for every $a\in \cJ^{\ens}\setminus \cI^{\ens}$.
	If $W=V|_{\cI^{\ens}}$, we thus have $V\simeq i_{!}^{\ens}(W)$.
	Hence, 
	\[
	F=i_{\cJ !}(V)\simeq i_{\cJ !}\circ  i_{!}^{\ens}(W)\simeq i_{!} \circ i_{\cI !}(W)
	\]
	which proves $(1)$, thus finishing the proof of \cref{image_fully_faithful_induction_PS}.
\end{proof}

\begin{prop}\label{image_fully_faithful_induction_Stokes}
	Let $\cX$ be an $\infty$-category. 
	Let $i \colon \cI \hookrightarrow  \cJ$ be a fully faithful functor in  $\PosFib$ over $\cX$.
	Let $\cE$ be a presentable stable $\infty$-category.
	Let $F\in \St_{\cJ,\cE}$. 
	Then, the following statements are equivalent :
	\begin{enumerate}\itemsep=0.2cm
		\item $F$ lies in the essential image of $i_! \colon \St_{\cI,\cE} \to \St_{\cJ,\cE}$.
		\item $i^*(F)$ lies in $\St_{\cI,\cE}$ and the counit map $i_{!}(i^*(F))\to F$ is an equivalence.
		\item  $(\Gr F)_a\simeq 0$  for every  $a\in \cJ^{\ens}$ not  in the essential image of $i^{\ens} \colon \cI^{\ens} \to \cJ^{\ens}$.
	\end{enumerate}
\end{prop}

\begin{proof}
The equivalence between (1) and (2) follows as in  \cref{image_fully_faithful_induction_PS}.
Assume that (1) holds. 
Then (3) holds in virtue of \cref{cor:compatibility_Gr_id}.
Assume that (3) holds.
We are doing to show that (2) holds.
Since $F$ is punctually split, \cref{image_fully_faithful_induction_PS} implies that 
$i^*(F)$ is punctually split and that the counit map $i_{!}(i^*(F))\to F$ is an equivalence.
Hence, we are left to show that $i^*(F)$  is cocartesian.
To do this, we can suppose that $\cX=\Delta^1$.
In that case, consider the commutative square 
\[ \begin{tikzcd}
		\cI_x \arrow[hook]{d}{i_x} \arrow[hook]{r}{j_x} & \cI \arrow[hook]{d}{i} \\
		\cJ_x \arrow[hook]{r}{j_x} & \cJ  \ .
	\end{tikzcd} \]
By \cref{prop:Stokes_functors_in_presence_of_initial_object}, the counit map $j_{x,!} j_x^*(F)\to F$ is an equivalence.
By \cref{lem:Gr_restriction}, the split functor $j_x^*(F) \colon \cJ_x \to \cE$ satisfies the conditions of \cref{image_fully_faithful_induction_PS}-(3).
Thus, there exists $G \colon \cI_x \to \cE$ such that $j_x^*(F)\simeq i_{x,!}(G)$.
Hence, we have
	\[
	i^*(F) \simeq i^* j_{x,!} j_x^*(F) \simeq i^* j_{x,!} i_{x,!}(G) \simeq i^* i_! j_{x,!}(G)\simeq j_{x,!}(G)
	\]
where the last equivalence follows from the fully faithfulness of $i \colon \cI \hookrightarrow  \cJ$.
Then $i^*(F)$ is cocartesian by	\cref{prop:Stokes_functors_in_presence_of_initial_object}.
\end{proof}

\begin{cor}\label{lem:elementarity_criterion}
	Let $\cX$ be an $\infty$-category. 
	Let $i \colon \cI\hookrightarrow \cJ$ be a fully faithful morphism in $\PosFib$ over $\cX$.
	Let $\cE$ be a presentable $\infty$-category.
	Assume that 
	\[i_{\cJ, !} \colon \St_{\cJ^{\ens},\cE} \to  \St_{\cJ,\cE} \] 
	is essentially surjective (resp. fully faithful).
	Then, so is 
	\[ i_{\cI, !} \colon \St_{\cI^{\ens},\cE}  \to  \St_{\cI,\cE}  \ .  \]
\end{cor}

\begin{proof}
Consider the commutative square 
\[
	\begin{tikzcd}
 \St_{\cI^{\ens},\cE}   \arrow[hook]{r}{i_!^{\ens}}     \arrow{d}{i_{\cI, !} }   & \St_{\cJ^{\ens},\cE}   \arrow{d}{i_{\cJ, !} }     \\
     \St_{\cI,\cE}    \arrow[hook]{r}{i_!}   & \St_{\cJ,\cE}   
		\end{tikzcd}
\]
whose horizontal arrows are fully faithful since $i \colon \cI\hookrightarrow \cJ$ is fully faithful.
In particular, if $i_{\cJ, !}$ is fully faithful so is $i_{\cI, !}$.
Assume that  $i_{\cJ, !}$  is essentially surjective.
Let $F \colon \cI \to \cE$ be a Stokes functor.
Write $i_!(F)\simeq i_{\cJ, !}(V)$ where $V \colon \cJ^{\ens} \to \cE$ is Stokes. 
By \cref{cor:compatibility_Gr_i_I!}, we have 
\[
\Gr i_!(F)\simeq \Gr  i_{\cJ, !}(V) \simeq V \simeq  \Gr V \ .
\]
By \cref{image_fully_faithful_induction_Stokes}, we deduce that $(\Gr V)_a \simeq 0$ for every  $a\in \cJ^{\ens}$ not  in the essential image of $i^{\ens} \colon \cI^{\ens} \to \cJ^{\ens}$.
By \cref{image_fully_faithful_induction_Stokes} again, there is a Stokes functor $W \colon \cI^{\ens} \to \cE$ such that $V\simeq i_!^{\ens}(W)$.
Then, 
\[
i_! i_{\cI, !}(W) \simeq i_{\cJ, !} i_!^{\ens}(W) \simeq i_!(F) \ .
\]
Since $i \colon \cI\hookrightarrow \cJ$ is fully faithful, we deduce that $F\simeq i_{\cJ, !}(W)$.
The proof of \cref{lem:elementarity_criterion} is thus complete.
\end{proof}

%


\subsection{Graduation and $t$-structures}

We now explore the properties of the relative graduation with respect to the $t$-structures of \cref{prop:t_structure_Stokes}.

\begin{prop}\label{prop:Gr_t_exact}
	Let $\cX$ be an $\infty$-category and let $p \colon \cI \to \cJ$ be a graduation morphism of cocartesian fibrations in finite posets over $\cX$.
	Let $\cE$ be a presentable stable $\infty$-category equipped with an accessible $t$-structure $\tau = (\cE_{\leqslant 0}, \cE_{\geqslant 0})$.
	If $\cI$, $\cI_p$ are $\cE$-bireflexive, then the relative graduation functor
	\[ \Gr_p \colon \St_{\cI, \cE} \to \St_{\cI_p, \cE} \]
	is $t$-exact.
\end{prop}

\begin{proof}
	The very definition of $\Gr_p$ (see \cref{defin_Gr}) and \cref{recollection:standard_t_structure} imply together that $\Gr_p$ is right $t$-exact.
	Let now $F \in ( \St_{\cI,\cE} )_{\leqslant 0}$.
	To check that $\Gr_p(F) \in ( \St_{\cI_p, \cE} )_{\leqslant 0}$, it suffices to show that for every $x \in \cX$ one has
	\[ j_x^\ast( \Gr_p(F) ) \in \Fun((\cI_p)_x, \cE) \ . \]
	By \cref{lem:Gr_restriction} and \cref{lem:Gr_restriction_points} we have a canonical equivalence
	\[ j_x^\ast( \Gr_p(F) ) \simeq \Gr_{p_x}( j_x^\ast(F) ) \ . \]
	We can therefore assume that $\cX$ is reduced to a point.
	Since $F$ is punctually split, we can find a functor $V \colon \cI^{\ens} \to \cE$.
	\Cref{lem:t_structure_Stokes_splitting} guarantees that $V$ takes values in $\cE_{\leqslant 0}$.
	Since $\cI_x$ is finite and $\cE_{\leqslant 0}$ is closed under finite sums, the conclusion follows from the fomula given in \cref{ex_Gr}.
\end{proof}

\begin{cor}\label{prop:Gr_creates_the_t_structures}
	In the setting of \cref{prop:Gr_t_exact}, a Stokes functor $F \colon \cI \to \cE$ is connective (resp.\ coconnective) with respect to the induced $t$-structure on $\St_{\cI,\cE}$ if and only if $\Gr_p(F)$ is connective (resp.\ coconnective).
\end{cor}

\begin{proof}
	It follows combining $t$-exactness and conservativity of $\Gr_p$, see \cref{prop:Gr_t_exact} and \cref{cor:section_Gr_conservative}.
\end{proof}

\subsection{Splitting criterion}\label{lower}

      The goal of this subsection is to establish a splitting criterion (\cref{fiber_product_is_split}), to be used in  the essential surjectivity part of the proof of  \cite[\cref*{Geometric_Stokes-induction_for_adm_Stokes_triple}]{Geometric_Stokes}.

\begin{construction}  \label{constr:F_minus_I}   
     Let $\cX$ be an $\infty$-category. 
     Let $i \colon \cI\hookrightarrow  \cJ$ and $k \colon \cK\hookrightarrow  \cJ$ be fully faithful functors in $\PosFib^f$ over $\cX$ such that $\cJ^{\ens}= \cI^{\ens} \sqcup \cK^{\ens}$.
     In particular, for every functor $G \colon \cI^{\ens} \to \cE$, we have
     \[
     G\simeq i_{!}^{\ens} i^{\ens \ast}(G) \oplus k_{!}^{\ens} k^{\ens \ast}(G)  \ .
     \]
     We denote by $\Delta(G)$ the split cofiber sequence 
     \begin{equation}\label{stupid_cofiber_sequnce}
     i_{!}^{\ens} i^{\ens \ast}(G)  \to G \to k_{!}^{\ens}  k^{\ens \ast}(G)  \ .
          \end{equation}
      We assume that $\cJ^{\ens}\to \cX$ is locally constant.
      Let $\cE$ be a presentable stable $\infty$-category. 
      Let $F\colon \cJ\to \cE$ be a functor.
     We suppose that the canonical morphism $ i^{\ens, \ast} i^{*}_{\cJ }(F)  \to  i^{\ens \ast}\Gr(F)$ admits a section
\begin{equation}\label{lower_running_assumption}
	\sigma \colon i^{\ens \ast}\Gr (F)    \to  i^{\ens, \ast} i^{*}_{\cJ }(F)   \ .
\end{equation} 
      By adjunction, $\sigma$ yields a morphism
\[
\tau \colon i_{\cJ !} i_{!}^{\ens} i^{\ens \ast} \Gr(F) \to  F
\]
in $\Fun(\cJ, \cE)$.
      We denote by $\Delta(F,\sigma)$ the following cofiber sequence 

\begin{equation}\label{eq:defFminusI}
\begin{tikzcd}
i_{\cJ !} i_{!}^{\ens}  i^{\ens \ast} \Gr(F)   \arrow{r}{\tau } & F  \arrow{r} &  F^{\backslash \cI}   \ .
\end{tikzcd}
\end{equation}

\end{construction}

\begin{rem}\label{lowering_I_and_pullback}
      By \cref{lem:Gr_restriction} and \cref{cor:induction_specialization_Beck_Chevalley}, observe that the formation of $F^{\backslash \cI}$  commutes with pull-back.
\end{rem}

\begin{lem}\label{lower_size_J}
      In the setting of  \cref{constr:F_minus_I}, the canonical morphism
\[
k^{\ens}_! k^{\ens,\ast}\Gr(F) \to \Gr(F^{\backslash \cI})
\]
is an equivalence.

\end{lem}
\begin{proof}
      Since $\Gr$ commutes with colimits, applying $\Gr$ to \eqref{eq:defFminusI} yields a cofiber sequence 
\[
 \Gr i_{!}i_{\cI !} i^{\ens \ast}  \Gr(F)   \to  \Gr(F)  \to  \Gr(F^{\backslash \cI})   \ .
\]
      By \cref{cor:compatibility_Gr_i_I!} and \cref{cor:compatibility_Gr_id}, we have  $\Gr i_{!}i_{\cI !}  i^{\ens \ast} \Gr(F) \simeq i^{\ens}_!i^{\ens,\ast}\Gr(F)$.
      Since we have
\[
\Gr(F) \simeq  i^{\ens}_!i^{\ens,\ast}\Gr(F) \oplus k^{\ens}_! k^{\ens,\ast}\Gr(F)  \ ,
\]
\cref{lower_size_J} thus follows.
\end{proof}

\begin{lem}\label{morphism_comes_from_graded}
      In the setting of  \cref{constr:F_minus_I}, the following hold:
	\begin{enumerate}\itemsep=0.2cm
	\item If $F$ is cocartesian, so is $F^{\backslash \cI}$.
	
	\item If $F$ is punctually split, so is $F^{\backslash \cI}$.
	
	\item If $F$ split,the cofiber sequences $\Delta(F,\sigma)$ and  $i_{\cJ,!}\Delta(\Gr(F))$ are equivalent.
	In particular, $F^{\backslash \cI}$ split.
	\end{enumerate}
\end{lem}
\begin{proof} 
      Item (1) follows immediately from the stability of $\Funcocart(\cJ,\cE)$ under colimits (\cref{cor:cocartesianization}).
      By \cref{lowering_I_and_pullback}, the formation of $F^{\backslash \cI}$  commutes with pull-back.
      Hence, (3) implies (2).
      We now prove (3) and assume that $F$ split.
      By \cref{section_Gr_gives_splitting}, the canonical  morphism $i^{*}_{\cJ }(F)  \to  \Gr(F)$ 
admits a section $\iota \oplus \kappa$.
      Then, the vertical arrows of the commutative square 
     	\[
	\begin{tikzcd}
i^{\ens}_! i^{\ens,\ast} 	i^{*}_{\cJ }(F) 	\arrow{r}\arrow{d} &  i^{*}_{\cJ }(F)  \arrow{d}  \\
	i^{\ens}_! i^{\ens,\ast}   \Gr(F)     \arrow{r}  \arrow[u, dashed,bend left = 80, "\sigma  \oplus 0"] &  \Gr(F)    \arrow[u, swap, dashed, bend right = 80, "\sigma \oplus \kappa"] 
	\end{tikzcd} 
	\]
 admit sections represented by dashed arrows.
      By adjunction, there is a commutative square
	\[
	\begin{tikzcd}
	F	\arrow{r}{\id}&F  \\
	i_{\cJ,!} 	 i^{\ens}_! i^{\ens,\ast}   \Gr(F)    \arrow{u}{\tau}  \arrow{r}   & i_{\cJ,!} \Gr(F)  \arrow{u} 
	\end{tikzcd} 
	\]
	whose right vertical arrow is an equivalence in virtue of \cref{section_Gr_gives_splitting}.  
  Item (3) is thus proved.
\end{proof}

\begin{cor}\label{induced_from_K}
      In the setting of  \cref{constr:F_minus_I}, assume that $F \colon \cJ \to \cE$ punctally split.
      Then, $F^{\backslash \cI}$ lies in the essential image of $k_! \colon \Fun(\cK,\cE) \to \Fun(\cJ,\cE)$.
\end{cor}

\begin{proof}
      By \cref{morphism_comes_from_graded}, we know that $F^{\backslash \cI}$ punctually split.
      By \cref{lower_size_J}, we have $(\Gr F^{\backslash \cI})(a)\simeq 0$ for every $a\in \cI$.
      Then, \cref{induced_from_K} follows from \cref{image_fully_faithful_induction_PS}.
\end{proof}

\begin{construction}\label{constr:F_minus_IL}  
      In the setting of \cref{constr:F_minus_I}, let $l \colon \cL\hookrightarrow  \cK$ and $m \colon \cM\hookrightarrow  \cK$ be fully faithful functors in $\PosFib^f$ over $\cX$ such that $\cK^{\ens}= \cL^{\ens} \sqcup \cM^{\ens}$.
      We suppose that the canonical morphism $ l^{\ens \ast} i^{*}_{\cJ }(F)  \to  l^{\ens \ast}\Gr(F)$ admits a section
     
\begin{equation}\label{lower_running_assumption}
	\lambda \colon l^{\ens \ast}\Gr (F)    \to  l^{\ens \ast} i^{*}_{\cJ }(F)   \ .
\end{equation}

     Let $\iota \colon \cI \cup \cL  \hookrightarrow \cJ$ be the full subcategory of $\cI$ spanned by objects of  $\cI$ and $\cL$.
     Then, the vertical arrows of the commutative square 
     	\[
	\begin{tikzcd}
	\iota^* i^{*}_{\cJ }(F) 	\arrow{r}\arrow{d} & \iota^* i^{*}_{\cJ }(F^{\backslash \cI})  \arrow{d}  \\
	\iota^*  \Gr(F)     \arrow{r}  \arrow[u, dashed,bend left = 80, "\sigma \oplus \lambda"] & \iota^*  \Gr(F^{\backslash \cI})    \arrow[u, swap, dashed, bend right = 80, "0\oplus \lambda"] 
	\end{tikzcd} 
	\]
     admit sections represented by dashed arrows.
By adjunction, we thus deduce a morphism of cofiber sequence 
\begin{equation}\label{eq:constr:F_minus_IL}
\Delta(F,\sigma \oplus \lambda) \to \Delta(F^{\backslash \cI},0\oplus \lambda)  \ .
\end{equation}

\end{construction}

\begin{lem}\label{iterated_lowering}
      In the setting of \cref{constr:F_minus_IL}, the natural transformation 
\[
F^{\backslash \cI\cup \cL}  \to (F^{\backslash \cI})^{\backslash \cI \cup \cL}
\]
deduced from \eqref{eq:constr:F_minus_IL} is an equivalence.
\end{lem}

\begin{proof}
      Immediate from \cref{lower_size_J} and \cref{cor:section_Gr_conservative}.
\end{proof}

\begin{notation}
We denote by $ \alpha_{\cI,\cL} \colon F^{\backslash \cI}\to F^{\backslash \cI\cup \cL} $   the canonical morphism obtained by composing $F^{\backslash \cI} \to (F^{\backslash \cI})^{\backslash \cI \cup \cL}$ with the inverse of $F^{\backslash \cI\cup \cL}  \to (F^{\backslash \cI})^{\backslash \cI \cup \cL}$ supplied by \cref{iterated_lowering}.
\end{notation}

\begin{lem}\label{in_the_essential_image_of_induction}
     In the setting of \cref{constr:F_minus_IL}, assume that $F^{\backslash \cI}$ split.
     Then there is  a commutative square 
\[ \begin{tikzcd}
	i_{\cJ,!}\Gr(F^{\backslash \cI})  \arrow{r} \arrow{d}{i_{\cJ,!}  \Gr(\alpha_{\cI,\cL})  } & F^{\backslash \cI} \arrow{d}{ \alpha_{\cI,\cL} } \\
	i_{\cJ,!}\Gr(F^{\backslash \cI \cup \cL})\arrow{r} & F^{\backslash \cI \cup \cL} 
\end{tikzcd} \]
whose horizontal arrows are equivalences.
\end{lem}

\begin{proof}
Recall that $\cJ^{\ens}= \cI^{\ens} \sqcup  \cL^{\ens} \sqcup \cM^{\ens}$.
Consider the commutative diagram 
\[ \begin{tikzcd}
i_{\cJ}^*(\alpha_{\cI,\cL})  \colon 	i_{\cJ}^*(F^{\backslash \cI})  \arrow{r} \arrow{d} & i_{\cJ}^* (F^{\backslash \cI})^{\backslash \cI \cup \cL} \arrow{d}     \arrow{r}{\sim}  &  i_{\cJ}^* (F^{\backslash \cI \cup \cL})  \arrow{d}   \\
\Gr(\alpha_{\cI,\cL}) \colon 	\Gr (F^{\backslash \cI}) \arrow{r} \arrow[u, swap, dashed, bend right = 60, "0\oplus \lambda \oplus  \mu' "]  & \Gr (F^{\backslash \cI})^{\backslash \cI \cup \cL}  \arrow{r}{\sim} \arrow[u, swap, dashed, bend right = 60, "0\oplus 0\oplus  \mu' "] &     \Gr(F^{\backslash \cI \cup \cL})   \arrow[u, swap, dashed, bend right = 60, "0\oplus 0\oplus  \mu"]  \ .
\end{tikzcd} \]
Since $F^{\backslash \cI}$ split, \cref{morphism_comes_from_graded} ensures that so do 
$(F^{\backslash \cI})^{\backslash \cI \cup \cL}$.
By \cref{iterated_lowering}, the functor $F^{\backslash \cI \cup \cL}$ split as well.
By \cref{section_Gr_gives_splitting}, we thus deduce the existence of the section $\mu$ represented as a dashed arrow.
Since the right horizontal arrows are equivalences, there exists a section $\mu'$ making the right square commutative.
On the other hand, \cref{lower_size_J} implies that $\Gr (F^{\backslash \cI})^{\backslash \cI \cup \cL}$  is a direct factor of $\Gr (F^{\backslash \cI})$.
Since $F^{\backslash \cI}$ split, we deduce from  \cref{morphism_comes_from_graded} the existence of a section $\lambda$ making the left square commutative.
   Then, \cref{in_the_essential_image_of_induction} follows from \cref{section_Gr_gives_splitting}.
\end{proof}

\begin{lem}\label{fiber_product_iso}
	In the setting of \cref{constr:F_minus_IL}, the morphism
\begin{equation}\label{eq:fiber_product_iso}
		F \to  F^{\backslash \cI}\times_ {F^{\backslash \cI \cup \cL}} F^{\backslash \cL}
\end{equation}
	is an equivalence.
\end{lem}

\begin{proof}
     Follows immediately from \cref{lower_size_J} and \cref{cor:section_Gr_conservative}.
\end{proof}

\begin{cor}\label{fiber_product_is_split}
	In the setting of \cref{constr:F_minus_IL}, the following are equivalent:
	\begin{enumerate}\itemsep=0.2cm
	\item the functor $F$ split;
	\item the functors $F^{\backslash \cI}$ and $F^{\backslash \cL}$ split. 
	\end{enumerate}
\end{cor}
\begin{proof}
       If (1) holds, so do (2) in virtue of \cref{morphism_comes_from_graded}.
       Assume that (2) holds.
	From \cref{fiber_product_iso}, we are left to show that 
	$F^{\backslash \cI}\times_ {F^{\backslash \cI \cup \cL}} F^{\backslash \cL}$ split. 
	Since $F^{\backslash \cI}$ and $F^{\backslash \cL}$ split, \cref{in_the_essential_image_of_induction} ensures that the diagram
	\[
	\begin{tikzcd}
F^{\backslash \cI}  \arrow{r}{\alpha_{\cI,\cL} } &  F^{\backslash \cI \cup \cL} &  F^{\backslash \cL}     \arrow[l,swap, "\alpha_{\cL,\cI}" ]
\end{tikzcd}
\]
	is equivalent to 
	\[
	\begin{tikzcd}[column sep=50pt]
	i_{\cJ,!} \Gr(F^{\backslash \cI})  \arrow{r}{i_{\cJ,!} \Gr(\alpha_{\cI,\cL}) } & i_{\cJ,!} \Gr(F^{\backslash \cI \cup \cL}) &  F^{\backslash \cL}     \arrow[l,swap,"i_{\cJ,!}\Gr(\alpha_{\cL,\cI})" ]   \ .
\end{tikzcd}
\]
	Since the induction functor $i_{\cJ,!}$ commutes with finite limits, \cref{fiber_product_is_split}  thus follows.
\end{proof}

\section{Level structures}\label{sec:level_structures}

We now introduce an axiomatization of the notion of level structure from the theory of good meromorphic flat bundles \cite{Mochizuki1}.
The key concept is that of level morphism for a morphism of cocartesian fibrations in posets.

\subsection{Level morphisms}

We start with the following pair of definitions:

\begin{defin}\label{defin_level_morphism_poset}
	A morphism of posets $p \colon \cI \to \cJ$ is a \emph{level morphism} if it is surjective and for every $a, b \in \cI$, we have
	\[ p(a) < p(b) \text{ in } \cJ \Rightarrow a < b \text{ in } \cI \ . \]
	\personal{See \cref{stability_lim_colim_ISt} for the reason surjectivity is required.}
\end{defin}

\begin{defin}\label{defin_level_morphism}
	Let $\cX$ be an $\infty$-category and let $p \colon \cI \to \cJ$ be a morphism in $\PosFib_\cX$.
	We say that $p$ is a \emph{level morphism} if for every $x \in \cX$, the induced morphism $p_x \colon \cI_x \to \cJ_x$ is a level morphism.
\end{defin}

\begin{eg}
Let $\cI \to \cX$ be an object of $\PosFib$.
Then, the morphisms of cocartesian fibrations $\id_\cI  \colon \cI \to \cI$ and $\cI \to \cX \times \ast \simeq \cX$  are level morphisms. 
\end{eg}

\begin{rem}
The class of level morphisms is stable under pullback.
\end{rem}

\begin{construction}\label{Ip}
Fix an $\infty$-category $\cX$ and let $p \colon \cI \to \cJ$ be a level graduation morphism in $\PosFib_\cX$.
Fix also a presentable stable $\infty$-category $\cE$.
Recall from \cref{construction:I_p} the following pullback diagram
\[ \begin{tikzcd}
	\cI_p \arrow{r} \arrow{d}{\pi} & \cI \arrow{d}{p} \\
	\cJ^{\ens} \arrow{r} & \cJ \ ,
\end{tikzcd} \]
as well as the commutative diagram
\[ \begin{tikzcd}
	\exp_\cE(\cI / \cX) \arrow{r}{\cE^p_!} \arrow{d}{\expGr_p} & \exp_\cE( \cJ / \cX ) \arrow{d}{\expGr} \\
	\exp_\cE( \cI_p / \cX ) \arrow{r}{\cE^\pi_!} & \exp_\cE( \cJ^{\ens} / \cX ) \ .
\end{tikzcd} \]
supplied by \cref{prop:Graduation_induction}.
It induces a canonical transformation
\[ \phi_p \colon \exp_\cE( \cI / \cX ) \to  \exp_\cE( \cJ / \cX ) \times_{\exp_\cE(\cI^{\ens} / \cX)} \exp_\cE( \cI_p / \cX ) \]
in $\PrFibL_\cX$.
Observe as well that combining Propositions \ref{prop:functoriality_exponential}-(2) and \cref{prop:exponential_graduation_cocartesian_edges}, we see that all the functors in the above square preserve cocartesian edges.
Thus, the same goes for $\phi_p$.
Since  $\Sigma_\cX \colon \PrFibL_\cX \to \PrL$ is a right adjoint, $\Sigma_\cX(\phi_p)$ is a functor
	\[ \Sigma_\cX(\phi_p) \colon \Fun(\cI, \cE) \to  \Fun( \cJ, \cE ) \times_{\Fun(\cI^{\ens}, \cE)} \Fun(\cI_p, \cE) \ , \]
	and Propositions \ref{prop:exponential_vs_global_induction} and \ref{prop:exponential_graduation_local_global_behavior}-(2) imply that it canonically coincides with the functor induced by the commutative diagram
	\[ \begin{tikzcd}
		\Fun( \cI, \cE ) \arrow{r}{p_!} \arrow{d}{\Gr_p} & \Fun(\cJ, \cE) \arrow{d}{\Gr} \\
		\Fun(\cI_p, \cE) \arrow{r}{\pi_!} & \Fun(\cJ^{\ens}, \cE) \ .
	\end{tikzcd} \]
\end{construction}

\begin{prop}\label{prop:exponential_full_faithfulness_level_induction}
	The functors $\phi_p$ and $\Sigma_\cX(\phi_p)$ are fully faithful.
\end{prop}
\begin{proof}
	Thanks to \cref{lem:full_faithfulness_cocartesian _fibrations}, we are immediately reduced to prove the statement when $\cX$ is a point.
	In this case, unraveling the definitions, we have to check that for every pair of functors $F, G \colon \cI \to \cE$ the square
	\begin{equation}\label{eq:full_faithfulness_level_induction}
		\begin{tikzcd}
			\Map_{\Fun(\cI,\cE)}( F, G ) \arrow{r} \arrow{d} & \Map_{\Fun(\cJ,\cE)}( p_!(F), p_!(G) ) \arrow{d} \\
			\Map_{\Fun(\cI_p, \cE)}( \Gr_p(F), \Gr_p(G) ) \arrow{r} & \Map_{\Fun(\cJ^{\ens}, \cE)}( \Gr p_!(F), \Gr p_!(G) )
		\end{tikzcd}
	\end{equation}
	is a pullback.
	Notice that the collection of functors $F$ for which the statement is true is closed under colimits.
	Invoking \cref{prop:generation_by_split_functors} and \cref{eg:split_functor} we can therefore assume without loss of generality that $F \simeq \ev_{a,!}^{\cI}(E)$ for some $a \in \cI$ and some $E \in \cE$.
	Notice that
	\[ p_! \big( \ev_{a,!}^\cI(E) \big) \simeq \ev_{p(a),!}^{\cJ}(E) \]
	and that \cref{cor:compatibility_Gr_i_I!} supplies canonical identifications
	\[ \Gr_p \big( \ev_{a,!}^\cI(E) \big) \simeq \ev_{a,!}^{\cI_p}(E) \qquad \text{and} \qquad \Gr \big( \ev_{p(a),!}^{\cJ}(E) \big) \simeq \ev_{p(a),!}^{\cJ^{\ens}}(E) \ . \]
	Thus \eqref{eq:full_faithfulness_level_induction} can be rewritten as follows:
	\[ \begin{tikzcd}
		\Map_\cE( E, G_a ) \arrow{r} \arrow{d} & \Map_\cE( E, (p_!(G))_a ) \arrow{d} \\
		\Map_\cE( E, \Gr_p(G)_a ) \arrow{r} & \Map_\cE( E, \Gr( p_!(G) )_a ) \ ,
	\end{tikzcd} \]
	and to prove that it is a pullback becomes equivalent to prove that for every $a \in \cI$ and every $G \colon \cI \to \cE$, the square
	\begin{equation}\label{eq:full_faithfulness_level_induction_II}
		\begin{tikzcd}
			G_a \arrow{r} \arrow{d} & (p_!(G))_a \arrow{d} \\
			\Gr_p(G)_a \arrow{r} & \Gr(p_!(G))_a
		\end{tikzcd}
	\end{equation}
	is a pullback in $\cE$.
	Since $\cE$ is stable, we see that the collection of functors $G$ for which the above square is a pullback is closed under colimits.
	Invoking once again \cref{prop:generation_by_split_functors} and \cref{eg:split_functor}, we can suppose that $G \simeq \ev_{b,!}^\cI(M)$, for some $b \in \cI$ and $M \in \cE$.
	We now proceed by analysis case-by-case:
	\begin{enumerate}\itemsep=0.2cm
		\item \emph{Case $p(a) < p(b)$ or $p(a)$ and $p(b)$ incomparable.} Since $p$ is a level morphism, this implies respectively that $a < b$ or that $a$ and $b$ are incomparable.
		In either cases, \eqref{eq:full_faithfulness_level_induction_II} becomes
		\[ \begin{tikzcd}
			0 \arrow{r} \arrow{d} & 0 \arrow{d} \\
			0 \arrow{r} & 0
		\end{tikzcd} \]
		which is indeed a pullback.

		\item \emph{Case $p(a) > p(b)$.} Since $p$ is a level morphism, this implies that $a > b$.
		Then \eqref{eq:full_faithfulness_level_induction_II} becomes
		\[ \begin{tikzcd}
			M \arrow{r}{\id_M} \arrow{d} & M \arrow{d} \\
			0 \arrow{r} & 0
		\end{tikzcd} \]
		which is indeed a pullback.
		
		\item \emph{Case $p(a) = p(b)$.} We then distinguish two further cases:
		\begin{enumerate}[(i)]\itemsep=0.2cm
			\item \emph{Case $a \geq b$.} Then \eqref{eq:full_faithfulness_level_induction_II} becomes
			\[ \begin{tikzcd}
				M \arrow{r}{\id_M} \arrow{d}{\id_M} & M \arrow{d}{\id_M} \\
				M \arrow{r}{\id_M} & M \ ,
			\end{tikzcd} \]
			which is indeed a pullback.
			
			\item \emph{Case $a < b$ or $a$ and $b$ incomparable.} Then \eqref{eq:full_faithfulness_level_induction_II} becomes
			\[ \begin{tikzcd}
				0 \arrow{r} \arrow{d} & M \arrow{d}{\id_M} \\
				0 \arrow{r} & M \ ,
			\end{tikzcd} \]
			which is indeed a pullback.
		\end{enumerate}
	\end{enumerate}
	Thus, the conclusion follows.
\end{proof}

\subsection{Level induction}

The goal of this subsection is to prove the following result:

\begin{thm}\label{prop:Level_induction}
	Let $\cX$ be an $\infty$-category and let $p \colon \cI \to \cJ$ be a level graduation morphism in $\PosFib_\cX$.
	Then the square
	\[ \begin{tikzcd}
		\exp_\cE^{\PS}( \cI / \cX ) \arrow{r}{\cE^p_!} \arrow{d}{\expGr_p} & \exp_\cE^{\PS}( \cJ / \cX ) \arrow{d}{\expGr} \\
		\exp_\cE^{\PS}( \cI_p / \cX ) \arrow{r}{\cE^\pi_!} & \exp_\cE^{\PS}( \cJ^{\ens} / \cX )
	\end{tikzcd} \]
	is a pullback square in $\hCoCart_\cX$.
	In particular, the induced square
	\[ \begin{tikzcd}
		\St_{\cI, \cE} \arrow{r}{p_!} \arrow{d}{\Gr_p} & \St_{\cJ,\cE} \arrow{d}{\Gr} \\
		\St_{\cI_p, \cE} \arrow{r} & \St_{\cJ^{\ens}, \cE}
	\end{tikzcd} \]
	is a pullback square in $\CAT_\infty$.
\end{thm}

\begin{proof}
	The second half follows directly from the first since $\Sigma_\cX^{\cocart} \colon \hCoCart_\cX \to \CAT_\infty$ is a right adjoint.
	Moreover, the straightening/unstraightening equivalence immediately reduces the proof of the first half to the case where $\cX$ is a point.
	In this case, we have to show that the top horizontal arrow of the commutative square
\[ \begin{tikzcd}
	\Fun^{\PS}(\cI, \cE) \arrow{r} \arrow{d} & \Fun^{\PS}( \cJ, \cE ) \times_{\Fun^{\PS}(\cI^{\ens}, \cE)} \Fun^{\PS}(\cI_p, \cE) \arrow{d}\\
	\Fun(\cI, \cE) \arrow{r} & \Fun( \cJ, \cE ) \times_{\Fun(\cI^{\ens}, \cE)} \Fun(\cI_p, \cE)\ .
\end{tikzcd} \]
is an equivalence.
Note that the vertical arrows are fully faithful.
From \cref{prop:exponential_full_faithfulness_level_induction}, the bottom arrow is fully faithful.
Thus, so is the top horizontal arrow.
We are thus left to show essentially surjectivity.
From \cref{cor:compatibility_Gr_i_I!}, the lateral faces of the following cube 
$$
	\begin{tikzcd}[row sep=scriptsize, column sep=scriptsize]
	\Fun(\cI^{\ens}, \cE)\arrow[rr, "p^{\ens}_!" ] \arrow{dd} \arrow{rd}{i_{\cI !}}   &&\Fun(\cJ^{\ens}, \cE)\arrow{dd} \arrow{rd}{i_{\cJ !}}\\
		& \Fun^{\PS}(\cI, \cE)   &&\Fun^{\PS}(\cJ, \cE) \arrow[from=ll, crossing over, "p_!", near start]  \arrow{dd}\\
	\Fun(\cI^{\ens}, \cE)  \arrow{rd}{i_{\cI_p !}}   \arrow[rr, "p^{\ens}_!" , near end]   &&\Fun(\cJ^{\ens}, \cE)  \arrow{rd}{\id}   \\
	&   \Fun^{\PS}(\cI_p, \cE)  \arrow[rr, "\pi_!"]   \arrow[from=uu, crossing over, "\Gr_p", near end]  && \Fun(\cJ^{\ens}, \cE)   \arrow[from=uu, crossing over, "\Gr", near end]
	\end{tikzcd} 
$$
are commutative. 
Hence, all  faces are commutative. 
We thus obtain a commutative square
$$
\begin{tikzcd}
	\Fun(\cI^{\ens}, \cE) \arrow{d}{i_{\cI !}}\arrow{r}& 	\Fun(\cI^{\ens}, \cE)\times_{	\Fun(\cJ^{\ens}, \cE)  } 	\Fun(\cJ^{\ens}, \cE)  \arrow{d} \\
	\Fun^{\PS}(\cI, \cE)  \arrow{r} &	\Fun^{\PS}(\cI_p, \cE)  \times_{	\Fun(\cJ^{\ens}, \cE) } 	\Fun^{\PS}(\cJ, \cE)  
\end{tikzcd} 
$$
Since $i_{\cI_p !} \colon \Fun(\cI^{\ens}, \cE)  \to  \Fun^{\PS}(\cI_p, \cE) $  and  $ i_{\cJ !} \colon \Fun(\cJ^{\ens}, \cE) \to \Fun^{\PS}(\cJ, \cE)  $  are essentially surjective by definition, we deduce that so is the right vertical arrow of the above square. 
Since the top horizontal arrow is an equivalence,   the conclusion  follows.
\end{proof}

\subsection{Level induction and Stokes detection}

\begin{construction}\label{RFil}
Fix an $\infty$-category $\cX$ and let $p \colon \cI \to \cJ$ be a level graduation morphism in $\PosFib_\cX$.
Fix also a presentable stable $\infty$-category $\cE$.
We consider the following commutative cube:
\[ \begin{tikzcd}
	{} & \Fun(\cI, \cE) \arrow{rr}{p_!} \arrow{dd}[pos=0.7]{\Gr_p} & & \Fun(\cJ, \cE) \arrow{dd}{\Gr} \\
	\St_{\cI, \cE} \arrow[crossing over]{rr}[pos=0.6]{p_!} \arrow{dd}{\Gr_p} \arrow[hook]{ur} & & \St_{\cJ, \cE} \arrow[hook]{ur} \\
	{} & \Fun(\cI_p, \cE) \arrow{rr}[pos=0.6]{\pi_!} & & \Fun(\cJ^{\mathrm{set}}, \cE) \\
	\St_{\cI_p, \cE} \arrow[hook]{ur} \arrow{rr}{\pi_!} & & \St_{\cJ^{\mathrm{set}}, \cE} \arrow[hook]{ur} \arrow[crossing over,leftarrow]{uu}[pos=0.4]{\Gr}
\end{tikzcd} \]
Passing to fiber products on the front and back squares, we obtain the following commutative square:
\begin{equation}\label{eq:Beck_Chevalley_zig_zag}
	\begin{tikzcd}
	\St_{\cI, \cE}   \arrow{r}{L_{\St}} \arrow[hook]{d}{i}&  \St_{\cJ, \cE} \times_{\St_{\cI^{\mathrm{set}}, \cE}} \St_{\cI_p, \cE}  \arrow[hook]{d}{j}\\
			\Fun(\cJ, \cE)\arrow{r}{L_{\Fil}} & \Fun(\cJ, \cE) \times_{\Fun(\cJ^{\mathrm{set}}, \cE)} \Fun(\cI_p, \cE) \ .
	\end{tikzcd}
\end{equation}
Since 
\[
\cC \coloneqq \Fun(\cJ, \cE) \times_{\Fun(\cJ^{\mathrm{set}}, \cE)} \Fun(\cI_p, \cE)  
\]
is a finite limit in $\Cat_\infty$ whose transitions functors commute with filtered colimits, filtered colimits in $\cC$ are computed objectwise.
Since $\cE$ is stable, we deduce that colimits in $\cC$ are computed objectwise.
Hence, since 
$p_!  \colon  \Fun(\cI, \cE) \to \Fun(\cJ, \cE)$  and $\Gr_p \colon \Fun(\cI, \cE) \to \Fun(\cI_p, \cE)$ 
commute with colimits, so does $L_{\Fil}$\personal{(Mauro) Assume $\cC_\bullet \colon K \to \Cat_\infty$ is a finite (or compact) diagram and that all transition maps commute with filtered colimits. Then filtered colimits in $\lim_K \cC_\bullet$ are computed objectwise. One cannot drop the filtered assumption, because of the formula for mapping spaces in the limit. However in the current case, everything is stable. So finite colimits are automatic}.
Thus, $L_{\Fil}$ admits a right adjoint 
\[
R_{\Fil} \colon \Fun(\cJ, \cE) \times_{\Fun(\cJ^{\mathrm{set}}, \cE)} \Fun(\cI_p, \cE)\to  \Fun(\cI, \cE)  \ .
\]
\begin{rem}\label{formula_RFil}
By abstract nonsense, $R_{\Fil}$  sends a triple $G= (F_1, F_2, \alpha)$ to the pullback square
\[ \begin{tikzcd}
	R _{\Fil}(G) \arrow{r} \arrow{d} & p^\ast(F_1) \arrow{d} \\
	\Gr_p^\ast(F_2) \arrow{r}{\alpha} & \Gr_p^\ast \pi^\ast \Gr(F_1) 
\end{tikzcd} 
\]
in $\Fun(\cI, \cE)$. 
\end{rem}
From \cref{prop:Level_induction}, the functor $L_{\St}$ in (\ref{eq:Beck_Chevalley_zig_zag}) is an equivalence.
Let $R_{\St}$ be an inverse.
Then for every $G\in  \St_{\cJ, \cE} \times_{\St_{\cI^{\mathrm{set}}, \cE}} \St_{\cI_p, \cE}$, the chain of equivalences
	\begin{align*}
		\Map(R_{\St}(G ), R_{\Fil}(G)) & \simeq \Map(G, L_{\Fil}(R_{\Fil}(G)))& \text{By \cref{prop:exponential_full_faithfulness_level_induction}}  \\
		& \simeq  \Map(R_{\Fil}(G), R_{\Fil}(G))  & 
	\end{align*}
gives rise to a canonical morphism
\begin{equation}\label{RSt_TFil}
R_{\St}(G) \to  R_{\Fil}(G)   \ .
\end{equation}
\end{construction}

\begin{prop}\label{Stokes_detection}
Let $\cX$ be an $\infty$-category. 
Let $p \colon \cI\to \cJ$ be a level graduation morphism in $\PosFib^f$ over $\cX$.
Let $\cE$ be a presentable stable $\infty$-category.
Let $F \colon \cI \to \cE$ be a functor.
Then the following are equivalent : 
\begin{enumerate}\itemsep=0.2cm
\item $F$ is a Stokes functor.
\item $\Gr_p(F) \colon \cI_p \to \cE$ and $p_!(F) \colon \cJ \to \cE$ are Stokes functors.
\end{enumerate}
\end{prop}
\begin{proof}
That (1) implies (2) follows from \cref{cor:stokes_functoriality} and \cref{Gr_of_Stokes}.
Assume that (2) holds.
Then $L_{\Fil}(F)$ lies in $\St_{\cI, \cE} \times_{\St_{\cI^{\mathrm{set}}, \cE}} \St_{\cI_p, \cE}$.
From (\ref{RSt_TFil})  applied to  $G\coloneqq L_{\Fil}(F)$, there is a zig-zag
\[
R_{\St}(L_{\Fil}(F)) \to  R_{\Fil}(L_{\Fil}(F))   \leftarrow F
\]
whose right arrow is an equivalence in virtue of \cref{prop:exponential_full_faithfulness_level_induction}.
Hence, there is a canonical morphism 
\[ 
\alpha \colon R_{\St}(L_{\Fil}(F))  \to F \ .
\]
Since $R_{\St}(L_{\Fil}(F))$ is a Stokes functor, we are left to show that $\alpha$ is an equivalence.
Since $\Gr_p(\alpha) \colon \Gr_p(R_{\St}(L_{\Fil}(F)))  \to \Gr_p(F)$ identifies canonically with the identity of $\Gr_p(F)$, we conclude from  \cref{cor:section_Gr_conservative} by conservativity of $\Gr_p$.
\end{proof}
\personal{We could axiomatize the situation by replacing poset finite by Grp is conservative}

\newpage

\begin{appendices}

\section{Compactness results for $\infty$-categories}

This part is to be understood as an appendix, collecting auxiliary results needed in the main body, mostly of categorical flavor.
At the same time, we use in a couple of points the language of the specialization equivalence that has been developed in \cref{sec:specialization_equivalence} to obtain important structural results for cocartesian fibrations, that are interesting in their own right.
See in particular \cref{compact_cocartesian_fibration}, \cref{prop:pullback_localization} and \cref{cor:relative_tensor_product}.

\subsection{Compactness in the unstable setting}

Inspired by the usual terminology in non-commutative geometry (see e.g.\ \cite[Chapter 11]{Lurie_SAG}), we introduce:

\begin{defin}\label{proper_category}
	We say that an $\infty$-category $\cC$ is
	\begin{enumerate}\itemsep=0.2cm
		\item \emph{compact} if it is a compact object in $\Cat_\infty$;
		
		\item \emph{proper} if for every $c,c'\in \cC$, the mapping space $\Map_{\cC}(c,c')$ is a compact object in $\Spc$.
	\end{enumerate}
\end{defin}

The first goal of this section is to prove the following:

\begin{thm}\label{compact_cocartesian_fibration}
	Let $\cX$ be an $\infty$-category and let $\cA \to \cX$ be a cocartesian fibration.
	Assume that $\cX$ is compact and that for every $x \in \cX$, the fiber $\cA_x$ is compact in $\Cat_\infty$.
	Then $\cA$ is compact in $\Cat_\infty$ as well.
\end{thm}

\begin{rem}
	See \cite[Remark 6.5.4]{Hermitian_K_theory_I} for an analogous statement for finite $\infty$-categories instead of compact ones.
\end{rem}

The proof will use the specialization equivalence.
Before giving it, we need a couple of preliminaries.

\begin{lem}\label{lem:compact_categories_product}
	\hfill
	\begin{enumerate}\itemsep=0.2cm
		\item Compact objects in $\Cat_\infty$ are closed under finite products.
		
		\item An $\infty$-category $\cX \in \Cat_\infty$ is compact if and only if for every filtered diagram $\cC_\bullet \colon I \to \Cat_\infty$ with colimit $\cC$, the canonical map
		\begin{equation}\label{eq:strong_compactness}
			\colim_i \Fun(\cX, \cC_i) \to  \Fun(\cX, \cC)
		\end{equation}
		is an equivalence in $\Cat_\infty$.
	\end{enumerate}
\end{lem}

\begin{proof}
	First we prove (1).
	Fix therefore two compact $\infty$-categories $\cX$ and $\cY$.
	We can suppose that $\cX$ and $\cY$ are retract of finite $\infty$-categories $\cX'$ and that $\cY'$, respectively.
	Then $\cX \times \cY$ is a retract of $\cX' \times \cY$, which in turn is a retract of $\cX' \times \cY'$.
	It is therefore sufficient to prove that the latter is again a finite $\infty$-category.
	This latter statement follows immediately from the fact that the products $\Delta^n \times \Delta^m$ are again finite.
	
	\medskip
	
	We now prove point (2). Since $\ast$ is compact, we see that the stated condition implies the compactness of $\cX$ by applying $\Map_{\Cat_\infty}(\ast, -)$ to \eqref{eq:strong_compactness}.
	As for the converse, since $\Cat_\infty$ is compactly generated by the standard simplexes and since $- \times \cX \dashv \Fun(\cX, -)$, it is in fact enough to prove that for every $[n] \in \mathbf \Delta$, the canonical map
	\[ \colim_i \Map_{\Cat_\infty}( \Delta^n \times \cX, \cC_i ) \to  \Map_{\Cat_\infty}( \Delta^n \times \cX, \cC ) \]
	is an equivalence.
	Since point (1) guarantees that $\Delta^n \times \cX$ is again compact, the conclusion follows.
\end{proof}

\begin{lem}\label{lem:filtered_colimits_cartesian_fibrations}
	Let $\cX$ be an $\infty$-category.
	Then:
	\begin{enumerate}\itemsep=0.2cm
		\item the forgetful functor
		\[ \mathrm U_{\cX} \colon \Cart_\cX \to \Cat_{\infty / \cX} \]
		commutes with filtered colimits;
		
		\item if $\cX$ is compact in $\Cat_\infty$, then the functor
		\[ \Sigma_\cX \colon \Cat_{\infty / \cX} \to  \Cat_\infty \]
		commutes with filtered colimits.
	\end{enumerate}
\end{lem}

\begin{proof}
	Notice that $\mathrm U_{\cX}$ is by definition faithful.
	Thus, to prove (1) it is enough to prove that for any filtered diagram $\cC_\bullet \colon I \to \Cart_\cX$, the following two statements hold:
	\begin{enumerate}[(i)]\itemsep=0.2cm
		\item the colimit $p \colon \cC \to \cX$ of $\mathrm U_\cX(\cC_\bullet) \colon I \to \Cat_{\infty / \cX}$ is a cartesian fibration;
		
		\item for every other cartesian fibration $q \colon \cD \to \cX$ equipped with a cone $f_\bullet \colon \cC_\bullet \to \cD$ in $\Cart_\cX$, the induced functor $f \colon \cC \to \cD$ preserves cartesian edges.
	\end{enumerate}
	For (i), it is enough to apply the definition.
	First, since the horns $\Lambda^n_i$ and the simplexes $\Delta^n$ are compact in $\Cat_\infty$, we see that inner fibrations are stable under filtered colimits.
	\personal{(Mauro) Notice that since $\sSet_{\mathrm{Joyal}}$ is a combinatorial and compactly generated model category, filtered colimits are automatically homotopy colimits. So we can reason with colimits in $\sSet$. I insist to leave this as a personal comment, because it's bad taste to recall that $\Cat_\infty$ is modeled over $\sSet$.}
	Second, write $\lambda_i \colon \cC_i \to \cC$ for the canonical maps.
	Since the diagram was filtered, we see that every object $c \in \cC$ is of the form $\lambda_i(c_i)$ for some $i \in I$ and some $c_i \in \cC_i$.
	Let $\alpha \colon x \to p(c)$ be a morphism in $\cX$.
	Since $p(c) \simeq p(\lambda_i(c_i)) \simeq p_i(c_i)$ and since $p_i$ is a cartesian fibration, we can find a $p_i$-cartesian lift $\beta_i \colon d_i \to c_i$ of $\alpha$ inside $\cC_i$.
	Set $d \coloneqq \lambda_i(d_i)$ and $\beta \coloneqq \lambda_i(\beta_i)$.
	We claim that $\beta$ is a $p$-cartesian lift of $\alpha$.
	To see this, for every $(j, u \colon i \to j) \in I_{i/}$, write $\lambda_u \colon \cC_i \to \cC_j$ for the induced functor.
	Consider then the following commutative square:
	\[ \begin{tikzcd}
		\colim_{(j,u) \in I_{i/}} \cC_{j /\lambda_u(\beta_i)} \arrow{r} \arrow{d} & \colim_{(j,u) \in I_{i/}} \big( \cC_{j /\lambda_u(d_i)} \times_{\cX_{/p(c)}} \cX_{/\alpha} \big) \arrow{d} \\
		\cC_{/\beta} \arrow{r} & \cC_{/d} \times_{\cX_{/p(c)}} \cX_{/\alpha} \ ,
	\end{tikzcd} \]
	where the colimits are computed in $\Cat_\infty$.
	Since $\lambda_u$ preserves cartesian edges, we see that the top horizontal map is an equivalence.
	It is therefore enough to prove that the vertical arrows are equivalence.
	Since the colimit is filtered, it commutes with fiber products, and therefore we are reduced to check that the canonical functors
	\[ \colim_{(j,u) \in I_{i/}} \cC_{j /\lambda_u(\beta_i)} \to  \cC_{/\beta} \qquad \text{and} \qquad \colim_{(j,u) \in I_{i/}} \cC_{j /\lambda_u(d_i)} \to  \cC_{/d} \]
	are equivalences.
	We deal with the one on the left, as the other follows by a similar argument.
	Since $\Cat_\infty$ is compactly generated by the standard simplexes, it is enough to prove that for every $\Delta^n$, the canonical map
	\[ \colim_{(j,u) \in I_{i/}} \Map_{\Cat_\infty}( \Delta^n, \cC_{j / \lambda_u(\beta_i)}) \to  \Map_{\Cat_\infty}( \Delta^n, \cC_{/\beta} ) \]
	is an equivalence.
	Unraveling the definition of the comma category and using the identification $\Delta^n \star \Delta^1 \simeq \Delta^{n+2}$, we see that this map is canonically identified with the upper left diagonal map in the following commutative cube:
	\[ \begin{tikzcd}[column sep=tiny]
		{} & \Map_\beta(\Delta^{n+2}, \cC) \arrow{rr} \arrow{dd} & & \Map(\Delta^{n+2}, \cC) \arrow{dd}{\ev_{n+1,n+2}} \\
		\colim_{(j,u) \in I_{i/}} \Map_{\lambda_u(\beta_i)}( \Delta^{n+2}, \cC_j ) \arrow[crossing over]{rr} \arrow{ur} \arrow{dd} & & \colim_{(j,u) \in I_{i/}} \Map(\Delta^{n+2}, \cC_j) \arrow{ur} \\
		{} & \ast \arrow{rr}[pos=0.45]{\beta} & & \Map( \Delta^1, \cC ) \\
		\ast \arrow[equal]{ur} \arrow{rr}{\lambda_u(\beta_j)} & & \colim_{(j,u) \in I_{i/}} \Map(\Delta^1, \cC_j) \arrow{ur} \arrow[leftarrow,crossing over]{uu}[swap,pos=0.65]{\ev_{n+1,n+2}} \ .
	\end{tikzcd} \]
	Notice that the front and the back squares are pullback by definition.
	It is therefore sufficient to check that the other diagonal maps are equivalences, and this follows directly from the fact that both $\Delta^{n+2}$ and $\Delta^1$ are compact in $\Cat_\infty$.
	This proves at the same time that $p \colon \cC \to \cX$ is a cartesian fibration, and that $p$-cartesian edges are exactly the morphisms of the form $\lambda_i(\beta_i)$ for some $p_i$-cartesian edge $\beta_i$ inside $\cC_i$.
	In particular, (ii) follows immediately.
	
	\medskip
	
	We now prove (2).
	Notice that $\Sigma_\cX$ is right adjoint to the functor $- \times \cX \colon \Cat_\infty \to \Cat_{\infty / \cX}$.
	It is therefore enough to verify that $- \times \cX$ commutes with compact objects.
	Recall from \cite[Lemma A.3.10]{Beyond_conicality} that an object in $\Cat_{\infty / \cX}$ is compact if and only if it is compact in $\Cat_\infty$ after forgetting the structural map to $\cX$.
	Since $\cX$ itself is compact, the conclusion follows from \cref{lem:compact_categories_product}-(1).
\end{proof}

We are now ready for:

\begin{proof}[Proof of \cref{compact_cocartesian_fibration}]
	Fix a filtered diagram $\cE_\bullet \colon I \to \Cat_\infty$ with colimit $\cE$.
	In virtue of \cref{lem:compact_categories_product}-(2), we have to prove that the canonical map
	\[ \colim_{I} \Fun(\cA, \cE_i) \to  \Fun(\cA, \cE) \]
	is an equivalence.
	Write $\Upsilon_{\cA}$ for the straightening of $\cA$ and recall from \cref{notation:fake_exponential} that we write $\cE^\cA_{\mathrm c}$ for the \emph{cartesian} fibration classifying the functor
	\[ \Fun(\Upsilon_\cA(-), \cE) \colon \cX\op \to  \Cat_\infty \ . \]
	We similarly define the cartesian fibrations $\cE^\cA_{i, \mathrm c}$.
	Consider the canonical map
	\[ \colim_I \cE^\cA_{i, \mathrm c} \to  \cE^\cA_{\mathrm c} \]
	in $\Cart_\cX$.
	To see that this map is an equivalence, it is enough to test that for each $x \in \cX$, the induced map between the fibers at $x$ is an equivalence.
	However, at the level of fibers at $x$, this map is canonically identified with
	\[ \colim_I \Fun(\cA_x, \cE_i) \to  \Fun(\cA_x, \cE) \ . \]
	Since $\cA_x$ is compact by assumption, we see \cref{lem:compact_categories_product}-(2) guarantees that this map is indeed an equivalence.
	
	\medskip
	
	We can now apply \cref{lem:filtered_colimits_cartesian_fibrations}-(1) to deduce that the canonical map
	\[ \colim_I \cE^\cA_{i, \mathrm c} \to  \cE^\cA_{\mathrm c} \]
	is an equivalence also when the colimit is computed in $\Cat_{\infty / \cX}$.
	At this point, the conclusion follows from the identifications
	\[ \Fun(\cA, \cE_i) \simeq \Sigma_\cX( \cE^\cA_{i, \mathrm c} ) \qquad \text{and} \qquad \Fun(\cA, \cE) \simeq \Sigma_\cX(\cE^\cA_{\mathrm c}) \ , \]
	and \cref{lem:filtered_colimits_cartesian_fibrations}-(2).
\end{proof}

\subsection{Compact and proper (co)limits}

One of the most fundamental results in category theory is the commutation of filtered colimits with finite limits in $\mathbf{Set}$ and in $\Spc$.
In fact, the finiteness condition can be relaxed, using various combinations of compactness and properness.

\begin{lem}\label{finite_colimit_and_limit_stable}
	Let $\cE$ be a stable complete and cocomplete $\infty$-category.
	Let $\cC$ be a compact $\infty$-category.
	Then:
	\begin{enumerate}\itemsep=0.2cm
		\item the functor $\colim_{\cC} \colon \Fun(\cC, \cE)\to \cE$ commutes with limits.
		\item the functor $\lim_{\cC} \colon \Fun(\cC, \cE)\to \cE$ commutes with colimits.
	\end{enumerate}
\end{lem}

\begin{proof}
	The two statements are dual to each other.
	It is therefore enough to prove the second.
	Because $\cE$ is stable, it is enough to prove that $\lim_\cC$ commutes with filtered colimits, for which we refer to \cite[Lemma 6.7.4]{Exodromy_coefficient}.
\end{proof}

\begin{lem}\label{compactness_Kan_extension_category}
	Let $f \colon \cA \to \cB$ be a functor between $\infty$-categories.
	Let $b \in \cB$. 
	Assume that $\cA$ is compact and that for every $b' \in \cB$, the mapping space $\Map_\cB(b,b')$ is compact.
	Then both $\cA \times_{\cB} \cB_{b/}$ and $\cA \times_\cB \cB_{/b}$ are compact.
\end{lem}

\begin{proof}
	Replacing $\cA$ and $\cB$ by $\cA\op$ and $\cB\op$ respectively we see that it is enough to argue that $\cA \times_\cB \cB_{b/}$ is compact.
	For this, observe first that since $\cB_{b/}\to \cB$ is a cocartesian fibration, the pullback $\cA\times_{\cB}\cB_{b/}\to \cA$ is a cocartesian fibration as well.
	Since $\cA$ is compact, we are left from \cref{compact_cocartesian_fibration} to show that the fibers of  $\cA\times_{\cB}\cB_{b/}\to \cA$ are compact, which holds by assumption on the mapping spaces of $\cB$.
\end{proof}

\begin{prop}\label{induction_limit_stable}
	Let $\cX$ be an $\infty$-category and let $p \colon \cA \to \cB$ be a morphism of cocartesian fibrations over $\cX$.
	Assume that for every $x \in \cX$, the $\infty$-category $\cA_x$ is compact and $\cB_x$ is proper.
	Let $\cE$ be a complete, cocomplete and stable $\infty$-category.
	Then the functor
	\[ p_! \colon \Fun(\cA, \cE) \to  \Fun(\cB, \cE) \]
	commutes with limits.
\end{prop}

\begin{proof}
	From \cref{cor:induction_specialization_Beck_Chevalley}, it is enough to treat the case where $\cX$ is a point.
	In that case for every $F \colon \cA\to \cE$ and every $b\in \cB$, we have by definition of left Kan extension
	\[ (p_!(F))(b)\simeq \displaystyle{\colim_{ \cA\times_{\cB}\cB_{/b}}} F|_{\cA\times_{\cB}\cB_{/b}} \]
	From \cref{compactness_Kan_extension_category}, the $\infty$-category $\cA\times_{\cB}\cB_{/b}$ is compact.
	Thus, \cref{induction_limit_stable} follows from \cref{finite_colimit_and_limit_stable} applied to $\cC = \cA\times_{\cB}\cB_{/b}$.
\end{proof}

\begin{rem}\label{induction_limit_stable_posets}
	The assumption on $\cB$ is always satisfied when the fibers of $\cB$ are posets.
\end{rem}

\section{Stability of localizations under cocartesian pullback}

In \cite[Proposition 2.1.4]{Hinich_DK_revisited}, Hinich proved that the pullback of a localization functor via a cocartesian fibration is again a localization functor.
The theory surrounding the specialization equivalence and cocartesian functors developed so far allows for a model-independent proof, which we now give.

\subsection{Preliminaries}

\begin{lem}\label{lem:g_local_via_specialization}
	Let $p \colon \cB \to \cY$ be a cocartesian fibration and let $\cE$ be a presentable $\infty$-category.
	Let $\gamma \colon x \to y$ be a morphism in $\cY$.
	Let $F \in \Fun(\cB_x, \cE)$ and $G \in \Fun(\cB_y, \cE)$, and let $\alpha \colon F \to G$ be a morphism in $\exp_\cE(\cB / \cY)$.
	The following statements are equivalent:
	\begin{enumerate}\itemsep=0.2cm
		\item for every $p$-cocartesian lift $\phi \colon a \to b$ of $\gamma$ in $\cB$, the induced morphism (see \cref{notation:specialization_on_morphisms})
		\[ \alpha(\phi) \colon F(a) \to  G(b) \]
		is an equivalence in $\cE$;
		
		\item $\alpha$ is a $p_\cE$-cartesian morphism in $\exp_\cE(\cB / \cY)$.
	\end{enumerate}
	In addition, $\alpha$ is an equivalence in $\exp_\cE(\cB / \cY)$ if and only if $\gamma$ is an equivalence and condition (1) holds.
\end{lem}

\begin{proof}
	Since $p_\cE \colon \exp_\cE(\cB / \cY) \to \cY$ is a cartesian fibration, a morphism $\alpha \colon F \to G$ in $\exp_\cE(\cB / \cY)$ is an equivalence if and only if it is $p_\cE$-cartesian and its image in $\cY$ is an equivalence.
	So the second half of the statement follows automatically from the equivalence between statements (1) and (2).
	Choose a factorization of $\alpha$ as
	\[ \begin{tikzcd}[column sep=small]
		F \arrow{rr}{\alpha} \arrow{dr}[swap]{\alpha_0} & & G \\
		& G' \arrow{ur}[swap]{\alpha_1}
	\end{tikzcd} \]
	where $\alpha_1$ is $p_\cE$-cartesian.
	Then as observed in \cref{notation:specialization_on_morphisms}, any $p$-cocartesian lift $\phi \colon a \to b$ of $\gamma$ induces via $\alpha_1$ an equivalence
	\[ \alpha_1(\phi) \colon G'(a) \simeq G(b) \ . \]
	It follows that condition (1) is equivalent to ask that for every $a \in \cB_x$ the morphism
	\[ \alpha_0(a) \colon F(a) \to  G'(a) \]
	is an equivalence in $\cE$.
	In turn, this condition is equivalent to ask that $\alpha_0$ is an equivalence in $\exp_\cE(\cB / \cY)$, and hence to condition (2).
\end{proof}

For later use, let us store the following consequence of \cref{lem:g_local_via_specialization}

\begin{cor}\label{Stokes_when_locally_constant}
	Let $p \colon \cA \to \cX$ be a locally constant cocartesian fibration (see \cref{def:locally_constant}).
	Let $\cE$ be a presentable $\infty$-category and let $F \colon \cA \to \cE$ be a cocartesian functor.
	Let $\sigma \colon \cX \to \cA$ be a cocartesian section.
	Then, $\sigma^*(F) \colon \cX\to \cE$ inverts every arrow of $\cX$.
\end{cor}

\begin{proof}
	Since $p \colon \cA \to \cX$ is locally constant, the same goes for the associated exponential fibration $p_\cE \colon \exp_\cE(\cA / \cX) \to \cX$.
	Fix a morphism $\gamma \colon x \to y$ in $\cX$, so that $\sigma(\gamma) \colon \sigma(x) \to \sigma(y)$ is a $p$-cocartesian lift of $\gamma$ in $\cA$.
	Choose a specialization morphism
	\[ \begin{tikzcd}[column sep=small]
		(\spe F)_x \arrow{r}{\beta} & G \arrow{r}{\alpha} & (\spe F)_y
	\end{tikzcd} \]
	for $F$ relative to $\gamma$.
	Then \cref{locally_constant} guarantees that $\beta$ is $p_\cE$-cartesian in $\exp_\cE(\cA / \cX)$.
	Thus, the result follows combining \cref{lem:g_local_via_specialization} and \cref{cor:specialization_on_morphisms}.
\end{proof}

\subsection{Hinich's theorem}

We are now ready for:

\begin{thm}[Hinich]\label{prop:pullback_localization}
	Let
	\[ \begin{tikzcd}
		\cA \arrow{r}{u} \arrow{d}{q} & \cB \arrow{d}{p} \\
		\cX \arrow{r}{f} & \cY
	\end{tikzcd} \]
	be a pullback square in $\Cat_\infty$, where $p$ is a cocartesian fibration.
	Assume that $f$ exhibits $\cY$ as a localization of $\cX$ at a collection of morphisms $W$.
	Then $u$ is a localization functor as well, and exhibits $\cB$ as localization of $\cA$ at the collection $W_\cA$ of cocartesian lifts of the arrows of $W$.
\end{thm}

\begin{proof}
	We apply the criterion given in \cite[Proposition 7.1.11]{Cisinski_Higher_Category}.
	To begin with, observe that if $\varphi \in W_\cA$ then $\varphi$ is $q$-cocartesian and therefore $u(\varphi)$ is $p$-cocartesian and lies over $f(q(\varphi))$ which is an equivalence in $\cX$ since $q(\varphi) \in W$.
	Thus $u(\varphi)$ must be an equivalence as well, i.e.\ $u$ inverts the arrows in $W_\cA$.
	
	\smallskip
	
	Next, $u$ is essentially surjective: indeed, if $b \in \cB$ is an element, we can find $x \in \cX$ and an equivalence $f(x) \simeq p(b)$, because $f$ is essentially surjective.
	But then $b$ defines an element in $\cB_{f(x)}$ and since the given square is a pullback, we have $\cB_{f(x)} \simeq \cA_x$.
	Thus, we can write $b \simeq u(a)$ for some $a \in \cA$.
	
	\smallskip
	
	Since a functor $g \colon \cC \to \cD$ is a localization if and only if $f\op \colon \cC\op \to \cD\op$ is a localization (see \cite[Proposition 7.1.7]{Cisinski_Higher_Category}), to complete the proof it is enough to prove that
	\[ u^\ast \colon \Fun(\cB, \Spc) \to  \Fun(\cA, \Spc) \]
	is fully faithful and the essential image consists of those functors $F \colon \cA \to \Spc$ that invert the arrows in $W_\cA$.
	We will more generally prove that this is the case for any presentable $\infty$-category $\cE$ in place of $\Spc$.
	\Cref{prop:global_functoriality}-(\ref{prop:global_functoriality:pullback}) allows to rewrite $u^\ast$ as
	\[ \Sigma(\cE^u) \colon \Fun_{/\cY}(\cY, \exp_\cE(\cB / \cY)) \to  \Fun_{/\cX}(\cX, \exp_\cE(\cA / \cY)) \ . \]
	In virtue of \cref{prop:functoriality_exponential}-(1), we can rewrite
	\[ \Fun_{/\cX}(\cX, \exp_\cE(\cA / \cY)) \simeq \Fun_{/\cY}(\cX, \exp_\cE(\cB / \cY)) \ , \]
	and under this identification $\Sigma(\cE^u)$ simply becomes
	\begin{equation}\label{eq:pullback_localization}
		f^\ast \colon \Fun_{/\cY}(\cY, \exp_\cE(\cB / \cY)) \to  \Fun_{/\cY}(\cX, \exp_\cE(\cB / \cY)) \ .
	\end{equation}
	Consider now the following commutative cube:
	\begin{equation}\label{eq:cocartesian_pullback_localization}
		\begin{tikzcd}
			{} & \ast \arrow{rr}{\id_\cY} \arrow[equal]{dd} & & \Fun(\cY, \cY) \arrow{dd}{f^\ast} \\
			\Fun_{/\cY}(\cY, \exp_\cE(\cB / \cY)) \arrow[crossing over]{rr} \arrow{dd}{f^\ast} \arrow{ur} & & \Fun(\cY, \exp_\cE(\cB / \cY)) \arrow{ur} \\
			{} & \ast \arrow{rr}[pos=0.4]{f} & & \Fun(\cX, \cY) \\
			\Fun_{/\cY}(\cX, \exp_\cE(\cB / \cY)) \arrow{rr} \arrow{ur} & & \Fun(\cX, \exp_\cE(\cB / \cY)) \arrow{ur} \arrow[leftarrow, crossing over]{uu}[swap,pos=0.7]{f^\ast} & \phantom{\Fun(\cX, \cY)} \ .
		\end{tikzcd}
	\end{equation}
	The bottom and the top squares are pullbacks by definition.
	Since $f$ is a localization, the functor
	\[ f^\ast \colon \Fun(\cY, \cY) \to \Fun(\cX, \cY) \]
	is fully faithful, which implies that the back square is a pullback as well.
	Thus, the front square is a pullback as well, and therefore the full faithfulness of \eqref{eq:pullback_localization} follows from the full faithfulness of
	\[ f^\ast \colon \Fun(\cY, \exp_\cE( \cB / \cY ) ) \to  \Fun(\cX, \exp_\cE( \cB / \cY ) ) \ , \]
	which holds because $f$ is a localization.
	
	\smallskip
	
	Since the front square is a pullback, we also deduce that a section $s \in \Fun_{/\cY}(\cX, \exp_\cE(\cA / \cY))$ lies in the essential image of $f^\ast$ if and only if it inverts all arrows in $W$.
	Via the specialization equivalence of \cref{prop:sections_exponential}, we deduce that a functor $F \in \Fun(\cB, \cE)$ lies in the essential image of $u^\ast$ if and only if $\cE^u \circ (\spe F) \colon \cX \to \exp_\cE(\cB / \cY)$ inverts all arrows in $W$.
	Fix $\gamma \colon x \to y$ in $W$.
	By assumption $f(\gamma)$ is an equivalence in $\cY$, so \cref{lem:g_local_via_specialization} shows that $\cE^u \circ (\spe F)$ inverts $\gamma$ if and only if
	\[ \big( \cE^u( \spe F ) \big)_\gamma \colon \big(\cE^u (\spe F)\big)_x \to  \big( \cE^u( \spe F ) \big)_y \]
	is $p_\cE$-cartesian in $\exp_\cE(\cB / \cY)$.
	Since $p_\cE \colon \exp_\cE(\cB / \cY) \to \cY$ is a cartesian fibration, it is actually enough to check that the above morphism is locally cartesian.
	Therefore, we can replace $\cB \to \cY$ by $\cB_{f(\gamma)} \to \Delta^1$, and since $\cB_{f(\gamma)} \simeq \cA_\gamma$, \cref{lem:g_local_via_specialization} further shows that it is enough to check that for every $q$-cocartesian lift $\phi \colon a \to a'$ of $\gamma$ in $\cA$, the morphism
	\[ (\spe F)_\gamma(\phi) \colon (\spe F)_x(a) \to  (\spe F)_y(a') \]
	is an equivalence in $\cE$.
	However, \cref{cor:specialization_on_morphisms} provides a canonical identification of this morphism with $F(\phi)$.
	In other words, $\cE^u \circ (\spe F)$ inverts $\gamma$ if and only if $F$ inverts all $q$-cocartesian lifts of $\gamma$.
	The conclusion follows.
\end{proof}

\section{Locally constant and finite étale fibrations}\label{sec:locally_constant_cocartesian_fibrations}

We collect in this section some material on cocartesian fibrations that generalize the idea of local constancy and finite covering in topology.

\subsection{Local constancy}

We start with the following definition:

\begin{defin}\label{def:abstract_local_systems}
	Let $\cX$ be an $\infty$-category and let $\cE$ be a presentable $\infty$-category.
	We write
	\[ \Loc(\cX; \cE) \coloneqq \Fun(\Env(\cX), \cE) \ . \]
\end{defin}

\begin{eg}
	Let $(X,P)$ be an exodromic stratified space.
	Then \cite[Theorem 0.3.1]{Beyond_conicality}
 implies that $\Env(\Pi_\infty(X,P)) \simeq \Pi_\infty(X)$.
	Therefore, $\Loc(\Pi_\infty(X,P);\cE)$ correspond via the exodromy equivalence exactly to $\cE$-valued hyperconstructible hypersheaves on $X$.
\end{eg}

\begin{notation}\label{notation:abstrat_local_systems_as_functors}
	Let $\cX$ be an $\infty$-category and let $\lambda_\cX \colon \cX \to \Env(\cX)$ be the canonical localization morphism.
	Then for every presentable $\infty$-category $\cE$, the functor
	\[ \lambda_\cX^\ast \colon \Loc(\cX;\cE) \to \Fun(\cX, \cE) \]
	is fully faithful.
	Given $L \in \Loc(\cX; \cE)$ we will often consider it implicitly as a functor $L \colon \cX \to \cE$ with the property of inverting every arrow in $\cX$.
\end{notation}

\begin{defin}\label{def:locally_constant}
	We say that a functor $p \colon \cA \to \cX$ of $\infty$-categories is a \emph{locally constant fibration} if it is a cocartesian fibration and its straightening $\Upsilon \colon \cX \to \CAT_\infty$ belongs to $\Loc(\cX;\CAT_\infty)$.
\end{defin}

The following simply follows unraveling the definitions:

\begin{lem}\label{lem:locally_constant_pullback}
	Locally constant fibrations are stable under pullback.
\end{lem}

It is possible to give a more intrinsic formulation of locally constant cocartesian fibrations as follows.

\begin{recollection}\label{recollection:unit_counit}
	Let $p \colon \cA \to \Delta^1$ be a cartesian and cocartesian fibration and let
	\[ f \colon \cA_0 \leftrightarrows \cA_1 \colon g \]
	be the induced adjunction.
	Write $\eta \colon \id_{\cA_0} \to g \circ f$ and $\varepsilon \colon f \circ g \to \id_{\cA_1}$ for the unit and the counit of this adjunction.
	It follows from \cite[Proposition 5.2.2.8]{HTT} that for every morphism $\phi \colon a \to b$ in $\cA$ lying over $0 \to 1$ in $\Delta^1$, there is a commutative diagram in $\cA$
	\begin{equation}\label{eq:unit_counit}
		\begin{tikzcd}
			{} & g(f(a)) \arrow{r}{\alpha} & f(a) \arrow{dr} \\
			a \arrow{ur}{\eta_a} \arrow{dr} \arrow{rrr}{\phi} & & & b \\
			{} & g(b) \arrow{r}{\beta} & f(g(b)) \arrow{ur}[swap]{\varepsilon_b}
		\end{tikzcd}
	\end{equation}
	where:
	\begin{enumerate}\itemsep=0.2cm
		\item $\alpha$ and $\varepsilon_b \circ \beta$ are $p$-cartesian;
		
		\item $\beta$ and $\alpha \circ \eta_b$ are $p$-cocartesian.
	\end{enumerate}
\end{recollection}

\begin{prop}\label{locally_constant}
	Let $p \colon \cA \to \cX$ be a cocartesian fibration and let $\Upsilon \colon \cX \to \CAT_\infty$ be its straightening.
	For every morphism $\gamma \colon x \to y$ in $\cX$, the following statements are equivalent:
	\begin{enumerate}\itemsep=0.2cm
		\item $p_{\gamma} \colon \cA_{\gamma}\to \Delta^1$ (see \cref{notation:simplexes_cocartesian_fibration}) is a cartesian fibration and an arrow in $\cA_{\gamma}$ is cocartesian if and only it is cartesian;
		
		\item $\Upsilon(\gamma) \colon \Upsilon(x) \to \Upsilon(y)$ is an equivalence in $\CAT_{\infty}$;
	\end{enumerate}
	In particular, $p$ is locally constant if and only if condition (1) holds for every morphism $\gamma$ in $\cX$.
\end{prop}

\begin{proof}
	Assume first that (1) holds.
	Since $p_\gamma$ is both Cartesian and coCartesian the functor $\Phi(\gamma)$ admits a right adjoint $R(\gamma) \colon \Upsilon(y) \to \Upsilon(x)$.
	Then \cref{recollection:unit_counit} implies that in diagram \eqref{eq:unit_counit} both $\alpha$ and $\alpha \circ \eta_a$ are $p$-cartesian lifts of $\gamma$, so the universal property of $p$-cartesian edges implies that $\eta_a$ must be an equivalence.
	The dual argument shows that $\varepsilon_b$ is an equivalence as well.
	It follows that $\Upsilon(\gamma)$ is an equivalence.
	
	Suppose conversely that $\Phi(\gamma)$ is an equivalence.
	Then it admits a right adjoint, which in turn implies that $p_{\gamma}$ is a cartesian fibration.
	Then in \cref{recollection:unit_counit} both $\eta$ and $\varepsilon$ are equivalences.
	It immediately follows that the cocartesian lift $a \to f(a)$ is also cocartesian, and that the cocartesian lift $g(b) \to b$ is also cartesian, whence the conclusion.
	\personal{(Other argument)
		We prove that unit and counit are equivalences.
		Let $a \in \Phi(x)$ be an object.
		Let $f \colon a \to b$ be a coCartesian lift of $\gamma$ starting at $a$ and let $g \colon a' \to b$ be a Cartesian lift of $\gamma$ ending at $b$.
		Then we have the following commutative diagram:
		\begin{equation} \label{eq:locally_trivial_fibrations}
			\begin{tikzcd}[row sep = small,ampersand replacement=\&]
				a \arrow{dr}{f} \arrow{dd}{\eta} \\
				{} \& b \\
				a' \arrow{ur}[swap]{g} \& \phantom{b} .
			\end{tikzcd}
		\end{equation}
		Then we have a contractible space of homotopies $b \simeq \Phi(\gamma)(a)$ and $a' \simeq R(\gamma)(b)$, and $\eta$ is a unit transformation for the adjunction $\Phi(\gamma) \dashv R(\gamma)$ at $a$.
		Since $f$ is coCartesian, the assumption implies that it is Cartesian as well.
		Therefore, $\eta$ is an equivalence.
		The dual argument shows that the counit must be an equivalence as well, and hence that $\Phi(\gamma)$ is an equivalence.
		Suppose conversely that $\Phi(\gamma)$ is an equivalence.
		Then it admits a right adjoint, which in turn implies that $p_{\gamma}$ is a cartesian fibration.
		Now let $f \colon a \to b$ be a cocartesian morphism in $\cA_{\gamma}$.
		Choose a cartesian lift $g \colon b \to a'$.
		We obtain once more the diagram \eqref{eq:locally_trivial_fibrations}.
		Since $\Phi(\gamma)$ is an equivalence, its unit is an equivalence, and therefore $f$ is equivalent to $g$, showing that $f$ is cartesian as well.
		The dual argument implies that if $f$ is a cartesian arrow, then it must be cocartesian as well.
		The proof is therefore complete.}
\end{proof}

\subsection{Finite étale fibrations}

We now introduce the following abstract formulation of the notion of finite covering in topology:

\begin{defin}\label{def:finite_etale_fibrations}
	We say that a cocartesian fibration between $\infty$-categories $f \colon \cY \to \cX$ is a \emph{finite étale fibration} if:
	\begin{enumerate}\itemsep=0.2cm
		\item it is locally constant;
		\item it is a cartesian fibration;
		\item the fibers of $f$ are finite sets.
	\end{enumerate}
\end{defin}

\begin{lem}\label{lem:finite_etale_fibrations_pullback}
	Finite étale fibrations are closed under pullback.
\end{lem}

Finite étale fibrations satisfy another important stability property, that we are going to explain now.

\begin{construction}
	Let $f \colon \cX \to \cY$ be a functor of small $\infty$-categories.
	Recall from \cref{recollection:pushforward_fibrations} the adjunction
	\[ \fcocartlowershriek \colon \CoCart_\cX \leftrightarrows \CoCart_\cY \colon f^\ast \ . \]
	Evaluating the unit of this adjunction on a cocartesian fibration $p \colon \cA \to \cX$, we obtain the following commutative square:\\
	\begin{equation}\label{eq:exceptional_pushforward_square}
		\begin{tikzcd}
			\cA \arrow{r}{f_\cA} \arrow{d}{p} & \fcocartlowershriek(\cA) \arrow{d}{q} \\
			\cX \arrow{r}{f} & \cY \ .
		\end{tikzcd}
	\end{equation}
	When $f$ is a localization, $f^\ast \colon \CoCart_\cY \to \CoCart_\cX$ is fully faithful.
	In this case, the counit $\fcocartlowershriek(\cX) \simeq \fcocartlowershriek(f^\ast(\cY)) \to \cY$ is an equivalence.
	Therefore, in this case, the structural map $q \colon \fcocartlowershriek(\cA) \to \cY$ is canonically identified with $\fcocartlowershriek(p)$.
\end{construction}

\begin{lem}\label{exceptional_pushforward_square_lem}
	Assume that $f$ exhibits $\cY$ as the localization of $\cX$ at a class of morphisms $W$.
	Let $\Upsilon_\cA \colon \cX\to \Cat_\infty$ be the straightening of $p \colon \cA \to \cX$.
	Then, the following are equivalent:
	\begin{enumerate}\itemsep=0.2cm
		\item the square \eqref{eq:exceptional_pushforward_square} is a pullback;
		
		\item the functor $\Upsilon_\cA \colon \cX\to \Cat_\infty$ maps $W$ to equivalences;
		
		\item For every $\gamma \in W$, the pullback $p_\gamma \colon \cA_\gamma \to \Delta^1$ (see \cref{notation:simplexes_cocartesian_fibration}) is a cartesian fibration and an arrow in $\cA_{\gamma}$ is cocartesian if and only it is cartesian.
	\end{enumerate}
\end{lem}

\begin{proof}
The equivalence between (1) and (2) follows from the universal property of the localization.
The equivalence between (2) and (3) follows from \cref{locally_constant}.
\end{proof}

\begin{cor}\label{finite_etaleness_descent_along_localization}
	Let $p \colon \cA \to \cX$ be a cocartesian fibration between $\infty$-categories.
	Let $f \colon \cX \to \cY$ be a functor exhibiting $\cY$ as the localization of $\cX$ at a class of morphisms $W$.
	Then, the following are equivalent:
	\begin{enumerate}\itemsep=0.2cm
		\item $p \colon \cA \to \cX$ is a finite étale fibration;
		
		\item the square \eqref{eq:exceptional_pushforward_square} is a pullback and $\fcocartlowershriek(p) \colon \fcocartlowershriek(\cA) \to \cY$ is a finite étale fibration.
	\end{enumerate}
	If these conditions are satisfied, the functor $f_\cA \colon \cA \to  \fcocartlowershriek(\cA)$ exhibits $\fcocartlowershriek(\cA)$ as the localization of $\cA$ at every morphism above $W$.
\end{cor}

\begin{proof}
	That (2) implies (1) follows from the preservation of finite étale fibrations under pullback from \cref{lem:finite_etale_fibrations_pullback}.
	Assume that (1) holds.
	Let $\lambda_{\cY} \colon \cY \to \Env(\cY)\simeq \Env(\cX)$ be the localization at every morphism.
	Since $p \colon \cA \to \cX$ is locally constant,  \cref{exceptional_pushforward_square_lem}-(2) is satisfied both for $(p,W)$ and $(\fcocartlowershriek(p),\Mor(\cY))$.
	Hence, there is a commutative diagram
	\[ \begin{tikzcd}
		\cA \arrow{r}\arrow{d}{p} & \fcocartlowershriek(\cA)  \arrow{d}{\fcocartlowershriek(p)} \arrow{r} & \lambda_{\cX,!}^{\mathrm{cc}}(\cA) \arrow{d}{\lambda_{\cX,!}^{\mathrm{cc}}(p)} \\
		\cX \arrow{r}{f} & \cY \arrow{r}{\lambda_{\cY}} & \Env(\cX)
	\end{tikzcd} \]
	whose squares are pullback squares.
	By \cref{lem:finite_etale_fibrations_pullback}, we are thus left to show that
	\begin{equation}\label{eq:finite_etaleness_descent_along_localization}
		\lambda_{\cX,!}^{\mathrm{cc}}(p) \colon \lambda_{\cX,!}^{\mathrm{cc}}(\cA) \to \Env(\cX)
	\end{equation} 
	is a finite étale fibration.
	Since the outer square is a pullback, the fibres of \eqref{eq:finite_etaleness_descent_along_localization} are finite sets.
	Local constancy is obvious since $\Env(\cX)$ is an $\infty$-groupoid.
	Note that  \eqref{eq:finite_etaleness_descent_along_localization} is an inner fibration as it is cocartesian. 
	To show that it is cartesian, it is enough to show \cite[Proposition 2.4.1.5]{HTT} that $\lambda_{\cX,!}^{\mathrm{cc}}(\cA)$ is an $\infty$-groupoid.
	To do this, it is enough to show that $\cA  \to \lambda_{\cX,!}^{\mathrm{cc}}(\cA)$ exhibits $\lambda_{\cX,!}^{\mathrm{cc}}(\cA)$ as the localization of $\cA$ at every morphism.
	Hence, we are left to show more generally that $\cA \to  \fcocartlowershriek(\cA)$ exhibits $\fcocartlowershriek(\cA)$ as the localization of $\cA$ at every morphism above a morphism of $W$.
	By \cref{prop:pullback_localization}, it is enough to show that every morphism in $\cA$ is $p$-cocartesian.
	This follows immediately from the fact that the fibers of $p \colon \cA \to \cX$ are discrete.
\end{proof}

\begin{cor}
	Let $f \colon \cX \to \cY$ be a localization functor.
	Then the adjunction
	\[ \fcocartlowershriek \colon \CoCart_\cX \leftrightarrows \CoCart_\cY \colon f^\ast \]
	restricts to an equivalence between the $\infty$-subcategories spanned by finite étale fibrations.
\end{cor}

\begin{proof}
	If $p \colon \cA \to \cX$ is a finite étale fibration, then so is $\fcocartlowershriek(p) \colon \fcocartlowershriek(\cA) \to \cY$ in virtue \cref{finite_etaleness_descent_along_localization} and the unit of $\fcocartlowershriek \dashv f^\ast$  applied to $p \colon \cA \to \cX$  is an equivalence.
	If $p \colon \cB \to \cY$  is a finite étale fibration, then so is $f^*(p) \colon f^*(\cA) \to \cX$ by \cref{lem:finite_etale_fibrations_pullback}.
	Since $f \colon \cX \to \cY$ is a localization, the counit of $f_!  \dashv f^\ast$ applied to $p \colon \cB \to \cY$  is automatically an equivalence.
\end{proof}

The link with topological covering maps is expressed by the following:

\begin{lem}\label{lem:finite_etale_fibrations_example}
Let $(X,P)$ be a stratified space and let $f \colon Y \to X$ be a continuous morphism.
	Assume that:
	\begin{enumerate}\itemsep=0.2cm
		\item $f \colon Y \to X$ is a finite covering map;
		\item there is a refinement $R\to P$ such that $(X,R)$ is conical with locally weakly contractible strata.
	\end{enumerate}
	Then $(Y,P)$ admits a conical refinement with locally weakly contractible strata and the induced map
	\[ \Pi_\infty(Y,P) \to \Pi_\infty(X,P) \]
	is a finite étale fibration.
\end{lem}

\begin{proof}
    Since $f$ is a local homeomorphism,  $(Y,R)$ is also conical with locally weakly contractible strata.
    Therefore, there is a commutative diagram 
    	\[ \begin{tikzcd}
		\Pi_\infty(Y,R) \arrow{r} \arrow{d}& \Pi_\infty(Y,P) \arrow{d}\\
		\Pi_\infty(X,R) \arrow{r}{r^{\mathrm{cc}}_!} & \Pi_\infty(X,P) 
	\end{tikzcd} \]
	in $\Cat_{\infty}$.
    Assume that the left arrow is a finite étale fibration.
    By \cref{finite_etaleness_descent_along_localization} we deduce the existence of a pullback square of finite étale fibrations 
    	\[ \begin{tikzcd}
		\Pi_\infty(Y,R) \arrow{r} \arrow{d}& r^{\mathrm{cc}}_!(\Pi_\infty(Y,R))    \arrow{d} \\
		\Pi_\infty(X,R) \arrow{r}{r^{\mathrm{cc}}_!} & \Pi_\infty(X,P) 
	\end{tikzcd} \]
    such that the top arrow exhibits  $r_!(\Pi_\infty(Y,R))$  as the localization of $\Pi_\infty(Y,R)$ at every arrow above
    an equivalence of $P$. 
    By \cite[Theorem 0.3.1]{Beyond_conicality}, we deduce the existence of a canonical equivalence
\[
r^{\mathrm{cc}}_!(\Pi_\infty(Y,P))\simeq  \Pi_\infty(Y,P) \ .
\]    
    Hence,  $\Pi_\infty(Y,P) \to \Pi_\infty(X,P)$ is a finite étale fibration.
    Thus, we are left to prove \cref{lem:finite_etale_fibrations_example} in the case where $(X,P)$ is conically stratified.
    In that case, so is $(Y,P)$.
	Therefore, we have the following pullback square of simplicial sets
	\[ \begin{tikzcd}
		\Sing_Q(Y) \arrow{r} \arrow{d} & \Sing(Y) \arrow{d} \\
		\Sing_P(X) \arrow{r} & \Sing(X) \ .
	\end{tikzcd} \]
	Since $f$ is a covering map, it is in particular a Serre fibration.
	Therefore, $\Sing(Y) \to \Sing(X)$ is a Kan fibration.
	It follows that the above square is a homotopy pullback, and therefore that
	\[ \begin{tikzcd}
		\Pi_\infty(Y,Q) \arrow{r} \arrow{d} & \Pi_\infty(Y) \arrow{d} \\
		\Pi_\infty(X,P) \arrow{r} & \Pi_\infty(X)
	\end{tikzcd} \]
	is a pullback in $\Cat_\infty$.
	By \cref{lem:finite_etale_fibrations_pullback}, we are left to prove \cref{lem:finite_etale_fibrations_example} when $P = \ast$ is the trivial stratification.
	We know that $\Sing(Y) \to \Sing(X)$ is a Kan fibration, so that $\Pi_\infty(Y) \to \Pi_\infty(X)$ is both a left and a right fibration.
	Since the base is an $\infty$-groupoid, it follows that it is locally constant in the sense of \cref{def:locally_constant}.
	Besides, for $x \in X$ we have a pullback
	\[ \begin{tikzcd}
		\Sing(Y_x) \arrow{r} \arrow{d} & \Sing(Y) \arrow{d} \\
		\{x\} \arrow{r} & \Sing(X)
	\end{tikzcd} \]
	of simplicial sets.
	Since the right vertical map is a Kan fibration, we deduce that it is a homotopy pullback, i.e.\ that
	\[ \{x\} \times_{\Pi_\infty(X)} \Pi_\infty(Y) \simeq \Pi_\infty(Y_x) \ . \]
	Since $f$ is a finite covering map, $Y_x$ is discrete, whence the conclusion.
\end{proof}

\section{Categorical actions}\label{sec:categorical_actions}

We collect some material on $\infty$-categorical actions that is needed throughout the text.

\subsection{Generalities}

We refer to \cite[\S4.8.1]{Lurie_Higher_algebra} for the theory of tensor products of presentable $\infty$-categories, that endows $\PrL$ with a symmetric monoidal structure $\PrLotimes$.
Fix an object $\cE^\otimes \in \CAlg(\PrLotimes)$.
We refer to $\cE^\otimes$ as a \emph{presentably symmetric monoidal $\infty$-category}.
In particular, we have an underlying tensor product
\[ \otimes_\cE \colon \cE \times \cE \to \cE \]
commuting with colimits in both variables and a tensor unit $\bbI_\cE \in \cE$.
We refer to an object in $\PrL_\cE \coloneqq \Mod_{\cE^\otimes}(\PrLotimes)$ as an \emph{$\infty$-categorical module over $\cE^\otimes$}.
Ignoring homotopy coherences, such an object can informally be described as an $\infty$-category $\cD$ equipped with an external tensor product
\[ \otimes \colon \cE \times \cD \to \cD \]
that commutes with colimits in both variables and that satisfies the usual module relations.
In particular, $\bbI_\cE \otimes (-) \colon \cD \to \cD$ comes with an identification with $\id_\cD$.
Similarly, a morphism $f \colon \cD \to \cD'$ of $\infty$-categorical $\cE^\otimes$-modules can informally be described as a functor $f$ equipped with homotopy coherent identifications
\[ f( E \otimes D ) \simeq E \otimes f(D) \ , \]
for $E \in \cE$ and $D \in \cD$.
Finally, \cite[Theorem 4.5.2.1]{Lurie_Higher_algebra} supplies $\PrL_\cE$ with an induced symmetric monoidal structure $\PrLotimes_\cE$.
In particular, given two $\infty$-categorical $\cE^\otimes$-modules $\cD$ and $\cD'$, we can form the relative tensor product
\[ \cD \otimes_\cE \cD' \in \PrLotimes_\cE \ . \]

\begin{recollection}\label{recollection:algebras_to_modules}
	It follows from \cite[Corollary 3.4.1.7]{Lurie_Higher_algebra} that a symmetric monoidal functor $f^\otimes \colon \cE^\otimes \to \cD^\otimes$ allows to see $\cD^\otimes$ as a $\infty$-categorical module over $\cE^\otimes$.
	The underlying tensor product is then informally defined as
	\[ E \otimes D \coloneqq f(E) \otimes_\cD D \ . \]
	Similarly, if
	\[ \begin{tikzcd}[column sep=small]
		{} & \cE^\otimes \arrow{dl}[swap]{f^\otimes} \arrow{dr}{g^\otimes} \\
		\cD^\otimes \arrow{rr}{h^\otimes} & & \cD^{\prime\otimes}
	\end{tikzcd} \]
	is a commutative triangle in $\CAlg(\PrLotimes)$, then $h \colon \cD \to \cD'$ inherits the structure of a $\cE$-linear functor.
\end{recollection}

\begin{recollection}\label{recollection:monoidal_structure_functor_category}
	Let $\cE^\otimes$ be a presentably symmetric monoidal $\infty$-category.
	It follows from \cite[Remark 2.1.3.4]{Lurie_Higher_algebra} that for every (small) $\infty$-category $\cA$, $\Fun(\cA, \cE)$ inherits a symmetric monoidal structure, that we denote $\Fun(\cA, \cE)^\otimes$.
	Informally speaking, given two functors $F, G \colon \cA \to \cE$, their tensor product is defined by the rule
	\[ (F \otimes G)(a) \coloneqq F(a) \otimes_\cE G(a) \ . \]
	Similarly, if $f \colon \cB \to \cA$ is a functor of $\infty$-categories, then
	\[ f^\ast \colon \Fun(\cA, \cE) \to \Fun(\cB, \cE) \]
	inherits a canonical symmetric monoidal structure.
\end{recollection}

\begin{lem}\label{lem:linearity_of_induction}
	Let $\cE^\otimes$ be a presentably symmetric monoidal $\infty$-category and let $f \colon \cA \to \cB$ be a cocartesian fibration.
	Reviewing $\Fun(\cB, \cE)$ as a $\Fun(\cA, \cE)^\otimes$-module via Recollections \ref{recollection:algebras_to_modules} and \ref{recollection:monoidal_structure_functor_category}, the left Kan extension functor
	\[ f_! \colon \Fun(\cB, \cE) \to \Fun(\cA, \cE) \]
	is $\Fun(\cA, \cE)^\otimes$-linear.
\end{lem}

\begin{proof}
	It follows from \cite[Proposition 2.5.5.1]{Lurie_SAG} that $f_!$ is an oplax symmetric monoidal functor when we see both $\Fun(\cB, \cE)$ and $\Fun(\cA, \cE)$ as symmetric monoidal $\infty$-categories.
	Using \cite[Corollary 3.4.1.5]{Lurie_Higher_algebra}, we reduce ourselves to check that for every $F \in \Fun(\cA, \cE)$ and every $G \in \Fun(\cB, \cE)$, the canonical map
	\[ f_!( f^\ast(F) \otimes G ) \to F \otimes f_!(G) \]
	is an equivalence.
	\personal{The use of \cite[Corollary 3.4.1.5]{Lurie_Higher_algebra} is a bit rough. One should add some operadic details here.}
	Since the tensor product of $\cE$ commutes with colimits in both variables, this follows from the formula for left Kan extensions provided by the dual of \cite[\cref*{exodromy-lem:RKE}]{Exodromy_coefficient}.
\end{proof}

\subsection{Universal monadicity for finite étale fibrations}

To motivate the results of this section, consider the following:

\begin{construction}\label{construction:relative_tensor_of_functor_categories}
	Fix a presentably symmetric monoidal $\infty$-category $\cE^\otimes$ and let
	\[ \begin{tikzcd}
		\cB \arrow{r}{u} \arrow{d}{q} & \cA \arrow{d}{p} \\
		\cY \arrow{r}{f} & \cX
	\end{tikzcd} \]
	be a pullback square in $\Cat_\infty$.
	Via \cref{recollection:monoidal_structure_functor_category}, we obtain a commutative square
	\[ \begin{tikzcd}
		\Fun(\cX, \cE)^\otimes \arrow{r}{f^\ast} \arrow{d}{p^\ast} & \Fun(\cY,\cE)^\otimes \arrow{d}{q^\ast} \\
		\Fun(\cA, \cE)^\otimes \arrow{r}{u^\ast} & \Fun(\cB, \cE) \ .
	\end{tikzcd} \]
	Combining \cite[Theorem 4.5.2.1 and Proposition 3.2.4.7]{Lurie_Higher_algebra}, we obtain a canonical comparison map
	\begin{equation}\label{eq:relative_tensor_of_functor_categories}
		\mu \colon \Fun(\cY, \cE) \otimes_{\Fun(\cX, \cE)} \Fun(\cA, \cE) \to \Fun(\cB, \cE) \ .
	\end{equation}
\end{construction}

\begin{warning}
	When $\cX = \ast$, the comparison map \eqref{eq:relative_tensor_of_functor_categories} is an equivalence.
	If both $f$ and $p$ are cocartesian fibrations, one can easily prove that inside $\Mod_{\Triv_\cX(\cE^\otimes)}(\PrFibLotimes_\cX)$ there is a canonical equivalence
	\[ \exp_\cE( \cY / \cX ) \otimes_{\Triv_\cX(\cE)} \exp_\cE( \cA / \cX ) \simeq \exp_\cE( \cB / \cX ) \ . \]
	However, the global section functor
	\[ \Sigma_\cX \colon \Mod_{\Triv_\cX(\cE^\otimes)}(\PrFibLotimes_\cX) \to \Mod_{\Fun(\cX, \cE)}( \PrLotimes ) \]
	is only lax monoidal.
	Because of this, the functor \eqref{eq:relative_tensor_of_functor_categories} is typically not an equivalence.
\end{warning}

The goal of this section is to show that the situation gets considerably better if $f$ is assumed to be a finite étale fibration and $\cE$ to be stable.
We start introducing some terminology:

\begin{defin}
	Let $f \colon \cC \to \cD$ and $g \colon \cD \to \cC$ be functors between $\infty$-categories.
	We say that $f$ and $g$ are \emph{biadjoints} if the adjunctions $f \dashv g$ and $g \dashv f$ hold.
\end{defin}

\begin{lem}\label{lem:finite_etale_fibration_biadjoint}
	Let $f \colon \cY \to \cX$ be a finite étale fibration and let $\cE$ be stable presentable $\infty$-category.
	Then the functors
	\[ f_! \colon \Fun(\cY, \cE) \to \Fun(\cX, \cE) \qquad \text{and} \qquad f^\ast \colon \Fun(\cX, \cE) \to \Fun(\cY,\cE) \]
	are biadjoints.
\end{lem}

\begin{proof}
	Fix a functor $F \colon \cY \to \cE$.
	Since $f \colon \cY \to \cX$ is a cocartesian fibration, the dual of \cite[\cref*{exodromy-lem:RKE}]{Exodromy_coefficient} provides for every $x \in \cX$ a natural equivalence
	\[ f_!(F)(x) \simeq \colim_{y \in \cY_x} F_y \ . \]
	Since $f$ is a finite \'etale fibration, $\cY_x \coloneqq \cY \times_\cX \{x\}$ is a finite set.
	Thus, since $\cE$ is stable, we deduce
	\[ f_!(F)(x) \simeq \bigoplus_{y \in \cY_x} F_y \ . \]
	Since $f$ is a cartesian fibration as well, \cite[\cref*{exodromy-lem:RKE}]{Exodromy_coefficient} yields
	\[ f_\ast(F)(x) \simeq \lim_{y \in \cY_x} F_y \simeq \bigoplus_{y \in \cY_x} F_y \ . \]
	Thus, $f_!$ and $f_\ast$ canonically agree, whence the conclusion.
\end{proof}

\begin{lem}\label{lem:finite_etale_fibration_universally_conservative}
	Let $f \colon \cY \to \cX$ be a finite étale fibration and let $\cE$ be stable presentable $\infty$-category.
	Then the composition
	\[ \id_{\Fun(\cY,\cE)} \to f^\ast \circ f_! \simeq f^\ast \circ f_\ast \to \id_{\Fun(\cY,\cE)} \]
	is an equivalence.
	In particular, $f_!$ is conservative.
\end{lem}

\begin{proof}
	Write $\alpha$ for the given composition.
	It is enough to prove that for every $x \in \cX$, $j_x^\ast(\alpha)$ is an equivalence in $\Fun(\cY_x, \cE)$.
	Using \cref{cor:induction_specialization_Beck_Chevalley} (applied with $\cA = \cY$, $\cB = \cX$ and $\cY = \{x\}$), we can therefore reduce ourselves to the case where $\cX$ consists of a single point.
	
	\medskip
	
	In this case, $\cY$ is just a set.
	Unraveling the definitions, we see that the unit of $f_! \dashv f^\ast$ evaluated on $F \colon \cY \to \cE$ sends $y \in \cY$ to the canonical inclusion
	\[ i_y \colon F(y) \to \bigoplus_{y' \in \cY_{f(y)}} F(y') \ , \]
	while the counit of $f^\ast \dashv f_\ast$ evaluated on $F$ sends $y \in \cY$ to the canonical projection
	\[ \pi_y \colon \bigoplus_{y' \in \cY_{f(y)}} F_{y'} \to F_y \ , \]
	whence the conclusion.
\end{proof}

The following is the main result concerning finite étale fibrations:

\begin{prop}[Universal monadicity for finite étale fibrations]\label{prop:finite_etale_fibration_monadic}
	Let $f \colon \cY \to \cX$ be a finite étale fibration and let $\cE$ be a stable presentable $\infty$-category.
	For every categorical $\Fun(\cX, \cE)$-module $\cD$, the induced functor
	\[ f_! \otimes \cD \colon \Fun(\cY, \cE) \otimes_{\Fun(\cX, \cE)} \cD \to \cD \]
	is monadic.
\end{prop}

\begin{proof}
	Using \cref{lem:linearity_of_induction}, we see that both $f^\ast \colon \Fun(\cX, \cE) \to \Fun(\cY,\cE)$ and $f_! \colon \Fun(\cY,\cE) \to \Fun(\cX, \cE)$ are $\Fun(\cX,\cE)$-linear.
	Besides, they are biadjoints to each other thanks to \cref{lem:finite_etale_fibration_biadjoint}.
	Therefore, we obtain well defined functors
	\[ f_! \otimes \cD \colon \Fun(\cY, \cE) \otimes_{\Fun(\cX, \cE)} \cD \to \cD \qquad \text{and} \qquad f^\ast \otimes \cD \colon \cD \to \Fun(\cY,\cE) \otimes_{\Fun(\cX,\cE)} \cD \ , \]
	that are still biadjoints to each other.
	Besides, \cref{lem:finite_etale_fibration_universally_conservative} implies that the composition
	\[ \id \to (f_! \otimes \cD) \circ (f^\ast \otimes \cD) \to \id \]
	is an equivalence, so it follows that $f_! \otimes \id_\cD$ is conservative.
	Therefore, it is monadic thanks to Lurie-Barr-Beck's theorem \cite[Theorem 4.7.3.5]{Lurie_Higher_algebra}.
\end{proof}

\begin{cor}\label{cor:relative_tensor_product}
	In the situation of \cref{construction:relative_tensor_of_functor_categories}, assume that $f \colon \cY \to \cX$ is a finite étale fibration.
	Then the comparison functor
	\[ \mu \colon \Fun(\cY, \cE) \otimes_{\Fun(\cX, \cE)} \Fun(\cA, \cE) \to \Fun(\cB, \cE) \]
	is an equivalence.
\end{cor}

\begin{proof}
	Notice that $u \colon \cB \to \cA$ is a finite étale fibration thanks to \cref{lem:finite_etale_fibrations_pullback}.
	Consider the following commutative triangle:
	\[ \begin{tikzcd}[column sep=tiny]
		\Fun(\cY, \cE) \otimes_{\Fun(\cX, \cE)} \Fun(\cA, \cE) \arrow{rr}{\mu} \arrow{dr}[swap]{f_! \otimes \Fun(\cA, \cE)} & & \Fun(\cB, \cE) \arrow{dl}{u_!} \\
		{} & \Fun(\cA, \cE) \ .
	\end{tikzcd} \]
	Using \cref{prop:finite_etale_fibration_monadic}, we see that both diagonal morphisms are monadic.
	To conclude that the horizontal arrow is an equivalence, it is enough by \cite[Corollary 4.7.3.16]{Lurie_Higher_algebra} to check that the Beck-Chevalley transformation
	\[ \mu \circ (f^\ast \otimes \Fun(\cA, \cE)) \to u^\ast \]
	is an equivalence.
	Since $u_!$ is conservative, it is enough to prove that the induced transformation
	\[ (f_! \circ f^\ast) \otimes \Fun(\cA, \cE) \simeq (f_! \otimes \Fun(\cA, \cE)) \circ (f^\ast \otimes \Fun(\cA, \cE)) \to u_! \circ u^\ast \]
	is an equivalence.
	Fix a functor $F \colon \cA \to \cE$ and an object $a \in \cA$.
	Set $x \coloneqq p(a)$ and write $\bbI$ for the tensor unit of $\Fun(\cA, \cE)$ (that is, the constant functor associated to the tensor unit $\bbI_\cE$ of $\cE$).
	Evaluating the source of the above transformation at $F$ and at $a$ yields
	\[ \big( f_! f^\ast(\bbI) \otimes F \big)(a) \simeq \Big( \bigoplus_{y \in \cY_x} I \Big) \otimes F(a) \ , \] 
	while
	\[ \big(u_! u^\ast(F) \big)(a) \simeq \bigoplus_{b \in \cB_a} F(a) \ . \]
	Since the square in \cref{construction:relative_tensor_of_functor_categories} is a pullback, $\cB_a \simeq \cY_{p(a)} \simeq \cY_x$, whence the conclusion.
\end{proof}

\section{Additional properties of cocartesian fibrations}

Finally, we collect some auxiliary results on cocartesian fibrations that are occasionally needed throughout the text.

\subsection{Global vs.\ local full faithfulness}

The following results provides a categorical local-to-global principle to test fully faithfulness:

\begin{prop}\label{lem:full_faithfulness_cocartesian _fibrations}
	Let $\cX$ be an $\infty$-category and let $f \colon \cA \to \cB$ be a morphism in $\PrFibL_\cX$.
	Then:
	\begin{enumerate}\itemsep=0.2cm
		\item $f$ is fully faithful if and only if for every $x \in \cX$ the induced functor $f_x \colon \cA_x \to \cB_x$ is fully faithful;
		
		\item if $f$ is fully faithful, then the same goes for
		\[ \Sigma_\cX(f) \colon \Fun_{/\cX}( \cX, \cA ) \to  \Fun_{/\cX}( \cX, \cB ) \ . \]
	\end{enumerate}
\end{prop}

\begin{proof}
	First we prove (1).
	Write $p \colon \cA \to \cX$ and $q \colon \cB \to \cX$ for the structural maps.
	Fix $a, a' \in \cA$ and set $x \coloneqq p(a)$ and $x' \coloneqq p(a')$.
	The morphism $f$ induces a canonical commutative triangle
	\[ \begin{tikzcd}[column sep=small]
		\Map_\cA( a, a' ) \arrow{rr}{\omega} \arrow{dr} & & \Map_\cB( f(a), f(a') ) \arrow{dl} \\
		& \Map_\cX(x, x') 
	\end{tikzcd} \]
	in $\Spc$.
	Thus, we see that $\omega$ is an equivalence if and only if for every $\gamma \colon x \to x'$ the fiber $\omega_\gamma$ is an equivalence.
	Let $a \to a_\gamma$ be a cocartesian lift of $\gamma$ inside $\cA$.
	Since $f$ preserves cocartesian edges, we see that $f(a) \to f(a_\gamma)$ is cocartesian in $\cB$.
	Thus, \cite[Proposition 2.4.4.2]{HTT} and the above commutative triangle supply a canonical identification of $\omega_\gamma$ with the map
	\[ \Map_{\cA_{x'}}( a_\gamma, a' ) \to  \Map_{\cB_{x'}}( f(a_\gamma), f(a') ) \]
	induced by $f_{x'} \colon \cA_{x'} \to \cB_{x'}$.
	Thus, if $f_{x'}$ is fully faithful, we deduce that $\omega$ is an equivalence.
	As for the converse, it suffices to observe that with the above notations, the square
	\[ \begin{tikzcd}
		\Map_{\cA_{x'}}( a_\gamma, a' ) \arrow{r} \arrow{d} & \Map_{\cB_{x'}}( f(a_\gamma), f(a') ) \arrow{d} \\
		\Map_{\cA}( a, a' ) \arrow{r}{\omega} & \Map_{\cB}( f(a), f(a') )
	\end{tikzcd} \]
	is a pullback.
	Thus, when $\gamma = \id_x$, we see that the full faithfulness of $f$ implies the full faithfulness of $f_x$.
	
	\medskip
	
	We now prove (2).
	Consider the following commutative diagram
	\[ \begin{tikzcd}
		\Fun_{/\cX}(\cX, \cA) \arrow{r} \arrow{d} & \Fun(\cX, \cA) \arrow{r} \arrow{d} & \Fun(\cX, \cX) \arrow[equal]{d} \\
		\Fun_{/\cX}(\cX, \cB) \arrow{r} & \Fun(\cX, \cB) \arrow{r} & \Fun(\cX, \cX) \ ,
	\end{tikzcd} \]
	whose rows are fibers sequences at $\id_\cX \in \Fun(\cX, \cX)$.
	Since fully faithful functors are stable under pullbacks, it suffices to prove that the middle vertical functor is fully faithful.
	This follows immediately from the assumption and from \cite[Proposition 5.1]{Gepner_Lax_colimits} (see also Lemma 5.2 in \emph{loc.\ cit.}).
\end{proof}

\begin{cor}\label{cor:colimit_PrL_fully_faithful}
	Let $\cC_\bullet \colon I \to \PrL$ be a filtered diagram.
	Let
	\[ \cC \coloneqq \colim_{i \in I} \cC_i \]
	be its colimit computed in $\PrL$ and denote $\iota_i \colon \cC_i \to \cC$ for the canonical maps.
	If all the transition maps in $\cC_\bullet$ are fully faithful, the same goes for each $\iota_i$.
\end{cor}

\begin{proof}
	Fix an index $i \in I$.
	Up to replacing $I$ by $I_{i/}$, we can suppose without loss of generality that $i$ is the initial object of $I$.
	Thus, we obtain a transformation $\cC_i \to \cC_\bullet$, where $\cC_i$ is seen as a constant diagram.
	Passing to the cocartesian unstraightenings, we obtain a morphism
	\[ f \colon \cC_i \times I \to  \mathrm{Un}_I(\cC_\bullet) \]
	of cocartesian fibrations over $I$.
	Our assumption implies that this functor is fully faithful fiberwise, and therefore \cref{lem:full_faithfulness_cocartesian _fibrations} guarantees that $f$ is itself fully faithful.
	Notice now that $\cC_i \times I$ and $\mathrm{Un}_I(\cC_\bullet)$ are also cartesian fibrations over $I$ and that
	\[ \cC_i \simeq \lim_{j \in I\op} \cC_i \simeq \Fun_{/I}^{\cart}(I, \cC_i \times I) \qquad \text{and} \qquad \cC \simeq \lim_{j \in I\op} \cC_j \simeq \Fun_{/I}^{\cart}(I, \mathrm{Un}_I(\cC_\bullet)) \ . \]
	Moreover, under these equivalences, $\phi$ induces the functor $\iota_i \colon \cC_i \to \cC$.
	We claim that $\phi$ preserves cartesian edges.
	Assuming this statement, we see that $f$ induces the following commutative diagram:
	\[ \begin{tikzcd}
		\cC_i \arrow{r} \arrow{d}{\iota_i} & \Fun(I, \cC_i) \arrow{d}{\Sigma_I(f)} \\
		\cC \arrow{r} & \Fun_{/I}(I, \mathrm{Un}_I(\cC_\bullet)) \ ,
	\end{tikzcd} \]
	whose horizontal arrows are fully faithful.
	Since $\Sigma_I(f)$ is fully faithful by \cref{lem:full_faithfulness_cocartesian _fibrations}, we conclude that $\iota_i$ is fully faithful as well.
	
	\medskip
	
	We are left to prove the claim.
	Let $j \to \ell$ be a morphism in $I$ and let $f_{j,\ell} \colon \cC_j \to \cC_\ell$ be the induced functor.
	It fits in the following commutative triangle
	\[ \begin{tikzcd}[column sep=small]
		{} & \cC_i \arrow{dl}[swap]{f_j} \arrow{dr}{f_\ell} \\
		\cC_j \arrow{rr}{f_{j,\ell}} & & \cC_\ell \ ,
	\end{tikzcd} \]
	where $f_j$ and $f_\ell$ are the functors induced by $i \to j$ and $i \to \ell$, respectively.
	Write $g_j$, $g_\ell$ and $g_{j,\ell}$ for their right adjoints.
	Unraveling the definitions, we have to check that the Beck-Chevalley transformation
	\[ f_j \to  g_{j,\ell} \circ f_\ell  \]
	is an equivalence.
	However, $f_\ell \simeq f_{j,\ell} \circ f_j$, and the unit $\id_{\cC_j} \to g_{j,\ell} \circ f_{j,\ell}$ is an equivalence because $f_{j,\ell}$ is fully faithful by assumption.
	Thus, the conclusion follows.
\end{proof}

\subsection{Inducing left adjointability from the base}

The following lemma provides a general mechanism to deduce left adjointability involving cocartesian fibrations from the case of trivial fibrations.
It plays an important role in the proof of the spreading out \cite[\cref*{Geometric_Stokes-thm:spreading_out}]{Geometric_Stokes} for Stokes analytic stratified spaces.

\begin{lem}\label{left_adjointability_reduction_trivial_cocart_fibration}
	Consider the commutative cube
	\[ \begin{tikzcd}
		\cD \arrow{rr} \arrow{dr} \arrow{dd} & & \cC \arrow{dr}{j} \arrow{dd}[near end]{p} \\
		{} & \cB \arrow[crossing over]{rr} & & \cA \arrow{dd} \\
		\cT \arrow{rr} \arrow{dr} & & \cZ \arrow{dr}{i} \\
		{} & \cY \arrow{rr} \arrow[leftarrow, crossing over]{uu} & & \cX 
	\end{tikzcd} \]
	whose vertical faces are pull-back diagrams.
	Assume that the vertical arrows are cocartesian fibrations.
	Let $a \in \cC$ and set $x \coloneqq p(a)\in \cZ$.
	Assume that the functor
	$$
	\cT \times_{\cZ} \cZ_{/x} \to  \cY \times_{\cX} \cX_{/i(x)}
	$$	
	is cofinal. 
	Then, the functor	
	$$
	\cD \times_{\cC} \cC_{/a} \to  \cB \times_{\cA} \cA_{/j(a)}
	$$		
	is cofinal.
\end{lem}
\begin{proof}
	Since the vertical faces of the above cube are pull-back, the following square  
	$$
	\begin{tikzcd}
		\cD \times_{\cC} \cC_{/a}  \arrow{r} \arrow{d}  &\cB \times_{\cA} \cA_{/j(a)}  \arrow{d}     \\
		\cT \times_{\cZ} \cZ_{/x}  \arrow{r}   &\cY \times_{\cX} \cX_{/i(x)}  
	\end{tikzcd} 
	$$
	is a pull-back.
	From \cite[2.4.3.2]{HTT}, its vertical arrows are cocartesian fibrations.
	Since cocartesian fibrations are smooth \cite[4.1.2.15]{HTT} and since the pull-back along a smooth map preserves cofinality \cite[4.1.2.10]{HTT}, \cref{left_adjointability_reduction_trivial_cocart_fibration} thus follows.
\end{proof}

%
%
%




\end{appendices}

\bibliographystyle{plain}
\bibliography{dahema}

\end{document}